\documentclass[9pt, leqno, fleqn]{amsart}

\usepackage[normalem]{ulem}
\usepackage[titletoc]{appendix}
\usepackage{lineno}

\usepackage{graphicx, psfrag, epstopdf}

\usepackage{hyperref}
\usepackage{comment}

\usepackage{dsfont}
\usepackage{amsmath}
\usepackage[fleqn]{mathtools}
\usepackage{amssymb}
\usepackage{graphicx}

\usepackage{amssymb}

\usepackage[all,cmtip]{xy}

\usepackage{amsmath}

\usepackage{amsfonts}

\usepackage{soul}

\usepackage{mathpazo}
\usepackage{color}
\usepackage{yfonts}

\usepackage{paralist}

\usepackage{stmaryrd}

\usepackage{amsxtra}

\def\Id{{\rm Id}}

\def\cU{\mathfrak U}

\newcommand{\setR}{\mathsf{R} }

\newcommand{\norm}[1]{\| #1\|}

\def\cI{          \mathcal I}

\def\cJ{          \mathcal J}

\def\cA{         A}

\def\cB{          \mathcal B}

\def\cD{          \mathcal D}

\def\cN{          \mathcal N}

\def\cW{         W}

\def\greg{ {\bf s} }

\def \Cent{  {\rm Cent} }

\def \R{{\mathbb R}}

\def \Z{{\mathbb Z}}
\def \N{{\mathbb N}}

\def\Trunc{ \mathfrak{R}  }

\def\BV{ \overline{D}_V }

\def\DV{ D_V }

\def\DZ{  D_{\tz} }

\newcommand{\T}{{\mathbb T}}
\newcommand{\prf}{{\begin{proof}}}
	\newcommand{\epf}{{\end{proof}}}

\newcommand{\bs}{{\mathbf s}}

\newcommand{\cE}{{\mathcal E}}

\newcommand{\cT}{{\mathcal T}}

\newcommand{\cP}{{\mathcal P}}

\newcommand{\cK}{{\mathcal K}}

\newcommand{\C}{{\mathbb C}}

\newcommand{\cQ}{{\mathcal Q}}

\newcommand{\tpl}{{T}}

\newcommand{\ary}{\begin{eqnarray}}
	\newcommand{\eary}{\end{eqnarray}}

\newcommand{\aryst}{\begin{eqnarray*}}
	\newcommand{\earyst}{\end{eqnarray*}}

\newcommand{\enmt}{\begin{enumerate}}
	\newcommand{\eenmt}{\end{enumerate}}

\usepackage{mathptmx}
\usepackage{bm}

\newtheorem{thm}{\bf Theorem}[section]

\newtheorem{lemma}[thm]{\bf Lemma}
\newtheorem{prop}[thm]{\bf Proposition}
\newtheorem{defi}[thm]{\bf Definition}

\newtheorem{cor}[thm]{\bf Corollary}

\theoremstyle{definition}

\newtheorem{rema}[thm]{\bf Remark}

\DeclareMathOperator{\diff}{Diff}

\theoremstyle{definition}

\def\bee{\begin{equation}}
	\def\eee{\end{equation}}

\newif\ifA
\newcommand\detail[1]{\ifA {\color{blue} \textbf{Details:} \\ #1}\fi}

\def\deltazero{  {  \hat\lambda_{g} }}

\def\sigmaNI{{\sigma_{g}}}

\def\deltasharp{  { \delta_{g} }}

\def\CNI{{C_{g}}}

\newcommand{\tinitial}{t_{\sharp}}

\newcommand{\omegalow}{\eta_{*}}
 
\def\esmall{{\epsilon_0}}

\def\Anosovflow{\mathfrak{U}}
\def\stabletransitive{\mathfrak{U}_0}

\def\adscale{    \mathfrak r}

\def\Barg{          \mathfrak B}

\def\CC{    \mathfrak C}

\def\midtor{  {\tor^{\flat}}}

\def\tor{  {\rm Tor}}

\def\flu{  {\rm Flu}}

\def\hDelta{  \bm{\Delta}     }

\def\bI{   \mathbb{I}}

\def\supp{{\rm supp}}

\def\Id{{\rm Id}}

\def\cJ{\mathcal{J}}

\def\bD{\mathbb{D}}

\def\hL{\mathbb{L}}

\def\hP{\mathbb{P}}

\def\hK{\mathbb{K}}

\def\PolyU{\mathrm{Poly_{U}}}

\def\PolyW{\mathrm{Poly_{W}}}

\def\Aff{\mathrm{Aff}}

\def\hK{\mathbb{K}}

\def\ker{{\rm Ker}}

\def\hH{\mathbb{H}}

\def\cL{\mathcal{L}}

\def\Bcone{ {\rm CONE} }

\def\Scone{ {\rm cone} }

\def\tq{{\tilde q}}
\def\tw{{\tilde w}}
\def\tz{{\tilde z}}
\def\tx{{\tilde x}}
\def\ty{{\tilde y}}

\def\hq{{\widehat q}}
\def\hw{{\widehat w}}
\def\hz{{\widehat z}}

\def\frq{ \mathfrak{q} }

\def\fp{ \texttt{p} } 
\def\fq{ \texttt{q} } 
\def\fw{ \texttt{w} }
\def\fz{ \texttt{z} }
\def\fx{ \texttt{x} }
\def\fy{ \texttt{y} }

\def\NGB{     U }

\def\ngb{          \Omega }

\def\ponM{{\fp} }

\def\estar{   e^*   }

\def\cO{  { \mathcal{O} } }

\def\twist{  { tor }}

\def \lambdag{  { \lambda_g } }

\def \extra{    \gamma_{\sharp}   }

\newcommand{\Diff}{{\rm Diff}}

		\newcommand{\pdvr}[2]
		{\dfrac{\partial^{#2} #1}{\partial \theta^{#2_1} \partial r^{#2_2}}}
		\newcommand{\pdvrs}[2]
		{\partial^{#2} #1 /\partial \theta^{#2_1} \partial r^{#2_2}}
		\newtheorem{question}{\sc Question}

		\definecolor{blue}{rgb}{0,0,1}
		\definecolor{red}{rgb}{1,0,0}
		\definecolor{green}{rgb}{0,.7,0}

		\DeclareFontFamily{U}{solomos}{}
		\DeclareErrorFont{U}{solomos}{m}{n}{10}
		\DeclareFontShape{U}{solomos}{m}{n}{
			<-> s*[1.1]  gsolomos8r
		}{}

\begin{document}

			\title[Perturbation of the time-1 map of a generic volume-preserving $3$-dimensional Anosov flow]{Perturbation of the time-1 map of a generic volume-preserving $3$-dimensional Anosov flow}
			

			\author{Masato Tsujii}
			\address{Masato Tsujii, Department of Mathematics, Kyushu University, Fukuoka, 819-0395, Japan}
			\email{tsujii@math.kyushu-u.ac.jp}
			
			\author{Zhiyuan Zhang}
			
			\address{Zhiyuan Zhang,
				Department of Mathematics, 
				Imperial College London, 
				London SW7 2AZ, 
				United Kingdom
			}
			\email{zhiyuan.zhang@imperial.ac.uk} 		
			\maketitle
			
			\begin{abstract}
				Let $\greg > 1$ be a large integer, and let $f$ be a diffeomorphism  sufficiently close in the $C^{\greg}$-topology to the time-1 map of a $C^{\greg}$ generic volume-preserving Anosov flow on a $3$-dimensional compact manifold.   We show that for any probability measure  $\mu$ with smooth density, $f^n_* \mu$ converges exponentially fast to a common limit measure with full support.
			    As corollaries, we show the following:
				\enmt
				\item $f$ is   topologically mixing, 
				\item  $f$ has a unique physical measure with basin of full Lebesgue measure, which is also the unique u-Gibbs state,
				\item if $f$ is volume preserving, then $f$ is exponentially mixing with respect to the volume form.
				\eenmt
				As applications, we give a class of time-1 maps of transitive Anosov flows non-approximable in $C^{\greg}$ by Axiom A maps, giving  negative answer to a question of Palis-Pugh (1974); the first example of a $C^{\bs}$-stably transitive time-1 map of Anosov flow, a question mentioned in Bonatti-Guelman (2010), Rodriguez Hertz (2010);  as well as the first example of a $C^{\bs}$-stably transitive diffeomorphism without periodic points.
			\end{abstract}
			
			\setcounter{tocdepth}{1}
			\tableofcontents

			\section{Introduction} \label{sec Introduction}

A diffeomorphism is said to be {\it transitive} if it has a dense orbit, and is said to be {\it stably transitive} (also known as persistently transitive, robustly transitive in the literatures) if this property holds for all of its small perturbations.
The study of stably transitive diffeomorphisms can be traced back to  the study of Anosov, Smale on structural stability (see \cite{An, Smale}).
A transitive Anosov diffeomorphism is always structurally stable, and as a result any transitive Anosov diffeomorphism is stably transitive. 
 
The first $C^1$-perisistent, non-Anosov, transitive diffeomorphism is constructed by Shub in 1968 (see \cite{Shu}  and \cite{HPS}).
Shub's example is  a skew product of {\it derived-from-Anosov} (DA-diffeomorphism) on $\T^4$.
In 1978,  Ma\~n\'e constructed a $C^1$-stably transitive DA diffeomorphism on $\T^3$ \cite{Man}.    These maps are homotopic to Anosov diffeomorphisms.
The existence of stably transitive diffeomorphisms isotopic to the identity was an open question until it was resolved by Bonatti and Díaz \cite{BD}  in 1995.
They proved that certain perturbation of the time-1 maps of Anosov flow is $C^1$-stably transitive.    They have introduced a dynamical model called \lq\lq blender\rq\rq, which is a special type of hyperbolic set.
The notion of blender has later found a wide range of applications, some of which goes well beyond the scope of the study of stable transitivity (for instances, see  \cite{ACW, Ber,  LT}).
 There are some other examples of non-uniformly hyperbolic stably transitive diffeomorphisms, for instance \cite{Pot, BV, CO}.

All of the above examples of stably transitive diffeomorphisms contain infinitely many  periodic orbits. 
The following question appears naturally:
\begin{question} \label{quest 1}
	Is there a stably transitive diffeomorphism without any periodic points ? 
\end{question}

A  place to look for answer is among the time-one maps of transitive Anosov flow, as studied in \cite{BD}. In fact, Question \ref{quest 1} is closely related to the following question mentioned in \cite{BG}:
\begin{question} \label{quest 2}
	Let $X$ be a transitive Anosov vector field that is not conjugate to a suspension, and let $f$ be the time-one map of its flow. Is $f$ robustly transitive ?  (The answer may depend on $X$)
\end{question}
As stated in \cite{BG}, 
{\it \lq\lq A positive answer to this question would provide an initial example of a robustly
	transitive diffeomorphism without periodic points.\rq\rq }
	More precisely, the existence of a single $X$ giving positive answer to Question \ref{quest 2} would already provide a positive answer to Question \ref{quest 1}.
Question \ref{quest 2} is also related to deciding whether the time-one map of the geodesic flow on a negatively curved surface can be robustly transitive,  an old question which appeared in \cite{GPS}, and reiterated in   \cite[Problem 4]{RHRHU} and \cite[Problem 1]{Wil}.
Related question about whether there can be stably transitive non-Anosov affine diffeomorphism can be found in \cite{HPS}, \cite{GPS}, \cite{RHRHU} and \cite{RH}. 
However, to our best knowledge, even the existence of a single robustly transitive partially hyperbolic diffeomorphism with isometric central direction is unknown (see \cite[Section 3]{RH})  before the present paper.

The following question of Palis and Pugh back in 1974 is also closely related to Question \ref{quest 2}.
\begin{question} \cite[Problem 20]{PP} \label{quest 3}
	Can the time-one map of an Anosov flow be approximated by an Axiom A diffeomorphism ?
\end{question}
As already pointed out in \cite{PP},  the answer to Question \ref{quest 3} is Yes  if the flow is a suspension of an Anosov diffeomorphism (however, this condition only holds for a nowhere dense set of flows). 
In \cite{BG}, Bonatti and Guelman have introduced a deformation method for approximation by Axiom A maps. Based on the strategy put forth in \cite{BG}, Shi has shown in his thesis \cite{Shi, Shi2} that every partially hyperbolic automorphism on non-abelian $\T^3$ can be $C^1$ approximated by structurally stable Axiom A maps.
In \cite{BG2}, the authors have found some constraint to a positive answer to Question \ref{quest 3}.
We refer the reader to \cite[Sections 1.9-1.10]{BG} for progresses and attempts (in both positive and negative directions) towards Question \ref{quest 3}. Despite of these progresses, as far as we know there was not yet an answer to Question \ref{quest 3} prior to our paper.

In this paper, we give a positive answer to Question \ref{quest 1}, and a positive answer to Question \ref{quest 2} for a generic volume preserving Anosov vector field $X$ in $3D$, albeit the topology considered in our theorem is the $C^{\greg}$-topology for a large integer $\greg > 1$. 
 In fact, our diffeomorphisms are  stably  topologically mixing.\footnote{A map $f$ is topologically mixing if for any open sets $U, V$,  $f^{n}(U) \cap V \neq \emptyset$ holds for all sufficiently large $n$. Topological mixing implies   transitivity.}

\begin{thm}  \label{thm stable transitivity}
	There is an integer $\greg > 0$ such that the following is true.
	Let $g$ be a volume preserving $3$-dimensional Anosov flow satisfying a $C^{\greg}$-open and $C^{\infty}$-dense condition. Then there exists an open neighborhood $\cU$ of $g^1$ in $\Diff^{\greg}(M)$ such that evey $f \in \cU$ is  topologically mixing  (in particular, $g^1$ is stably transitive in $\Diff^{\greg}(M)$).
\end{thm}

We immediately obtain as a corollary the following partial answer to Question \ref{quest 3}\footnote{This proof is communicated to us by Jiagang Yang.}. As far as we know, this is the first negative answer to Question \ref{quest 3}, in any topology.

\begin{cor}  \label{cor 1}
	Let $\greg$ be given by  Theorme \ref{thm stable transitivity}, and 
	let $g$ be a volume preserving Anosov flow given in Theorme \ref{thm stable transitivity} on a $3$-dimensional manifold different from $\T^3$. Then $g^1$ cannot be approximated in the $C^{\greg}$-topology by Axiom A maps.	
\end{cor}
\begin{proof}
	Indeed, a transitive Axiom A map is an Anosov diffeomorphism, and it is known that $\T^3$ is the only $3$-dimensional manifold which supports an Anosov diffeomorphism (see \cite{Fra, New}).
\end{proof}

Theorem \ref{thm stable transitivity} follows quickly from our main result, Theorem \ref{thm main}, which is about the speed of convergence to equilibrium.
Somewhat surprisingly, although Questions \ref{quest 1} to \ref{quest 3} are of topological nature, our method for proving  Theorem \ref{thm main} is purely analytical.

\subsection{Main results}

A $C^1$ flow $g = \{  g^t: M \to M \}_{t \in \R}$ on a compact Riemannian manifold $M$ is  \underline{{\em Anosov}} if there is a continuous flow-invariant splitting
$TM = E^s \oplus N_g \oplus E^u$ such that $N_g$ integrates to the flow-line foliation $\cN_g$, and  there exist $C, \lambda > 0$ such that for any $t > 0$
\aryst
\norm{Dg^t|_{E^s}} < Ce^{-\lambda t} \ \mbox{ and } \
\norm{Dg^{-t}|_{E^u}} < Ce^{-\lambda t}.
\earyst
We denote by $\Anosovflow(M)$ the set of $C^{\infty}$,  volume-preserving Anosov flows on a compact manifold $M$.

In the rest of the paper, we fix a 3-dimensional boundaryless compact Riemannian manifold $M$, equipped with a normalized volume measure, denoted by $Leb$, i.e., $Leb(M) = 1$.
Given a diffeomorphism $f \in \diff^{\infty}(M)$,  we denote by $\cL_f: C^\infty(M) \to C^\infty(M)$ the {\it transfer operator} of $f$ with respect to $Leb$:
\ary \label{eq lowerboundtransferop}
\cL_f u(p) = u(f^{-1}(p)) / (\det Df)(f^{-1}(p)). 
\eary
By definition, we have the following equation:
\ary \label{main identitetransferop}
\int \cL_f^n u \cdot v dLeb =  \int u \cdot v\circ f^n  dLeb, \quad n \geq 1.
\eary
Our main result in this paper is the following.
\begin{thm}\label{thm main}
		There exist  integers $\greg, r > 0$,
	and a $C^{\greg}$-open and $C^\infty$-dense subset $\stabletransitive \subset \Anosovflow(M)$ such that for any $g \in \stabletransitive$, there exist $\kappa_g  > 0$ and an open neighborhood $\mathfrak{V}_g$ of $g^1$ in $\Diff^{\greg}(M)$ such that for any $f \in \mathfrak{V}_g$, there is an $f$-invariant measure $\nu_f$ such that for any $u, v \in C^\infty(M)$, we have 
	\aryst
	\Big	|\int u \cdot v \circ f^n dLeb - \int u d Leb \int v d \nu_f \Big | <   C(f) e^{-n \kappa_g} \| u \|_{C^r}  \| v \|_{C^r}, \quad n \geq 0.
	\earyst
\end{thm}

In other words, we have shown that the push-forwards of a given probability measure on $M$ with smooth density by iterates of $f$ will converge exponentially fast to a limit measure.  
We will deduce Theorem \ref{thm main} from a more precise result,  Theorem \ref{lem main}, in Section \ref{sec Anisotropic Sobolev space}.

The exponential convergence presented in Theorem \ref{thm main} corresponds precisely to \lq\lq{\it volume is almost exponentially mixing}\rq\rq, a notion studied in \cite{Mal} and \cite{BORH}. 
Under this exponential convergence,  Ben Ovadia and Rodriguez-Hertz proved in  \cite{BORH}   that the limit measure $\nu_f$ is a SRB measure, i.e., the measure entropy $h_{\nu_f}(f)$ equals to the sum of positive Lyapunov exponents of $\nu_f$; and Maldonado proved  in \cite{Mal} that the measurable transformation $(M, \nu_f, f)$ is Bernoulli.

For $C^2$ volume preserving partially hyperbolic diffeomorphisms, by the work of Burns-Wilkinson \cite{BW} on  Pugh-Shub's Stable Ergodicity conjecture, under the center bunching condition, {accessibility} implies K-property, in particular, mixing. However, proving quantitative mixing for partially hyperbolic diffeomorphisms is a challenging problem (see \cite{Wil}, \cite[Conjecture 1.8]{DWD}). 
There are   results on quantitative mixing for partially hyperbolic maps under various assumptions (see \cite{CL,   Dol4}), but they are not satisfied by all maps in our region.
 An immediate consequence of Theorem \ref{thm main} is the following new result on quantitative mixing.  
\begin{thm} \label{thm exp}
	Let   $\greg > 0$ and  a generic volume preserving 3D Anosov flow $g$  be given by Theorem \ref{thm main}. Then
	any  volume preserving partially hyperbolic diffeomorphism $f$ that is sufficiently close to $g^1$ in $C^{\greg}$ is exponentially mixing with respect to the volume. 
\end{thm}

In the following subsections, we will use Theorem \ref{thm main}  to present a new mechanism of stable transitive and a new mechanism for the unique existence of physical measure and u-Gibbs state beyond the conservative setting. 

\subsection{Stable transitivity}

We can deduce Theorem \ref{thm stable transitivity} from Theorem \ref{thm main} by the following simple argument. We mention that functional-analytic approach to stable transitivity for endomorphisms was previously used in \cite{Zha}.

\begin{proof}[Proof of Theorem \ref{thm stable transitivity}]
	Let $\cU \subset \Diff^\infty(M)$ be an open neighborhood of $g^1$.  We take an arbitrary $f \in \cU$.
	We have the following.
	\begin{lemma} \label{lem nufhasfullsupport}
		By letting $\cU$ be sufficiently small, the  measure $\nu_f$ in Theorem \ref{thm main} has full support. 
	\end{lemma}
	\begin{proof}
		By Theorem \ref{thm main}, \eqref{main identitetransferop} and by letting $\cU$ be sufficiently small, we have  for any non-negative non-zero function $v \in C^\infty(M)$,
		\ary \label{eq expsmalldifference}
		|\int \cL_f^n 1\cdot v dLeb - \int v d\nu_f| < C_{g, v} e^{-n \kappa_g}, \quad n \geq 0,
		\eary
		for   $\kappa_g > 0$ given by Theorem \ref{thm main}, and a constant $C_{g, v}  > 0$.
		
		Since $g$ is volume preserving, by letting $\cU$ be sufficiently small, depending only on $g$ and $\kappa_g$, we may assume that $e^{-\frac{1}{2}\kappa_g} \leq |\det f| \leq e^{\frac{1}{2}\kappa_g}$.
		By \eqref{eq lowerboundtransferop} we have	$\cL_f^n 1 \geq e^{- \frac{1}{2} n \kappa_g}$,
		and consequently,
		\aryst
		\int \cL_f^n 1\cdot v dLeb \geq e^{- \frac{1}{2} n \kappa_g} \int v dLeb.
		\earyst
		By \eqref{eq expsmalldifference} and letting $n$ be sufficiently large, we obtain  $\int v d\nu_f  > 0$.
		Since $v$ is arbitrary, we deduce  that $\nu_f$ has full support. We have completed the proof.
	\end{proof} 
Now take two arbitrary open sets $U, V \subset M$, and smooth bump functions $\varphi_U, \varphi_V$ supported in $U$, $V$ respectively.  Then  Theorem \ref{thm main} implies that $\int \varphi_U \cdot \varphi_V \circ f^n dLeb > 0$ for all sufficiently large $n$. This implies topological mixing, and thus completes the proof.
\end{proof}

\subsection{Physical measures and u-Gibbs states}
The following notion is well-known (see for instance \cite{Pal}).

\begin{defi} \label{def physicalmeasure}
	Given a diffeomorphism $f \in C^1(M)$ and an $f$-invariant measure $\nu$, we define the \underline{{\em basin}} of $\nu$ by $\cB_f(\nu) = \{ p \in M \mid n^{-1} \sum_{i=0}^{n-1} \delta_{f^{i}(p)} \mbox{ tends to $\nu$ in the weak star topology} \}$. We say that an $f$-invariant measure $\nu$ is a \underline{{\em physical measure}} for $f$ if $\cB_f(\nu)$ has positive Lebesgue measure.
\end{defi}

A closely related notion is following.

\begin{defi} \label{def phdandugibbs}
	An $f$-invariant measure $\nu$ is an \underline{{\em u-Gibbs state}}  if the conditional measure of $\nu$ along the unstable foliation of $f$ is absolutely continuous.
\end{defi}

For partially hyperbolic diffeomorphisms, physical measure may not exist  (see \cite{CYZ}).  However, there is always at least one u-Gibbs state. 
When the u-Gibbs state is unique, then such measure is necessarily a physical measure. In general,  an u-Gibbs state with negative center Lyapunov exponents is a physical measure.  Recently, Katz proved in \cite{Ka} that certain generalized Gibbs states  are necessarily the unique physical measure for Anosov diffeomorphism,
using a measure-rigidity method originated in the seminal work of Eskin-Mirzakhani about measure classification in Teichm\"uller dynamics.
In particular, his result implies that in dimension $3$,  any u-Gibbs state with positive center Lyapunov exponent is a physical measure provided some rather weak non-integrability condition  in \cite{EPZ} holds.   
In the case of almost neutral center Lyapunov exponent, Bortolotti proved in  \cite{Bor} the existence of physical measures assuming the stable foliaton is Lipschitz and some transversality condition as in \cite{Tsu}.  However, the Lipschitz assumption in \cite{Bor} does not hold in general, even within volume preserving maps. In this paper, we prove the following new result on physical measure and u-Gibbs state.

\begin{thm}  \label{thm uniqueustate}
	We may choose $\stabletransitive$ in Theorem \ref{thm main} to satisfy the following property:  for any $g \in \stabletransitive$,
	there exists an open neighborhood $\cU$ of $g^1$ in $\Diff^{\infty}(M)$  such that for $f \in \cU$, the measure $\nu_f$ in Theorem \ref{thm main} is the unique u-Gibbs measure of $f$ (hence $\nu_f$ is ergodic), whose basin is of full Lebesgue measure. Moreover, $\nu_f$ is a SRB measure, i.e., $h_{\nu_f}(f)$ is the sum of $\nu_f$'s non-negative Lyapunov exponents, and that $(M, \nu_f, f)$ is Bernoulli.
\end{thm}

We will provide the proof of Theorem \ref{thm uniqueustate} in Section \ref{subsec Quasi-compactness operator and uniform spectral gap}.   
We mention that in a similar setting,  Dolgopyat proved in \cite{Dol2} that along a generic one-parameter deformation of a time-1 map of a geodesic flow on a negatively curved surface, all sufficiently small perturbations admit a unique physical measure with non-zero Lyapunov center exponents. Our result is different in the sense that we need some generic condition for the initial Anosov flow, but our result gives a genuine open neighborhood of the time-one map.

\subsection{General approach and difficulties} \label{subsec DandS}

An immediate corollary of
Theorem \ref{thm main} is that  $g^t$  is stably exponentially mixing among the flows. It is now understood, due to the seminal works \cite{Che, Dol, L}, that exponential mixing is caused by the uniform  jointly non-integrability of the stable and unstable foliations. One way to prove this corollary is in \cite{TZ} which uses Dolgopyat's method based on the {\it Markov partitions} and Laplace transform. However it seems difficult to apply such method to the perturbed map as it is unclear how we find Markov partition for a partially hyperbolic diffeomorphism, and perhaps a more fundamental obstacle is that such method    requires a good control of arbitrarily high iterates of the map, making it unsuitable for studying perturbations.

Another proof of the corollary is given in \cite{Tsu},  based on the analysis of the transfer operator $\cL_f$ on certain {\it Anisotropic Banach Space} $\hH_f$ (this space  is in fact a Hilbert space).
This approach can be traced back to the original works of \cite{L, GL, AGT, BT} on the decay of correlations of Anosov diffeomorphisms and hyperbolic flows without using the  Markov partitions. Other important applications of this approach include \cite{GLP, BDL, FT}.
Such method is more suitable for studying perturbations (see for instance \cite{BL, GL, Zha}).

Still, it is not clear at all that in the general setting the uniform non-integrability condition, which depends on all small scales, would work for an open set of maps.
Indeed, this is a reason that we work with $3$-dimensional flows. It was first observed in \cite{Tsu} and later generalized in \cite{TZ} that in dimension $3$, the uniform non-integrability condition can be described by \lq\lq template functions\rq\rq, which is a family of functions satisfying certain equivariance condition (for other recent applications using templates, see \cite{EPZ, GLRH}). 
The method used in \cite{Tsu} requires only some condition on the template functions at a compact set of scales. 
In the present paper, we will see that this argument, after much effort, can be generalized to nearby diffeomorphisms.

Our construction of the Hilbert space here is closely linked to the one in \cite{Tsu}, which is a continuation of the ideas developed in a series of works  \cite{ FT, Tsu2, Tsu3, Tsu4}.
However, two new difficulties appear in the study of perturbed maps:

1. All the previous works about Anosov flows use at least the smoothness of the flow-line foliation, as we often need to represent a  function as a superposition of functions having comparable frequencies along the flow line. 
This appears to be a problem since for the perturbed map, the center foliation is usually only H\"older, and it may even be {\it pathological}, i.e., the volume form can disintegrate to atomic measures along the center leaves (see \cite{AVW}).

2. For Anosov flow, the dynamics along the center foliation is isometric. In all of the previous works, this translation structure is used in a crucial way.
However for a perturbed map, the center dynamics are usually nonlinear, mixed with expansions and contractions.  In \cite{FT}, the authors can consider wave-packets mostly concentrated within a  $|\omega|^{- \alpha^{\perp}}$-ball along the hyperbolic directions and within a $|\omega|^{ - \alpha^{\parallel}}$-interval along the center direction, where $\omega$ denotes the frequency along the flow, and $0 \leq \alpha^{\parallel}  < \alpha^{\perp} < 1$, $\frac{1}{2} \leq \alpha^{\perp} < 1$ (see \cite[(2.5)]{FT}). Unfortunately, this choice is in general not enough to exploit the non-integrability which may be useful only at scale $|\omega|^{-1/2 \pm \epsilon}$,  requiring $\alpha^{\perp} < 1/2$. In this case, it is also not enough for a linear approximation of the center dynamics since we would need $\alpha^{\parallel} > 1/2$ for that.

\subsection{Outline of the proof} \label{subsec outline}

To resolve these difficulties, we introduce a family of coordinates called {\it Normal Central Charts}.  Each coordinate system parameterizes an open set of the manifold by a domain in $\R^3 = \R^2_w \oplus \R_z$; and the $z$-foliations of all these coordinate systems are mapped to a common smooth foliation close to the center foliation\footnote{For $f$ close to the time-1 map of an Anosov flow, we can take this common smooth foliation as the flow-line foliation. In this case, Normal Central Charts are also flow charts. We introduce this extra common smooth foliation since this construction works for general partially hyperbolic diffeomorphisms, for which we expect to find more applications.}. Moreover, we require that the transformation between two such coordinate systems is a translation along the $z$-coordinate; and the dual central bundle $E^{c}_* = (E^s \oplus E^u)^{\perp}$ written in a Normal Center Chart takes a similar form as in the {\it Normal Coordinate System} in \cite{TZ}.	 This helps us to understand the small-scale twisting of $E^s$, resp. $E^u$, along $W^u$, resp. $W^s$. Then we will build upon Normal Central Charts a {\it Dynamical Wave-Packet Transformation} $\Barg$, expressing each smooth function on $M$ as a superposition of localized functions, indexed by certain open neighborhood $\Gamma$ of the pseudo-diagonal of $\R^{3 + 3} \times \R^3$  (namely, $(p, Q, p)$ with $p, Q \in \R^3$).    Instead of studying $\cL_f$, we take a large integer $\tinitial$ and lift $\cL^{\tinitial}_f$ by $\Barg$ to an operator $\hL_f$ on $\Gamma$. Then we choose  carefully some weight function on $\Gamma$ so that $\hL_f$ becomes quasi-compact with respect to the weighted norm (see Theorem \ref{lem main}).  Once this is established, Theorem \ref{thm main} then follows from a standard argument.

Let us describe the wave-packets in more details.  
A wave-packet $\varphi$ is represented by some $(q, P, \fq) \in \Gamma$ where $q, \fq \in \R^3$ and $P = (\xi, \eta) \in \R^{2+1}$. We use   $\fq$   to record the base point of a chart, mapping the $z$-line and $w$-plane  respectively to almost center direction and almost hyperbolic directions at $\fq$ (this is necessary because $E^u$, $E^s$ depend sensitively on points).  In the chart based at $\fq$,  $\varphi$ is centered at $q$, with frequency along $z$-line centered near $\eta$, and with frequency along $w$-plane centered near $\xi$. There is a dynamically-interesting scale $\adscale(  \langle \eta \rangle, q) \sim \langle \eta \rangle^{- 1/2 \pm \epsilon}$ so that  $\varphi$ is mostly concentrated within a $\adscale(\langle \eta \rangle, q) $-ball along the $w$-plane, and within a $\langle \eta \rangle^{- \beta_1}$-interval along the $z$-line for some $\beta_1 \in (0,1)$ close to $1$ (see \eqref{def beta1}).  
We will choose $\adscale( \langle \eta \rangle, q)$ so that  the oscillation of $E^{c}_*$ within a $\adscale(\langle \eta \rangle, q)$ neighborhood of $q$ is of order $\Delta(\langle \eta \rangle, q)\adscale(\langle \eta \rangle, q) \ll \langle \eta \rangle^{- \beta_1}$.  A crucial observation is that both   $\Delta(\langle \eta \rangle, q)$  and $\adscale(\langle \eta \rangle,q)$ depend mildly on $q$  within a $C \adscale(\langle \eta \rangle, q)$-ball. This last point roughly corresponds to the {\it temperate property} of H\"ormander \cite[Definition 18.5.1]{Hor} in microlocal analysis.  

We now give an overview of the proof of the quasi-compactness of $\hL_f$.
We will estimate the correlation $\langle u, \hL_f v \rangle$ for $u$, $v$ with localized supports in $\Gamma$. The estimate is divided into three cases (as usual, we may assume that at least one of the supports of $u, v$ is outside of a compact set):

\noindent{$\bullet$} When $\supp(v)$ is transported by the dynamic to be separated from $\supp(u)$, the correlation is very small (see Proposition \ref{prop offdiagonal}),

\noindent{$\bullet$} When $\supp(v)$ is transported by the dynamic to be overlapped with $\supp(u)$, and that at least one of the supports is away from  the {\it dual center locus} which is related to $E^c_*$ (see  Section \ref{sec distanceonphasespaceNo1}), we may use the weight function to exploit hyperbolicity (see Proposition \ref{lem contractionweight}),

\noindent{$\bullet$} The main difficulty is found when both supports are close to the dual center locus.  Unlike in the previous cases, where rather loose estimates suffice, here we need to exploit a strong form of uniform non-integrability condition $({\bf NI})_{\sigma, C}$ in \cite{Tsu} to get cancellation (see Proposition \ref{lem formerL6.8}).
%
%

			\subsection{Plan of the paper}
			
			After we review some basic notions and results   in Section \ref{sec Basic properties}, we will introduce in Section \ref{sec Templates functions and Adapted scales} some dynamical quantities which are  important  for our argument. In Section \ref{sec The dual  central bundle in local chart}, we will introduce Normal Central Charts, and then we will use them in Section \ref{sec DWPT: Construction and Properties} to define Dynamical Wave-Packet Transform. In Section \ref{sec Anisotropic Sobolev space}, we define the Anisotropic Hilbert space and state  the main estimate Theorem \ref{lem main} from which we deduce  Theorems \ref{thm main} and  \ref{thm uniqueustate}. In Section \ref{sec distanceonphasespace}, we break the proof of Theorem \ref{lem main} into several lemmata, whose proofs   are given in Sections \ref{sec distanceonphasespace} - \ref{sec lem formerL6.8}.
			
			The paper is mostly self-contained: the only recent results used without proofs are \cite[Prop 3.21]{EPZ} and \cite[Theorem 2.15]{Tsu}.
			
							\subsection*{Acknowledgement}
			We are grateful to Jiagang Yang for telling us the connection between Question \ref{quest 2} and Question \ref{quest 3}; to Sylvain Crovisier for telling us the results in \cite{CYZ}; to  Amadeus Maldonado and Aaron Brown  for telling us the results in \cite{Mal};  to  Snir Ben Ovadia for telling us the results in \cite{BORH};  to Yi Shi for explaining the results in \cite{BG, Shi}; and to  Artur Avila, Dmitry Dolgopyat, Rafael Potrie, Carlangelo Liverani  and Federico Rodriguez Hertz for related general discussions.

			\subsection*{Notation}
			We denote $\N = \{0, 1, 2, \cdots \}$.
			Throughout this paper, we let $M$ be a $3$-dimensional compact, connected Riemannian manifold,
			and let $g$ denote an element of $\Anosovflow(M)$, the set of  $C^{\infty}$,    volume-preserving Anosov flows on  $M$.  For every $r \in \Z_{+} \cup\{\infty\}$, we let $\Diff^{r}(M)$ denote the $C^r$ diffeomorphism group of $M$ equipped with Whitney's $C^r$ topology, induced by   $d_{\diff^r(M)}$.

	 We denote by $C$ a generic large positive constant, which may vary from line to line, but ultimately depend only on the $g$.   
			Given two non-negative real constants $A, B$, we write $A \lesssim B$ or $A = \cO(B)$ if there is a constant $C' > 1$ depending only on  $g$ such that $ A/ B \leq C'$. We write $A \sim B$ if we have $A \lesssim B$ and $B \lesssim A$. 
 More generally, we denote by $C(c_1, c_2, \cdots)$ a   generic  large positive constant, which may vary from line to line, but ultimately depend only on parameters $c_1, c_2, \cdots$ (in addition to $g$), and write
  $A \lesssim_{c_1, c_2, \cdots} B$, resp. $A = \cO_{c_1, c_2, \cdots}(B)$, $A \sim_{c_1, c_2, \cdots} B$ if the implicit constants depend on parameters $c_1, c_2, \cdots$ (in addition to $g$).  
    We will write $A \gg_{c_1, c_2, \cdots} 1$, resp. $A \ll_{c_1, c_2, \cdots} 1$ to mean that we can take $A$ to be an {\it arbitrary} positive parameter that is bigger, resp. smaller, than certain constant depending on $c_1, c_2, \cdots$ (in addition to $g$).

			We will use the Japanese bracket $\langle x \rangle = (x^2 + 1)^{1/2}$ for a constant $x \in \R$.
		
 
			Given a bounded measurable function $P : \R^{n} \to \C$,  we denote
			\aryst
			M(P): L^{2}(\R^n) \to L^2(\R^n), \quad M(P)u = Pu.
			\earyst
			
			Given a diffeomorphism $T: U \to V$ between two open domains of $\R^n$, we denote
			\aryst
			\CC_T: L^2(V) \to L^2(U), \quad \CC_T u = u \circ T.
			\earyst
			
 			For any $x \in \R^n$, we will treat the co-tangent vector $V = (V_1, \cdots, V_m) \in T^*_{F(x)}\R^m$ as a row vector and  treat the tangent vector $v = (v_1, \cdots, v_n) \in T_x\R^n$ as a column vector. 
			Given a diffeomorphism $F = (F_1, \cdots, F_m) : \R^n \to \R^m$ and $x = (x_1, \cdots, x_n) \in \R^n$, we denote
			\aryst
			(DF)_x = 
			\begin{bmatrix}
				\partial_{x_1} F_1 &  \cdots & \partial_{x_n} F_1  \\
				\vdots &    \ddots & \vdots \\
				\partial_{x_1} F_m &   \cdots & \partial_{x_n} F_m
			\end{bmatrix}
			\earyst
			so that  $(DF)_x^*(V) = V \cdot (DF)_x$, and $(DF)_x(v) = (DF)_x \cdot v$. With these conventions,  we have
$\langle (DF)^*_{x}(V) , v  \rangle = \langle V , (DF)_x (v )  \rangle = V \cdot (DF)_x \cdot v^{T}$.

			\section{Basic properties and standing assumptions} \label{sec Basic properties}

			In the rest of the paper, we let $g$ denote a $C^{\infty}$,  volume-preserving Anosov flow on  $M$. We let $\greg > 1$ be a large integer to specify the topology (that is, $C^{\greg}$) under which we perturb $g$.
 Although not necessary for this section, we will introduce in the course of the proof a large integer parameter $\tinitial > 0$ which may be determined only on $g$ so that the hyperbolicity of $g^{\tinitial}$ becomes evident.
 	We let  $\esmall  > 0$ be a small parameter which will be determined later depending only on $g$ and $\tinitial$. 
The parameters $\greg, \tinitial, \esmall$ will be determined in Theorem \ref{lem main}. 			 Throughout this paper, we use
$c(\esmall)$ to denote a constant in $(0, 1)$, which may vary from line to line, but ultimately can be chosen to depend only on $g$ and $\esmall$, and tends to $0$ as $\esmall$ tends to $0$.
			We let $f \in \Diff^{\greg}(M)$ denote an arbitrary map such that
			\ary \label{eq fclosetog1}
			d_{\Diff^{\greg}(M)}(f, g^{1}) < \esmall.
			\eary

			We will let $\esmall$ be sufficiently small so that $f$ is a {\it partially hyperbolic diffeomorphism}: there is a $Df$-invariant continuous splitting $TM = E^s_f \oplus E^c_f \oplus E^u_f$ such that for every $p \in M$, we have
			\aryst
			\| Df|_{E_f^s(p)} \| < \min(1,  \|Df|_{E_f^c(p)}\|) \leq \max(1,  \|Df|_{E_f^c(p)} \| ) < \|Df|_{E_f^u(p)} \|.
			\earyst
			 Here $E^s_f, E^c_f, E^u_f$ are respectively the stable, center, and unstable subspace of $f$.
			It is known that $E^s_f$, resp. $E^u_f$, uniquely integrates to the stable foliation $\cW^s_f$, resp. unstable foliation $\cW^u_f$.
			We say that $E^{cs}_f = E^c_f \oplus E^s_f$, resp. $E^{cu}_f = E^{c}_f \oplus E^u_f$, is the center-stable subspace, resp. the center-unstable subspace of $f$.
			Moreover, it is known that by letting $\esmall$ be sufficiently small, 
			the subspaces $E^{cs}_f$, $E^{cu}_f$ and $E^c_f$ integrate to respectively the center-stable foliation $\cW^{cs}_f$, the center-unstable foliation $\cW^{cu}_f$ and the center foliation $\cW^c_f$ (see \cite{HPS} for this assertion).
			We may assume without loss of generality that

			\begin{center}
				{\it $E^{s}_f, E^{u}_f$ and $E^c_f$ are all orientable, and $f$ preserves the orientation of each bundle.}
			\end{center}
			Indeed, we can consider a finite cover $\hat M$ of $M$, and a lift $\hat f: \hat M \to \hat M$ of $f$, so that the above assertion is satisfied for $\hat f$ in place of $f$. Moreover, $\hat M$ can be chosen to depend only on $g$ for all $f$ sufficiently close to $g$. The flow $g$ lifts to an Anosov flow $\hat{g}$ on $\hat{M}$; and any $f$ close to $g^{1}$ lifts to a map $\hat{f}$ close to $\hat{g}^{1}$. This completes the reduction from $(f, M)$ to $(\hat{f}, \hat{M})$.

			In the following, we abbreviate $E^s_f$, $E^u_f$, $E^c_f$, $\cW^s_f$, $\cW^u_f$, etc., as $E^s$, $E^u$, $E^c$, $\cW^s$ and $\cW^u$, etc., respectively when there is no ambiguity about the map.
			Without loss of generality, we will always assume that   we have
			\aryst
			\| Df^{-1} |_{E^c_f}  \|, \|Df|_{E^c_f}\| < e^{\esmall},
			\earyst
			and, there exist $\hat\lambda_g > \lambda_g > 0$ depending only on $g$ such that $\| Df^{-1} \|, \| Df \| < e^{\hat\lambda_g}$ and  
			\ary \label{eq deflambdag}
	 	\| Df^{-1} |_{E^u_f}  \|, \|Df|_{E^s_f}\| < e^{ - \lambda_g}.
			\eary

For each $p \in M$ and $n \in \Z$, we denote
\aryst
 \mu^{*}(p, n) = \| Df^n|_{E^{*}(p)} \|, \quad * \in \{s, c, u\}.
\earyst 
			For later purposes, we define
			\aryst
			E^{c}_* = (E^u \oplus E^s)^{\perp} =  \{ \gamma \in T^{*} M  \mid \ker( \gamma ) = E^u \oplus E^s \}.
			\earyst
			Since $E^u$ and $E^s$ are invariant under $Df$, we see that $E^c_*$ is invariant under $(Df)^*$.

			The following result, which can be found in \cite{HPS}, is well-known.
			
			\begin{lemma} \label{lem regularitiesofthefoliation}
			The invariant bundles  $E^s_f, E^u_f$ and $E^c_f$ are all uniformly $(1 - c(\esmall))$-H\"older continuous.
			\end{lemma}

			\subsection{Non-stationary normal form} \label{subsec Non-stationary normal form}
			
			{\it Non-stationary normal coordinate}  is a collection of $C^{\greg}$ immersions $\{ \imath^u_p: \R \to \cW^u_f(p) \}_{p \in M}$  
			such that 
			\enmt
			\item  as a map from $M$ to $C^{\greg-1}([-1,1], M)$,   the map $p \mapsto \imath^u_p|_{[-1,1]}$ is $(1-c(\esmall))$-H\"older continuous. The H\"older norms of these maps are uniformly bounded for all $f$ close to $g^1$,
			\item  $\imath^u_p(0) = p$ and $D\imath^u_p(0)$ is the vector of norm $1$
			representing the direction of $E^u(p)$ (recall that we have assumed the orientability of the invariant subspaces),
			\item $(\imath^u_{f^n(p)})^{-1} f^n \imath^u_p(\tau) =  \mu^u(p, n) \tau$ for any $\tau \in \R$,  $n \in \Z$ with  $\mu^u(p, n) =  \norm{Df^n|_{E^u(p)}}$.
			\eenmt
			We refer the readers to \cite{KK} for a proof of the existence of non-stationary normal coordinates.
			In an analogous way, we define $\imath^s_p$ for the stable foliation 
			$\cW^s_f$.

			By our hypothesis, it is clear that for any $n \geq 0$, we have 
			\ary \label{eq almostconversative}
			| \log \mu^{c}(p, n)| , | \log \mu^{u}(p, n) + \log \mu^s(p, n) | \leq n\esmall + C.
			\eary

			\section{Template function, adapted scale, torsion and fluctuation} \label{sec Templates functions and Adapted scales}

			\subsection{Template function and Normal coordinate system}
			
			Throughout this paper, we denote the collection of real-valued affine  transformation on $\R$ by $\Aff = \{ x \mapsto a x + b \mid a, b \in \R \}$.
			 
			\begin{prop} \label{prop normalcoordinatesystem} 
				After a change of metric of $M$ if necessary (depending only on $g$),
				for any $\greg \geq 4$,
				there exist  $C = C(g) > 0$ and a neighborhood $\mathfrak{V}_g$ of $g^1$ in $\diff^{\greg}(M)$, and,  for each $f \in \mathfrak{V}_g$,  a family of $C^{\greg - 1}$ embeddings $\{ \imath_p:  (- 4 \|Df\|, 4 \|Df\|)^3 \to M \}_{p \in M}$, which we call a {\em \underline{normal coordinate system}}, such that:
				\enmt
				\item \label{item imathsmoothnorms}
				The supremum of the $C^{\greg - 1}$-norms of $\imath_p$ over all $p \in M$ is bounded from above by $C$.    Moreover,
				as a map from $M$ to $C^{3}( [  -  3 \|Dg^1\|, 3 \|Dg^1\| ]^3, M )$, the map   $p \mapsto \imath_p$ has $(1 - c(\esmall))$-H\"older norm bounded from above by $C$\footnote{It is known that  the Whitney topology on $C^{3}( [ -  3 \|Dg^1\|, 3 \|Dg^1\| ]^3, M )$ is metrisable, and we fix an arbitrary metric comparable with the topology.}; and the collection of maps $\{ f \mapsto \imath_p \}_{p \in M}$ are uniformly continuous,   
				\item  \label{itm imathp0order}  for every $p \in M$, for every $(x, y, z) \in (- \|Df\|, \|Df\|)^3$, we have 
				\ary
				&\imath_p(0,0,0) = p, \quad \imath_p(x,0,0) = \imath^u_p(x),  \quad \imath_p(0,y,0) = \imath^s_p(y),  \label{eq imath1} 
				\eary
                \item \label{itm imath1storder} for every $p \in M$, the map
                \aryst
             F_p =    (F_{p, 1}, F_{p, 2}, F_{p,3}) :=  \imath_{f(p)}^{-1} f \imath_p  : (- 1,  1)^3 \to \R^3
                \earyst
                is a well-defined $C^{\greg - 1}$ embedding
                such that
                \ary
                && \partial_z F_{p,2}(\cdot,0, 0) \equiv 0, \quad \partial_z F_{p,1}(0,\cdot, 0) \equiv 0, \\
                && \partial_y F_{p,2}(\cdot,0, 0) \equiv \mu^s(p, 1), \quad \partial_x F_{p,1}(0,\cdot, 0) \equiv \mu^u(p, 1), \\
                &&  \partial_z F_{p,3}(\cdot,0, 0) \equiv \mu^c(p, 1), \quad \partial_z F_{p,3}(0,\cdot, 0) \equiv \mu^c(p, 1),  \\
                \mbox{ and } \ &&  \partial_y F_{p, 3}(\cdot,0, 0)  \in \Aff, \quad \partial_x F_{p, 3}(0,\cdot, 0) \in \Aff.
                \eary
				\eenmt
			\end{prop}
			For Anosov flows, the parallel statement is included in \cite[Lemma 4.2]{TZ}. 
			The proof of Proposition \ref{prop normalcoordinatesystem} is implicitly contained in  \cite{EPZ}. 			
			We defer the proof of Proposition \ref{prop normalcoordinatesystem} to Appendix \ref{sec normalcoordinate}.

			\begin{rema} \label{rem D2Fvanish}
				By \eqref{eq imath1} in  Proposition \ref{prop normalcoordinatesystem}, we have 
			 $\partial^2_x F_{p, 3}(0,0,0)  = \partial^2_y F_{p, 3}(0,0,0)  =  0$.
			\end{rema}
			
			In the rest of the paper, we  fix $\{ \imath_p \}_{p \in M}$.
			Given $p \in M$, we will define $F_p$ as in Proposition \ref{prop normalcoordinatesystem}.
			We define $(\vartheta^u_{ p}, \vartheta^s_{ p}) : (- 3 \|Df\|, 3 \|Df\|)^3  \to \R^2$ by setting for each $w \in (- 3 \|Df\|, 3 \|Df\|)^2$ and each $z \in (- 3 \|Df\|, 3 \|Df\|)$ 
			\ary \label{def varthetap}
		  (\vartheta^u_{ p}(w, z), \vartheta^s_{ p}(w, z), 1) \in (\imath_p)^*(E^c_*(\imath_p(w, z))). 
			\eary
			By definition, it is clear that  $ \vartheta^u_{ p}(0, 0, 0) = \vartheta^s_{ p}(0, 0, 0) = 0$.

			 The following notion corresponds to \cite[Definition 2.8]{Tsu} (see also  \cite[Definition 4.15]{TZ}).
			 \begin{defi}[Template function] \label{def templatefunction}
			 	We define functions $\tpl^s_{p}, \tpl^u_{p} :  (- 3 \|Df\|, 3 \|Df\|)  \to \R$ by  
			 	\aryst
			 	    \tpl^s_{p}(\tau) = \vartheta^s_{p}(\tau, 0, 0), \quad		     	\tpl^u_{p}(\tau) = \vartheta^u_{p}(0, \tau, 0).
			 	\earyst
			 \end{defi}
			 
We define   $\twist : M \to \R$ by 
$$\twist(p) =		\partial_x	 \partial_y F_{p, 3}(0,0,0).$$
	Since $E^c_*$ is $(Df)^*$-invariant, we have 
	$$(\vartheta^u_{f(p)}(F_p(w, z)), \vartheta^s_{f(p)}(F_p(w, z)), 1)  \cdot DF_p(w, z) \in \R  (\vartheta^u_{ p}(w, z), \vartheta^s_{ p}(w, z), 1).$$
	Then by Proposition \ref{prop normalcoordinatesystem},  we deduce that for every $p \in M$ 
	\ary \label{eq psisequivariant}
  \mu^c(p, 1)  \tpl^s_p(\tau) -  \mu^s(p, 1)   \tpl^s_{f(p)}(\mu^u(p, 1) \tau)    = \twist(p) \tau.
	\eary
		We have a parallel expression relating $\tpl^u_p$ and $\tpl^u_{f(p)}$.
	By Proposition \ref{prop normalcoordinatesystem}\eqref{item imathsmoothnorms}, we have
	\ary \label{eq cisHolder}
	| \twist(p) - \twist(p') | \lesssim_{\esmall}  d(p, p')^{1- c(\esmall)}.
	\eary

			\subsection{Torsion and Fluctuation}
 
We introduce two quantities that measure the fluctuation of the stable subbundle $E^s$ along a local unstable manifold of $p\in M$ at the scale $0<\rho\ll 1$. Roughly we write  $\tor^s(\cdot)$ for an averaged torsion and $\flu^s(\cdot)$ for deviation from that average (See Remark \ref{rem:Flu}).

			\begin{defi} \label{def torsion}
				Let $p \in M$. For every  $\rho > 1$ we set
				\aryst
				\tor^s(p, \rho) = 0, \quad \flu^s(p, \rho) = 1.
				\earyst
				We extend the definitions of $\tor^s(p, \rho)$ and $\flu^s(p, \rho)$ for $\rho \in (0, 1)$ using the following formula:
				\begin{align}
					  \label{eq torsionandc}
					\mu^c(p, 1) 	\tor^s(p, \rho)  &=	\mu^s(p, 1)   \mu^u(p, 1)    \tor^s(f(p), \mu^u(p, 1)\rho)  + \twist(p), \\
					\label{eq fludef}
				 \flu^s(p, \rho) &=   	\frac{\mu^s(p, 1)}{ \mu^c(p, 1)} \flu^s(f(p), \mu^u(p, 1) \rho).
				\end{align}
				We define $\tor^u(p, \rho)$ and $\flu^u(p, \rho)$ in a similar way.
			\end{defi}

				Given $p \in M$ and $\rho \in (0, 1]$, 
				let integer $n =n(p,\rho) >  0$ be defined by inequalities
				\ary \label{eq rhoandn}
				\mu^u(p, n-1) \leq \rho^{-1} < \mu^u(p, n).
				\eary
				We can deduce from \eqref{eq fludef} and  a simple induction that
\ary \label{eq def flus}
 \flu^s(p,\rho) = \mu^s(p, n) / \mu^c(p,n).
\eary
					By \eqref{eq psisequivariant}, we have for every $0<\rho<1$  
				\aryst
  	 \tpl^s_p(\tau)  - \tor^s(p, \rho) \tau  =  
					\frac{\mu^s(p, 1)}{ \mu^c(p, 1)} ( \tpl^s_{f(p)}(\mu^u(p, 1) \tau) -   \tor^s(f(p), \mu^u(p, 1)\rho)   \mu^u(p, 1) \tau ).
					\earyst
			By iterating \eqref{eq torsionandc}, we deduce that  for  $0\le m\le n(p,\rho)$ 
				\ary \label{eq torexpression0}
			\mu^c(p, m) \tor^s(p, \rho)     =  \twist(p, m)  + \mu^s(p, m)   \mu^u(p, m)    \tor^s(f^m(p), \mu^u(p, m)\rho)
				\eary
				where  
				\ary \label{eq cpmexpression}
 			  \twist(p, m) =  \sum_{i = 0}^{m-1}	 \mu^s(p, i)   \mu^u(p, i)    \mu^c(f^{i+1}(p), m-1-i) \twist(f^i(p)).
				\eary
 Note that the last term on the right hand side of \eqref{eq torexpression0} vanishes when $m=n(p,\rho)$. 
 The term $\twist(p, m)$ has the following simple meaning: if we denote $\hat{F} = (\hat{F}_{p, 1}, \hat{F}_{p, 2}, \hat{F}_{p, 3}) = \imath_{f^m(p)}^{-1} f^m \imath_p$, then 
 \aryst
\twist(p, m)  = \partial_x \partial_y \hat{F}_{p, 3}(0, 0, 0).
 \earyst
 Similar to \eqref{eq def flus}, we let $n > 0$ be given by $\mu^s(p, - (n-1)) \leq \rho^{-1} < \mu^s(p, -n)$, and we have
 \aryst
  \flu^u(p,\rho) = \mu^u(p, -n) (\mu^c(p,-n))^{-1}.
 \earyst
 Similar to \eqref{eq torexpression0}, we have 
\ary \label{eq torexpression02}
\mu^c(p, m) \tor^u(p, \rho)     = \twist(p, m)  +  \mu^s(p, m)   \mu^u(p, m)    \tor^u(f^m(p), \mu^s(p, m)\rho).
\eary
	\begin{rema}\label{rem:Flu}
	    The geometric meaning of these quantities is the following:  for any $\rho \in (0,1)$,
	\aryst
	\sup_{\tau \in (- \rho, \rho)} |  \tpl^s_p(\tau) - \tor^s(p, \rho)  \tau | \lesssim \flu^s(p, \rho).
	\earyst
	We have a similar expression involving $\tpl^u_p$, $\tor^u$ and $\flu^u$.
	\end{rema}

	It is not hard to see that whenever $\rho \in (0, 1]$ we have 
	\begin{align}
&  \rho^{1 + c(\esmall)} \lesssim	\flu^s(p, \rho), \flu^u(p, \rho) \lesssim \rho^{1 - c(\esmall)}, \ \    | \tor^s(p, \rho) |,  | \tor^u(p, \rho) | = \cO( \rho^{- c(\esmall)} \langle \log \rho \rangle ), \label{eq flutorbasicestimate} \\
&  \mbox{ and } \quad 	K^{1 - c(\esmall)}  \lesssim 		\flu^{*}(p, K \rho) / \flu^{*}(p, \rho) \lesssim K^{1 + c(\esmall)}, \ * \in \{s, u\}, 1 \leq K \leq \rho^{-1}.
	\end{align} 
	
We also have the following.

			\begin{lemma} \label{lem torsionareclose}
				We have 
				\aryst
			&&	|  \tor^s(p, K \rho) - \tor^s(p, \rho) |  \lesssim  \langle \log K \rangle K^{c(\esmall)} 	\flu^s(p, \rho)  / \rho \quad \mbox{for any  $K>1$}, \\
			&&	| \tor^s(p, \rho) - \tor^s(p', \rho) | \lesssim  \langle \log  \langle d(p, p') / \rho \rangle \rangle    \langle d(p, p') / \rho \rangle^{c(\esmall)}   \flu^s(p, \rho) / \rho.
				\earyst
				We have a parallel estimate for $\tor^u$ in terms of $\flu^u$.
			\end{lemma}
			
			\begin{proof}
				We defer the proof to Appendix \ref{sec normalcoordinate}.
			\end{proof}

			\subsection{Adapted scale} \label{subsec adaptedscale} 
		 Below we set up a scale $\adscale(\omega,p)$ that depend on a real parameter $\omega\ge 1$ and a point $p\in M$. Some preliminary explanation should be in order. We would like to set up a scale $\rho=\adscale(\omega,p)$ so that the non-integrability between the stable and unstable foliation (observed in the center direction) restricted in the $\rho$-neighborhood of $p$ is approximately $1/\omega$. To this end, we would like to have that the fluctuation of the stable manifolds $W^s_{\rho}(q)$ for $q\in W^u_{\rho}(p)$ (or $W^u_{\rho}(q)$ for $q\in W^s_{\rho}(p)$) is proportional to $1/\omega$. Suppose  $\rho=\mu^u(p,n)^{-1}$ for some $n\ge 1$. Then $f^n$ sends $W^s_\rho(q)$ for $q\in W^u_\rho(p)$ to $W^s_{\rho'}(q')$ for $q'=f^n(q)\in W^u_1(f^n(p))$ with $\rho'\sim \mu^s(p,n)\rho$. Since $W^u_1(f^n(p))$ is of unit scale, the fluctuation of $W^s_{\rho'}(q')$ for $q'\in W^u_1(f^n(p))$ in the center direction should be proportional to $\rho'$. If we pull-back this picture by $f^n$, the fluctuation of $W^s_{\rho}(q)$ for $q\in W^u_\rho(p)$ in the center direction should be proportional to  
   \aryst
   \rho'\cdot \mu^c(p,n)^{-1}=\mu^s(p,n)\cdot \rho \cdot \mu^c(p,n)^{-1} =
   \frac{\mu^s(p,n)}{ \mu^u(p,n) \cdot \mu^c(p,n)}. 
   \earyst
This consideration leads to the following definition.

			\begin{defi}[Adapted scales] \label{def Adapted scales}
				For every $\omega \geq 1$, for every $p \in M$, 
				let $n_u = n_u(p, \omega) \geq 0$ be the smallest   integer $n \geq 0$ such that
				\aryst
				\mu^u(p, n)  (\mu^s(p, n))^{-1} \mu^c(p,n) \geq \omega,
				\earyst
				and let  $n_s = n_s(p, \omega) \geq 0$ be the smallest   number $n \geq 0$ such that
				\aryst
				\mu^s(p, -n)  (\mu^u(p, -n))^{-1} \mu^c(p, -n) \geq \omega.
				\earyst
				We define the {\em \underline{$\omega$-adapted scale at $p$}} to be
				\ary \label{eq defofadscale}
				\adscale(\omega, p) :=   \min( ( \mu^u(p, n_u) )^{-1},   (  \mu^s(p, -n_s) )^{-1}  ).
				\eary 
			\end{defi}
			
				\begin{rema}\label{rem:torfluadscale}
			The geometric motivation behind Definition \ref{def Adapted scales} is the following: 
			For any $\omega \geq 1$, any $p \in M$, for $\rho =  \adscale(\omega, p)$, we have 
			\ary \label{lem adaptedscalefluctuation} 
			\omega  \rho  \max( \flu^u(p,   \rho ),  \flu^s(p, \rho ) )   \sim 1.
			\eary
 	  Let $m$ be the smallest integer with $\mu^u(p, m) \rho \geq 1$. Clearly, $m \geq n_u$.
			By Definition \ref{def torsion}, for $\theta :=  \omega \rho \mu^s(p, m) ( \mu^c(p, m)  )^{-1} \in (0, 1)$, we have 
			$$   \omega  \rho ( \tpl^s_p( \rho \tau) - \tor^s(p, \rho) \rho  \tau ) = \theta T^s_{f^m(p)}( \tau )$$ which, as a  function of $\tau \in (-1, 1)$, has  $(1- c(\esmall))$-H\"older norm bounded by $\cO_{\esmall}(1)$.		 	 A parallel statement holds for $\tpl^u_p$, $\tor^u(p, \rho)$ as well.
			\end{rema}
			
			In the later part of the paper, $\omega$ in Definition \ref{def Adapted scales} will roughly correspond to the frequency of a function along the center direction.
	The following estimate is clear from the definition, and will be used throughout the paper (sometimes without mentioning) 
				\begin{align}
  & 	 \omega^{-1/2 + c(\esmall)} \lesssim	\adscale(\omega, p) \lesssim \omega^{-1/2 -  c(\esmall)},  \  \ \ \  (\frac{\omega' }{ \omega })^{- 1/2 - c(\esmall)} \lesssim  \frac{\adscale(\omega', p) }{\adscale(\omega, p) }  \lesssim  (\frac{\omega' }{ \omega })^{- 1/2 + c(\esmall)}   \label{eq adscalesquarerootgrowth} \\
	&	  \label{eq pushforwardofadscale}
   \mbox{ and } \quad 		 | \log	\frac{ 	\adscale(\omega, f^n(p))}{ 	\adscale(\omega, p)} | \leq  c(\esmall) n +  C.
				\end{align}

The following lemma  follows quickly from Lemma \ref{lem torsionareclose} and Remark \ref{rem:torfluadscale}.   We omit its proof.

\begin{lemma} \label{lem mainapproximation1} 
	Let $C_0 \geq 1$. For any $\esmall \ll_{C_0} 1$,  any $p \in M$ and  any $\rho \in (0, 1)$,
 any $y_0 \in (- C_0 \rho, C_0 \rho )$
	\aryst
	\sup_{y \in (y_0 - \rho, y_0 + \rho)}	| \tpl^u_{p}(y) -  \tpl^u_{p}(y_0)  - \tor^u(p, \rho ) ( y - y_0 ) |  \lesssim   \flu^{s}(p, \rho).
	\earyst
	We have parallel statements for $\tpl^s_{p}$ in place of $\tpl^u_{p}$.
\end{lemma}

			\subsection{Local geometry of $E^c_*$}
			
			The following lemma describes the local geometry of $E^c_*$ in a normal coordinate system.

			\begin{lemma} \label{lem mainapproximation}
				For any $c \in (0, 1/4)$,  by letting $\esmall \ll_{c} 1$, we have  for any $p \in M$,
				 any $\rho \in (0, 1)$, any $1 \leq C_1 \leq C_0 \leq   \rho^{-1/4}$,  any $w  \in B((0, 0),  C_0 \rho )$, any  $z \in   (-   \rho^{1 + c },   \rho^{1 + c })$,    
				and any $\tau \in (-  C_1 \rho,  C_1 \rho)$, we have
				\begin{align}
				| \vartheta^u_p(w + (\tau, 0), z) -  \vartheta^u_p(w, 0) |  &\lesssim  (C_0 \rho)^{1 + c / 2}, \label{eq thetaualongu1} \\
				|  \vartheta^s_p(w + (0, \tau), z)  - \vartheta^s_p(w, 0)   | &\lesssim (C_0  \rho)^{1 + c / 2}, \label{eq tildethetasalongs2} \\
				| \vartheta^s_p(w + (\tau, 0), z) -  \vartheta^s_p(w, 0) - \tor^s(p, \rho) \tau | &\lesssim_{\esmall}  C_0^{c(\esmall)}  C_1  \flu^s(p, \rho), \label{eq thetasalongu3} \\
				\label{eq tildethetaualongs4}
				|  \vartheta^u_p(w + (0, \tau), z)  - \vartheta^u_p(w, 0) - \tor^u(p, \rho) \tau | &\lesssim_{\esmall}  C_0^{c(\esmall)}  C_1 \flu^u(p, \rho).
				\end{align}
			\end{lemma}
			\begin{proof}
	 	We defer the proof to Appendix \ref{sec normalcoordinate}.
			\end{proof}

The torsions $\tor^u(p,\rho)$ and $\tor^s(p, \rho)$ do not have intrinsic geometric meanings as they depend on the choice of local charts $\imath_p$, though their difference will do. Thus it is a little more convenient to use a local chart in which the quantities corresponding to $\tor^u(p,\rho)$ and $\tor^s(p,\rho)$ are balanced. We realize this by precomposing a map to $\imath_p$. 

In the following, for each $\omega \geq 1$ and $p \in M$, we always denote
$$\rho = \adscale(\omega, p).$$ 
We set
\aryst
	\tor^{\dagger}(p, \rho) :=  (\tor^s(p, \rho) + \tor^u(p, \rho) )/2.
\earyst
The following map $\imath_{p, \omega}: (-1,1)^3 \to \R^3$ will be used to rectify $\imath_p$ and balance the $z$-coordinate fluctuations along $W^s$ and $W^u$:
			\begin{align} \label{def imathetap}
	 	\imath_{p, \omega }(x, y, z) = (x, y, z -  \tor^{\dagger}(p, \rho) xy).
			\end{align}

			We set $(w, z') = \imath_{p, \omega }(w, z)$, and define
			\begin{align} \label{def varthetapeta}  
			 (\vartheta^u_{ p, \omega }(w, z), \vartheta^s_{p, \omega }(w, z), 1) := (\imath_{p, \omega})^*(\vartheta^u_{ p}(w, z'), \vartheta^s_{p}(w, z'), 1) 
			\in   (\imath_{p, \omega})^* (\imath_p)^*(E^c_*(\imath_p (w, z'))).  
			\end{align}
			By direct computations, we see
			\ary \label{eq varthetauspdef}
			 (\vartheta^u_{ p, \omega}(w, z), \vartheta^s_{p,\omega}(w, z)) 
			= (\vartheta^u_{ p}(w, z')  , \vartheta^s_{p}(w, z') ) -  \tor^{\dagger}(p, \rho)  (y, x). 
			\eary

			Define  
			\begin{align}
			\midtor(p, \rho)  &= (\tor^s(p, \rho) - \tor^u(p, \rho))/2, \label{def delta}  \\
			\Delta(\omega, p)  &= \max( \flu^s(p, \rho),  \flu^u(p, \rho),   | \midtor(p, \rho) |  \rho )\  \text{with} \ \rho =  \adscale(\omega, p).  \label{def Delta}
			\end{align}
			Here $\Delta(\omega, p)$ captures the remaining deviation after the balanced rectification.
 
			By \eqref{eq varthetauspdef} and Lemma \ref{lem mainapproximation}, 	for any $c > 0$, any $\esmall \ll_{c} 1$,  any $\omega$ with $| \omega | \gg_{c } 1$, and any $1 \leq C_0 \leq \rho^{-1/4}$,  and any $(x, y, z) \in ( - C_0 \rho, C_0 \rho )^2 \times  ( -   \rho^{1+c},    \rho^{1+c})$, we have
			\begin{align}  \label{lem mainapproximation0}
		 & 	(\vartheta^u_{ p, \omega }(x, y, z), \vartheta^s_{ p, \omega }(x, y, z))  \\
			=& \midtor(p, \rho )  ( -   y,   x) + ( \tpl^u_p( y ),  \tpl^s_p(x ) ) +  \cO( C_0^{1 + c/2} \omega^{ - (1/2 + c / 8)})  
			= \cO( C_0^{1 + c(\esmall)} \Delta(\omega, p)   ). \nonumber
			\end{align}
			The last equality above follows from Lemma \ref{lem torsionareclose} and Remark \ref{rem:torfluadscale}.

By \eqref{eq flutorbasicestimate} and \eqref{eq adscalesquarerootgrowth}, it is clear that 
			 \ary \label{eq Deltaomegarange}
 	 \omega^{-1/2 - c(\esmall)} \lesssim_{\esmall} \Delta(\omega, p) \lesssim_{\esmall}    \omega^{-1/2 + c(\esmall)}.
\eary
		 The following is a basic property of   $\Delta$.
\begin{lemma} \label{lem behaviorofDeltalaongorbit}
For any  $\omega, \omega' \geq 1$, any integer $n \geq 0$, and any $p \in M$,  we have 
\aryst
\Big | \log	\frac{ 	\Delta(\omega', f^n(p))  }{ 	\Delta(\omega, p)  }  \Big | \leq   c(\esmall)n  +     C \langle \log (n + 1) \rangle + C  \langle \log \langle \omega / \omega' \rangle \rangle.
\earyst
			\end{lemma}

			\begin{proof}
				We defer the proof to Appendix \ref{sec normalcoordinate}.
			\end{proof}

				Finally, we state an important property of $\adscale$ and $\Delta$, which we call {\em temperate property}.   Among other things, we see that for a fixed $\omega$, both $\adscale(\omega, p)$ and  $\Delta(\omega, p)$ change slowly as $p$ varies.   
				
			\begin{lemma}[Temperate property of $\adscale$ and $\Delta$] \label{lem adaptedscaleistame}
				For   any $\omega, \omega' \geq 1$ and any $p, p' \in M$,  we have 
				\aryst
				&&   | \log	\frac{ 	\adscale(\omega', p')}{ 	\adscale(\omega, p)} |  \leq     c(\esmall) \log \langle \frac{d(p, p')}{\adscale(\omega, p)} \rangle + C  \langle \log \langle \omega / \omega' \rangle \rangle, \\
				&&	 | \log \frac{ 	\Delta(\omega', p')}{ 	\Delta(\omega, p)} |  \leq   \log \langle \log \langle \frac{d(p, p')}{\adscale(\omega, p)} \rangle \rangle +  c(\esmall) \log \langle \frac{d(p, p')}{\adscale(\omega, p)} \rangle + C  \langle \log \langle \omega / \omega' \rangle \rangle.
				\earyst
			\end{lemma}
			
			\begin{proof}
				We defer the proof to Appendix \ref{sec normalcoordinate}.
			\end{proof}
%

		\subsection{A transversality condition}

			Now we can introduce the genericity condition in Theorem \ref{thm main}.

			\begin{defi}[Property $({\bf NI})$] \label{def NIrho}
				Given  $\sigma > 0$ and $ C_2 > C_1 \geq 1$, we say that a $3$-dimensional partially hyperbolic diffeomorphism $f$ satisfies $({\bf NI})_{\sigma, C_1}$ (resp. $({\bf NI})_{\sigma, C_1, C_2}$) if for every $p \in M$,  every $b \in \R$ with $|b| \geq C_1$ (resp. $|b| \in [C_1,C_2]$),  every $c \in \R$, every $h \in [1, \| Df  \| ]$ we have  $
				 | \frac{1}{2 h} \int_{- h}^{h} \exp( i b \tpl^{s}_{p}(\tau)  + i c \tau)  d\tau | < |b|^{- \sigma} $
				  and $  | \frac{1}{2 h} \int_{- h}^{h}  \exp( i b \tpl^{u}_{p}(\tau)  + i c \tau)  d\tau | < |b|^{-\sigma}$.
				  
				  We say that a $3$-dimensional Anosov flow $g$ satisfies $({\bf NI})_{\sigma, C_1}$ (resp. $({\bf NI})_{\sigma, C_1, C_2}$)  if $g^1$ does.
			\end{defi}

			Recall that $T^s_p$ and $T^u_p$ are related to $f$ by Definition \ref{def templatefunction} and Proposition \ref{prop normalcoordinatesystem}.  Moreover, by construction,   as maps in $C^0([-1, 1], \R)$, $T^s_p$ and $T^u_p$ depend continuously on $f$ in $\diff^{3}(M)$.  
			Hence for any $\sigma > 0$ any $C_2 > C_1 \geq 1$, $({\bf NI})_{\sigma, C_1,C_2}$ is an open condition: the set of  $3$-dimensional partially hyperbolic diffeomorphism $f \in \Diff^3(M)$  satisfying $({\bf NI})_{\sigma, C_1,C_2}$ is an open subset of  $\Diff^3(M)$.
			
			The following is essentially  proved in \cite{Tsu}.
			\begin{thm} \label{thm prevalenceofoscillation}
				There is a $C^{3}$-open and $C^\infty$-dense subset $\stabletransitive \subset \Anosovflow(M)$ such that for any $g \in \stabletransitive$, there exist some $\sigmaNI > 0$ and $\CNI > 1$ such that $g$ satisfies  $({\bf NI})_{\sigmaNI, \CNI}$.
			\end{thm}
			\begin{proof}
				In \cite[Theorem 2.15]{Tsu}, it is stated that the first inequality in Definition \ref{def NIrho} holds for a $C^{3}$ open and $C^{\infty}$ dense set of $g$. By symmetry, the second inequality in  Definition \ref{def NIrho} holds for a $C^{3}$ open and $C^{\infty}$ dense  set of the inverse of $g$. We conclude the proof by taking their intersection.
			\end{proof}

			\section{Dual  central bundle in normal central chart} \label{sec The dual  central bundle in local chart}

			In this section, we will introduce a system of {\it Normal Central Chart} (see Definition \ref{def normalcentralchart}). 
			These charts have two convenient features:
			(1) within the adapted scale, the geometry of the dual central bundle $E^c_*$ in a normal central chart approximates well the one in normal coordinate chart (see Proposition \ref{prop normalcoordinatesystem} and Lemma \ref{lem mainapproximation}); 
			(2) all the normal central charts share a common almost central foliation (see Definition \ref{def centralchart}). These properties are crucial for the construction of {\it Dynamical Wave-Packet Transformation} in Section \ref{sec DWPT: Construction and Properties}.

			\subsection{A finite covering by central charts}  \label{sec coveringcentralcharts}

			It is known that there exist  continuous closed cone fields $\{ C^u(p) \}_{p \in M}$ and  $\{ C^s(p) \}_{p \in M}$ such that for every $f$ close to $g^1$ in  $\diff^{1}(M)$, for every $p \in M$
			\aryst
			Df ( C^u(p) \setminus \{0\} ) \Subset C^u(f(p)) \ \mbox{ and } \ 			Df^{-1} ( C^s(p) \setminus \{0\} ) \Subset C^s(f^{-1}(p)).
			\earyst
			We have the following definition, relative to a smooth foliation $\tilde{\cW}^c_f$ that is close to $\cW^c_f$ in the following sense: let $\tilde{E}^c_f = T\tilde{\cW}^c_f$, then $\tilde{E}^c_f(p)$ is uniformly close to $E^c_f(p)$ for all $p \in M$.
			\begin{defi}\label{def centralchart}
				We say that a $C^{\greg}$-diffeomorphic embedding of an open domain $U \subset M$ into $\R^3$  
				\aryst
				\kappa: U \to \kappa(U) \subset \R^3, \quad \kappa(p) = (x,y,z)
				\earyst
				is a {\em \underline{central chart}} (with respect to $\tilde{\cW}^c_f$) if 
				\enmt
				\item   \label{itm smoothchartcentralisometry} $D\kappa( \tilde{E}^c_f ) = \R\partial_z$, $D\kappa|_{\tilde{E}^c_f}$ is isometric everywhere,
				\item \label{itm smoothchartcone}  for every $p \in U$ we have   $\det ( D\kappa(p) ) = 1$, and 
				\aryst
&&				D\kappa(E^u(p) \setminus \{0\} )  \Subset D\kappa(  C^u(p) \setminus \{0\} ) \Subset  \{  (V_x, V_y, V_z) \mid    |V_y|,  |V_z| \leq |V_x| / 4  \}, \\
&&				D\kappa(E^s(p) \setminus \{0\} )  \Subset D\kappa(  C^s(p) \setminus \{0\} ) \Subset  \{  (V_x, V_y, V_z) \mid    |V_x|, |V_z| \leq  |V_y| / 4  \}.
				\earyst
				\eenmt
				In this case, we call $\tilde W^c_f$ an {\em \underline{almost central foliation}} for $f$.
			\end{defi}
			In Definition \ref{def centralchart},  the essential part is \eqref{itm smoothchartcentralisometry}.
			In particular, for any two central charts $\kappa_1, \kappa_2$ with respect to $\tilde{\cW}^c_f$, 
			the map $\kappa_2 \kappa_1^{-1}$ is a skew product of form $(x, y, z) \mapsto ( F_1(x,y), z + F_2(x,y))$ on its domain of definition.
			In our setting, since we only consider the set of $f$ close to $g^{1}$, we will simply choose the almost central foliation $\tilde{\cW}^c_f$ to be the flow line foliation  of $g$. In this case, we have $\tilde{E}^c_f = N_g$, and $\kappa$ is a flow box chart.

			In the rest of this paper, we denote  
			\aryst
			\Omega = (-1, 1)^{2 + 1}, \quad  \widetilde\Omega = (- 2, 2)^{2 + 1}, \quad    \widehat\Omega = (- 4, 4)^{2 + 1}.
			\earyst
			We fix a finite collection of central charts  $\kappa_a : \widehat{U}_a \to  \widehat\Omega$ with respect to  $\tilde{\cW}^c_f$   
			as $a$ ranging from a finite set $A$. 
			We denote 
			\aryst
			U_a = \kappa_a^{-1}(\Omega), \quad \widetilde{U}_a = \kappa_a^{-1}(\widetilde{\Omega}).
			\earyst
			We require that the collection $\{ U_a  \mid a \in \cA \}$ forms an open covering of $M$. 
			Fix a partition of unity  $\{ \varphi_a  : M \to [0, 1] \mid a \in A \}$ by $C^\infty$ functions subordinate to  $\{ U_a  \mid a \in \cA \}$, i.e.,
			\ary \label{eq sumofrhoa}
			\supp(\varphi_a) \subset U_{a} \mbox{ for every $a \in A$,} \mbox{ and } \sum_{a \in \cA} \varphi_a  \equiv 1.
			\eary
			For each $a \in \cA$, we choose a $C^\infty$ compactly supported function $\widetilde\varphi_a  : M \to [0, 1]$ such that 
			\ary \label{eq sumofrhoatilde}
			\ \supp( \widetilde\varphi_a) \subset  U_a \ \mbox{ and } \  \widetilde\varphi_a \equiv 1 \mbox{ on $\supp(\varphi_{a})$.}
			\eary

			As explained in Subsection \ref{subsec DandS},  we fix a constant  
			\ary \label{def beta1}
			\beta_1 = 1 -   1/1000.
			\eary
			 	We also denote 
			 	\ary   \label{def c0}
			 	c_0 = (1 - \beta_1)/20.
			 	\eary
 			
			Let $f$ satisfy the conditions in Section \ref{sec Basic properties}. Recall that $\adscale$ and $\Delta$ are given by \eqref{eq defofadscale} and \eqref{def Delta} respectively.
			For every $a \in A$, for every $\eta \in \R$, for every $(w, z) \in \widehat\Omega$,  for every $\lambda > 0$, we define  
			\aryst
			&&  \Delta_a(\eta, w, z) :=  \Delta(\langle \eta \rangle, \kappa_a^{-1}(w, z) ), \
			 r_a(\eta, w, z) :=  \adscale(\langle \eta \rangle, \kappa_a^{-1}(w, z)), \\
	&&     \ngb^{\lambda }_a(\eta, w, z) := B(w,   \lambda  r_a(\eta, w, z)) \times B(z,   \lambda  \langle \eta \rangle^{-\beta_1})   \cap \widehat{\Omega},  \label{eq Omegadeltaa} \\
		\mbox{ and }	 &&   \NGB^{\lambda }_a(\eta, w, z) := B((w, z),   \lambda  r_a(\eta, w, z))  \cap \widehat{\Omega}.
			\earyst

		In the rest of this paper, we let $\tinitial \geq 1$ be a large integer whose value  will be determined later (see Theorem \ref{lem main}). Moreover, throughout the entire paper, we will always assume that 
		\ary \label{eq esmallandtinitial}
		\esmall \ll_{\tinitial} 1.
		\eary
 
		Given $a, a' \in A$ satisfying $	f^{\tinitial}(\widehat{U}_a) \cap \widehat{U}_{a'} \neq \emptyset$, we denote
\ary \label{eq faa'}
f^{\tinitial}_{a  \to a'} := \kappa_{a'} f^{\tinitial}  \kappa_{a}^{-1}.
\eary
Throughout this section we write $F = (F_{w}, F_z) = 	f^{\tinitial}_{a  \to a'}$.   By \eqref{eq esmallandtinitial},   we have 
\ary \label{eq skewproduct}
\| \partial_{z} F_w \|  \leq e^{C \tinitial} \esmall  \leq 1 \ \mbox{ and } \ 1/2 < \partial_{z} F_z < 2.
\eary

Define 
\aryst
  \Scone_{+} = \{ (x, y) \in \R^2 \mid |y| \leq  |x|/4  \},  \  & \Scone_{-}  = \{ (x, y) \in \R^2 \mid |x| \leq  |y|/4 \}.
\earyst
By Definition \ref{def centralchart}\eqref{itm smoothchartcone} and by requiring that $\tinitial$ be large, we may assume that
\ary \label{eq hyperbolicityofwFw}
\partial_w F_w ( \R^2 \setminus  \Scone_{-} ) \Subset {\rm Int}( \Scone_+ )  \ \mbox{ and }  \  (\partial_w F_w)^{-1} (  \R^2 \setminus  \Scone_{+} ) \Subset {\rm Int}( \Scone_- )
\eary 
and 
\ary \label{eq expansionoutsidecones}
\inf_{v \in \R^2 \setminus  \Scone_{-}}  ( \| \partial_w F_w(v) \| / \| v \| ) > e^{\lambda_g \tinitial} \ \mbox{ and } \inf_{v \in \R^2 \setminus  \Scone_{+}}  ( \| ( \partial_w F_w)^{-1}(v) \| / \| v \| ) > e^{\lambda_g \tinitial}.
\eary
Moreover, we may require that \eqref{eq hyperbolicityofwFw} and \eqref{eq expansionoutsidecones} hold   if $\partial_w F_w$ is replaced by any $\esmall$-close linear map.

			\subsection{Normal central charts} 
	
			In this subsection we introduce a family of center charts, in which the local geometry of the invariant splittings has good properties (see Proposition \ref{lem mainpropertyofestar}).
			
			To motivate our new definition,
			 we define $\estar_{a} = (\theta^u_{a}, \theta^s_{a}): \Omega \to \R^2$, for each $a \in A$, by the equation
			\aryst
			(\estar_{a}(w, z), 1) = (\theta^u_{a}(w, z), \theta^s_{a}(w, z), 1) \in (D\kappa_a^{-1})^{*}(E^c_*(\kappa_{a}^{-1}(w, z))).
			\earyst
			We would like a good understanding of the geometry of $E^c_*$, like we do in Lemma \ref{lem mainapproximation} for $(\vartheta^u_{ p}, \vartheta^s_{ p})$ in chart $\kappa_a$, or at least in some other central chart with respect to $\tilde{W}^c_f$.
		An immediate issue  is that the chart $(\imath_{p} \circ  \iota_{ \langle \eta \rangle, p})^{-1}$ needs not be a central chart since $E^c$ is in general not even Lipschitz.
	We will see that this is not a problem if we prescribe the correct derivatives to a central chart.

		In the rest of the paper, we will often consider two parameters $\eta, \eta' \in \R$, both correspond to the frequencies along the center direction. We will use the notations
		\ary \label{eq etamaxetamin}
		\eta_{max} = \max(\langle \eta \rangle, \langle \eta' \rangle), \quad \eta_{min} = \min(\langle \eta \rangle, \langle \eta' \rangle).
		\eary

		\begin{prop} \label{lem derivativeboundsforhatF} 
			For every $a \in A, (w, z) \in \widehat\Omega, \eta \in \R$, 
			there is a   diffeomorphism $ H_{a, w, z, \eta} : \R^{2 + 1} \to \R^{2 + 1}$ of form 		 
			$H_{a,  w, z, \eta}(\tw, \tz) =   (\tw, \tz + d_{a, w, z, \eta}(\tw) )$
			so that $H_{a, w, z, \eta}(w, z) = (w, z)$, and the following is true.
			For any $a, a' \in A$ so that $F = (F_w, F_z) := f^{\tinitial}_{a \to a'}$ is well-defined, and any $(w, z, \eta), (w', z', \eta') \in \widehat\Omega \times \R$, we let $\eta_{max}$, $\eta_{min}$ be given by \eqref{eq etamaxetamin},
			 and let
			\ary
			 \label{eq hatF}
			\hat F = ( \hat F_w,  \hat F_z ) :=  H_{a', w', z', \eta'} \circ  F \circ H_{a, w, z, \eta}^{-1}.
			\eary
			Assume that  $(\tw, \tz) \in   \NGB^{\Lambda}_a(\eta, w, z)   \cap  \hat{F}^{-1} ( \NGB^{\Lambda}_{a'}(\eta', w', z') ) \neq \emptyset$ for some $\Lambda > 1$. Then 
			\enmt
			\item \label{itm 1inL42}  We have
			\emph{
				\ary
				\label{eq partialzFw}
				&&   \| D^k   \hat F  \|_{U_a^{\Lambda}(\eta, w, z)} =    \cO_{k, \tinitial, \esmall}(    \Lambda^{k+2} \eta_{max}^{k c(\esmall)}  ), \ 1 \leq   k \leq \greg, \\
		 	\label{eq partialwFz}
		 	 &&\| \partial_{w}  \hat F_z(\tw, \tz)  \|   = \cO_{ \tinitial, \esmall}(   \Lambda^{50}   \eta_{max}^{c(\esmall)}    \eta_{min}^{-1/2 } ),
				\eary
			}
			\item \label{itm 2inL42}    
		If $\langle \eta \rangle / \langle \eta' \rangle \in (1/100, 100)$, $\Lambda \leq   e^{10 \hat {\lambda}_g \tinitial}$ and $  | z - \tz |, | z' - \hat{F}_z(\tw, \tz) | \leq \langle \eta \rangle^{ - 1/2 - c}$   for some $c > 0$,    then for every $\langle \eta \rangle  \gg_{\tinitial, c} 1$ and $\esmall \ll_{\tinitial, c} 1$, we have   
			\emph{
				\ary
				\label{eq partialw^2Fz}
	 			\| \partial_w^2 \hat{F}_z(\tw, \tz) \|  =
				\cO(   \tinitial^2    r_a(\eta, w, z)^{-1} \Delta_a(\eta, w, z) ).
				\eary
			}  
			\eenmt
		\end{prop}

		 Proposition \ref{lem derivativeboundsforhatF} gives us the estimates we need for the core argument. We will also use the following lemma which gives some loose estimates on a larger region for $(\tw, \tz)$ when $\eta_{max}/\eta_{min}$ is large.   
		 \begin{lemma} \label{lem etaeta'different}
		 	{\em 
		 		Let $c > 0$.
		 		Assume that  $\Lambda \leq \eta_{max}^{c}$, $\eta_{max} \gg_{c, \tinitial} 1$, $\esmall \ll_{c, \tinitial} 1$ and $\hat{F}( \NGB^{\Lambda}_a(\eta, \fw, \fz) ) \cap \NGB^{\Lambda}_{a'}(\eta', \fw', \fz') \neq \emptyset$. If we have either
		 		(1) $| \eta | \geq 10  | \eta' |$ and $(\tw, \tz) \in U_{a}^{\Lambda}(\eta, \fw, \fz)$,   or
		 		(2) $| \eta' |> 10 | \eta |$ and $\hat{F}(\tw, \tz) \in U_{a'}^{\Lambda}(\eta', \fw', \fz')$, 
		 		then we have $| \eta - \eta' \partial_z \hat{F}_z(\tw, \tz) | \geq \eta_{max} / 2$. 
		 	}
		 \end{lemma}

		 We will give the proofs of Proposition \ref{lem derivativeboundsforhatF} and Lemma \ref{lem etaeta'different}  after we give the definition and basic properties of  the charts $H_{a, w, z, \eta}$ in Proposition \ref{prop decompositionofkappaapeta}.  Before that, we give some remarks about certain properties of $\hat{F}$ that will follow quickly from the construction of $H_{a, w, z, \eta}$.

		\begin{rema} \label{rema DwhatFw}
			
			In \eqref{eq Cbound}, \eqref{eq expressionofHawzeta} in Proposition \ref{prop decompositionofkappaapeta} below, we will see that $D^2d_{a, w, z, \eta}( \tw  ) = \cO_{\esmall}(\langle  \eta \rangle^{c(\esmall)})$ for every $\tw$.
		We collect some related estimates here, which depends only on the skew product structure of $H_{a, w, z, \eta}$ and its derivative norm:
		
		\noindent{(1)}
		By  \eqref{eq skewproduct}, we see that $( \partial_w \hat{F}_w )(\tw, \tz)   = ( \partial_w F_w )( H_{a, w, z, \eta}^{-1}(\tw, \tz) ) + \cO(e^{C \tinitial }\esmall)$ for every $\tw \in B(w, C \langle \eta \rangle^{- c})$.

 	\noindent{(2)} 
 	By  \eqref{eq skewproduct},  for every  $(\tw, \tz) \in  \hat{F}^{-1} ( \NGB^{\Lambda}_{a'}(\eta', w', z') )$, we have    $| \partial_z \hat{F}_z (\tw, \tz) | \leq 3 +  C(\esmall) \Lambda \langle \eta' \rangle^{-1/2 + c(\esmall)}$.
		\end{rema}

		\begin{rema} \label{rema worksforinverseofF}
			We will see that the definition of $H_{a, w, z, \eta}$ is the same if we replace $f$ by $f^{-1}$.
		Consequently, the same estimates in Proposition \ref{lem derivativeboundsforhatF} hold for $(\hat{F})^{-1}$ in place of $\hat{F}$, provided parallel conditions.
		\end{rema}

 	By definition, we see that, restricted to a neighborhood of $\kappa_a^{-1}(w,z)$, $H_{a, w, z, \eta} \circ \kappa_a$ is also a central chart with respect to $\tilde{\cW}^c_f$.  By a slight abuse of notation, we introduce the following definition.
		\begin{defi}[Normal Central Chart] \label{def normalcentralchart}
			For every $a \in A$,  every $(w, z) \in \widehat\Omega$ and   every $\eta \in \R$, the map $H_{a, w, z, \eta} : \R^{2 + 1} \to \R^{2 + 1}$ given by Proposition \ref{lem derivativeboundsforhatF} is called a {\em \underline{normal central chart}} with base point  $(w, z)$. 
		\end{defi}
		We now begin the construction of normal central chart $H_{a, w, z, \eta}$ in Proposition \ref{lem derivativeboundsforhatF}.
		Given some $a \in A$, $(w, z) \in \widehat\Omega$ and some $\eta \in \R$.
		We set $p = \kappa_a^{-1}(w, z) \in M$ and  define
		\ary \label{eq defkappaaomegap} 
		\kappa_{a, p, \eta} =  \kappa_a \circ  \imath_{p} \circ  \iota_{p, \langle \eta \rangle}.
		\eary
		Note that we have	$\kappa_{a, p, \eta}(0, 0, 0) = (w, z)$.
		The  normal central chart $H_{a, w, z, \eta}$ in Proposition \ref{lem derivativeboundsforhatF} will be given by the following proposition,  constructed by an approximation of the chart $\kappa_{a,  p, \eta}$.
		\begin{prop}  \label{prop decompositionofkappaapeta}
				For every $a \in A$,
		there is a continuous map  $C_{a}$ from $\R \times \widehat\Omega$ to the space of quadratic forms on $\R^2$, satisfying 
		\ary \label{eq Cbound}
		\|  C_a(\eta, w, z ) \| = \cO_{\esmall}(\langle \eta \rangle^{c(\esmall)} ), \quad (\eta, w, z) \in \R \times \widehat\Omega
		\eary
		such that the following is true.
		Let 	$H_{a,  w, z, \eta}$ be given by	$H_{a,  w, z, \eta}(\tw, \tz) =  (\tw, \tz + d_{a, w, z, \eta}(\tw) )$ with
		\ary
		\label{eq expressionofHawzeta}  
		 d_{a, w, z, \eta}(\tw)  =     \estar_{a}(w, z) \cdot ( \tw - w )  + C_{a}(\eta, w, z)(\tw - w).
		\eary
		Then the map $ R_{a, w, z, \eta} =  (R_{w}, R_{z})   := \kappa_{a,  p , \eta}^{-1} \circ H_{a, w, z, \eta}^{-1}$
		satisfies that $R_{a, w, z, \eta}(w, z) = (0, 0, 0)$  and 
		\aryst
		 A_{a}(w, z) := \partial_w R_w(w, z) \ \mbox{ and } \ D_a(w, z) := \partial_z R_z(w, z)
		 \earyst
	  are independent of $\eta$, and satisfy
	  \aryst
	  \| A_a(w, z) \|, \| A_a(w, z)^{-1} \|  = \cO(1) \ \mbox{ and } \ D_a(w, z) \sim 1.
	  \earyst
		Moreover,  for any $\Lambda > 1$,  any $(\tw, \tz) \in \NGB_a^{\Lambda}(\eta, w, z)$,  we have 
		\ary
		\label{eq Rproperty1}
  && \| \partial_w R_z (\tw, \tz) \| =  \cO_{\esmall}(  \Lambda^5  \langle \eta \rangle^{-1/2 + c(\esmall)}), \   \| \partial_w^2 R_z (\tw, \tz) \|   = \cO_{\esmall}(  \Lambda^6  \langle \eta \rangle^{-1/2 + c(\esmall)}), \\
  		\label{eq Rproperty1.5}
  && \|	\partial_{w} R_w(\tw, \tz)   \|,   | \partial_z R_z(\tw, \tz) - \partial_z R_z(w, z) |   = \cO_{\esmall}( \langle   \Lambda^5     \langle \eta \rangle^{-1/2 + c(\esmall)}  \rangle ), \\
  \label{eq Rproperty3}
  &&   \| D^k R (\tw, \tz)  \| = \cO_{k, \esmall}(   \Lambda^{k + 2}   \langle \eta \rangle^{k c(\esmall) }  ), \quad 1 \leq k \leq \greg - 1.
		\eary  
			\end{prop}
		 
\begin{rema} \label{rem moreonR}
	From the proof of  Proposition \ref{prop decompositionofkappaapeta}, we see that the definition of $H_{a, w, z, \eta}$ is the same if we replace $f$ by $f^{-1}$.
					Denote $   R^{-1} = (\tilde{R}_w, \tilde{R}_z )$. By analogous argument, we see that \eqref{eq Rproperty3}  holds for $R^{-1}$ in place of $R$; and \eqref{eq Rproperty1}, \eqref{eq Rproperty1.5} hold for $(\tilde{R}_w, \tilde{R}_z ) $ in place of $(R_w, R_z)$. 
			\end{rema}

			\begin{rema} 
 	Let $\phi(w, z)$ be a smooth bump function of $w \in \R^2$ and $z \in \R$ such that $\partial_w \phi = \cO(r_a(\eta, w, z)^{-1})$ and $\partial_z \phi = \cO( \langle \eta \rangle^{\beta_1})$.
				Then the first equality of \eqref{eq Rproperty1}  shows that  $\phi \circ H_{a, w, z, \eta}(\tw, \tz)$  have derivatives $ \cO(r_a(\eta, w, z)^{-1})$ along $\tw$, and derivative $\cO( \langle \eta \rangle^{\beta_1})$ along $\tz$ restricted to $U^1_a(\eta, w, z)$.
				We also need the second equality of \eqref{eq Rproperty1}   because we will consider a oscillatory integral with a nonlinear phase (see \eqref{eq cK}).  The upper bound on $\partial_w^2 R_z$ will be used to deduce the upper bound \eqref{eq partialw^2Fz} in Proposition \ref{lem derivativeboundsforhatF}, which is important to treat the nonlinearity.
			\end{rema}

			\begin{proof}[Proof of Proposition \ref{prop decompositionofkappaapeta}]

%
%

 			Define $H_{a, w, z, \eta}$ by \eqref{eq expressionofHawzeta} where 	$C_a(\eta, w, z) $ is the quadratic form on $\R^2$ to be determined later. Then $H_{a, w, z, \eta}(\tw, \tz) = (\tw, \tz +  \langle \estar_a(w, z),    \tw - w \rangle + C_a(\eta, w, z)(\tw- w) )$ and
				\aryst
				  R_{a, w, z, \eta} = (R_w, R_z) =  \imath_{p,  \langle \eta \rangle}^{-1} \circ (\kappa_a \circ  \imath_{p})^{-1} \circ H_{a, w, z, \eta}^{-1}. 
				\earyst
				We will show that for some $C_a(\eta, w, z)$ satisfying \eqref{eq Cbound}, we have $\partial_w R_{z}(w, z) = \partial^2_w R_{z}(w, z)  =  0$.
			
			Denote $$D((\kappa_a \circ  \imath_{p})^{-1} )(w, z) = 	 \begin{bmatrix}
				\tilde{A}_{a, w, z} & \tilde{E}_{a, w, z}  \\  \tilde{B}_{a, w, z} & \tilde{D}_{a, w, z}
			\end{bmatrix}.$$
			 All the entries above, as well as $1/\tilde{D}_{a, w, z}$, are of size $\cO(1)$.
			Substitute the expressions for $H_{a, w, z, \eta}$ and $\imath_{p, \langle \eta \rangle}$   into the above formula for $R_{a, w, z, \eta}$, we deduce
				\aryst
				D R_{a, w, z, \eta }(w, z) =
				 \begin{bmatrix}
					{\rm Id} & 0 \\ 0 & 1
				\end{bmatrix} \cdot 
			\begin{bmatrix}
				\tilde{A}_{a, w, z} & \tilde{E}_{a, w, z}  \\  \tilde{B}_{a, w, z} & \tilde{D}_{a, w, z}.
			\end{bmatrix} \cdot
				 \begin{bmatrix}
					{\rm Id} & 0 \\ - \estar_a(w, z) & 1
				\end{bmatrix}.
				\earyst
				Note that $\partial_w R_w(w, z) = 	\tilde{A}_{a, w, z}  - \tilde{E}_{a, w, z}  \estar_a(w, z) $ and $\partial_z R_z(w, z) =  \tilde{D}_{a, w, z}$,  which are independent of $\eta$.
							Moreover, it is clear that $\partial_w R_w(w, z), \partial_{z} R_z(w, z) = \cO(1)$.

				We have $\partial_w R_{z}(w, z)  =  \tilde{B}_{a, w, z} - \tilde{D}_{a, w, z}   \estar_a(w, z)$. 				By \eqref{eq imath1} in Proposition \ref{prop normalcoordinatesystem},  we have $(D\imath_p)^*(E^c_*(p)) \in \R (0, 0, 1)$. Then $(\estar_a(w, z), 1) \in (D(\kappa_a \circ \imath_p)^{-1})^*(D\imath_p)^*(E^c_*(p)) \in \R(\tilde{B}_{a, w, z}, \tilde{D}_{a, w, z})$.  Consequently, $\partial_w R_{z}(w, z) = 0$.
			
			   Now we compute $\partial_w^2 R_{z}(w, z)$. 
			   To simplify computations, we write $F_0 = (F_{0, w}, F_{0, z}) = ( \imath_{p, \langle \eta \rangle })^{-1} $,  $F_1 = (F_{1, w}, F_{1, z}) =  (\kappa_a \circ  \imath_{p})^{-1} $ and  $F_2 = (F_{2, w}, F_{2, z}) =   H_{a, w, z, \eta}^{-1}$.
			   Note that $\partial_{w}F_{0, z}$, $\partial_{z}^2 F_{0, z} $, $ \partial_{z}\partial_{w} F_{0, z}$ all vanish at $(0, 0, 0)$.  Since $ R_{z}  = F_{0, z} \circ F_1 \circ F_2$, we have
				\aryst
				\partial_w^2 R_{z}(w, z) =  (\partial_w^2 F_{0, z})   ( 0, 0, 0) (\partial_w (F_{1, w}\circ F_2)(w, z), \partial_w (F_{1, w}\circ F_2)(w, z) ) + \partial_{w}^2 (F_{1, z} \circ F_2)(w, z).
				\earyst
				On the right hand side above,	
				the first term is of size $\cO_{\esmall}( \langle \eta \rangle^{c(\esmall)} )$ by \eqref{eq flutorbasicestimate}; and, by $ \partial^2_w F_{2, w}(w, z) = 0$, the second term $\partial_{w}^2 (F_{1, z} \circ F_2)(w, z)$ writes 
				\begin{align*}
		&  \partial^2_w F_{1, z}(0, 0, 0) ( \partial_w F_{2, w}(w, z), \partial_w F_{2, w}(w, z) ) + \partial_w\partial_{z} F_{1, z}(0, 0, 0) ( \partial_w F_{2, w}(w, z), \partial_w F_{2, z}(w, z) ) \\
				+&  \partial_z\partial_w F_{1, z}(0, 0, 0) ( \partial_w F_{2, z}(w, z), \partial_w F_{2, w}(w, z)   )  + \partial^2_z  F_{1, z}(0, 0, 0) ( \partial_w F_{2, z}(w, z), \partial_w F_{2, z}(w, z)   ) \\
			+&  \partial_z F_{1, z}(0, 0, 0) \partial^2_w F_{2, z}(w, z).
				\end{align*}
				Note that all the terms above, aside from $ \partial_z F_{1, z}(0, 0, 0) \partial^2_w F_{2, z}(w, z)  = \tilde{D}_{a, w, z}  C_a(\eta, w, z)$, are of size $\cO(1)$.
			 Hence the unique $C_a(\eta, w, z)$ solving $\partial_w^2 R_z(w, z) = 0$ satisfies \eqref{eq Cbound}. We will fix such $C_a(\eta, w, z)$. 
				   
				   We can deduce \eqref{eq Rproperty3} from \eqref{eq Cbound}, \eqref{eq flutorbasicestimate} and direct computations.
	           Then it is straightforward to deduce \eqref{eq Rproperty1} and  \eqref{eq Rproperty1.5}    from $\partial_w R_{z}(w, z) = \partial^2_w R_{z}(w, z)  =  0$  and \eqref{eq Rproperty3}.
			\end{proof}

			 We now show that Proposition \ref{lem derivativeboundsforhatF} and Lemma \ref{lem etaeta'different} hold for the charts given by Proposition \ref{prop decompositionofkappaapeta}.

\begin{proof}[Proof of Proposition \ref{lem derivativeboundsforhatF}]  

	 Define $H_{a, w, z, \eta}$ by Proposition \ref{prop decompositionofkappaapeta}. Then $d_{a, w, z, \eta}(\tw) = \langle \estar_{a}(w, z),  \tw - w \rangle  + C_{a}(\eta, w, z)(\tw - w)$.  
	 By $(\tw, \tz) \in  \NGB^{\Lambda}_a(\eta, w, z) $, we deduce from   \eqref{eq Cbound} in Proposition \ref{prop decompositionofkappaapeta} and \eqref{eq adscalesquarerootgrowth} that  $\| H_{a, w, z, \eta}^{-1} (\tw, \tz) - (w, z) \| = \cO_{\esmall}( \Lambda^2 \eta_{max}^{c(\esmall)} \langle \eta \rangle^{-1/2 })$. 	
	   \detail{
	 	Let $(w'', z'') \in   \NGB^{\Lambda}_a(\eta, w, z)  \cap  (\hat{F})^{-1}( \NGB^{\Lambda}_{a'}(\eta', w', z' ))$. Then
	 \aryst
	 \|  H_{a', w', z', \eta'} F(w,z) - (w', z')  \| \leq \|  H_{a', w', z', \eta'} F(w,z) -  H_{a', w', z', \eta'} F H_{a, w, z, \eta}^{-1} (w'',z'')  \| + \| H_{a', w', z', \eta'} F H_{a, w, z, \eta}^{-1}(w'',z'')  - (w', z')  \|   \\
	 \leq Ce^{2 \lambda_g \tinitial}  \eta_{max}^{6c} \langle \eta \rangle^{-1/2 } + \eta_{max}^{c} \eta_{min}^{-1/2 + c}
	 \earyst
	 } By $\NGB^{\Lambda}_a(\eta, w, z)  \cap  (\hat{F})^{-1}( \NGB^{\Lambda}_{a'}(\eta', w', z' )) \neq \emptyset$, we deduce $\| F H_{a, w, z, \eta}^{-1} (\tw, \tz)   -    (w', z')  \| = \cO_{\esmall} (   \Lambda^2 \eta_{max}^{c(\esmall)}  \eta_{min}^{-1/2 } )$.
	 Note that $D^3 H_{a, w, z, \eta} = D^3 H_{a', w', z', \eta'} = 0$. 
	Then \eqref{eq partialzFw} follows immediately from Proposition \ref{prop decompositionofkappaapeta} and the chain rule.
	
Now we assume in addition that $(\tw, \tz) \in   \NGB^{\Lambda}_a(\eta, w, z)   \cap  \hat{F}^{-1} ( \NGB^{\Lambda}_{a'}(\eta', w', z') )$.  
	 It remains to prove   \eqref{eq partialwFz}  and \eqref{eq partialw^2Fz}. 
   Define $R_{a, w, z, \eta}$, $R_{a', w', z', \eta'}$ as in Proposition \ref{prop decompositionofkappaapeta}. 
	Denote 	$p = \kappa_a^{-1}(w, z)$, $p' = \kappa_{a'}^{-1}(w', z')$. Denote $F_0 = 	\imath_{p'}^{-1}  \circ f^{\tinitial} \circ \imath_{p}$ and $\tilde{F} = (\tilde{F}_w, \tilde{F}_z) = (\imath_{p', \langle \eta' \rangle })^{-1} \circ  F_0 \circ  \imath_{p, \langle \eta \rangle}$.
By Proposition \ref{prop decompositionofkappaapeta} and \eqref{eq defkappaaomegap}, we have
\aryst
\hat{F}  =  H_{a', w', z', \eta'} \circ  F \circ H_{a, w, z, \eta}^{-1}    = (R_{a', w', z', \eta'} )^{-1} \kappa_{a', p', \eta'}^{-1} \circ   \kappa_{a'} f^{\tinitial}  \kappa_{a}^{-1} \circ  \kappa_{a, p, \eta} R_{a, w, z, \eta} 
= (R_{a', w', z', \eta'} )^{-1}  \tilde{F} R_{a, w, z, \eta}.
\earyst
We will use the above formula to reduce the problem to estimates on $\partial_w\tilde{F}_z$ and $\partial^2_w\tilde{F}_z$. 

 	By  \eqref{eq Rproperty3} in   Remark \ref{rem moreonR} and  \eqref{eq flutorbasicestimate}, we see that $R_{a, w, z, \eta} ( \NGB_a^{\Lambda}(\eta, w, z) ) \subset U_1 := B((0, 0, 0),  C(\esmall) \Lambda^2  \eta_{max}^{c(\esmall)} \eta_{min}^{-1/2} )$,   $\imath_{p, \langle \eta \rangle}(U_1) \subset U_2 :=  B((0, 0, 0), C(\esmall) \Lambda^4 \eta_{max}^{c(\esmall)} \eta_{min}^{-1/2}   )$, $F_0(U_2) \subset U_3 := B((0, 0, 0), C(\esmall) e^{C \tinitial} \Lambda^4   \eta_{max}^{c(\esmall)} \eta_{min}^{-1/2})$, $(\imath_{p', \langle \eta' \rangle})^{-1} U_3 \subset U_4 := B((0, 0, 0), C(\esmall) e^{C \tinitial} \Lambda^{8}  \eta_{max}^{c(\esmall)} \eta_{min}^{-1/2} )$.
	Thus we may consider only the restrictions $R_{a, w, z, \eta}|_{  \NGB_a^{\Lambda}(\eta, w, z) }$, $\imath_{p, \langle \eta \rangle}|_{U_1}$, $(\imath_{p', \langle \eta' \rangle})^{-1} |_{U_3}$ and $(R_{a', w', z', \eta'})^{-1} |_{U_4}$ in the rest of the proof.
 All the norms of various derivatives mentioned in the rest of the proof should be    understood as for such restrictions.

	By  \eqref{eq Rproperty1} to \eqref{eq Rproperty3} in Proposition \ref{prop decompositionofkappaapeta}  and Remark \ref{rem moreonR}, we have 
\ary
	\label{eq partial1whatFtildeF}
 	\|	\partial_w \hat{F}_z   \|   \lesssim     \|	\partial_w  \tilde{F}_{ z}  \| +  \cO_{\tinitial, \esmall}(    \Lambda^{50}    \eta_{max}^{c(\esmall)}  \eta_{min}^{- 1/2 } ).
\eary
	We have  $F_0 = F_1 \circ F_2 $ where $F_1 = \imath_{p'}^{-1}  	\imath_{p''}$, $F_2  =  \imath_{p''}^{-1}   f^{\tinitial} \imath_{p}$  with $p'' = f^{\tinitial}(p)$.  
Denote $F_i = (F_{i, w}, F_{i, z}) = (F_{i, x}, F_{i, y}, F_{i, z}) $ for $i \in \{0, 1,2\}$.

 Denote $\rho = \adscale(\langle \eta \rangle, p)$ and $\rho' = \adscale(\langle \eta' \rangle, p')$. By \eqref{def imathetap} we have
	\ary \label{eq tildeFzexpression}
	\tilde{F}_z(\tw, \tz) = F_{0, z}(\tw, \tz' ) + \tor^{\dagger}(p', \rho') F_{0, x}( \tw, \tz' ) F_{0, y}(\tw, \tz'  ) \ \mbox{ with } \ \tz' =  \tz - \tor^{\dagger}(p, \rho)  \tx \ty.
	\eary
	We may bound the derivatives of $F_1, F_2$ by Proposition \ref{prop normalcoordinatesystem}. Then  we have   
	\ary 
	&& 	\partial_w F_{0, w} = \begin{bmatrix}
		\mu^u(p, \tinitial) & 0 \\ 0 & \mu^s(p, \tinitial)
	\end{bmatrix} +   \cO_{\tinitial, \esmall}(  \Lambda^4  \eta_{max}^{c(\esmall)}  \eta_{min}^{- 1/2} ), \label{eq partialF01}  \\
	&& 	\partial_z F_{0, z} = \mu^c(p, \tinitial) +   \cO_{\tinitial, \esmall}(   \Lambda^4  \eta_{max}^{c(\esmall)}  \eta_{min}^{- 1/2 }),  \label{eq partialF02} \\
  &&
\partial_w F_{0, z},    \partial_{w} \partial_{z} F_{0, z},  \partial^2_{x} F_{0, z}, \partial^2_{y} F_{0, z}  =  \cO_{\tinitial, \esmall}(   \Lambda^4    \eta_{max}^{c(\esmall)}  \eta_{min}^{- 1/2 }),  \label{eq partialF03} \\ 
&&   \| \partial_z F_{0, w} \|,  \| D^2 F_0 \| =   \cO_{\tinitial}(  1  ). \label{eq partialF04} 
	\eary
 We see that  \eqref{eq partialwFz} follows immediately from  \eqref{eq partial1whatFtildeF}, \eqref{eq tildeFzexpression} and \eqref{eq partialF04}.

We now show  \eqref{eq partialw^2Fz}.  
Since $\eta_{min}  \gg_{\tinitial} 1$ in item \eqref{itm 2inL42}, we can replace the error terms     $\cO_{\tinitial, \esmall}( \Lambda^4  \eta_{max}^{c(\esmall)}  \eta_{min}^{- 1/2} )$   in  \eqref{eq partialF01} to \eqref{eq partialF04} by   $ \cO(  \langle \eta \rangle^{-1/4})$.
	Then by Proposition \ref{prop decompositionofkappaapeta}, we have   $ \|	\partial_{w} R_w  \|,  \| \partial_z R_z  \|   = \cO(1)$  and $\| \partial_w R_z \| = \cO(\langle \eta \rangle^{-1/4})$.  Similar estimates hold for $R^{-1} = (\tilde{R}_w, \tilde{R}_z)$ in place of $(R_w, R_z)$. Then
	\ary
	\label{eq partial2whatFtildeF}
	\|	\partial^2_w  \hat{F}_{z}   \|   \lesssim     \|	\partial^2_w  \tilde{F}_{ z}  \|  + \cO( \langle \eta \rangle^{-1/4} ).
	\eary
We may deduce by  \eqref{eq tildeFzexpression} that
	\ary \label{eq 2ndpartialtildeF}
	\| \partial_x^2  \tilde{F}_z  \|,  \| \partial^2_{y} \tilde{F}_z \|  =  \cO( \langle \eta \rangle^{-1/4}).
	\eary

	Recall the notation $\twist(p, m) $ given by \eqref{eq cpmexpression}.
	\begin{lemma} \label{lem 2ndpartialofF0}
		Under \eqref{itm 2inL42},  we have  
\emph{	$\partial_x \partial_y F_{0, z}(\tw, \tz' ) = \twist(p, \tinitial)  +    \cO(   \tinitial^2 r_a(\eta, w, z)^{-1} \Delta_a(\eta, w, z)  )$.	 }
	\end{lemma}
	\begin{proof}  
		 By Proposition \ref{prop normalcoordinatesystem},
	we have $\partial_x \partial_y  F_{2, z} =  \twist(p, \tinitial) + \cO(\langle \eta \rangle^{-1/4})$,  $\| \partial_{w} F_{1, z} \|, \| \partial_{w} F_{2, z} \|   = \cO(\langle \eta \rangle^{-1/4})$, and 
		\aryst
		\partial_z F_{1, z}  = 1 + \cO(\langle \eta \rangle^{-1/4}), \	\partial_w F_{2, w} =  \begin{bmatrix}
			\mu^u(p, \tinitial) & 0 \\ 0 & \mu^s(p, \tinitial)
		\end{bmatrix} + \cO(\langle \eta \rangle^{-1/4}).
		\earyst
		Then we have  $\partial_x \partial_y F_{0, z}  =  \twist(p, \tinitial)   +  \mu^u(p, \tinitial) \mu^s(p, \tinitial) \partial _x \partial_y F_{1, z}  \circ F_2  + \cO(\langle \eta \rangle^{-1/4})$.   
		By construction, we have $( \vartheta^u_{p'} \circ F_1, \vartheta^s_{p'} \circ F_1, 1 ) DF_1 \in \R ( \vartheta^u_{p''}, \vartheta^s_{p''}, 1 )$.
	By Lemma \ref{lem mainapproximation},   Lemma \ref{lem behaviorofDeltalaongorbit},  Lemma \ref{lem adaptedscaleistame} and Lemma \ref{lem torsionareclose},  we deduce $\|  \partial_w F_{1, z} \| = \cO(\tinitial  \Delta_a(\eta, w, z)   )$. 
	By $\| D^3 F_{1, z} \| = \cO( | \eta |^{c(\esmall)})$ and intermediate value theorem,   we see $\|  \partial_w^2 F_{1,z} \|  =  \cO( \tinitial^2   r_a(\eta, w, z)^{-1} \Delta_a(\eta, w, z)    )$.
 	This  concludes the proof.
	\end{proof}

	By \eqref{eq tildeFzexpression}, \eqref{eq partialF03} and 	Lemma \ref{lem 2ndpartialofF0}, we have
	\begin{align}
	& \partial_{x}\partial_{y} \tilde{F}_z(\tw, \tz) 	=  \partial_{x}\partial_{y} F_{0, z}(\tw, \tz') -  \partial_z F_{0, z}(\tw, \tz')    \tor^{\dagger}(p, \rho)  +  \tor^{\dagger}(p', \rho') \partial_{x} F_{0, x}(\tw, \tz')   \partial_{y} F_{0, y}(\tw, \tz')  + \cO(\langle \eta \rangle^{-1/5}  )  \nonumber \\
	&=  \twist(p, \tinitial) - \mu^c(p, \tinitial)  \tor^{\dagger}(p, \rho)  + \mu^u(p, \tinitial) \mu^s(p, \tinitial)  \tor^{\dagger}(p', \rho')  +   \cO(   \tinitial    r_a(\eta, w, z)^{-1} \Delta_a(\eta, w, z) ). \label{eq partialxytildeFzterm0}
	\end{align}
	Then by \eqref{eq torexpression0} and \eqref{eq torexpression02}, we have 
	\aryst
	 \twist(p, \tinitial) =  \mu^c(p, \tinitial)  \tor^{\dagger}(p, \rho) -  \mu^u(p, \tinitial) \mu^s(p, \tinitial) (\tor^s( f^{\tinitial}(p), \rho_s ) + \tor^u( f^{\tinitial}(p), \rho_u ) )/2
	\earyst
	where  $\rho_s = \mu^u(p, \tinitial) \rho$ and $\rho_u = \mu^s(p, \tinitial)  \rho$. 
	Then we can write the expression in \eqref{eq partialxytildeFzterm0} as
	\aryst
  \frac{1}{2}	\mu^u(p, \tinitial) \mu^s(p, \tinitial) ( \tor^{s}(p', \rho') -  \tor^{s}( f^{\tinitial}(p),   \rho_s ) +  \tor^{u}(p', \rho') -  \tor^{u}( f^{\tinitial}(p),    \rho_u ) ) +   \cO(   \tinitial    r_a(\eta, w, z)^{-1} \Delta_a(\eta, w, z)  ).
	\earyst
	By Lemma \ref{lem torsionareclose}, the above line is bounded by $\cO(   \tinitial   r_a(\eta, w, z)^{-1} \Delta_a(\eta, w, z) )$.  This gives  
		\ary   \label{eq 2ndpartialtildeF2} 
	\| \partial_{x} \partial_{y} \tilde{F}_z \| \lesssim   \tinitial    r_a(\eta, w, z)^{-1} \Delta_a(\eta, w, z).
	\eary
	We conclude  the proof of \eqref{eq partialw^2Fz}  by \eqref{eq partial2whatFtildeF}, \eqref{eq 2ndpartialtildeF}  and  \eqref{eq 2ndpartialtildeF2}.
\end{proof}

		 \begin{proof}[Proof of Lemma \ref{lem etaeta'different}]
	By direct computations, we obtain
	\aryst
	\partial_z \hat{F}_z(\tw, \tz) = \partial_z F_z(H_{a, w, z, \eta}^{-1}(\tw, \tz)) + D d_{a', w', z', \eta'}( \hat{F}_w(\tw, \tz) ) \partial_z \hat{F}_w(\tw, \tz).
	\earyst
	Since $\esmall$ is small, we have $1/2 < | \partial_z F_z  | < 2$. 	By \eqref{eq skewproduct}, we have $ \partial_z \hat{F}_w(\tw, \tz) = \cO(e^{C \tinitial} \esmall)$.  
	Note that under either (1) or (2), we always have $| \hat{F}_w(\tw, \tz)  - \fw' | \lesssim e^{C\tinitial}\eta_{max}^{c}  \langle \eta' \rangle^{-1/2}$. Then by Proposition \ref{prop decompositionofkappaapeta}, we obtain $ D d_{a', w', z', \eta'}( \hat{F}_w(\tw, \tz) )  = \cO( e^{C \tinitial} \eta_{max}^{2c}  \langle \eta' \rangle^{-1/2} )$. Thus we have 
	\aryst
	| \eta - \eta' \partial_z \hat{F}_z(\tw, \tz) | \geq | \eta - \eta'  \partial_z F_z(H_{a, w, z, \eta}^{-1}(\tw, \tz)) | -  e^{C \tinitial} \eta_{max}^{1/2+ 3c}.
	\earyst
	By letting $ \eta_{max} \gg_{c, \tinitial} 1$, it is direct to see that $| \eta - \eta' \partial_z \hat{F}_z(\tw, \tz) |  \geq \eta_{max} / 2$. 
\end{proof}

			\subsection{Local geometry of $E^c_*$ in normal central charts}  
			In this section, we will describe the geometry of $E^c_*$ in normal central charts.


			\begin{defi} \label{def estarlocal}
				Given $a \in A$, $\eta \in \R$, $(w, z) \in \widehat \Omega$ and  $(\tw, \tz) \in   H_{a, w, z, \eta}(\widehat\Omega) \cap \widehat\Omega$,
				we will define  $\estar_a(\eta, w, z \mid \tw, \tz) = (\theta^u_a(\eta, w, z \mid \tw, \tz), \theta^s_a(\eta, w, z \mid \tw, \tz)) \in \R^2$ by 
				\begin{align*}
				(\estar_{a}(\eta, w, z \mid \tw, \tz), 1)   := (\estar_{a}( H_{a, w, z, \eta}^{-1}( \tw, \tz )), 1) \cdot D(H^{-1}_{a, w, z, \eta})(\tw, \tz).
				\end{align*}
			\end{defi}
 	By \eqref{eq expressionofHawzeta} in Proposition \ref{prop decompositionofkappaapeta}, we have $\estar_a(\eta, w, z \mid w, z) = (0, 0)$.

Assume that $f^{\tinitial}_{a \to a'}$ is well-defined, and $ \hat{F}$ is given by  \eqref{eq hatF}.
By the invariance of $E^c_*$, we have
 \aryst
 (\estar_{a'}(\eta', w', z' \mid  \hat{F}( \tw, \tz ) ), 1) \cdot D\hat{F}(\tw, \tz) \in \R (\estar_{a}(\eta, w, z \mid \tw, \tz ), 1).
 \earyst
 In other words, we have
 \begin{align} \label{eq pullbackofestarinmodifiedchart}
 & \estar_{a'}(\eta', w', z' \mid  \hat{F}( \tw, \tz ) ) \cdot \partial_w \hat F_w(\tw, \tz) + \partial_w \hat F_z(\tw, \tz)  \\
 =& ( \estar_{a'}(\eta', w', z' \mid  \hat{F}( \tw, \tz ) ) \cdot \partial_z \hat F_w(\tw, \tz) + \partial_z \hat F_z(\tw, \tz)  ) \estar_{a}(\eta, w, z \mid  \tw, \tz ).   \nonumber
 \end{align}
 By Remark \ref{rem moreonR}, we see that an identity parallel to \eqref{eq pullbackofestarinmodifiedchart} holds when we consider $f^{- 1}$ in place of $f$. 
 
%

\detail{

We have 
\aryst
( \estar_a, 1 ) \sim (D \kappa_a^{-1})^* E^c_* = (D \kappa_a^{-1})^*  (D(\imath_p \imath_{p, \langle \eta \rangle})^{-1} )^*   D(\imath_p \imath_{p, \langle \eta \rangle} )^* E^c_*   \\
=  D(R_{a, w, z, \eta} H_{a, w, z, \eta})^*  (\vartheta^u_{ p, \langle \eta \rangle }, \vartheta^s_{p, \langle \eta \rangle }, 1)   =  (DH_{a, w, z, \eta})^*  (DR_{a, w, z, \eta})^*  (\vartheta^u_{ p, \langle \eta \rangle }, \vartheta^s_{p, \langle \eta \rangle }, 1)   
\earyst
Then
\aryst
	(\estar_{a}(\eta, w, z \mid \tw, \tz), 1)  =  (DH_{a, w, z, \eta}^{-1})^*  \estar_a  = (DR_{a, w, z, \eta})^*  (\vartheta^u_{ p, \langle \eta \rangle }, \vartheta^s_{p, \langle \eta \rangle }, 1).
\earyst

Then
\aryst
\estar_{a}(\eta, w, z \mid \tw, \tz) = (\vartheta^u_{ p, \langle \eta \rangle }, \vartheta^s_{p, \langle \eta \rangle })(R_w(\tw, \tz)) \partial_w R_w(\tw, \tz) + \partial_w R_z(\tw, \tz).
\earyst

We let 
\aryst
L_{a, w, z, \eta} = \partial_{z}\partial_w R_{z}
\earyst
}

			\begin{prop}[Main properties of $\estar_a$] \label{lem mainpropertyofestar}
				Let $a \in A$ and $\eta \in \R$. 
				Then   
				\enmt
				\item \label{itm rangeofestar}
			For every $1 \leq C_0 \leq  \langle \eta \rangle^{1/100}$, every $(w, z) \in \widehat\Omega$,  
			and   every $(\tw, \tz)$  in the domain of $\estar_a(\eta, w, z \mid \cdot)$ with $(\tw, \tz) \in   B(w, C_0 r_a(\eta, w, z)) \times B(z, C_0 \langle \eta \rangle^{- \beta_1 + c_0})$,  we have
				\begin{align}
					&  (\widetilde \theta^u_a(\eta, w, z \mid \tw, \tz),  \widetilde \theta^s_a(\eta, w, z \mid \tw, \tz))   :=  \estar_a(\eta, w, z \mid \tw, \tz) \cdot    A_{a}(w, z)^{-1}    \nonumber   \\
				& 
			= D_{a}(w, z)^{-1}	(\vartheta^u_{  \kappa_a^{-1}(w, z),  \langle \eta \rangle },  \vartheta^s_{  \kappa_a^{-1}(w, z),  \langle \eta \rangle })(	A_{a}(w, z)( \tw - w), 0)  +     \cO(  C_0^{7} \langle \eta \rangle^{- 3/4}). \label{eq estaraapprox}
				\end{align}
 	Consequently,    for every $C_1  \in [1, C_0]$, for any other $(\tw', \tz')$  in the domain of $\estar_a(\eta, w, z \mid \cdot)$ with $\| \tw' - \tw \| \leq  C_1 r_a(\eta, w, z)$ and  $| \tz' - \tz | <  \langle \eta \rangle^{-  3 / 4}$,   we have
 	$$\|  ( \estar_a(\eta, w, z \mid \tw, \tz) - \estar_a(\eta, w, z \mid \tw', \tz') )  \cdot A_{a}(w, z)^{-1}  \| \cdot  \Delta_a(\eta, w, z)^{-1} = \cO_{\esmall}(C_0^{c(\esmall)} C_1 ).$$  
 	In particular, we have  	$\estar_a(\eta, w, z \mid \tw, \tz)  \cdot A_{a}(w, z)^{-1} \cdot  \Delta_a(\eta, w, z)^{-1} = \cO_{\esmall}(C_0^{1 + c(\esmall)} )$.
			    \item \label{itm NIinequality}
			Moreover, for any $c > 0$, for sufficiently small $\esmall$, for any $f$ satisfying \eqref{eq fclosetog1} and   
			$({\bf NI})_{\sigmaNI, \CNI, \hat\CNI}$
			for some $\sigmaNI > 0$, $\CNI > 1$ and some $\hat\CNI \gg C_0^{2c} \CNI$ (as in Definition \ref{def NIrho}),
				for any $(w, z) \in \widehat \Omega$, 	for $p = \kappa_a^{-1}(w, z)$,  for an integer $m \geq 1$,   for any $0 < h_0 <1$ and $C_1 \langle \eta \rangle^{-3/4}  < h_1 < 1$  satisfying  
				\aryst
			1/2 \leq	h_0  \mu^u(p, m) \leq 2 \quad \mbox{ and } \quad C^c_0 C_g \ll h_1^{-1} \mu^s(p, m) / \mu^c(p, m)  \ll C_0^{-c} \hat{C}_g,
				\earyst
			for any $(x'', y'', z'') \in B((0, 0), C_0 r_a(\eta, w, z)) \times \langle \eta \rangle^{-3/4}$, we have
				\aryst
		 \frac{1}{2 h_0} \sup_{L \in \R }	\Big |	\int_{- h_0}^{h_0}  \exp( i ( h_1^{-1} \widetilde \theta^s_a(\eta, w, z \mid      w +  A_{a}(w, z)^{-1}(x'' + t, y''), z''      )  + L t  ) )  dt  \Big | < C_1^{-\sigma}.
				\earyst
				We have a parallel statement for $\widetilde \theta^u_a$ in place of $\widetilde   \theta^s_a$.
				\eenmt
			\end{prop}
			
			\begin{proof}  
				Let  $R = (R_w, R_z) = R_{a, w, z, \eta}$ be given by Proposition \ref{prop decompositionofkappaapeta}.
				By definition,   $ (DH_{a, w, z, \eta}^{-1})^*   (D \kappa_a^{-1})^* E^c_*$ contains both $(\estar_{a}(\eta, w, z \mid \cdot ), 1)  =  (DH_{a, w, z, \eta}^{-1})^* ( \estar_a, 1)$ and $D(R_{a, w, z, \eta}  )^*  (\vartheta^u_{ p, \langle \eta \rangle }, \vartheta^s_{p, \langle \eta \rangle }, 1)$ with $p = \kappa_a^{-1}(w, z)$  (see \eqref{def varthetapeta}). 
				Then we have 
				\aryst
				&& (\vartheta^u_{  p,  \langle \eta \rangle },  \vartheta^s_{ p, \langle \eta \rangle }) \circ R \cdot \partial_{w} R_w(\tw, \tz) + \partial_w R_z(\tw, \tz) \\
				&=&			(  (\vartheta^u_{  p,  \langle \eta \rangle },  \vartheta^s_{ p, \langle \eta \rangle } ) \circ R \cdot \partial_{z} R_w(\tw, \tz) + \partial_z R_z(\tw, \tz) ) \estar_{a}(\eta, w, z \mid \tw, \tz).
				\earyst
				By   the proof of Proposition \ref{prop decompositionofkappaapeta}, we see that $\partial_w R_z(w, z) = 0$.
			By \eqref{eq Rproperty1}, \eqref{eq Rproperty3} in Proposition \ref{prop decompositionofkappaapeta}, we have
				\aryst
				 \partial_w R_z(\tw, \tz) = \cO( C^6_0 \langle \eta \rangle^{-1/2 + 1/8} \cdot   C_0 \langle \eta \rangle^{-1/2 + 1/8} + C^4_0 \langle \eta \rangle^{1/8} \cdot C_0 \langle \eta \rangle^{- \beta_1 + c_0} ) = \cO( C_0^7 \langle \eta \rangle^{-3/4} ).
				\earyst
			Then it is direct to  deduce \eqref{eq estaraapprox} in  item \eqref{itm rangeofestar} from Proposition \ref{prop decompositionofkappaapeta} and Lemma \ref{lem mainapproximation}.
 		Then the rest statements of Item  \eqref{itm rangeofestar}  follow from    \eqref{lem mainapproximation0} and \eqref{eq Deltaomegarange}. 
 		By Item   \eqref{itm rangeofestar}, we may reduce Item \eqref{itm NIinequality} to proving
 		$$	 \frac{1}{2 h_0} \sup_{L \in \R }	\Big |	\int_{- h_0}^{h_0}  \exp( i ( h_1^{-1}  D_{a}(w, z)^{-1}  \vartheta^s_{p, \langle \eta \rangle }(     x'' + t, y'',  0    )  + L t  ) )  dt  \Big | < C_1^{-\sigma},$$
 		which follows from   Definition \ref{def torsion} and Definition \ref{def NIrho}.
			\end{proof}

			\section{Dynamical wave-packet transform: Construction and Properties} \label{sec DWPT: Construction and Properties}

			In the rest of the paper, we will deal with various functions on products of Euclidean spaces. 
			In order to distinguish the roles of different variables, we will often write a function $u \in L^2(\R^{n} \times \R^{m})$ for integers $n, m \geq 1$  as $u( p \mid q)$ where $p \in \R^{n}$ and $q \in \R^{m}$.  More generally, given positive integers $n_1, n_2, \cdots, n_k $, we may write a function $u \in L^2(\R^{n_1} \times \cdots \times \R^{n_k})$ as $u(p_1 \mid  \cdots \mid p_k)$ for $p_i \in \R^{n_i}$, $1 \leq i \leq k$.
			Throughout this section, we fix an arbitrary $a \in A$.

%
 Recall that we let $\tinitial  > 1$ be a large integer. 
We denote
\ary
\label{eq extra}
\extra = e^{ - 3 \deltazero  \tinitial}.
\eary
 We set
\ary
&&\Gamma_a = \{  (w, z,  \eta \mid \fw, \fz) \in \widehat\Omega \times  \R \times \widehat\Omega \mid (w, z) \in   \ngb^{1   }_a( \extra \eta, \fw, \fz) \},  \label{eq Gammaa}  \\
&&\hat{\Gamma}_a = \{ (w, z, \xi, \eta \mid \fw, \fz) \in  \widehat\Omega \times \R^{2+1} \times  \widehat\Omega  \mid (w, z,  \eta \mid \fw, \fz)  \in \Gamma_a  \}. \label{eq hatGammaa}
\eary
Note that by Lemma \ref{lem adaptedscaleistame},  \eqref{eq adscalesquarerootgrowth} and \eqref{eq esmallandtinitial},   for every $(w, z, \eta \mid \fw, \fz) \in  \Gamma_a$, we have 
\ary 
\label{eq rafwfzrawzclose}
&& r_a(\eta, w, z) \sim r_a(\eta, \fw, \fz), \\
\label{eq Randrratio}
&& r_{a}(\extra \eta, \fw, \fz) \gtrsim e^{\hat\lambda_g \tinitial} r_{a}(\eta, \fw, \fz).
\eary
This convenient observation will be often used without mentioning.

In this section, we will define a linear operator $\Barg_a $ that sends every $u \in C^{\infty}_{c}(\widehat\Omega)$ to a bounded function $\Barg_a  u$ on $\hat\Gamma_a$. Informally speaking, $\Barg_a u$ describes how $u$ is represented as a superposition of a collection of   localized functions, which may be called \lq\lq dynamical wave-packet\rq\rq, indexed by $\hat\Gamma_a$: 
the function we associate to $(w, z, \xi, \eta \mid \fw, \fz) \in \hat\Gamma_a$ is 
the pull-back by $H_{a, \fw, \fz, \eta}$ of a function of form  
$$(\tw, \tz) \mapsto e^{ i (\xi, \eta) \cdot ((\tw, \tz) - (w, z))   } \widetilde \phi_a^{\dagger}( \eta, w, z \mid \fw, \fz \mid \tw, \tz) $$
where $\widetilde \phi_a^{\dagger}$ is a smooth function with support centered at $(w, z)$.
In the rest of the paper, we will always use  typewriter letters  such as $(\fw, \fz)$ to denote the base point of a normal central chart (see Definition \ref{def normalcentralchart}); use the normal letters such as $(w, z)$ to denote the  center of a wave-packet; and use  letters with tilde  such as $(\tw, \tz)$ to denote the variables of a wave-packet. 

For every $a \in A$, every $\eta \in \R$, every $(\fw, \fz), (w, z) \in \widehat\Omega$ and every $\ell > 0$, we denote 
\begin{align} \label{eq formuaofPhia}  
& \Phi_{a, \ell}(\eta, \fw, \fz \mid w, z) \\
=&   \langle \extra \eta \rangle^{\beta_1 / 2} \langle \langle \extra \eta \rangle^{\beta_1} (z - \fz) \rangle^{- 2 \ell} \cdot  r_{a}( \extra \eta, \fw, \fz)^{- 1}  \langle  r_{a}( \extra \eta, \fw, \fz)^{-1} ((w, z) - (\fw, \fz) ) \rangle^{- 2 \ell}, \nonumber \\
\label{eq formuaofPhia2}        
 &   \widehat\Phi_{a, \ell}(\eta, w, z \mid \tw, \tz) 
= \langle \eta \rangle^{\beta_1/2}  \langle \langle \eta \rangle^{\beta_1} (\tz - z) \rangle^{-  2 \ell} \cdot   r_{a}(\eta, w, z)^{-1}     \langle  r_{a}(\eta, w, z)^{-1} (\tw - w) \rangle^{- 2 \ell}.   
\end{align}
It is clear that we have 
\ary 
   \int   \Phi_{a, \ell}(\eta, \fw, \fz \mid w, z)^2 dw dz,  \int  \widehat \Phi_{a, \ell}(\eta, w, z \mid \tw, \tz)^2 d\tw d\tz  \in ( C(\ell)^{-1},  C(\ell) ).  \label{eq integralphiaupperbound0}
\eary
The following  estimates, whose proofs are deferred to Appendix \ref{app DWPTcomputations}, will be used extensively in Sections \ref{sec proofoflem formerL6.3} to \ref{sec lem formerL6.8}.   
\begin{lemma} \label{lem intPhiadwdz}
	For every integer $\ell > 3$, we have 
	{\em
		\begin{align}
			&	 \int   \Phi_{a, \ell}(\eta, \fw, \fz \mid w, z)^2 d\fw d\fz \lesssim  C(\ell), \label{eq integralphiaupperbound} \\
			& \int \widehat\Phi_{a, \ell}(\eta, w, z \mid \tw, \tz)  \widehat\Phi_{a, \ell}(\eta, w, z \mid \hw, \hz)   dw dz \lesssim  C(\ell)   \langle \langle \eta \rangle^{\beta_1}(\hz - \tz) \rangle^{-\ell} \langle  r_{a}(\eta, \tw, \tz)^{-1}(\hw - \tw) \rangle^{-\ell}. \label{eq integralphiaupperbound2}
		\end{align}
	}
\end{lemma} 

We will introduce some notations to simplify the discussions. We define several differential operators as follows. At $(\eta, w, z \mid \fw, \fz \mid  \tw, \tz ) \in \R \times \widehat\Omega \times \widehat\Omega \times \widehat\Omega$, we define
\aryst
\bD_{w} =  r_a( \eta, w, z) \partial_{w}, \   \bD_{\tw} =  r_a(\eta, w, z) \partial_{\tw},  \  \bD_{\tz}  =   \langle \eta \rangle^{ - \beta_1} \partial_{\tz}.
\earyst

We now state the main result of this section. 
\begin{thm}[Dynamical Wave-Packet Transformation] \label{thm DWPT}
	There exist two collections of functions \emph{$\{ \phi_{a}(\eta, w, z \mid \fw, \fz \mid \cdot)  \}_{(w, z,  \eta \mid \fw, \fz) \in \hat\Gamma_a}$} and \emph{$\{ \phi^{\dagger}_{a}(\eta, w, z \mid \fw, \fz \mid \cdot)  \}_{(w, z,  \eta \mid \fw, \fz) \in \hat\Gamma_a}$},
	a bounded linear operator $\Barg_a :  L^2(\widetilde \Omega)  \to L^2(\hat\Gamma_a)$ given by
	\emph{
		\aryst
		(\Barg_a  u)( w, z, \xi, \eta  \mid  \fw, \fz)   
		= \int e^{- i (\xi, \eta) \cdot (H_{a,  \fw, \fz, \eta}(\tw, \tz) -  (w, z))   }  \phi_a(  \eta, w, z \mid  \fw, \fz \mid \tw, \tz)    u(\tw, \tz) d\tw d\tz,
		\earyst
	}
	and a bounded linear operator $\Barg_a^{\dagger} :  L^2(\hat\Gamma_a) \to L^2(\widetilde \Omega)$  given by  
	\emph{
		\aryst 
		( \Barg_a^{\dagger}  v) (\tw, \tz)  =  \int e^{ i (\xi, \eta) \cdot (H_{a, \fw, \fz, \eta}(\tw, \tz) - (w, z))   }  \phi_a^{\dagger}( \eta, w, z \mid \fw, \fz \mid \tw, \tz)  v( w, z, \xi, \eta \mid  \fw, \fz )  d\fw d\fz dw dz d\xi d\eta
		\earyst
	}
	such that the following are true:
	\enmt
	\item \label{thm DWPT item 1}	Let $\psi : \R^{2+1+2+1} \times \R^{2+1} \to \R$ be a bounded function such that  $\psi \equiv 1$ on   $ \hat{\Gamma}_a$. 
	Then for any $u \in C^{\infty}_c(\Omega)$ we have  $\Barg_a ^{\dagger} \psi \Barg_a u  = u$,
	\item  \label{thm DWPT item 2} Let \emph{ $\widetilde{ \phi}_a(\eta, w, z \mid \fw, \fz \mid  \tw, \tz ) = \phi_a (\eta, w, z \mid \fw, \fz \mid H_{a, \fw, \fz, \eta}^{-1}(\tw, \tz))$ }. 
	Then any {\em$(\eta, w, z \mid \fw, \fz \mid \tw, \tz) \in \supp(\widetilde\phi_a)$} satisfies that {\em $(w, z, \eta \mid \fw, \fz) \in   \Gamma_a$} and  {\em  $\tw \in B( w,   r_a(\eta, \fw, \fz)  / 4  )$}.  
	Moreover for any integers $l, k, \ell \geq 0$,      we have 
	\emph{
	\aryst
&&	|  \bD_{\tz}^l \bD_{\tw}^k  \big( \widetilde{ \phi}_a(\eta, w, z \mid {\fw}, \fz \mid \tw, \tz )  \big)  |  \leq C(l, k, \ell)     \widehat \Phi_{a, \ell}(\eta, w, z \mid \tw, \tz) \Phi_{a, \ell}(\eta, \fw, \fz \mid w, z), 	 \\
&&	 | \bD_{w}^{k}  \widehat{ \phi}_a(\eta, w, z \mid \fw, \fz \mid \tw, \tz )   | \leq C( k, \ell)      ( \langle  \extra \eta \rangle / \langle \eta \rangle)^{k/2}   \widehat \Phi_{a, \ell}(\eta, w, z \mid \tw + w, \tz) \Phi_{a, \ell}(\eta, \fw, \fz \mid w, z) 
	\earyst
	where $\widehat\phi_a(\eta, w, z \mid \fw, \fz \mid \tw, \tz) =  \widetilde{ \phi}_a(\eta, w, z \mid \fw, \fz \mid \tw + w, \tz )$.
}
The same statements hold for  $\widetilde{ \phi}_a^{\dagger} $ in place of  $\widetilde{ \phi}_a$ where   \emph{ $\widetilde{ \phi}_a^{\dagger} (\eta, w, z \mid \fw, \fz \mid \tw, \tz ) =  \phi_a^{\dagger}(\eta, w, z \mid \fw, \fz \mid H_{a, \fw, \fz, \eta}^{-1}(\tw, \tz))$}.
	\eenmt
\end{thm}

From Theorem \ref{thm DWPT}\eqref{thm DWPT item 2}, we see that $\phi_a(\eta, w, z \mid \fw, \fz \mid \cdot)$ and $\phi^{\dagger}_a(\eta, w, z \mid \fw, \fz \mid \cdot)$ are Schwartz functions.   Hence we may define $\Barg_{a}u$ pointwise when $u$ is only a distribution.
 The following is a corollary of Theorem \ref{thm DWPT},  proved using integration-by-parts. We omit this standard argument.  
\begin{cor} \label{cor DWPT item3}
 	For each $s \in \R$, we denote by $\| \cdot \|_{H^s}$ the Sobolev norm of order $s$ on $\R^{3}$.   
Then there is a function $k : \R \to \Z$ with $k(s) \to +\infty$ as $s \to +\infty$ such that for any $s \in \R$, for any \emph{$(w, z, \xi, \eta \mid \fw, \fz) \in \hat\Gamma_a$}, we have
\emph{ $$| \langle \xi \rangle^{k(s)} \langle \eta \rangle^{k(s)} \Barg_a u(w, z, \xi, \eta \mid \fw, \fz) |\leq C(g, s, \tinitial) \| u \|_{H^s}.$$}
The same  statement holds for $(\Barg_a^{\dagger})^*$ in place of $\Barg_a$. 
\end{cor}

Subsections \ref{subsec DWPT step 1} to \ref{subsec DWPT step 3} below are dedicated to the proof of Theorem \ref{thm DWPT}.

			\subsection{Step I: Decomposition along the center} \label{subsec DWPT step 1}

			Given $(\eta, w, z) \in \R \times \widehat\Omega$, we let $H_{a, w, z, \eta}$ be the normal central coordinate given by Definition \ref{def normalcentralchart}.
			Recall that $H_{a, w, z, \eta}$ fixes $(w, z)$ and preserves the volume.

%
%
%
%

			For any $D > 0$, any $(\eta, w, z) \in \R \times \widehat{\Omega}$, we let $\rho_{D, \eta, w, z} : \R^{2+1} \to [0, 1]$ be a $C^{\infty}$  function such that 
			\ary \label{eq  propertyrhoetawz1}
			&& \supp(\rho_{D, \eta, w, z}) \subset \ngb_a^{D}(\eta, w, z), \quad \rho_{D, \eta, w, z} \equiv 1 \mbox{ on } \ngb_a^{D/2}(\eta, w, z)  \\
 \label{eq  propertyrhoetawz2}
		\mbox{ and } &&	\|  \bD_{\tw}^{l_1} \bD_{\tz}^{l_2}  \rho_{D, \eta, w, z}(\tw, \tz) \| \leq C( l_1, l_2 ) D^{- l_1 - l_2 }, \quad  l_1, l_2 \geq 0.
			\eary
			It is clear that such function $\rho_{D, \eta, w, z}$ exists.

			\begin{lemma} \label{lem phidefandproperty}
				There is a $C^{\infty}$   function
				$\chi_a : \R \times  \widehat{\Omega} \times \widetilde{\Omega}  \to \R$  such that the following is true:
				\enmt
				\item  \label{eq integralphi} \emph{  $\int_{\widehat\Omega} \chi_a(\eta, \fw, \fz \mid  \tw, \tz  ) \cdot  r_a( \extra \eta, \fw, \fz)^{-1} \langle \extra \eta \rangle^{\beta_1/2}  d\fw d\fz = 1$ for every $(\eta, \tw, \tz ) \in \R \times  \widetilde{\Omega}$, }
				\item \label{eq supportphi}   
				denote
				 \emph{  $		\widetilde{\chi}_a(\eta,  \fw, \fz \mid \tw, \tz ) := \chi_a(\eta,   \fw, \fz \mid H_{a, \fw, \fz, \eta}^{-1}(\tw, \tz) )$.				 	 } 
				 	 Then \emph{
				 	 	\aryst
				 	 	\supp( \widetilde{\chi}_a(\eta,  \fw, \fz \mid \cdot  ) ) \subset \ngb_a^{1/ 8}(\extra \eta, \fw, \fz),
				 	 	\earyst
				 	 } and for any integer $k \geq 0$, for any $(\tw, \tz) \in \widetilde\Omega$, we have 
				 	 \emph{ 	
				 	 \aryst 
				 	 &&   |   \bD_{\tw}^{k} \big ( \widetilde{\chi}_a(\eta, \fw, \fz \mid \tw, \tz )  \big ) | \leq C( l, k )     (  \langle  \extra \eta \rangle / \langle \eta \rangle )^{k/2}     r_a(\extra \eta, \fw, \fz)^{ - 1} \langle \extra \eta \rangle^{ \beta_1 / 2}.
				 	 \earyst 
				 	 }
				\eenmt
			\end{lemma}
			
			\begin{proof}   
	We define
	\aryst
	D(\eta, \tw, \tz) = \int_{ \widehat{\Omega}}  r_{a}(\extra \eta, \fw', \fz')^{- 2} \langle \extra \eta \rangle^{\beta_1}   \rho_{ 1/8,   \extra  \eta , \fw', \fz' } ( 	H_{a, \fw', \fz', \eta}(\tw, \tz) )  d\fw' d\fz'.
	\earyst
	Note that for any $(\fw', \fz')$ with $\rho_{ 1/8,   \extra  \eta , \fw', \fz' } ( 	H_{a, \fw', \fz', \eta}(\tw, \tz) )   \neq 0$, we have   
	$$(\tw, \tz ) \in H_{a, \fw', \fz', \eta}^{-1}(  \ngb_a^{1/8}( \extra \eta, \fw', \fz') ) \subset  \NGB^{C }_a(\extra\eta, \fw', \fz').$$
	The last inclusion above follows from \eqref{eq Cbound} in Proposition \ref{prop decompositionofkappaapeta},  \eqref{eq esmallandtinitial} and \eqref{eq adscalesquarerootgrowth}.
  By Lemma \ref{lem adaptedscaleistame} and   that $\estar_a$ is $(1 - c(\esmall))$-H\"older,   we deduce $D(\eta, \tw, \tz)  \sim 1$   for every $(\eta, \tw, \tz ) \in \R \times \widetilde{\Omega}$.   We define
				\aryst
				\chi_a(\eta,  \fw, \fz \mid \tw, \tz )  :=    r_{a}(\extra \eta, \fw, \fz)^{- 1} \langle \extra \eta \rangle^{\beta_1 / 2}  \rho_{ 1/8,   \extra  \eta , \fw, \fz } ( 	H_{a, \fw, \fz, \eta}(\tw, \tz) ) / D(\eta, \tw, \tz ).
				\earyst
				Now it is immediate to verify item \eqref{eq integralphi} and the first statement in  item   \eqref{eq supportphi}.    
				

				It remains to verify the second statement of item \eqref{eq supportphi}. 
				The statement is clear for $ l = 0$. 
 			Take $(\fw, \fz) \in \widehat\Omega$ and  $(\tw, \tz) \in  \ngb_a^{1/ 8}(\extra \eta, \fw, \fz)$. 
				  Let  $(\fw', \fz') \in \widehat\Omega$ satisfy that 
				$\rho_{1/8, \extra \eta, \fw', \fz'} (H_{a, \fw', \fz', \eta} H_{a, \fw, \fz, \eta}^{-1}(\tw, \tz) ) \neq 0$.   Then we have
				\ary \label{eq BintersectsB}
			 \NGB_a^{C    }(\extra \eta, \fw, \fz)  \cap   \NGB_a^{C   }(\extra \eta, \fw', \fz')   \neq \emptyset.
				\eary
				By Lemma \ref{lem adaptedscaleistame} and \eqref{eq esmallandtinitial},   $r_a( \extra \eta, \fw, \fz) \sim r_a( \extra \eta, \fw', \fz')$.  
				Note that  
				$$H_{a, \fw', \fz', \eta} H^{-1}_{a, \fw, \fz, \eta}(\tw, \tz) = (\tw, \tz + d_{a, \fw', \fz', \eta}(\tw) - d_{a, \fw, \fz, \eta}(\tw) ). $$
				By  \eqref{eq BintersectsB} , \eqref{eq Cbound} in Proposition \ref{prop decompositionofkappaapeta}, and   that $\estar_a$ is $(1 - c(\esmall))$-H\"older,    we have  
	             \aryst
			 	\| \partial_{\tw} ( d_{a, \fw', \fz', \eta}(\tw) - d_{a, \fw, \fz, \eta}(\tw)  )  \|  \lesssim  \langle \eta \rangle^{ (1 - \beta_1)/10 } \langle \extra \eta \rangle^{-1/2 + c(\esmall)}
				\earyst
				and
				$	\| \partial^{2}_{\tw} ( d_{a, \fw', \fz', \eta}(\tw) - d_{a, \fw, \fz, \eta}(\tw)  )  \|  \lesssim   \langle \eta \rangle^{ (1 - \beta_1)/10 }$ and $	\partial^{l}_{\tw} ( d_{a, \fw', \fz', \eta}(\tw) - d_{a, \fw, \fz, \eta}(\tw)  )  = 0$ if $l \geq 3$.
				Finally, we note that  by \eqref{eq adscalesquarerootgrowth}, 
			 	  $r_a( \eta, \fw, \fz ) /   r_a( \extra \eta, \fw, \fz) \lesssim  (  \langle  \extra \eta \rangle / \langle \eta \rangle )^{1/2 }$.   
				Then we have $ \bD_{\tz}^l  \bD_{\tw}^k D \lesssim C(k, l)    (  \langle  \extra \eta \rangle / \langle \eta \rangle )^{k/2 }   $.
		The second statement of item \eqref{eq supportphi} follows from direct computations.
			\end{proof}

						The following is a variant of \cite[Lemma 1]{GU}, whose proof is deferred to Appendix \ref{app DWPTcomputations}.   
			\begin{lemma} \label{lem decompositiononthereal}
				 There exist functions $\rho_{\parallel}, \rho_{\parallel}^{\dagger}   \in  L^{\infty}(\R, L^2 (\R ))$ such that: 
				\enmt
				\item \label{eq supporth} For every $\eta \in \R$, we have  $\supp(  \rho_{\parallel}^{\dagger}(\eta, \cdot) ) \subset B(0,   \frac{1}{8} \langle \eta \rangle^{-\beta_1})$, 
				\item \label{lem partialtildeh} Let $\widetilde{\rho_{\parallel}}(\eta, z) = e^{- i \eta z} \rho_{\parallel}(\eta, z)$ and $ \widetilde{\rho_{\parallel}}^{\dagger}(\eta, z) = e^{- i \eta z} \rho_{\parallel}^{\dagger}(\eta, z)$.
				Then we have
				\aryst
				| \partial_{z}^{l}  \widetilde{\rho_{\parallel}}^{\dagger}(\eta,  z)|, 
				| \partial_{z}^{l}  \widetilde{\rho_{\parallel}}(\eta, z)| \leq C(l, \delta) \langle \eta \rangle^{\beta_1 l + \beta_1/2 }   \langle \langle \eta \rangle^{\beta_1} z  \rangle^{- l}, \quad  z \in \R, l \geq 0,
				\earyst 
				\item \label{eq VVdagger} 
			The operators $V, V^{\dagger}$ given by
				\aryst
				V u(\eta, z) = \int \rho_{\parallel}(\eta, z - \tz) u(\tz) d\tz, \quad 
				V^{\dagger} v(\tz) = \int \rho_{\parallel}^{\dagger}(\eta, \tz - z)   v(\eta, z) d\eta dz
				\earyst
				extend to bounded linear operators between the  $L^2$ spaces, satisfying $V^{\dagger} V = {\rm Id}$.
				\eenmt
			\end{lemma}
%

			Let  $\rho_{\parallel}^{\dagger}$ and $\rho_{\parallel}$ be given by Lemma \ref{lem decompositiononthereal}. Let $\chi_a$ be given by Lemma \ref{lem phidefandproperty}.
			Denote
			\ary
		\check{\Gamma}_a = \{ (\eta, \fw, \fz \mid \tw, z  ) \mid \eta \in \R, (\fw, \fz) \in \widehat{\Omega}, (\tw, z) \in \ngb_a^{1 / 4}(\extra \eta, \fw, \fz)  \}.
			\eary
			Then we define $\cT_1 : C^{\infty}_{c}(   \Omega ) \to   L^{\infty}(\R \times \widehat\Omega \times (-1, 1)^2  \times \R)$ by 
			\aryst
			\cT_1 u(\eta, \fw, \fz \mid \tw, z ) =    r_a( \extra \eta, \fw, \fz)^{-1} \langle \extra \eta \rangle^{\beta_1/2}  \int  \rho_{\parallel}(\eta, z - \tz -  d_{a, \fw, \fz, \eta}( \tw) )   u(\tw, \tz) d\tz,
			\earyst
			and we define  $\cT_1^{\dagger} :   C^{\infty}_c(\check\Gamma_a)    \to  L^{\infty}(  \Omega)$ by 
			\aryst
			\cT_1^{\dagger} v(\tw, \tz) = \int  \rho_{\parallel}^{\dagger}(\eta,  \tz + d_{a, \fw, \fz, \eta}(\tw) - z )  \chi_a(\eta,  \fw, \fz \mid \tw, z - d_{a, \fw, \fz, \eta}(\tw) )  v(\eta, \fw, \fz \mid \tw, z )   d\eta dz d\fw d\fz.
			\earyst
	Here $d_{a, \fw, \fz, \eta}$ is given by Proposition \ref{lem derivativeboundsforhatF}.  We also denote by $\cT_0: L^{\infty}(\R \times \widehat\Omega \times (-1, 1)^2  \times \R) \to L^{\infty}(\check\Gamma_a)$ the linear operator  associated to the inclusion $\check\Gamma_a \subset \R \times \widehat\Omega \times (-1, 1)^2  \times \R$.

 \detail{
 \aryst
&&  \cT_1^{\dagger} \cT_1 u(\tw, \tz) = \int   \rho_{\parallel}^{\dagger}(\eta,  \tz + d_{a, \fw, \fz, \eta}(\tw) - z )  \chi_a(\eta,  \fw, \fz \mid \tw, z - d_{a, \fw, \fz, \eta}(\tw) ) \\
&&  1_{\check{\Gamma}_a}(\eta, \fw, \fz \mid \tw, z  ) \cT_1 v(\eta, \fw, \fz \mid \tw, z )   d\eta dz d\fw d\fz \\
&& =  \int   \rho_{\parallel}^{\dagger}(\eta,  \tz + d_{a, \fw, \fz, \eta}(\tw) - z )  \chi_a(\eta,  \fw, \fz \mid \tw, z - d_{a, \fw, \fz, \eta}(\tw)  )   1_{\check{\Gamma}_a}(\eta, \fw, \fz \mid \tw, z  )    r_a( \extra \eta, \fw, \fz)^{-1} \langle \extra \eta \rangle^{\beta_1/2}  \\
&&    \rho_{\parallel}(\eta, z - \hz -  d_{a, \fw, \fz, \eta}( \tw) )   u(\tw, \hz) d\hz   d\eta dz d\fw d\fz  \\
&& =  \int   \rho_{\parallel}(\eta, z - \hz   )   \rho_{\parallel}^{\dagger}(\eta,  \tz  - z ) 
\Big [ \int  1_{\check{\Gamma}_a}(\eta, \fw, \fz \mid \tw, z + d_{a, \fw, \fz, \eta}(\tw) )   \chi_a(\eta,  \fw, \fz \mid \tw, z  )    r_a( \extra \eta, \fw, \fz)^{-1} \langle \extra \eta \rangle^{\beta_1/2}  d\fw d\fz \Big ]     \\
&&  u(\tw, \hz) d\hz   d\eta dz  \\
&& =  \int   \rho_{\parallel}(\eta, z - \hz   )   \rho_{\parallel}^{\dagger}(\eta,  \tz  - z ) 
     u(\tw, \hz) d\hz   d\eta dz  = u(\tw, \tz).
 \earyst
 
To verify the second to last equality, we have used that:
 
 1. By Lemma \ref{lem phidefandproperty}\eqref{eq supportphi}, for every $(\eta, \fw, \fz \mid \tw, z )\in \supp(\chi_a)$, we have $(\eta, \fw, \fz \mid \tw, z + d_{a, \fw, \fz, \eta}(\tw) ) \in \check\Gamma_a$,
 
 2.  
 For every $(\tw, \tz) \in \Omega$, and $z$ such that $\rho^{\dagger}_{\parallel}(\eta, \tz - z) \neq 0$,   we have $(\tw, z) \in \widetilde\Omega$. Then by Lemma \ref{lem phidefandproperty}\eqref{eq integralphi}, we have
 $$ \int    \chi_a(\eta,  \fw, \fz \mid \tw, z  )    r_a( \extra \eta, \fw, \fz)^{-1} \langle \extra \eta \rangle^{\beta_1/2}  d\fw d\fz  = 1.$$
 
 }

	It is direct to deduce from Lemma \ref{lem decompositiononthereal} the following.
	\begin{lemma} \label{eq T1isbounded}
						 $\cT_1$ and  $\cT_1^{\dagger}$  extend to   bounded operators from $L^2$ to $L^2$. Moreover, $\cT_1^{\dagger}  \cT_0  \cT_1 = {\rm Id}$. 
	\end{lemma}

	\begin{proof}
		
		By Lemma \ref{lem phidefandproperty}\eqref{eq supportphi}, we see that any $(\eta, \fw, \fz \mid \tw, z) \in \supp(\chi_a)$ satisfies that $(\eta, \fw, \fz \mid \tw, z + d_{a, \fw, \fz, \eta}(\tw)) \in \check\Gamma_a$.	 Moreover, by Lemma \ref{lem decompositiononthereal}\eqref{eq supporth}, for any $(\tw, \tz) \in \Omega$ and any $z \in \R$ with  $\rho^{\dagger}_{\parallel}(\eta, \tz - z) \neq 0$,   we have $(\tw, z) \in \widetilde\Omega$. 
		Using these observations, the identity in the lemma follows from the claim and Lemma \ref{lem phidefandproperty}\eqref{eq integralphi} and Lemma \ref{lem decompositiononthereal}\eqref{eq VVdagger}.

		To prove the boundedness of $\cT_1$ ad $\cT_1^{\dagger}$, we note that  
		\begin{align*}
		\cT_1 u(\eta, \fw, \fz \mid \tw, z) &=    r_a( \extra \eta, \fw, \fz)^{-1} \langle \extra \eta \rangle^{\beta_1/2} ( V u(\tw, \cdot))(\eta, z - d_{a, \fw, \fz, \eta}(\tw)), \\
		(\cT_1^{\dagger})^{*} u(\eta, \fw, \fz \mid \tw, z) &=    \overline{\chi_a(\eta,  \fw, \fz \mid \tw, z - d_{a, \fw, \fz, \eta}(\tw) ) } ((V^{\dagger} )^* u(\tw, \cdot))(\eta, z - d_{a, \fw, \fz, \eta}(\tw)).
		\end{align*}
		Then it is direct to verify the $L^2$-boundedness of these operators by Lemma \ref{lem decompositiononthereal}.
	\end{proof}

			From now on, we will regard $\cT_1$, resp. $\cT_1^{\dagger}$, as operators between the corresponding $L^2$ spaces.

			\subsection{Step II: Decomposition along the hyperbolic directions}
		 \label{subsec DWPT step 2}

			\begin{lemma} \label{lem rhoaperpproperty}
				 There is a $C^{\infty}$ function $\rho^a_{ \perp}:   \R \times \widehat\Omega  \times \R^{2 + 2 + 1}  \to \R$ such that the following are true:
				\enmt
				\item \label{eq w''z''closetow'z'} 
				we have   
				\emph{
				\aryst
				(\rho_{\perp}^a)(\eta, \fw, \fz  \mid  w, \tw, z ) \neq 0    \implies ( w, z) \in   \ngb_{a}^{1}( \extra \eta, \fw, \fz)   \mbox{ and }   \tw \in B( w,    r_a(\eta, \fw, \fz)  / 4  ),
				\earyst
			}
				\item\label{eq integralofrhosquare} 
				we have 
						\emph{
				\ary
					\label{eq integralofrho^222}
						&  		      \sup_{\eta, \fw, \fz, \tw, z} \int (\rho_{\perp}^a)^2(\eta, \fw, \fz \mid w,  \tw, z  ) dw  < C.
				\eary
			}
		 	Moreover, if \emph{ $(\eta, \fw, \fz \mid \tw, z) \in \check\Gamma_a$, } then we have
				\emph{
				\ary \label{eq integralofrho^2}
				&  \int (\rho_{\perp}^a)^2(\eta, \fw, \fz \mid w,  \tw, z  ) dw  = 1,
				\eary
			}
				\item \label{eq rhoaperpsmoothbound} 
				for any integers $l, \ell \geq 0$,    we have
				\emph{
				\aryst
	  \|    \bD_{\tw}^{l} \rho_{\perp}^a(\eta, \fw, \fz \mid w, \tw, z  )  \|  \leq C(l, \ell)     \langle r_a(\eta, \fw, \fz)^{-1}( \tw - w) \rangle^{- \ell} \Phi_{a, \ell}(\eta, \fw, \fz \mid w, z),
				\earyst
			}
				\item  \label{eq rhoaperpsmoothbound2} 
			for any  integers $l, \ell \geq 0$,  we have 
				\emph{
					\aryst
					 \| \bD^l_{w} \widehat\rho_{\perp}^a(\eta, \fw, \fz \mid w,  \tw, z )  \|   \leq C(l, \ell)      ( \frac{\langle  \extra \eta \rangle}{ \langle \eta \rangle})^{l /2}         \langle r_a(\eta, \fw, \fz)^{-1} \tw  \rangle^{-\ell} \Phi_{a, \ell}(\eta, \fw, \fz \mid w, z)
					\earyst
					where $\widehat \rho_{\perp}^a(\eta, \fw, \fz \mid w,  \tw, z) = \rho_{\perp}^a(\eta, \fw, \fz \mid w,  \tw + w, z ) $.
			}
				\eenmt
			\end{lemma}
			
			\begin{proof}

                Given $(\eta, \fw, \fz) \in \R \times \widehat\Omega$,
				we let $\rho^{\perp}_{\eta, \fw, \fz} : \R^{2} \to \R_{\geq 0}$ be a $C^{\infty}$ smooth function such that 
				\aryst
				&&  \supp(\rho^{\perp}_{\eta, \fw, \fz}) \subset  B(0, r_a(\eta, \fw, \fz) / 4   ),  \quad \rho^{\perp}_{\eta, \fw, \fz} \equiv  1  \mbox{ on } B(0, r_a(\eta, \fw, \fz)/ 8) \\
			\mbox{ and }    &&  \| D^{l}   \rho^{\perp}_{\eta, \fw, \fz} \| \leq C(  l) r_a( \eta, \fw, \fz)^{- l }, \quad  l \geq 0.
				\earyst
				It is clear that such function $\rho^{\perp}_{\eta, \fw, \fz}$ exists.

			Let  $\rho_{D, \eta, w, z}$ be as in Subsection \ref{subsec DWPT step 1}.	We define
\aryst
\tilde\rho(\eta, \fw, \fz \mid w,  \tw, z  ) = 
  r_a(\eta, \fw, \fz)^{- 1}       \rho^{\perp}_{\eta, \fw, \fz}( \tw - w   )   \rho_{1, \extra \eta, \fw, \fz}(  w, z ).
\earyst

				Take $(\eta, \fw, \fz \mid \tw, z) \in \check\Gamma_a$.
				We have $(\tw, z) \in \ngb_a^{1/4}(\extra \eta, \fw, \fz)$. 
				For all $w$ with $\rho^{\perp}_{\eta, \fw, \fz}(\tw - w) > 0$, we have $\| \tw - w \| \leq r_a(\eta, \fw, \fz) /  4$.  Then for such $w$ we have $ (w, z)  \in  \ngb_a^{1/2}(\extra \eta, \fw, \fz)$, and as a result,  $\rho_{1, \extra \eta, \fw, \fz}(  w, z )  = 1$.
	  Then  for  any  $(\eta, \fw, \fz \mid \tw, z) \in \check\Gamma_a$
				\ary \label{eq lowerboundD}
				D(\eta, \fw, \fz \mid   \tw,  z ) :=  \int  \tilde\rho^2( \eta, \fw, \fz \mid w,   \tw, z ) dw  \in ( 10^{-3}, 10^3).
				\eary
				By straightforward computations, we obtain 
				\ary \label{eq upperboundDD}
	 	|	\partial_{\tw}^{l} D(\eta, \fw, \fz \mid   \tw,  z )  |  \lesssim C(l) r_a( \extra \eta, \fw, \fz)^{- l } .
				\eary
				We can extend $D$ smoothly so that \eqref{eq lowerboundD} and \eqref{eq upperboundDD} hold for all $(\tw, z)$.\footnote{ For instance, we may set  $	D(  \eta, \fw, \fz \mid   \tw,  z  ) = Q( \int  \tilde\rho^2(  \eta, \fw, \fz \mid w,   \tw, z  ) dw)$ where $Q \in C^{\infty}(\R, \R)$ is  non-decreasing  with $Q(x) = x$ for  $x \in (1/C', C')$;  $Q(x) = 1/(2C')$ for $x \in (-\infty, 1/(2C'))$; and $Q(x) = 2C'$ for $x \in (2C', \infty)$.}
				We set
				\aryst
				\rho^a_{\perp}( \eta, \fw, \fz \mid w,   \tw, z  ) := \frac{\tilde\rho( \eta, \fw, \fz \mid w,   \tw, z   )}{D( \eta, \fw, \fz \mid   \tw,  z  )^{1/2}}.
				\earyst
				
				We now verify items \eqref{eq w''z''closetow'z'} to \eqref{eq rhoaperpsmoothbound2}.
				Fix an arbitrary $(\eta,  \fw, \fz, w,  \tw, z )$ in the support of $\rho_{\perp}^a$. By construction, we have $\rho^{\perp}_{\eta, \fw, \fz}(\tw- w) \neq 0$ and $\rho_{1, \extra \eta, \fw, \fz}(w, z) \neq 0$. Thus $\tw \in B( w,  r_a(\eta, \fw, \fz) / 4 )$ and  $( w, z) \in \ngb^{1}_{a}( \extra \eta, \fw, \fz)$. This gives \eqref{eq w''z''closetow'z'}.
				Clearly we have \eqref{eq integralofrho^2} in \eqref{eq integralofrhosquare}. 
 			It is straightforward to verify   items \eqref{eq rhoaperpsmoothbound} and \eqref{eq rhoaperpsmoothbound2} by \eqref{eq lowerboundD}, \eqref{eq upperboundDD} and direct computations.  
			\end{proof}

			Let $\rho^a_{ \perp}$ be given by Lemma \ref{lem rhoaperpproperty}.
			Then we define $\cT_2 :   C^{\infty}_c( \R \times \widehat\Omega \times (-1, 1)^2 \times \R ) \to  L^{\infty}(\hat\Gamma_a)$ by  
			\aryst 
			\cT_2 u(w, z, \xi, \eta \mid   \fw, \fz) = \int e^{- i \xi (\tw - w) } \rho^a_{\perp}(  \eta, \fw, \fz \mid w,   \tw, z  )    u(\eta, \fw, \fz \mid \tw, z) d\tw,
			\earyst
			and we define  $\cT_2^{\dagger} :    C^{\infty}_c( \hat\Gamma_a ) \to  L^{\infty}(\check\Gamma_a)  $ by
			\aryst
			\cT_2^{\dagger} v (\eta, \fw, \fz \mid \tw, z ) =  \int e^{i \xi (\tw - w)} \rho^a_{\perp}(  \eta, \fw, \fz \mid w,   \tw, z  )    v( w, z, \xi, \eta \mid \fw, \fz )   dw d\xi.
			\earyst

			We have the following lemma.
			\begin{lemma} \label{lem T2isbounded}
 		  $\cT_2$ and  $\cT_2^{\dagger}$  extend to  bounded operators from $L^2$ to $L^2$.  Moreover, for any $\psi$ in  Theorem  \ref{thm DWPT}\eqref{thm DWPT item 1}, we have   $\cT_2^{\dagger} M(\psi) \cT_2 =  \cT_0$. 
			\end{lemma}
			
			\begin{proof}
The boundedness of these operators follow from \eqref{eq integralofrho^222} in Lemma \ref{lem rhoaperpproperty} and the boundedness of Fourier transform and its inverse. The identity follows from \eqref{eq integralofrho^2} in Lemma \ref{lem rhoaperpproperty} and direct computations. We omit the details since they are straightforward.
			\end{proof}

			\subsection{Step III: Putting together}
		 \label{subsec DWPT step 3}
 
			We define $\rho_{\parallel}^{\dagger}$ and $\rho_{\parallel}$ using Lemma \ref{lem decompositiononthereal}; define  $\chi_a$ using  Lemma \ref{lem phidefandproperty}; and define  $\rho^a_{\perp}$ using  Lemma \ref{lem rhoaperpproperty}.
			We define $\cT_1, \cT^{\dagger}_1$ and $\cT_2, \cT^{\dagger}_2$  as in Subsections \ref{subsec DWPT step 1} and \ref{subsec DWPT step 2}.

			Now we  construct $\Barg_a$ and $\Barg_a^{\dagger}$ in  Theorem \ref{thm DWPT}.			 By Lemma \ref{eq T1isbounded} and Lemma \ref{lem T2isbounded}, we can
		    define $\Barg_a :  L^{2}( \Omega)  \to L^{2}(\hat\Gamma_a)$ by
		    \aryst
		    \Barg_a = \cT_2 \cT_1.
		    \earyst
			By straightforward computations, we may express $\Barg_a$  as in Theorem \ref{thm DWPT}  with 
			\begin{align} \label{eq formulaofphi} 
			 \phi_a( \eta, w, z \mid  \fw, \fz \mid \tw, \tz)  =& \widetilde{\rho}_{\parallel}(\eta,  z - \tz -  d_{a,   \fw, \fz, \eta}(\tw) )  \rho^a_{\perp}(\eta, \fw, \fz \mid w,    \tw, z  ) \cdot    r_a( \extra \eta, \fw, \fz)^{-1}   \langle \extra \eta \rangle^{\beta_1/2}. 
			\end{align}
			Define $\Barg_a^{\dagger} :  L^{2}(\hat\Gamma_a) \to  L^{2}(    \Omega)$ by
			\aryst
			\Barg_a^{\dagger} =  \cT_1^{\dagger} \cT_2^{\dagger}.
		   \earyst
			Then by straightforward computations, we may express $\Barg^{\dagger}_a$  as in Theorem \ref{thm DWPT}  with 
			\begin{align} \label{eq formulaofphidagger}
			\phi_a^{\dagger}(  \eta, w, z \mid \fw, \fz \mid \tw, \tz) =    \widetilde{\rho_{\parallel}}^{\dagger}(\eta,  \tz + d_{a,  \fw, \fz, \eta}(\tw) -  z )  \rho^a_{\perp}(\eta,  \fw, \fz \mid w,  \tw, z )  \chi_a(\eta,  \fw, \fz \mid \tw, z - d_{a, \fw, \fz, \eta}(\tw) ).
			\end{align}

			\begin{proof}[Proof of Theorem \ref{thm DWPT}:]
		 
			We deduce from Lemma \ref{lem T2isbounded}, Lemma \ref{eq T1isbounded} and the definitions of $\Barg_a$ and $\Barg_a^{\dagger}$  
			\aryst
			\| \Barg_a \|_{L^2 \to L^2}, \| \Barg_a^{\dagger} \|_{L^2 \to L^2} \leq C.
			\earyst
			 Theorem \ref{thm DWPT}\eqref{thm DWPT item 1} also follows from Lemma \ref{eq T1isbounded}  and Lemma \ref{lem T2isbounded}.
    Theorem \ref{thm DWPT}\eqref{thm DWPT item 2}  follows   from Lemma \ref{lem decompositiononthereal}\eqref{lem partialtildeh}, Lemma \ref{lem phidefandproperty}\eqref{eq supportphi} and Lemma \ref{lem rhoaperpproperty}\eqref{eq rhoaperpsmoothbound}.
			\end{proof}

			\section{Anisotropic Sobolev space} \label{sec Anisotropic Sobolev space}
			
			In this section, we construct a certain Hilbert space $\hH_f$ that is well-adapted to the dynamics of $f$. We follow closely the ideas introduced in \cite{Tsu, FT}.

			\subsection{Weight function and Anisotropic Sobolev space} \label{subsec WeightAnisotropic}


			In this section,
			we  define weight functions ($W_a$ in \eqref{eq defWa}) which are crucial in the construction of the Hilbert spaces mentioned in the introduction. 
			We have some natural descriptions of the weight functions in the regions $\log\langle \eta \rangle \gg \log\langle \xi \rangle$ and $ \log \langle \eta \rangle \ll \log \langle \xi \rangle$ (more precisely, see \eqref{eq Wsigmaatworegions}), and we will choose some appropriate interpolations inbetween. There is no canonical way for making such choices. We will give a construction involving relatively simple expressions \eqref{eq weightfunction}.

	In the rest of the paper, we assume that $g$ has property $({\bf NI})_{\sigmaNI, \CNI}$ for some $\sigmaNI > 0$  and $\CNI > 1$.

			We denote 
			\ary \label{eq defhDelta}
		 \hDelta_{a}(\eta, w, z)    =
		  \langle \eta \rangle 
		 \Delta_a(  \eta, w, z).
			\eary
			By \eqref{eq Deltaomegarange}, \eqref{eq adscalesquarerootgrowth} and \eqref{lem adaptedscalefluctuation}, we see that  
			\ary \label{eq hDeltaomegarange}
	&&		  \hDelta_{a}(\eta, w, z) \lesssim_{\esmall}	\langle \eta \rangle^{1/2+ c(\esmall) }, \\
		  \label{eq Deltaadscale>1}
	&&			 1 \lesssim \hDelta_{a}(\eta, w, z)	r_a(\eta, w, z)   \lesssim_{\esmall} \langle \eta \rangle^{c(\esmall) }.
			\eary

Define $\Bcone_{+} = \{ (x, y) \in \R^2 \mid |y| \leq  |x|  \}$ and $\Bcone_{-}  = \{ (x, y) \in \R^2 \mid |x| <  |y| \}$.
		 	We  define $\Theta^{\pm} : \R^2 \setminus \{(0, 0)\} \to [0, 1]$ by 
		 	$\Theta^{+} = 1_{\Bcone_{+}}$ and $\Theta^{-}= 1_{\Bcone_{-}}$.  
		 	
			For each $a \in A$,   each $Q = (w, z, \xi, \eta \mid \fw, \fz) \in \hat\Gamma_a$, we set  $\zeta = \xi  -   \eta  \estar_a(\eta, \fw, \fz \mid w, z)$, and define
		 	\ary
		 	&& \beta_2 = 3/4, \\
		 	 \label{eq weightfunction}
			&&	W_{a, *}(Q) =    \langle  \zeta / \max(   \hDelta_{a}(\eta, \fw, \fz) ,  \langle \zeta \rangle^{\beta_2} ) \rangle, \\
			&&	 \Theta^{\pm}_a(Q) =  \Theta^{\pm} (  \zeta   ).
			\eary

			By \eqref{eq Gammaa} and \eqref{eq hatGammaa}, we see that 
			if $(w, z, \xi, \eta \mid \fw, \fz) \in \hat\Gamma_a$, then
		$(w, z) \in \ngb^{1 }_a(\extra \eta, \fw, \fz)$.  
		 Then by Proposition \ref{lem mainpropertyofestar}\eqref{itm rangeofestar} and \eqref{eq esmallandtinitial}, we have   
		 \ary \label{etaestarADO1}
    \| \eta \estar_a(\eta, \fw, \fz \mid w, z)      \hDelta_{a}(\eta, \fw, \fz)^{-1}  \|    =  \cO( \extra^{-1} ).
		 \eary
		 Then by comparing $\langle \xi \rangle$ and $\langle \zeta \rangle$, we may deduce from \eqref{eq weightfunction} and  \eqref{eq Deltaomegarange}    that
       	for every $Q = (w, z, \xi, \eta \mid \fw, \fz) \in \hat\Gamma_a$ such that $\log \max( \langle \xi \rangle, \langle \eta \rangle )\gg_{\esmall} \tinitial$, we have
			\begin{align}
  	 W_{a, *}(Q) \label{eq Wsigmaatworegions}  \sim
			\begin{cases}
				\langle  \xi  /     \hDelta_{a}(\eta, w, z)   \rangle, &     \log\langle \eta \rangle \geq (2 + c(\esmall)) \beta_2 \log\langle \xi \rangle, \\
				\langle  \xi       \langle \xi \rangle^{ - \beta_2}  \rangle, &   \log \langle \eta \rangle \leq  (2 - c(\esmall))  \beta_2 \log \langle \xi \rangle.
			\end{cases} 
			\end{align}
			We fix an absolute constant $\alpha_0$ to be determined later, such that
			\ary \label{eq defalpha0}
			 0 < \alpha_0 \ll 1.
			 \eary
			We can define the weight function mentioned in the introduction as
			\ary \label{eq defWa}
			W_a = ( ( \Theta^{+}_a W^{\alpha_0}_{a, *} )^2 + ( \Theta^{-}_a W^{- \alpha_0}_{a, *} )^{2} )^{1/2}.
			\eary
			
			\begin{defi}[Anisotropic norm] \label{def anisotropicnorm}
				For each $a \in A$, we define 
				\aryst 
				  \| v  \|_{\hH_{a}} :=  \|  W_a v  \|_{L^2}, \quad 
		    \| v  \|_{\hH^*_{a}} :=  \| W_a^{- 1} v  \|_{L^2}.
				\earyst
				We define $\hH$, resp.  $\hH^*$, as the Hilbert space obtained as the completion of the direct sum $\bigoplus_{a \in A} C^{\infty}_c( \hat{\Gamma}_a )$ with respect to the norm
				\aryst
				\| (u_{a})_{a \in A} \|_{\hH} :=  \left( \sum_{a \in A}  \|   u_a \|_{\hH_a}^2  \right)^{1/2}, \quad \mbox{resp. } 	\quad			\| (u_{a})_{a \in A} \|_{\hH^*} :=   \left( \sum_{a \in A}  \|   u_a \|_{\hH^*_a}^2  \right)^{1/2}.
				\earyst
			\end{defi}
			We write $\hH$, $\hH^*$ as $\hH_f$, $\hH^*_f$ respectively when we want to emphasize their dependence   on $f$.

			\subsection{Lifted transfer operator} \label{subsec decompositionofthefunctions}
			
			We now define a collection of functions following the construction in \cite[Section 5.2]{Tsu}. 
		Let $\varphi_a$, $\kappa_a$ be defined at the beginning of Section \ref{sec The dual  central bundle in local chart}.	We define
			\aryst
		&&	\bI : C^{\infty}(M) \to \prod_{a \in A}L^{\infty}(\hat\Gamma_a ),  \\
	&& \bI(u) = (  u_{a})_{a \in A} \quad 	\mbox{where}  \quad
		 u_{a} =    \Barg_{a} ( (\varphi_{a}  \cdot u) \circ \kappa_{a}^{-1}).
			\earyst
			Let  $ \widetilde{ \varphi}_{a}$ be defined at the beginning of Section \ref{sec The dual  central bundle in local chart}.
			We define 
			\aryst
			&& \bI^{\dagger} : \prod_{a \in A} L^\infty( \hat\Gamma_a  ) \to  C^{\infty}(M), \\
			&&	\bI^{\dagger} ( (  u_{a})_{a \in A} )
			=  \sum_{a \in A} \widetilde{ \varphi}_{a}  \cdot \Barg^{\dagger}_{a}( u_{a}) \circ \kappa_{a}
			\earyst
The following is true. 
			\begin{lemma} \label{lem resoludeidentite}
				 $\bI$ and $\bI^{\dagger}$ extends to bounded linear operators between $L^2(M)$ and $ \bigoplus_{a \in A} L^2( \hat\Gamma_a  )$. Moreover,
				$\bI^{\dagger}  \circ \bI = \Id$ on $C^\infty(M)$.
			\end{lemma}
			\begin{proof}
				This follows immediately from Theorem \ref{thm DWPT}, \eqref{eq sumofrhoa}, \eqref{eq sumofrhoatilde} and the definitions of $\bI$, $\bI^{\dagger}$.
			\end{proof}
			We define
			\aryst
			\hL = \bI \circ \cL_{f}^{\tinitial}  \circ \bI^{\dagger}. 
			\earyst
			By Lemma \ref{lem resoludeidentite} and \eqref{eq esmallandtinitial}, we have $\| \hL \|_{L^2 \to L^2} \lesssim \| \cL_{f}^{\tinitial} \|_{L^2 \to L^2} \lesssim 1$.
			A simple computation shows that	$\hL  = \sum_{a, a' \in A} \hL_{a \to a'}$ 
		where we set $\hL_{a \to a'} = 0$ whenever $f^{\tinitial}_{a \to a'}$ in \eqref{eq faa'} is not defined; otherwise we set
		\aryst
&&	\hL_{a \to  a'} u = \Barg_{a'} \CC_{( f^{\tinitial}_{a \to a'} )^{-1}}  M(\rho_{a \to a'} )  \Barg^{\dagger}_{a} u_a =   \Barg_{a'} (( \rho_{a \to a'} \Barg^{\dagger}_{a} (u_a ) )\circ (f^{\tinitial}_{a \to a'})^{-1}), \\
	\mbox{ where }	&&	\rho_{a \to a'} =    (\varphi_{a'} \cdot \det D f^{- \tinitial} ) \circ \kappa_{a'}^{-1}   \circ  f^{\tinitial}_{a \to a'} \cdot \widetilde\varphi_{a}    \circ \kappa_{a}^{-1}.
			\earyst 
	  It is clear that 
			\ary \label{eq smoothnessofrho}
			\| D^k \rho_{a \to a'}  \| \lesssim e^{C k \tinitial}, \quad 0 \leq k \leq \greg - 1.
			\eary

			By Lemma \ref{lem resoludeidentite} and straightforward computations,  we have 
			\ary \label{eq duality}
			\langle u_1, u_2 \rangle = \langle  (\bI^{\dagger}) \bI(u_1), u_2  \rangle = \langle \bI (u_1), (\bI^{\dagger})^* u_2 \rangle
			\eary
			where 	  $(\bI^{\dagger})^* :  C^{\infty}(M)  \to \prod_{a \in A} L^\infty( \hat\Gamma_a  )  $ is given by  
			\aryst
			(\bI^{\dagger})^* u =   (    (\Barg_a^{\dagger})^* \CC_{\kappa_a^{-1}}  M(  \widetilde{ \varphi}_{a} )  u    )_{a \in A}.
			\earyst

			We write $\bI$, $\hL$ as $\bI_f$, $\hL_f$ respectively when we want to emphasize their dependences on $f$.
			
 
 For any $D > 0$, we denote $\hat\Gamma_a^{< D} = \{  (w, z, \xi, \eta \mid \fw, \fz) \in \hat\Gamma_a \mid \max(\| \xi \|, | \eta |) < D \}$ and
 we define a projection $\Pi_{D} : \hH_f \to \hH_f$ by
 $ \Pi_{D} (  (  u_{a})_{a \in A}  ) =  (     1_{\hat\Gamma_a^{< D} }  u_{a})_{a \in A} $.
			The following statement follows immediately from Corollary  \ref{cor DWPT item3}.
 			\begin{lemma} \label{lem hHandsmoothnorm}
 	We have the following:
					\enmt
					\item \label{itm 1lem63}
	For every $s \in \R$ and every $D > 0$, we have $\| 1_{\Pi_{D}} \bI(u) \|_{L^{\infty}} \lesssim_{D, s, \tinitial} \| u \|_{H^s(M)}$,
					\item \label{itm 2lem63}
	There exists $s_0 > 0$ such that for any $u \in H^{s_0}(M)$,  we have $\|  (\bI^{\dagger})^*( u ) \|_{\hH^*_f}, \|  \bI ( u ) \|_{\hH_f} \lesssim \| u \|_{H^{s_0}(M)}$.  Consequently, we have  
				$\| u \|_{H^{- s_0}(M)} \lesssim \|  \bI( u ) \|_{\hH_f}$ by duality. 
				\eenmt
			\end{lemma}

			\subsection{Quasi-compactness and uniform spectral gap} \label{subsec Quasi-compactness operator and uniform spectral gap}

			The central result of this paper is the following estimate, whose proof will be deferred to Subsection \ref{subsec proofofthemainlemma}. 
			\begin{thm}[Essential Spectral Gap]\label{lem main}
				There is an integer $\greg \geq 4$ such that the following is true.
					Assume that $g \in  \Anosovflow(M)$  has  property $({\bf NI})_{\sigmaNI, \CNI}$
				for some $\sigmaNI > 0$, $\CNI > 1$. 
%
				Then there exists a constant $\kappa_{g}  > 0$ depending only on $g$ such that if we take $\tinitial > 0$ to be a sufficiently large integer,  and then let $\esmall > 0$ be sufficiently small depending on $g$ and $\tinitial$, and let $\omegalow  > 1$ be sufficiently large depending on $g$ and $\tinitial$, the following claim holds for every $f \in \diff^{\greg}(M)$ with $d_{  \diff^{\greg}(M) }(f, g^{1}) < \esmall$: we have 
				\aryst
				\|   (1 -  \Pi_{\omegalow}) \hL_f  \|_{  \hH_f \to \hH_f }  < \exp(-\tinitial \kappa_{g}).
			\earyst 
			\end{thm}

			We are now ready to deduce Theorem  \ref{thm main} and Theorem \ref{thm uniqueustate} from Lemma \ref{lem hHandsmoothnorm} and Theorem \ref{lem main}.
			\begin{proof}[Proofs of Theorem \ref{thm main} and Theorem \ref{thm uniqueustate}]   
						We denote by $H_f$ the completion of $C^{\infty}(M)$ with respect to the norm $\| u \|_{H_f}  := \| \bI_f (u)  \|_{\hH_f}$.
  By Lemma \ref{lem resoludeidentite}, we have
				\aryst
				\cL_f^{\tinitial} = \bI^{\dagger}_f \circ   \hL_f  \circ \bI_f.
				\earyst
 	Since $\cL_{f}^{\tinitial}$ extends to a bounded operator on $H^s(M)$ for every $s \in \R$, we     deduce from  Theorem \ref{thm DWPT}\eqref{thm DWPT item 2} and Lemma \ref{lem hHandsmoothnorm} that  $\bI^{\dagger}_f   \circ \Pi_{\omegalow} \hL_{f} \circ \bI_f$ is a compact operator on $H_f$.  
				Then by Theorem \ref{lem main}, the transfer operator $\cL_f^{\tinitial}$ extends to a bounded linear operator on $H_f$. 
		  		Moreover, the essential spectral radius of $\cL_f^{\tinitial}$ on $H_f$ is at most $\exp(-  \kappa_{g} \tinitial)$. On the other hand, $1 \in Sp(\cL_f^{\tinitial} : H_f \to H_f)$ because $\langle (1 -  \cL_f^{\tinitial}  ) u , 1 \rangle = 0$ for every $u \in C^{\infty}(M)$.    Thus $\cL_f^{\tinitial}$ has $1$ as an eigenvalue with finite multiplicity. 		
			
				If $f = g^1$, then as $g$ is mixing, there exists some $\kappa'_g > 0$ such that $\cL_{g^1}^{\tinitial}$ has a simple eigenvalue at $1$ with an eigenfunction $1$, and has no other eigenvalue in the complement of $B(0, \exp(- \kappa'_g \tinitial) )$.  We now outline the argument in  \cite[Section 6.5]{Tsu} showing such property persists as we perturbe $f$ away from $g^1$.  Let $W_f$ denote the closure of the collection of $u \in C^{\infty}(M)$ with $\int u dLeb = 0$ in $H_f$. It is clear that $W_f$ is invariant under $\cL_f^{\tinitial}$.
				Assume to the contrary that there exist   $(f_n)_{n=1}^{\infty}$ in $\diff^{\bs}(M)$ converging to $g^1$, and distributions $(u_n \in W_{f_n} )_{n = 1}^{\infty}$ such that $\|  u_n \|_{H_{f_n}} = 1$ and $ \cL^{\tinitial}_{f_n} u_n  = E_n u_n $ for some $E_n \in \C$ with $|E| \geq 1 - 1/n$.   
By Lemma \ref{lem hHandsmoothnorm}\eqref{itm 2lem63}, there is $s_0 > 0$ such that  $\{ u_n \}_{n = 1}^{\infty}$ is  bounded in $H^{-s_0}(M)$. Since the inclusion $H^{-s_0}(M) \subset H^{- 2s_0}(M)$ is compact,  after passing to a subsequence, we may assume that $\lim_{n} u_{n} =  u_{\infty} \in H^{-2s_0}(M)$ and $\lim_{n} E_{n} = E$ with $|E| = 1$.   It is direct to see that $\cL^{\tinitial}_{g^1} u_{\infty} = E u_{\infty}$.  

We claim that $u_{\infty} \neq 0$. To see this, we denote  $U^{0}_{n} =  \Pi_{\eta_*} \bI_{f_n} u_n$ for each $n \geq 1$. 
For all sufficiently large $n$, we have
\aryst
\| U_n^0 \|_{\hH_{f_n}} \geq  \| \bI_{f_n} u_n \|_{\hH_{f_n}}  -   \| (1 - \Pi_{\eta_*}) \bI_{f_n} u_n \|_{\hH_{f_n}}  = 1 -  E_n^{-1} \| (1 - \Pi_{\eta_*}) \hL_{f_n}  \bI_{f_n} u_n \|_{\hH_{f_n}}.
\earyst
 Then we deduce by Theorem \ref{lem main} that for all sufficiently large $n$, we have 
 $$\| U_n^0 \|_{\hH_{f_n}} \geq (1 - e^{ - \kappa_g \tinitial}) /2.$$
 On the other hand, by definition, it is clear that 
 $U_n^{0}$ converges to $\Pi_{\eta_*} \bI_{g^1} u_{\infty}$ uniformly, hence $ \| \Pi_{\eta_*} \bI_{g^1} u_{\infty} \|_{\hH_{g^1}} > 0$. This implies that $u_{\infty} \neq 0$ and gives the claim.
 	However, $\cL^{\tinitial}_{g^1}$ has no eigenvalue $E$ of modulus $1$ in $H^{-2s_0}(M)$ with zero average. This is a contradiction.
			We conclude that for all $f$ sufficiently close to $g^1$ in $\diff^{\bs}(M)$, the operator $\cL_{f}^{\tinitial}$ has a simple eigenvalue at $1$, and no other eigenvalue outside of $B(0, \exp( -  \kappa''_g \tinitial ) )$ for some $\kappa''_g > 0$.

				Denote by $\varphi_f$ the spectral projection of the constant function $1$ to the eigenspace at the eigenvalue $1$. 
				It is standard to see that   $\langle \varphi_f,  1 \rangle = 1$, and the eigenfunction $\varphi_f$ is represented by an $f$-invariant probability measure, denoted by $\nu_f$.  
		  The inequality in Theorem  \ref{thm main} is a standard consequence of the spectral gap of $\cL_{f}^{\tinitial}$ in $H_f$. This finishes the proof of  Theorem  \ref{thm main}.

				To prove Theorem \ref{thm uniqueustate}, we will first show that $\varphi_f$ is the unique u-Gibbs state.  Let $s_0$ be given by Lemma \ref{lem hHandsmoothnorm}.
				We take an arbitrary $v \in C^\infty(M)$,  some integer $n > 0$, and an arbitrary local unstable manifold $\mathfrak{L}$ supported in a single coordinate chart.  Let $m_\mathfrak{L}$ denote the normalized Lebesgue measure on $\mathfrak{L}$.
				By taking the convolution between $m_{\mathfrak{L}} $ and a smooth mollifier, we construct a function $u \in C^{\infty}(M)$ with $\int u dLeb = 1$ such that $u$ is localised in a $e^{-  n \kappa''_g / (2 s_0)}$-neighborhood (with respect to the metric on $M$) of  $\mathfrak{L}$, and $  \| u \|_{H^{s_0}(M)} \lesssim \| u \|_{C^{s_0}} \lesssim e^{n\kappa''_g/ 2}$.  
				Consequently, by letting $\esmall$ be small, we have
				\aryst
				| \int v \circ f^n dm_{\mathfrak{L}}  - \int v \circ f^n \cdot u dLeb |  \lesssim e^{n \esmall -  n \kappa''_g / (2 s_0)}  \|v\|_{C^{1}}  \lesssim  e^{ -  n \kappa''_g / (4 s_0)}  \|v\|_{C^{1}}.
				\earyst
				By Lemma \ref{lem hHandsmoothnorm},  we have    $\|u\|_{H_f} \lesssim   \| u \|_{H^{s_0}(M)} \lesssim e^{n\kappa''_g/ 2}$.
				On the other hand, we have seen that
				\aryst
				| \int v \circ f^n \cdot u dLeb - \int v d\nu_f|  \lesssim e^{ - n \kappa''_g} \|  (\bI^{\dagger})^*  v \|_{\hH_f^*}  \|u\|_{H_f}  \lesssim e^{ -  n \kappa''_g / 2} \|  (\bI^{\dagger})^*  v \|_{\hH_f^*}.
				\earyst
				Thus we have $	|\int v \circ f^n dm_{\mathfrak{L}} - \int v d\nu_f|$ tends to $0$ when $n$ tends to infinity.
				Since $\mathfrak{L}$ is taken arbitrarily, we deduce that $f$ has a unique u-Gibbs state.
				By \cite[Corollary 2]{Dol3} (see also \cite[Corollary 1.2]{CYZ}\footnote{This step is communicated to us by Sylvain Crovisier.} for a parallel statement in $C^1$-regularity),   we deduce that $\nu_f$ is the unique physical measure for $f$ whose basin is of full Lebesgue measure. 
 	Note that the for $f$,  volume is almost exponentially mixing  as defined in \cite{Mal} and \cite{BORH}. By the main result in \cite{BORH}, $\nu_f$ is a SRB measure; and by the main result in \cite{Mal},  $(M, \nu_f, f)$ is Bernoulli.  This finishes the proof of  Theorem \ref{thm uniqueustate}.  
			\end{proof}

			\section{Proof of Theorem \ref{lem main}} \label{sec distanceonphasespace}

			\subsection{A measurement of separation in  phase space} \label{sec distanceonphasespaceNo1}

			Throughout this section, we fix  $a, a' \in A$ so that $F = f^{\tinitial}_{a \to a'}$ is well-defined.  
			Denote $\rho = \rho_{a \to a'}$. Then we have 
			$\hL_{a \to a'} = \Barg_{a'}  \CC_{F^{-1}}   M( \rho )  \Barg_{a}^{\dagger}$.
			The kernel of $\hL_{a \to a'}$, denoted by $\cK$,  is given by
			\ary   \label{eq cK}
	 	\cK( w', z', \xi', \eta', \fw', \fz' \mid w, z, \xi, \eta, \fw, \fz) =   e^{ - i (\xi, \eta) \cdot (w, z) + i (\xi', \eta') \cdot (w', z')}   	   
 	\int  	 	e^{i  S }   u   d\tw d\tz  
			\eary
			where
			\begin{align} \label{eq defhatFandS} 
				& \hat{F} = (\hat{F}_w, \hat{F}_z) = H_{a',  \fw', \fz', \eta'} F H_{a,  \fw, \fz, \eta}^{-1}, \quad		S = S(\xi, \eta, \xi', \eta', \tw, \tz) = (\xi, \eta) \cdot (\tw, \tz)   -  (\xi', \eta') \cdot  \hat{F}(\tw, \tz), \\
						  \label{eq amplitude}
				& u = \widetilde{\phi}_{a'}(  \eta', w', z' \mid  \fw', \fz' \mid \hat{F}(  \tw, \tz ) )         \widetilde{\phi}^{\dagger}_a(  \eta, w, z \mid  \fw, \fz \mid   \tw, \tz )     \rho_F(\tw, \tz - d_{a, \fw, \fz, \eta}(\tw) ).
			\end{align}

			{\noindent {\bf Ideas of the proof:}}
			
			For each $Q = (w, z, \xi, \eta, \fw, \fz) \in \hat\Gamma_a$, the associated wave-packet is supported in a parallelepiped $\check{D}(Q)$ of side-lengths $r_a(\eta, \fw, \fz) \times r_a(\eta, \fw, \fz) \times \langle \eta \rangle^{- \beta_1}$ centered at $(w, z)$,  with Fourier transform supported in a parallelepiped  $\hat{D}(Q)$ of side-lengths $r_a(\eta, \fw, \fz)^{-1} \times r_a(\eta, \fw, \fz)^{-1} \times \langle \eta \rangle^{\beta_1}$. 
Suppose we have a metric $\bf{g}$  on $T^* \R^3$ under which every $\check{D}(Q) \times \hat{D}(Q)$ is within unit size. Then it would be natural to expect that the kernel of the lift of transfer operator decays rapidly as the distance $\widetilde{d}$ between 
			$Q$ and $f^{*}(Q')$
			associated to the metric   $\min\{ {\bf g}, {\bf g}\circ   (f^{-1})^{*} \}$ grows.  
			
		An important role in the proof is played by the  collection of $Q \in \hat\Gamma_a$ with $ \xi  = \eta  \estar_{a}(\eta, \fw, \fz \mid w, z)$,  which we call the \lq\lq {\it dual center locus}\rq\rq  \ since they correspond to $E^c_*$.  This is natural since $E^c_*$ is preserved by  $Df^*$.
		We would like to decompose $\hL_{a \to a'}$ into three parts via dividing the domain of $\cK$.		
		More precisely, each $(Q, Q') \in  \hat\Gamma_a \times  \hat\Gamma_{a'}$ is in one of the following three parts:
			 \enmt
			 \item[Part 1.] both $Q$, $Q'$ are close to the dual center locus, and $\widetilde{d}(Q, Q')$ is small. This part requires fine estimates, and in particular, property $({\bf NI})$. This corresponds to Proposition \ref{lem formerL6.8},
			 \item[Part 2.] one of $Q, Q'$ is far away from the dual center  locus, and $\widetilde{d}(Q, Q')$ is small. 
			 On this part, we will use the weight function and hyperbolicity. 
			 This corresponds to Lemma \ref{lem contractionweight}.
			 \item[Part 3.] $\widetilde{d}(Q, Q')$ is big. On this part the kernel decays rapidly. This corresponds to Proposition \ref{prop offdiagonal}.
			 \eenmt
			 Theorem \ref{lem main} will be deduced by putting together these three types of estimates.

			 Our implementation below will deviate from the above idea by a few technical issues. The main reason is that it is hard to find a metric ${\bf g}$ with the required properties, and the thresholds dividing different cases depend sensitively on the local geometry of $E^c_*$. 
		
			 In the following, we will introduce a measurement $d_f$ of separation of any pair $(Q, Q') \in \hat\Gamma_a \times \hat\Gamma_{a'}$, whose precise definition will appear in  Subsection \ref{subsec PreciseDef}, Definition \ref{def Distancedf}, which plays the role of $\widetilde{d}$ in the abve heuristics.     
			 For each $i \in \{ 1, 2, 3\}$  and each $C > 0$, we will define a subset  $\Trunc_{ (i)}^{\leq C} \subset \hat\Gamma_{a} \times \hat\Gamma_{a'}$.  Roughly speaking,  
		$\Trunc_{ (1)}^{ \leq C}$ stands for a bounded set; 
			 	$\Trunc_{ (2)}^{ \leq C}$ corresponds to $(Q, Q')$ for which $d_f(Q, Q')$ is small; and $\Trunc_{ (3)}^{ \leq C}$ corresponds to $(Q, Q')$ that are both close to the dual center locus.
			 The precise definitions of    $( \Trunc_{ (i)}^{\leq C} )_{i = 1}^{3}$ will be deferred to  Subsection \ref{subsec PreciseDef}.
			 We let $\Trunc_{ (i)}^{ >  C} = ( \hat\Gamma_{a} \times \hat\Gamma_{a'} ) \setminus \Trunc_{ (i)}^{\leq C}$.

Let $\setR_{}$ be the collection of Borel measurable functions $R : \hat\Gamma_{a} \times \hat\Gamma_{a'} \to [0, 1]$. For each $R \in \setR_{}$, we define $\hL^{R}_{a \to a'}$ to be the operator with kernel $\cK^{R}_{}$ given by
\aryst
\cK^{R}_{}(Q' \mid Q) = R(Q, Q') \cK(Q' \mid Q).
\earyst 
We let $\setR_{ (i)}^{> C}$, resp.  $\setR_{ (i)}^{\leq C}$, denote the collection of $R \in \setR_{}$ such that $\supp(R) \subset \Trunc^{> C}_{ (i)}$, resp. $\supp(R) \subset \Trunc^{\leq C}_{ (i)}$.
We will use these functions to  truncate the kernel of $\hL_{a \to a'}$ into pieces, that are estimated separately.

 Recall that $g$  has  property $({\bf NI})_{\sigmaNI, \CNI}$
 for some $\sigmaNI > 0$, $\CNI > 1$.   
		 From now on,  we denote  
		\ary \label{eq deltasharpdefine}
		\deltasharp  =   \min(\sigmaNI, 1)  \lambdag  / 30000.
		\eary
	We now explain the order we choose the parameters in Theorem  \ref{lem main} and related policies:
	
	$\bullet$ we let $\greg \geq 1$ be a large universal constant, and let $\kappa_g > 0$ be depending only on  $g$,

	$\bullet$ we let $\tinitial \geq 1$ be an integer sufficiently large depending on $g$  so that $f^{\tinitial}$ exhibits enough hyperbolicity, and $\tinitial 	\deltasharp  \gg 1$,

	$\bullet$ we let $\esmall > 0$ be a small parameter  chosen freely after $\tinitial$ is given. 	 
	All the implicit constants appeared in the course of the proof will be uniformly   bounded over all $f$ satisfying \eqref{eq fclosetog1},
  
 $\bullet$  we let $\omegalow   > 1$ be a large parameter chosen freely after $\tinitial$ and $\esmall$ are given. In particular, we may require $\log \omegalow \gg \tinitial$.


			 In Section \ref{sec lem formerL6.8}, we will prove the following lemma, which    is analogous to \cite[Lemma 6.8]{Tsu}.
					This is the core of our argument and will use property $({\bf NI})$.

					\begin{prop}[Contraction from Non-Integrability] \label{lem formerL6.8}   
For any  sufficiently large $\tinitial$, there exists $\esmall > 0$ such that for  any  $\omegalow$ sufficiently large depending on $g, \tinitial$, 
there is a function   $R_0 \in \setR_{}$ with $R_0 \equiv 1$    on $\Trunc_{ (1)}^{>  \omegalow } \cap \Trunc_{ (2)}^{< e^{\deltasharp \tinitial / 4} } \cap  \Trunc_{ (3)}^{\leq  e^{\deltasharp \tinitial}}$, such that
							\aryst
							\|  	\hL^{R_0}_{a \to a'}  \|_{L^2 \to L^2}   <   \exp( -   4 \deltasharp   \tinitial ).
							\earyst
						\end{prop}

			 We will prove in Section \ref{sec proofofLemma73} the following.
			 
			 \begin{lemma}[Contraction from Anisotropic weight] \label{lem contractionweight} 
			 For any  sufficiently large  $\tinitial$,  there exists $\esmall > 0$ such that for  any  sufficiently large  $\omegalow$  depending on $g, \tinitial$,  and  any $R \in \setR_{ (1)}^{> \omegalow } \cap \setR_{ (2)}^{< e^{\deltasharp \tinitial / 3}} \cap  \setR_{ (3)}^{> e^{\deltasharp \tinitial}}$  
			 	\ary 
			 	\|  \hL^{R}_{a \to a'}  \|_{\hH_{a} \to  \hH_{a'}} \leq e^{ -    \alpha_0 \deltasharp \tinitial / 4000 } 	\|  \hL^{R}_{a \to a'}  \|_{L^2 \to L^2}.
			 	\eary
			 \end{lemma}

			 	We will prove in Section \ref{sec proofoflem formerL6.3} the following.
			 \begin{prop}[Off-diagonal estimates] \label{prop offdiagonal}
			 	For any integer $l \geq 1$, there is $\greg_0(l) > 0$ such that if $\greg \geq \greg_0(l)$ then the following holds.
			 	For  any sufficiently large $\tinitial$,  there exists $\esmall > 0$ such that  for any  sufficiently large  $\omegalow$  depending on $g, \tinitial$,    any $D \geq e^{\deltasharp \tinitial / 4}$, 
			 	any $R \in  \setR_{ (1)}^{>  \omegalow } \cap \setR_{ (2)}^{\geq  D}$, we have   
			 	$$\| \hL_{a \to a'}^{R} \|_{L^2  \to L^2  }, \| \hL_{a \to a'}^{R} \|_{\hH_{a} \to \hH_{a'} } \leq C( l) D^{- l}.$$
			 \end{prop}

%
%
%
%
%

			\subsection{Precise definitions of $d_f$ and  $( \Trunc_{ (i)}^{\leq C} )_{i = 1}^{3}$} \label{subsec PreciseDef}

			\
			 
			 	Given   $Q = (w, z, \xi, \eta \mid \fw, \fz) \in \hat\Gamma_a$ and   $Q' = (w', z', \xi', \eta' \mid \fw', \fz') \in \hat\Gamma_{a'}$,   
			 we will denote
			 \ary
			 \label{eq Rminmax}	&& R_{min} = \min( r_a( \extra \eta, \fw, \fz), r_{a'}( \extra \eta', \fw', \fz') ), \ \ R_{max} = \max( r_a(\extra \eta, \fw, \fz), r_{a'}( \extra \eta', \fw', \fz') ), \\ 
			 \label{eq rminmax}		&& r_{min} = \min( r_a(\eta, \fw, \fz), r_{a'}(\eta', \fw', \fz') ), \ \ r_{max} = \max( r_a(\eta, \fw, \fz), r_{a'}(\eta', \fw', \fz') ), \\ 
			 \label{eq etaminmax}	&&  \eta_{min} =  \min( \langle \eta \rangle,   \langle \eta' \rangle ), \ \  \eta_{max} =  \max( \langle \eta \rangle,   \langle \eta' \rangle ), \ \ \hDelta_{max} =  \max( \hDelta_a(\eta, \fw, \fz),  \hDelta_{a'}(\eta', \fw', \fz') ).
			 \eary
			 We define $\zeta 	=  \xi  - \eta  \estar_{a}(\eta, \fw, \fz \mid w, z)$. 
			 In other words, $(\zeta, 0)$ is the projection of $(\xi, \eta)$ to $\R^2 \times \{0\}$ along the line $\R (\estar_{a}(\eta, \fw, \fz \mid w, z), 1)$.
			 Similarly, we define $\zeta'  = \xi' -  \eta' \estar_{a'}(\eta', \fw', \fz' \mid w', z')$. Set
			 \ary 
			 \label{eq Deltamax}		&&  \xi_{max} = \max(\langle \xi \rangle, \langle \xi' \rangle), \ \  \zeta_{max} = \max(\langle \zeta \rangle, \langle \zeta' \rangle ).
			 \eary

			\noindent{$\bullet$}  We first describe a region $\Sigma_D$ (see \eqref{def SigmaD}) on which we can exploit  hyperbolicity.   We will see, among other things, that the kernel $\cK$ decays rapidly on $\Sigma_D$ (see Definition \ref{def Distancedf} and Proposition \ref{prop offdiagonal}).  
			Given an arbitrary constant $D > 0$. We define
			\aryst
			\hat\Gamma_a( D )  =  \{ Q = (w, z, \xi, \eta \mid \fw, \fz) \in \hat\Gamma_a \mid  
			| \zeta | \leq D  \hDelta_{a}(\eta, \fw, \fz)  \}.
			\earyst				 
			We define
			\ary \label{eq Trunc0leqD}
				\Trunc_{ (3)}^{\leq D} =  \hat\Gamma_a(D)\times \hat\Gamma_{a'}(D).
			\eary


			We let $\Trunc_{D, -}$ be the set of $(Q, Q') \in \hat\Gamma_a \times \hat\Gamma_{a'}$ satisfying one of the following:
			(1) $\zeta' \in \Bcone_{+}$, $\zeta \in \Bcone_{-}$ and  $\| \zeta' \| \leq \| \zeta \|$,
			(2) both $\zeta'$, $\zeta$ are in $\Bcone_{-}$ and 
			$$ | \langle \zeta, ( 0, 1 ) \rangle | >  \min( e^{\lambda_g \tinitial},  D)^{- 1/100}  | \langle \zeta', (0, 1) \rangle |.$$

			We let $\Trunc_{D, +}$ be the set of $(Q, Q') \in \hat\Gamma_a \times \hat\Gamma_{a'}$ satisfying one of the following:
			(1) $\zeta' \in \Bcone_{+}$, $\zeta \in \Bcone_{-}$ and $\| \zeta' \| >  \| \zeta \|$,
			(2) both $\zeta'$, $\zeta$ are in $\Bcone_{+}$ and 
			$$ | \langle \zeta', (1, 0) \rangle | >  \min( e^{\lambda_g \tinitial},  D)^{ - 1/100}  | \langle \zeta, (1, 0) \rangle |.$$
			The term $\min( e^{\lambda_g \tinitial},  D)$ we put here is convenient for unifying discussions for small $D$ and large $D$.

			We define 
			\ary \label{def SigmaD}
			\Sigma_D = \Trunc_{ (3)}^{ > D} \cap (\Trunc_{D, +} \cup \Trunc_{D, -}).
			\eary 
 We will see that    $\Sigma_D$ corresponds to a subset of Part 3.  
			The precise formulas in the definition of $\Sigma_D$ will be used to deduce Proposition \ref{prop WeightandDf} and Proposition \ref{lem DvSDvF}, but  will not be used in other parts of the paper.

			\noindent{$\bullet$}   	Now we describe the \lq\lq Central\rq\rq  \ region on which our estimates will depend sensitively on the local geometry of $E^{c}_*$. 
Define
			\ary
			\Cent = \{ (Q, Q') \in \hat\Gamma_a \times \hat\Gamma_{a'}  \mid \eta_{max}^{\beta_1 + c_0} >  \xi_{max} \  \mbox{ and } \  \eta_{max}  < 10 \eta_{min} \}.
			\eary
			
			\noindent{$\bullet$} We are ready to define $d_f$, the measurement of separation  between $Q$ and  $Q'$.

			\begin{defi} \label{def Distancedf}
				Given $(Q, Q') \in \hat\Gamma_a \times \hat\Gamma_{a'}$, we define  $d_{f}(Q, Q')$ in the following way.
				
				If $(Q, Q') \in \Cent$, we let $d_{f}(Q, Q')$ denote the biggest constant $1 \le D \leq   \eta_{max}^{c_0}$ such that at least one of the following holds:
				\enmt
				\item  \label{itm distance1}
				{\em  
					either  $\hat{F}( \NGB^{\extra^{- 1}}_a(\eta, \fw, \fz) ) \cap \NGB^{\extra^{- 1}}_{a'}(\eta', \fw', \fz') = \emptyset$,}
				\item     \label{itm distance2}   
				{\em	 $(Q, Q') \in  \Trunc_{ (3)}^{ \leq D} $, and we have $| \eta - \eta' \partial_z \hat{F}_z(\fw, \fz) |   > D  \eta_{max}^{\beta_1} + 3   \xi_{max} $, }
				\item  \label{itm distance3}
				{\em $(Q, Q') \in  \Sigma_D$.}
				\eenmt
				
				If $(Q, Q') \notin \Cent$, we let $d_{f}(Q, Q')$ denote the biggest constant $1 \le D \leq  D_{max} := \max(\xi_{max} / \eta_{max}^{\beta_1}, \eta_{max}^{c_0})$ such that at least one of the following holds:
				\enmt
				\item     \label{itm distance12}
				{\em  
					$\hat{F}( \NGB^{D_{max}}_a(\eta, \fw, \fz) ) \cap \NGB^{D_{max}}_{a'}(\eta', \fw', \fz') = \emptyset$,}
				\item           \label{itm distance22}
				{\em we have $| \eta - \eta' \partial_z \hat{F}_z(\fw, \fz) |   > D  \eta_{max}^{\beta_1} + 3   \xi_{max} $ if $| \eta | \geq | \eta' |$; and $|  \eta  - \eta' \partial_z \hat{F}_z( \hat{F}^{-1}(\fw, \fz) ) |   > D  \eta_{max}^{\beta_1} + 3   \xi_{max} $ if $| \eta | < | \eta' |$, }
				\item    \label{itm distance32}
				{\em $(Q, Q') \in \Sigma_D$.}
				\eenmt
			\end{defi}
			In particular,  any $(Q, Q') \in \Sigma_D$ satisfies $d_f(Q, Q') \geq D$.

						\noindent{$\bullet$} 
					Finally, for every $D > 0$, we define the following subsets of $\hat\Gamma_{a} \times \hat\Gamma_{a'}$:
			\enmt
			\item we let $\Trunc_{ (1)}^{\leq D}$ denote the set of $(Q,Q')$ satisfying $\max( \xi_{max}, \eta_{max}) \leq D$,
			\item we let $\Trunc_{ (2)}^{\leq D}$  denote the set of  $(Q, Q')$ such that   $d_f(Q, Q') \leq D$.
			\eenmt
		We set $\Trunc_{ (i)}^{ >  D} = ( \hat\Gamma_{a} \times \hat\Gamma_{a'} ) \setminus \Trunc_{ (i)}^{\leq D}$, and define $\Trunc_{ (i)}^{ \geq  D}$, $\Trunc_{ (i)}^{ <  D}$ in a similar way.

%
%

			\subsection{Proof of Theorem \ref{lem main}} \label{subsec proofofthemainlemma}
			
We let $\greg = \greg_0(1)$ be given by Proposition \ref{prop offdiagonal}, and let $\deltasharp$ be given by \eqref{eq deltasharpdefine}.
			Let $R_0$ be given by Proposition \ref{lem formerL6.8}. Then  by Proposition \ref{lem formerL6.8} and by \eqref{eq defalpha0}, we have 
		\ary \label{eq typeii} 
			\| \hL^{R_0}_{a \to a'}  \|_{\hH_{a} \to \hH_{a'}}   \leq
	e^{2 \alpha_0 \deltasharp \tinitial} 	\| \hL^{R_0}_{a \to a'}  \|_{L^2 \to L^2}  \leq    \exp( ( 2 \alpha_0 \deltasharp -    4 \deltasharp    )   \tinitial ) \leq  \exp(  -  2  \deltasharp     \tinitial  ). 
		\eary
		Since $R_0 \equiv 1$ on $\Trunc_{ (1)}^{>  \omegalow } \cap \Trunc_{ (2)}^{<  e^{\deltasharp \tinitial / 4}} \cap  \Trunc_{ (3)}^{\leq  e^{\deltasharp \tinitial}}$, we may write  $1_{\hat\Gamma_a \times \hat\Gamma^{> \omegalow}_a}$ as a sum $R_0 + R_1 + R_2$ such that
		$R_1 \in    \setR_{ (1)}^{> \omegalow } \cap \setR_{ (2)}^{\geq   e^{\deltasharp \tinitial / 4} }$ and $R_2 \in  \setR_{ (1)}^{> \omegalow } \cap \setR_{ (2)}^{< e^{\deltasharp \tinitial / 3}} \cap  \setR_{ (3)}^{> e^{\deltasharp \tinitial}} $.
		Then 
		\ary \label{eq splittingintothree} 
		(1 -  \Pi_{\omegalow}  ) \hL_{a \to a'}  =  \hL_{a \to a'}^{R_0}  +  \hL_{a \to a'}^{R_1}  +  \hL_{a \to a'}^{R_2}.
		\eary
 	By Proposition \ref{prop offdiagonal} (with $l = 1$), we have 
		\ary \label{eq typei}
		\| \hL_{a \to a'}^{R_1} \|_{L^2 \to L^2 },	\| \hL_{a \to a'}^{R_1} \|_{\hH_{a} \to \hH_{a'} } \lesssim  e^{- \deltasharp \tinitial / 4}.
		\eary
	By \eqref{eq typeii}, \eqref{eq splittingintothree} and \eqref{eq typei},  we have $\|  \hL^{R_2}_{a \to a'}  \|_{L^2 \to L^2} = \cO(1)$.
 Then by Proposition \ref{lem contractionweight}, we have 
	\ary \label{eq typeiii} 
			 				\|  \hL^{R_2}_{a \to a'}  \|_{\hH_{a} \to \hH_{a'}} \lesssim e^{ -    \alpha_0 \deltasharp \tinitial / 4000 }.
	\eary
  We conclude the proof of Theorem \ref{lem main} for   $\kappa_g = \alpha_0 \deltasharp  / 5000 $    by \eqref{eq typeii}, \eqref{eq typei}  and  \eqref{eq typeiii}, and   by letting $\tinitial \gg  \deltasharp^{-1}$.

  \section{Proof of Lemma \ref{lem contractionweight}}  \label{sec proofofLemma73}
  
  We need the   following  crucial property of $d_f(Q, Q')$.
  
  \begin{prop} \label{prop WeightandDf} 
  	There exist   $C_2 = C_2(g) > 0$, $C_3 = C_3(g) > 0$ and $\kappa' = \kappa'(g) > 0$  such that for any $\tinitial$ sufficiently large, and any $D \geq e^{\deltasharp \tinitial /3 }$, the following is true:
  	\enmt
  	\item 	 \label{itm WQWQ'1}
  	for any $(Q, Q') \in  \Trunc_{ (1)}^{>  \omegalow } \cap  \Trunc_{ (2)}^{< D}$, we have 
  	\ary \label{eq WaQWa'Q'0}
  	W_{a'}(Q')	\leq D^{C_2} W_a(Q).
 	\eary
  	\item  \label{itm WQWQ'2}
  	for any $(Q, Q') \in \Trunc_{ (1)}^{>  \max( \omegalow, D^{C_3} )    } \cap \Trunc_{ (2)}^{<  D} \cap \Trunc_{ (3)}^{ >  D^{3}}$, we have 
  	\ary \label{eq WaQWa'Q'}
  	W_{a'}(Q')	\leq e^{-  \alpha_0 \deltasharp \tinitial / 2000} W_a(Q).
  	\eary
  	\eenmt
  \end{prop}

  \begin{proof} 
  	
  	For  any $C_3 > 0$, any $(Q, Q') \in   \Trunc_{ (1)}^{\leq    D^{C_3}  } \cup   \Trunc_{ (3)}^{\leq  D^{3}}$, we clearly have \eqref{eq WaQWa'Q'0} for some $C_2 = C_2(g, C_3 ) > 0$. 
  	From now on, we assume that $(Q, Q') \in \Trunc_{ (1)}^{> \max( \omegalow, D^{C_3} ) }  \cap \Trunc_{ (2)}^{<  D}  \cap \Trunc_{ (3)}^{ >  D^{3}}$ for some large $C_3 > 0$ to be determined later.
  	We will prove \eqref{eq WaQWa'Q'}, which clearly implies  \eqref{eq WaQWa'Q'0}.

  	By letting $C_3$ be large, we may assume that $D_{max}  = \max(\xi_{max} / \eta_{max}^{\beta_1}, \eta_{max}^{c_0}) > D$. Thus item \eqref{itm distance1} in Definition \ref{def Distancedf}  fails (for $(Q, Q')$ in both $\Cent$ and its complement). 
  	
  	Since item \eqref{itm distance3}  in Definition \ref{def Distancedf}  fails for $D$, one of the following cases hold:
  	\enmt
  	\item[$(i)$]  \label{case 1} $\zeta' \in \Bcone_{-}$ and $\zeta \in \Bcone_{+}$,
  	\item[$(ii)$] \label{case 2} both $\zeta'$, $\zeta$ are in $\Bcone_{-}$ and 
  	$ | \langle \zeta, ( 0, 1 ) \rangle | \leq  \min( e^{\lambda_g \tinitial},  D)^{- 1/100}  | \langle \zeta', (0, 1) \rangle |$,
  	\item[$(iii)$] \label{case 3} both $\zeta'$, $\zeta$ are in $\Bcone_{+}$ and 
  	$| \langle \zeta', (1, 0) \rangle | \leq \min( e^{\lambda_g \tinitial},  D)^{ - 1/100}  | \langle \zeta, (1, 0) \rangle |$.
  	\eenmt
  	We will only detail the proof under  $\|  \zeta \| \geq \|    \zeta'  \|$ since a parallel argument applies when $\|  \zeta \| < \|   \zeta'  \|$.
  It is clear that $(ii)$ fails, and thus we have either $(i)$ or $(iii)$. In both cases, $\zeta \in \Bcone_{+}$.

  	We divide the proof into two cases.

  	\noindent Case I.   Assume that $\zeta_{max}^{\beta_2} \geq   \hDelta_{max}$. 
 	Since $\omegalow$ is large, we have  $\zeta_{max}  \geq  \omegalow^{1/2}$.
It is clear that $W_{a, *}(Q) \sim \langle \zeta \rangle^{1  - \beta_1}$. Moreover, by definition, we have $W_{a', *}(Q') \lesssim \langle   \zeta'  \rangle^{1 - \beta_2}$.
  	
  	Assuming $(i)$,   we have $W_{a}(Q) / W_{a'}(Q') \geq W_{a}(Q) = W_{a, *}(Q)^{\alpha_0} \gtrsim \omegalow^{( 1 - \beta_1 ) \alpha_0/2}$. In particular, \eqref{eq WaQWa'Q'} is clear. Now assume  $(iii)$ holds.  
  	Then  we have $W_{a}(Q) / W_{a'}(Q') = ( W_{a, *}(Q) / W_{a', *}(Q') )^{\alpha_0} \gtrsim (\langle \zeta \rangle / \langle   \zeta'  \rangle)^{( 1 - \beta_2)\alpha_0} \gtrsim  \min( e^{\lambda_g \tinitial},  D)^{( 1 - \beta_2)\alpha_0 / 100}$. 
  	We obtain \eqref{eq WaQWa'Q'} as $D \geq e^{\deltasharp \tinitial /3 }$.

  	\noindent Case II.      Now  assume $\zeta_{max}^{\beta_2}  <   \hDelta_{max}$.    We claim that $\eta_{max}^{\beta_1 + c_0} > \xi_{max}$.  Indeed, assume to the contrary that  $\eta_{max}^{\beta_1 + c_0} \leq \xi_{max}$. 
  	In this case, we have   $\xi_{max} \sim \zeta_{max} >  \omegalow^{\beta_1 + c_0}$.
  	By $\beta_2  > (1/2 + c) / (\beta_1 + c_0)$ for  $c = \beta_2 c_0$, we have 
  	$\hDelta_{max} \lesssim  \eta_{max}^{1/2 + c} \leq     \xi_{max}^{(1/2 + c) / (\beta_1 + c_0) } \ll  \zeta_{max}^{\beta_2}$.
  	This contradicts our hypothesis. The claim is proved. In particular,   $\eta_{max} > \omegalow$ and  $D_{max} = \eta_{max}^{c_0}$.

  	We also claim that $\eta_{max} < 10 \eta_{min}$.   Indeed,  if $\eta_{max} \geq  10 \eta_{min}$, then by $\eta_{max} > \omegalow$ and Lemma \ref{lem etaeta'different}, we see that item \eqref{itm distance22}  in Definition \ref{def Distancedf}  holds, contradicting our hypothesis.

  	By the claims $\eta_{max}^{\beta_1 + c_0} >  \xi_{max}$,  $\eta_{max} < 10 \eta_{min}$ (that is, $(Q, Q') \in \Cent$) and that  item \eqref{itm distance1}  in Definition \ref{def Distancedf}  fails, we must have 
  	$\hat{F}( \NGB^{1/\extra }_a(\eta, \fw, \fz) ) \cap \NGB^{1/\extra }_{a'}(\eta', \fw', \fz') \neq \emptyset$.
  	Then by Lemma \ref{lem behaviorofDeltalaongorbit} and \eqref{eq esmallandtinitial}, we have  $1/\tinitial	\lesssim	\hDelta_a(\eta, \fw, \fz) / \hDelta_{a'}(\eta', \fw', \fz') 
  	\lesssim \tinitial$.
  	Then by $\xi_{max}^{\beta_2}  <   \hDelta_{max}$, we have 
  	\aryst
  1/\tinitial	\lesssim 	W_{a, *}(Q) / \langle \hDelta_a(\eta, \fw, \fz)^{-1}    \zeta \rangle,   W_{a', *}(Q') / \langle \hDelta_{a}(\eta, \fw, \fz)^{-1}   \zeta'  \rangle  \lesssim \tinitial.
  	\earyst
  	Now we can follows a similar argument as in Case I, using the hypothesis that $(Q, Q') \in  \Trunc_{ (3)}^{ >  D^{3}}$ to deduce \eqref{eq WaQWa'Q'}.  This concludes the proof.
  \end{proof}

  We can now start the proof of Proposition \ref{lem contractionweight}.
  Given an integer $m \in \Z$, we set
  \aryst
  \cI_m = \{ Q \in \hat\Gamma_{a}   \mid   W_a(Q) \in [ e^{m},  e^{m+1} )  \}.
  \earyst
  Then $\cJ = \{ \cI_m \}_{m \in \Z}$ is a partition of $\hat\Gamma_a$ into disjoint subsets.  We have
  \aryst
  \| u \|_{\hH_a}^2 \sim \sum_{m \in \Z} e^{2 m} \| \Pi_{\cI_m}  u \|_{L^2}^2.
  \earyst
  Similarly, we define $\cJ' = \{ \cI'_m \}_{m \in \Z}$ for $W_{a'}$ in place of $W_{a}$, and define  $\| u \|_{\hH_{a'}}^2$   similarly.
  
Denote $\kappa' =  \alpha_0 \deltasharp \tinitial / 2000$.
  Without loss of generality, we may assume that $\omegalow > e^{C_3 (1  + \deltasharp  ) \tinitial}$ where $C_3$ is given by Proposition \ref{prop WeightandDf}. By Proposition \ref{prop WeightandDf} (with $D = e^{\deltasharp \tinitial / 3}$) and by letting $\tinitial$ be large, any $m, m' \in \Z$ such that $(\cI_m \times \cI'_{m'}) \cap \Trunc_{ (1)}^{> \omegalow} \cap \Trunc_{ (2)}^{\leq e^{\deltasharp \tinitial / 3}}  \cap  \Trunc_{ (3)}^{> e^{\deltasharp \tinitial}} \neq \emptyset$  must also satisfy $m >  m'  +   \kappa'     \tinitial - 4$.

  Given $m, m' \in \Z$, we denote 
  \aryst
  \hL_{m \to m'}  = \Pi_{\cI'_{m'}} \circ \hL^{R}_{a \to a'}   \circ \Pi_{\cI_m}.
  \earyst
  By the above discussion, we have $\hL_{m \to m'} = 0$ if $m \leq m' +    \kappa'    \tinitial - 4$.
  Given $u \in C^{\infty}_c(\hat\Gamma_{a})$, we have  
  \aryst
  \|  \hL^{R}_{a \to a'}  u \|_{\hH_{a'}}^2 \sim \sum_{m' \in \Z}  e^{2m'} \|  \Pi_{\cI'_{m'}} \circ  \hL^{ R}_{a \to a'}  u \| _{L^2}^2 = \sum_{m' \in \Z}    \|  e^{m'}  \sum_{m > m' +     \kappa'       \tinitial - 4 }  \hL_{m \to m'} u  \| _{L^2}^2.
  \earyst   
  Then by Cauchy's inequality, we have 
  \aryst
  && \| e^{m'}  \sum_{m   > m' +   \kappa'   \tinitial - 4}  \hL_{ m \to m' } u   \|_{L^2}^2  =	\| \sum_{ m   > m' +      \kappa'   \tinitial - 4 }   e^{-m + m'}    \cdot e^{m} \hL_{m \to m'} u   \|_{L^2}^2 \\
  &\lesssim&  ( \sum_{ m   > m' +   \kappa'   \tinitial - 4}   e^{- 2m + 2m'}  )  (  \sum_{m \in  \Z  } e^{2m} \| \hL_{ m \to m' } u \|_{L^2}^2 ) \lesssim e^{ - 2   \kappa'  \tinitial}  \sum_{m \in  \Z  } e^{2m} \| \hL_{ m \to  m' } u \|_{L^2}^2.
  \earyst
  Then we have 
  \aryst
  \|  \hL^{R}_{a \to a'}  u \|_{\hH_{a'}}^2     &\lesssim&     e^{ - 2   \kappa'   \tinitial}   \sum_{m'}  \sum_{m} e^{2 m    } \| \hL_{ m \to  m' } u \|_{L^2}^2 \\
  &=&   e^{ - 2   \kappa'    \tinitial}    \sum_{m} e^{2m} \sum_{m' \in \Z}   \| \hL_{ m \to  m' } u \|_{L^2}^2  = e^{ - 2   \kappa'    \tinitial}    \sum_{m} e^{2m}     \|  \hL^{R}_{a \to a'}   \circ \Pi_{\cI_m} u \|_{L^2}^2.
  \earyst
  The last line above is bounded by
  \aryst
  C e^{ - 2   \kappa'    \tinitial}    	\|  \hL^{R}_{a \to a'}  \|_{L^2 \to L^2}^2  \sum_{m} e^{2m}     \|   \Pi_{\cI_m} u \|_{L^2}^2 \sim	C e^{ - 2  \kappa'   \tinitial}   	\|  \hL^{R}_{a \to a'}  \|_{L^2 \to L^2}^2   \| u  \|_{\hH_{a}}^2.
  \earyst
  This concludes the proof of Lemma \ref{lem contractionweight}.

			\section{Proof of  Proposition \ref{prop offdiagonal}} \label{sec proofoflem formerL6.3}

 We will decompose $\setR_{ (1)}^{>  \omegalow } \cap \setR_{ (2)}^{> D}$ in Proposition \ref{prop offdiagonal} into the following smaller pieces on which $d_f(Q, Q')$ takes comparable values. 			We denote
			\aryst
			&& \Trunc_{D} := \Trunc_{ (1)}^{> \omegalow } \cap ( \Trunc_{ (2)}^{> D} \setminus \Trunc_{ (2)}^{> 2D} ), \\
			&& 	\Trunc_{D, in} = 	\Trunc_{D} \cap  \Trunc_{ (3)}^{\leq 2 D} \ \mbox{ and } \  	\Trunc_{D, out} = 	\Trunc_{D} \cap  \Trunc_{ (3)}^{> D}.
			\earyst
			The main estimate in this section is the following.
					\begin{prop}\label{prop mainprop}
	    	{\em
						For any integer $l \geq 1$, there is $\greg_1(l) > 0$ such that if $\greg \geq \greg_1(l)$ then the following holds. 
						For any sufficiently large $\tinitial$,     any sufficiently large $\omegalow$,   and any $D \geq e^{\deltasharp \tinitial / 4}$
				\enmt
				\item \label{itm 1 Prop81} we have
				\aryst
				\sup_{Q' \in \hat\Gamma_{a'}} \int_{Q : (Q, Q') \in \Trunc_{D, out}}  | \cK(Q' \mid Q) |    dQ,  \  \sup_{Q \in \hat\Gamma_a} \int_{Q' : (Q, Q') \in \Trunc_{D, out}}  | \cK(Q' \mid Q) |   dQ'    \leq C(l) D^{- l},
				\earyst 
				\item \label{itm 2 Prop81} for any  $R \in  \setR_{}$  satisfying  $\| \partial^{k}_{\xi} \partial^{k'}_{\xi'} R \| \leq C(k, k') (D \hDelta_{a}(\eta, \fw, \fz) )^{- k} (D \hDelta_{a'}(\eta', \fw', \fz'))^{- k'}$ for every $k, k' \in \N$ and  with $\supp(R) \subset  \Trunc_{D, in}$, we have
				\aryst
			\| \hL_{a \to a'}^{R} \|_{L^2  \to L^2  } \leq C( l) D^{- l}.
				\earyst
				\eenmt
			}
			\end{prop} 
%
			
			\begin{proof}[Proof of Proposition \ref{prop offdiagonal}]
				To prove the estimate for $\| \hL_{a \to a'}^{R} \|_{L^2  \to L^2  }$, it suffices to note that 
				we can write $R$ as a sum $\sum_{k \in \N} ( R^{(k)}_{ in} + R^{(k)}_{  out} )$ of functions in $\setR_{}$ such that $\supp(R^{(k)}_{  out}) \subset  \Trunc_{2^k D, out} $ and  $\supp(R^{(k)}_{  in}) \subset  \Trunc_{2^{k} D, in}$  and $R^{(k)}_{  in}$ satisfies the derivative bounds  Proposition \ref{prop mainprop}\eqref{itm 2 Prop81}  for $2^{k}D$ in place of $D$.  Then we can apply Proposition \ref{prop mainprop}   and   Schur's test to bound $\| \hL_{a \to a'}^{R} \|_{L^2  \to L^2  }$.
				 By  letting $l$ in Proposition \ref{prop mainprop} be large relative to $C_2$ in \eqref{eq WaQWa'Q'0} of Proposition \ref{prop WeightandDf}, we deduce $\| \hL_{a \to a'}^{R} \|_{\hH_{a} \to \hH_{a'} } \leq C( l) D^{- l / 2}$.
			\end{proof}

		\subsection{Plan of the proof of Proposition \ref{prop mainprop}}

 		 We may write $\cK = \cK_{+}  + \cK_{-} $ where $ \cK_{+}  = \cK   1_{| \eta | \geq | \eta' |}$, $\cK_{-} =  \cK   1_{| \eta | < | \eta' |}$. 
 		 	 We fix some $(Q, Q') \in \Trunc_D$ and denote   $Q = (w, z, \xi, \eta \mid \fw, \fz) \in \hat\Gamma_a$ and   $Q' = (w', z', \xi', \eta' \mid \fw', \fz') \in \hat\Gamma_{a'}$. We will use the notations $R_{max}, R_{min}, \cdots$ in \eqref{eq Rminmax} to \eqref{eq Deltamax}.
 		 Since $\Trunc_D \subset \Trunc_{ (1)}^{> \omegalow }$, we have 
 		 \ary \label{eq ximaxetamaxbig}
 		 \max( \xi_{max}, \eta_{max} ) > \omegalow.
 		 \eary
 		 We will only detail the estimates for $\cK_{+}$ since the estimates for $\cK_{-}$ follow from   a parallel argument.

 		 We start with some general discussions about the set of $(\tw, \tz)$ to consider.
 		By Theorem \ref{thm DWPT}, for the integrand in \eqref{eq cK} to be non-zero, we may assume from now on that 
 	\ary
 	\| \tw - w \| \leq r_a(\eta, \fw, \fz)   \ \mbox{ and } \ 
 	\| \hat{F}_w(\tw, \tz) - w' \| \leq r_{a'}(\eta', \fw', \fz'). \label{eq tw-wboundhatFtwtz-w'bound}
 	\eary
 	
 			 We denote
 	\aryst
 	D_{-1} = 1, \ 	 D_0 =   D^{1/3000}   \ \mbox{ and } \  D_k = 2^k D_0, \ k \geq 1.
 	\earyst
 	We define $R_{-1} = \emptyset$, and define for each integer $K \geq 0$ 
 	\ary \label{eq RK}
 	R_K = \{ (\tw, \tz) \in \widehat\Omega \mid \mbox{\eqref{eq tw-wboundhatFtwtz-w'bound} holds and }   \langle \eta \rangle^{ \beta_1} | z - \tz   |  + \langle \eta' \rangle^{\beta_1} | z' -  \hat{F}_z(\tw, \tz)  |  \leq D_K   \}.
 	\eary
 	Note that $R_K$ depends on $Q, Q'$ through $\hat{F}$. Clearly, $R_{K} \subset R_{K'}$ if $K < K'$.
By  \eqref{eq rafwfzrawzclose}, we have 
 	\ary  \label{eq rangeoftwtzinRK}
 	R_K \subset \ngb_a^{D_K}(\eta, w, z) \cap (\hat F)^{-1}(  \ngb_{a'}^{D_K}(\eta', w', z')  ), \quad K \geq 0.
 	\eary
 	By  definition (see \eqref{eq Gammaa} and \eqref{eq hatGammaa}) and \eqref{eq rafwfzrawzclose},     we have
 	\ary
 	(\tw, \tz) \in \ngb_{a}^{   D_K  / \extra   }(\eta, \fw, \fz) \ \mbox{ and } 	\hat{F}(\tw, \tz) \in \ngb_{a'}^{   D_K / \extra  }(\eta', \fw', \fz'). \label{eq rangeoftwtzinchartsQQ'}
 	\eary
 	By Lemma \ref{lem adaptedscaleistame},  \eqref{eq esmallandtinitial},  \eqref{eq adscalesquarerootgrowth} and $\extra = e^{- 3\deltazero \tinitial}$ (see \eqref{eq extra}),  for any $K \geq 0$ such that $R_K \neq \emptyset$, we have
 	\ary
 	R_{max} / R_{min}, r_{max} / r_{min}   \lesssim     D_K (\eta_{max} / \eta_{min}).   \label{eq Rmaxminrmaxmin}   
 	\eary

 We will use the following technical lemma.
 \begin{lemma} \label{lem rangeoffwfz}
There exist an absolute constant $C' > 0$ and $C = C(g) > 1$ and a ball $W(Q', \eta)$ in $\R^2$ of radius $\cO(D^2_K R )$, and an interval $Z(Q', \eta)$ in $\R$ of length $\cO( D^{C'}_K \langle  e^{C \tinitial} / \eta_{min}^{c_0}  \rangle \langle \extra \eta_{min} \rangle^{- \beta_1} )$ such that every \emph{ $(\fw, \fz)$ }  satisfying $R_K \neq \emptyset$ for some   $(w, z)$ with \emph{$(w, z, \eta \mid \fw, \fz) \in \Gamma_a$} belongs to $W(Q', \eta) \times Z(Q', \eta)$. 
 \end{lemma}
 \begin{proof} 
 	We let $C'$ denote a large absolute constant, and let $C$ denote a large constant depending only on $g$, both of which may vary from line to line.  	We let $(w_0, z_0) = F^{-1}(w', z')$.
 	
 	By \eqref{eq rangeoftwtzinRK}, $R_K \neq \emptyset$ implies that 
$H_{a, \fw, \fz, \eta}^{-1}(\Omega_a^{D_K}(\eta, w, z)) \subset U^{C D_K^2}_a(\eta, w, z)$ intersects 
 	$$F^{-1}( H_{a', \fw', \fz', \eta'}^{-1}(\Omega_{a'}^{D_K}(\eta', w', z')) ) \subset F^{-1}( U_{a'}^{C D_K^2}(\eta', w', z') ).$$
 Thus   $\| (\fw, \fz) - (w_0, z_0) \| \lesssim e^{\deltazero  \tinitial} D_K^2 r_{max}$. 
 
 Define $\widetilde{H}: \R^{2+1} \to \R^{2+1}$ by  $\widetilde{H}(\tw, \tz) = (\tw, \tz + \estar_a(w_0, z_0)(\tw - \fw))$. We denote $B := \widetilde{H}	\circ H_{a, \fw, \fz, \eta}^{-1}(\Omega_a^{D_K}(\eta, w, z)) $.  Denote by $\pi_{z} : \R^{2+1} \to \R$ the projection to the last coordinate.
  By the $(1-c(\esmall))$-H\"older continuity of $\estar_a$,  we see that  $\pi_z(B) \subset B(\fz,   D_K^{C'} \langle  e^{C \tinitial} / \eta_{min}^{c_0}  \rangle   \langle \extra \eta \rangle^{- \beta_1} )$.
For $G = (G_w, G_z) =   H_{a', \fw', \fz', \eta'}  F \widetilde{H}^{-1}  $, we have $G(B) \cap  \Omega_{a'}^{D_K}(\eta', w', z') \neq \emptyset$.    As \eqref{eq partialwFz} in Proposition \ref{lem derivativeboundsforhatF}, we have $\| \partial_w G_z \| =  \cO ( e^{C \tinitial}   D_K^{C'}   \eta_{max}^{c_0}    \eta_{min}^{-1/2 } )$. Since $| \eta | \geq | \eta' |$, we deduce that $\pi_z( G(B) ) \subset B(G_z(\fw, \fz), D_K^{C'} \langle  e^{C \tinitial} / \eta_{min}^{c_0}  \rangle   \langle \extra \eta_{min} \rangle^{- \beta_1})$. 
Then the range of $\fz$ with the above property is contained in an interval of length $\cO( D^{C'}_K \langle  e^{C \tinitial} / \eta_{min}^{c_0}  \rangle \langle \extra \eta_{min} \rangle^{- \beta_1} )$.
 \end{proof}

 	After some preparations in Subsections \ref{subsec Ibp} to \ref{subsec partialsofu}, we will   detail in Subsection \ref{subsec proofitm1prop81}  the estimate for the integrals in  Proposition \ref{prop mainprop}\eqref{itm 1 Prop81}  restricted to $(Q, Q') \in \Cent$ for $\cK_{+}$ in place of $\cK$.
 	The estimate for the integrals in   Proposition \ref{prop mainprop}\eqref{itm 1 Prop81} restricted to $(Q, Q') \in \Cent^c$ follows from a similar argument, and will only be outlined  in Subsection \ref{subsec nonCent}.
 In Subsection \ref{subsec proofitm2prop81}, we will prove Proposition \ref{prop mainprop}\eqref{itm 2 Prop81}.

     	\subsection{Integration-by-parts} \label{subsec Ibp}

		 We let $K_0$ be the largest integer $K \geq 1$ with 
		 \ary \label{eq whatisK0}
		 D^{400}_K  \leq  D_{max} :=   \max( \xi_{max} / \eta_{max}^{\beta_1}, \eta_{max}^{c_0} ).
		 \eary
		 		Note that we have the following.  
		 \begin{lemma} \label{lem RKtoitm1}
		 	If $R_{K_0} \neq \emptyset$, then item \eqref{itm distance1} in Definition \ref{def Distancedf} fails (regardless of whether $(Q, Q') \in \Cent$).
		 \end{lemma}
		 \begin{proof}
		 	If $(Q, Q') \notin \Cent$, then clearly  item \eqref{itm distance1} fails by $D_{max} > D_{K_0}$ and \eqref{eq rangeoftwtzinRK}. Now assume $(Q, Q') \in \Cent$.
 		  We also assume to the contrary that  item  \eqref{itm distance1}  holds.
 		Since $\eta_{max}^{\beta_1 + c_0} >  \xi_{max}$ and  $\eta_{max} < 10 \eta_{min}$,   we have $D_{K_0} \leq \eta_{max}^{c_0}$ and $\eta_{max} > \omegalow$ by \eqref{eq ximaxetamaxbig}. Then by $\beta_1 = 1 - 20 c_0 > 1 - 1/10$  and \eqref{eq adscalesquarerootgrowth}, we  have $| \tz - z | <  D_{K_0} \langle \eta \rangle^{- \beta_1}  \lesssim  r_a(\eta, w, z)$.
 By \eqref{eq adscalesquarerootgrowth} and  $10 \eta_{min} > \eta_{max} \geq \omegalow$, we deduce in a similar way that  $| z' -  \hat{F}_z(\tw, \tz)  |    \lesssim r_{a'}(\eta', w', z')$. 
 By \eqref{eq tw-wboundhatFtwtz-w'bound}, we obtain
 \ary \label{eq rangeotwtzandimage1}
 && (\tw, \tz) \in   B(w, r_a(\eta, \fw, \fz)) \times B(z, C \langle \eta \rangle^{- \beta_1 + c_0}) \subset  \NGB^{C}_a(\eta, w, z) \\
 \mbox{ and  } && \hat{F}(\tw, \tz) \in   B(w', r_{a'}(\eta', \fw', \fz')) \times B(z', C  \langle \eta' \rangle^{- \beta_1 + c_0}) \subset   \NGB^{C}_{a'}(\eta', w', z'). \nonumber 
 \eary
 Then by \eqref{eq rafwfzrawzclose}, we have 
 \ary \label{eq rangeotwtzandimage}
 (\tw, \tz) \in  \NGB^{1/ \extra }_a(\eta, \fw, \fz) \ \mbox{ and  } \hat{F}(\tw, \tz) \in  \NGB^{1/ \extra }_{a'}(\eta', \fw', \fz'). 
 \eary
 Then item \eqref{itm distance1} fails, contradicting the hypothesis. This finishes the proof.
		 \end{proof}
		 
  The following says that   for $(Q, Q') \in \Sigma_D$  (see \eqref{def SigmaD}), $\xi$ and $\xi' \cdot \partial_{w}\hat{F}_w$ are well-separated.

 \begin{lemma} \label{lem DvSDvF}
 	Assume {\em $Q = (w, z, \xi, \eta \mid \fw, \fz) \in \hat\Gamma_a$} and   {\em $Q' = (w', z', \xi', \eta' \mid \fw', \fz') \in \hat\Gamma_{a'}$} satisfy $(Q, Q') \in \Sigma_D$. Then there is a $C^{\infty}$ vector field {\em $V(\eta, \fw, \fz, \eta', \fw', \fz' \mid \cdot)  : \widehat\Omega \to \R^2$} such that
			\ary \label{eq propertiesofvC}
\| V \|_{C^0} \leq 1, \    \|  D V \| \lesssim  e^{ C \tinitial} \ \mbox{ and } \     \|  D^k V \| \lesssim_{k, \tinitial, \esmall}  \eta_{max}^{k c(\esmall)}, k \geq 2
\eary	
and for any $\widetilde \zeta, \widetilde \zeta' \in \R^2$ satisfying
\ary \label{tildezetatildezeta'}
\| \zeta - \widetilde \zeta \|, \| \zeta' - \widetilde \zeta' \|   \leq    \min( e^{\lambda_g \tinitial},  D)^{- 1/100}  \zeta_{max},
\eary
 we have
 	\ary \label{eq gooddirection+}
 	|\langle  \widetilde \zeta -  \widetilde \zeta' \partial_w \hat{F}_w(\tw, \tz), V \rangle|  \gtrsim       \min( e^{\lambda_g \tinitial},  D)^{- 1/100}   \langle \partial_{V} \hat{F}_w(\tw, \tz) \rangle   \zeta_{max}.
 	\eary
 \end{lemma}
		 
\begin{proof}

 	If $(Q, Q') \notin \Trunc_{D, -}$, then by \eqref{def SigmaD}, $(Q, Q') \in  \Trunc_{ (3)}^{ > D}  \cap \Trunc_{D, +}$ and we define $V \equiv (1, 0)$. Then  \eqref{eq propertiesofvC} holds trivially, and  it is direct to verify  that   \eqref{eq gooddirection+} holds.
 	
 Now assume that $(Q, Q') \in  \Trunc_{ (3)}^{ > D}  \cap  \Trunc_{D, -}$.
			We fix a partition of unity on $\widehat\Omega$ by bump functions $\{ \varphi_j  \}_{j \in J}$ with supports of diameters at most  $e^{- C_* \tinitial}$ with some $C_{*} \gg 1$.  In particular, $\sum_{j \in J} \varphi_j \equiv 1$ on $\widehat\Omega$.
			For each $j \in J$, we choose an arbitrary $(w_j, z_j) \in \supp(\varphi_j)$, and  let $V_j = \begin{bmatrix} v_j \\  1 \end{bmatrix}  \in \R^2$ be parallel to $(\partial_w F_w(w_j, z_j) )^{-1}  \begin{bmatrix} 0 \\  1 \end{bmatrix}$. 
			By  Definition \ref{def centralchart}\eqref{itm smoothchartcone},  we have  $| v_j  | \leq 1$, and	$\partial_w  F_w(w_j, z_j) V_{j}  = \begin{bmatrix} 0 \\  \mu_j \end{bmatrix}$   for some $\mu_j \in (0, e^{- \lambda_g \tinitial})$.
			We define
			\aryst
			V(\eta, \fw, \fz, \eta', \fw', \fz' \mid \tw, \tz) :=   \sum_{j \in J} \varphi_j( \tw, \tz - d_{a, \fw, \fz, \eta}(\tw)) V_j.
			\earyst
			When there is no ambiguity, we omit the parameters $\eta, \fw, \fz, \eta', \fw', \fz'$ from the notations. 	By construction, we have $V =  \begin{bmatrix} v \\  1 \end{bmatrix}$ where $ v(  \tw, \tz) =  \sum_{j \in J} \varphi_j( \tw, \tz - d_{a, \fw, \fz, \eta}(\tw)) v_j$.		
			Then it is direct to verify \eqref{eq propertiesofvC}. 
			By letting $C_* \gg 1$, and letting $\tinitial \gg_{C_*} 1$, we have  
			\ary \label{eq Vcontracted0C}
			\|  \partial_{w}  F_w(\tw, \tz - d_{a, \fw, \fz, \eta}(\tw))  \cdot V(\tw, \tz) \| \leq   e^{- \lambda_g \tinitial / 2} < 1  \mbox{ if $(\tw, \tz- d_{a, \fw, \fz, \eta}(\tw)) \in \widehat\Omega$}.
			\eary 
	Recall that on $ \Trunc_{ (3)}^{ > D} $, we have $\zeta_{max} > D \hDelta_{max}$.
	 Restricted to $ \Trunc_{D, -}$, by \eqref{tildezetatildezeta'}, we have $\langle \widetilde\zeta \rangle \sim \langle \zeta \rangle$.
	  Then it is direct to verify $ 	|\langle \widetilde \zeta -  \widetilde \zeta' \partial_w \hat{F}_w(\tw, \tz), V \rangle|  \gtrsim  \langle  \zeta  \rangle$.  Since we have $\langle \zeta \rangle \gtrsim  \min( e^{\lambda_g \tinitial},  D)^{- 1/100} \langle \zeta' \rangle$ for  $(Q, Q') \in \Trunc_{D, -}$,  \eqref{eq gooddirection+} is clear.
\end{proof}

	
 	 By \eqref{eq skewproduct}, \eqref{eq rangeoftwtzinchartsQQ'}, \eqref{eq propertiesofvC}, Remark \ref{rema DwhatFw} and   Proposition \ref{lem derivativeboundsforhatF},   
	we deduce  
\ary
&&	\|  \partial_z^l \partial_{V}^k \hat{F} \|_{U_a^{D_K / \extra }(\eta, \fw, \fz)} \lesssim_{\tinitial, l, k,   \esmall}     D_K^{k+l+2} \eta_{max}^{(l + k) c(\esmall)}, 	\label{eq partialsFbound0}   \\
&& 	| \partial_z \hat{F}_z(\tw, \tz) |  \leq 3   +  	   C(\esmall ) \extra^{-1} D_K \eta_{min}^{-1/2 + c(\esmall)}, \label{eq partialzFzbound0} \\
&&							| \partial_z \hat{F}_w(\tw, \tz) | \leq  e^{C\tinitial}\esmall \leq 1/2, \label{eq partialFbound0} \\
&&												| \partial_V \hat{F}_z(\tw, \tz) | \lesssim_{\tinitial,  \esmall}   D_K^{50}   \eta_{max}^{c(\esmall)}    \eta_{min}^{-1/2}. \label{eq partialwFzbound0}
\eary

			We define
			\begin{align*}
			&	\DV = r_{min} \partial_{V}, \ \ \  \DZ = \eta_{max}^{-  \beta_1} \partial_{\tz}, \\
			\mbox{ and } \ \ \  &	\cD_{V} =  \frac{1 - i  \DV S \cdot \DV }{1 +   \| \DV S \|^2 },   \ \ \ 
				\cD_{\tz} =   \frac{1 - i \DZ S \cdot \DZ  }{1 +     | \DZ S |^2 }.
			\end{align*}
			Then $\cD_{V} (\exp(iS)) = \cD_{\tz} ( \exp(iS) ) = \exp(iS)$. Moreover, by direct computations, we see that 
			\begin{align*}
				\cD_{V}^* &=  \langle \DV S \rangle^{-2}   -    \DV  ( i \langle \DV S \rangle^{-2} \DV S   )   -    i \langle \DV S \rangle^{-2} \DV S \cdot  \DV,   \\
				\cD_{\tz}^* &=    \langle \DZ S \rangle^{-2}   -  \DZ  ( i \langle \DZ S \rangle^{-2} \DZ S   )  -    i \langle \DZ S \rangle^{-2} \DZ S  \DZ.
			\end{align*}

			 We want to exploit the oscillation of $S$ restricted to the set of $(\tw, \tz) \in R_K$, and apply integration-by-parts with respect to $\cD_{V}$ and $\cD_{\tz}$ to the expression of $\cK$ in \eqref{eq cK}. 
			We also need to estimate the partial derivatives of various terms in \eqref{eq cK}. 
			In Subsection \ref{subsec oscphases}, we will estimate the coefficients in the differential operators  $(\cD_{\tz}^*)^{l} (\cD_{V}^*)^k $, which are expressed in terms of derivatives of $S$.
			In Subsection \ref{subsec partialsofu} we will estimate the derivatives of $u$ with respect to $\DZ^l \DV^k$. The following tableau summarises the structure of our argument. 
\begin{table}[h]
	\centering
				\begin{tabular}{|c|c|c|}
					\hline
					\textbf{Region} & \textbf{Derivatives of $S$} & \textbf{Derivatives of $u$} \\ \hline
					$\Trunc_{D, out} \cap  \Cent$   &  Lemma \ref{lem smallK} & Corollary \ref{cor partialofphia'hatF}\eqref{itm 1911}  \\ \hline
					$ ( \Trunc_{D, out} \setminus  \Cent ) \cap \Sigma_{D} $  & Lemma \ref{lem smallK2} & Corollary \ref{cor partialofphia'hatF}\eqref{itm 1911}  \\ \hline
					$ ( \Trunc_{D, out} \setminus  \Cent ) \setminus \Sigma_{D} $  & Lemma \ref{lem smallK0} & Corollary \ref{cor partialofphia'hatF}\eqref{itm 2911}  \\ \hline
					$\Trunc_{ (3)}^{\leq D}$ & Lemma \ref{lem smallK0} &  Corollary \ref{cor partialofphia'hatF}\eqref{itm 2911} \\   \hline
				\end{tabular}
			\end{table}

					\subsection{Oscillatory phases} \label{subsec oscphases}

			Restricted to $R_K \setminus R_{K-1}$ with $K \geq K_0$,  we will only need the following (rather loose) estimate.
			
			\begin{lemma}  \label{lem bigK} 
			 		Let $(\tw, \tz) \in R_K$  for  some  integer $K  \geq K_0$. Then    for   any  $(l, k) \in \N^2 \setminus \{(0, 1), (1, 0), (0, 0)\}$ 
				\aryst
				 \|	 \DZ^l \DV^k S(\tw, \tz) \| \lesssim_{k, l}  D_K^{ k + l  + 403}.
			\earyst
			\end{lemma}
			\begin{proof}
	We deduce by \eqref{eq partialsFbound0} that
				\aryst
				 	 \|	  \DZ^l \DV^k S(\tw, \tz) \|	&\lesssim_{\tinitial, l, k,  \esmall}&  \eta_{max}^{-(1/2 - c(\esmall))k - l \beta_1 }	 \cdot  D_K^{k + l + 2} \eta_{max}^{ (k + l )c(\esmall)} 	(\| \xi' \| + \| \eta' \|)    \\
				&\lesssim_{\tinitial, l, k, \esmall}&     D_K^{k + l + 2} (\xi_{max} \eta_{max}^{- 1} +1 )   \eta_{max}^{ 1 + (k + l )c(\esmall) -(1/2 - c(\esmall))k - l \beta_1 } \\
				&\lesssim_{\tinitial, l, k,  \esmall}&    D_K^{k + l + 402}     \eta_{max}^{ 1 + (k + l )c(\esmall) -(1/2 - c(\esmall))k - l \beta_1 }.
				\earyst
				We conclude the proof using that $D_K \geq D_{K_0} \gtrsim \omegalow^{c_0/400}$, $k + l > 1$ and that $\omegalow$ can be chosen freely independent of $\tinitial, \esmall$.
			\end{proof}

			We now estimate the oscillation of $S$ restricted to $R_K$ with $0 \leq K   \leq K_0$.
  We start with $\DZ S$.

						\begin{lemma} \label{eq ratioetamaxminbig}
				Let $(\tw, \tz) \in R_K$. We have the following:
				\enmt
				\item \label{itm 1 DZSbound}   If $K \geq K_0$ and $\eta_{max} >   D_K^{500/(1-\beta_1)}  \eta_{min} $, then we have
				\aryst
				| \DZ S(\tw, \tz)  |  \gtrsim  \eta_{max}^{ 1 - \beta_1 }, 
				\earyst
				\item \label{itm 2 DZSbound1} If $K \leq K_0$, $(Q, Q') \in \Cent$, and item \eqref{itm distance2} in  Definition \ref{def Distancedf}  holds, then we have  {\em
				\ary \label{eq secondlowerboundpartialzS1}
				| \DZ S(\tw, \tz) |    \gtrsim  \langle \eta_{max}^{-\beta_1}( \eta - \eta' \partial_z \hat{F}_z(\fw, \fz) ) \rangle  \gtrsim D,
				\eary
}
				\item \label{itm 2 DZSbound2} If $K \leq K_0$, $(Q, Q') \notin \Cent$, and item \eqref{itm distance2}  in Definition \ref{def Distancedf}  holds, then we have {\em
\ary \label{eq secondlowerboundpartialzS2}
| \DZ S(\tw, \tz) |   \gtrsim   \langle \eta_{max}^{-\beta_1}( \eta - \eta' \partial_z \hat{F}_z(\fw, \fz) ) \rangle   \gtrsim  D^{1/2} \max(\eta_{max}, \xi_{max})^{c_0 / 2}.
\eary
}
				\eenmt
			\end{lemma}
			\begin{proof}
 		Assume that $K \geq K_0$ and $\eta_{max} >    D_K^{ 500/(1-\beta_1)}  \eta_{min} $.
				By \eqref{eq partialzFzbound0}, \eqref{eq ximaxetamaxbig} and by letting $\omegalow$ be large,  we have 
				\ary \label{eq upeta'partialzFz}
				| \partial_z \hat{F}_z(\tw, \tz) |  \leq 3 +  C  \extra^{-1}   D_K   \eta_{min}^{- 1/3 }   \leq (\eta_{max} / \eta_{min})^{1/2}.
				\eary
				Then by \eqref{eq partialFbound0} and  \eqref{eq whatisK0} we have
				\ary \label{eq ptzSlowerbound2}
				& & 	| \partial_{\tz} S (\tw, \tz) |  \geq	\eta_{max} - \eta_{min} (\eta_{max}/ \eta_{min})^{1/2}  - \xi_{max}  \geq \frac{1}{2} \eta_{max}  -  D^{400}_K \eta_{max}^{\beta_1} \geq   \frac{1}{4} \eta_{max}^{\beta_1}  \cdot \eta_{max}^{(1- \beta_1)}.
				\eary
				This gives   \eqref{itm 1 DZSbound}.

				Assume that $K \leq K_0$.
				We detail the proof under $| \eta | \geq | \eta' |$ as the other case is similar.
				By \eqref{eq partialsFbound0}, we have
				\ary \label{eq differencezFz}
				| \eta' \partial_z \hat{F}_z(\tw, \tz) -  \eta' \partial_z \hat{F}_z( \fw, \fz ) | \lesssim_{\tinitial}  \eta_{max} \cdot   D_K^5  \eta_{max}^{- 1 / 4}.
				\eary

				First we assume that $(Q, Q') \in \Cent$ and item \eqref{itm distance2} in  Definition \ref{def Distancedf}  holds.
				In particular $\eta_{max} <   10  \eta_{min}$ and $\xi_{max} \leq \eta_{max}^{\beta_1 + c_0}$. Then  $D_K \leq D_{K_0} \lesssim \eta_{max}^{c_0/400}$ and $\eta_{max} \geq \omegalow$.  		
				By \eqref{eq differencezFz}, we have $	| \eta' \partial_z \hat{F}_z(\tw, \tz) -  \eta' \partial_z \hat{F}_z( \fw, \fz ) |  \lesssim \eta_{max}^{\beta_1}$.
				Then  by \eqref{eq partialFbound0}, we have
				$$
				| \partial_{\tz} S (\tw, \tz) |  \geq  | \eta - \eta' \partial_z \hat{F}_z(\fw, \fz) |  -  C \eta_{max}^{\beta_1} - \xi_{max}.
				$$ 
				Then it is direct to  deduce  \eqref{eq secondlowerboundpartialzS1} from item \eqref{itm distance2}. 	This completes the proof of   \eqref{itm 2 DZSbound1}.

		        Now we assume that $(Q, Q') \notin \Cent$ and item \eqref{itm distance2} in  Definition \ref{def Distancedf}  holds.  We may divide the proof into two cases.
		        
		        (1) $\eta_{max}^{\beta_1 + c_0} \leq  \xi_{max}$. Then by \eqref{eq differencezFz} and \eqref{eq whatisK0}, we have $| \eta' \partial_z \hat{F}_z(\tw, \tz) -  \eta' \partial_z \hat{F}_z( \fw, \fz ) | \leq  \xi_{max}/2$.
		        By item \eqref{itm distance2} and  \eqref{eq partialFbound0}, we have
		        \ary \label{eq ptzSlowerbound}
		        &  &   | \partial_{\tz} S (\tw, \tz) |  \geq  | \eta - \eta' \partial_z \hat{F}_z(\fw, \fz) |  -   \xi_{max} \geq  D \eta_{max}^{\beta_1}  +  2\xi_{max}.
		        \eary
		        In this case,  \eqref{eq secondlowerboundpartialzS2} is deduce from  Cauchy's inequality and that $\xi_{max}^{c_0}  \eta_{max}^{\beta_1}  \leq \xi_{max}$.
		        
		        (2)   $\eta_{max}^{\beta_1 + c_0} >  \xi_{max}$. Then  $\eta_{max} \geq    10 \eta_{min} $ and $D_{K_0}^{400} \sim \eta_{max}^{c_0}$. Note that by definition, $D \leq \eta_{max}^{c_0}$.
		        In this case we can apply Lemma \ref{lem etaeta'different} to deduce $| \eta_{ max}^{- \beta_1}  \partial_{\tz} S(\tw, \tz)  | \gtrsim \eta_{max}^{1 - \beta_1}$ and \eqref{eq secondlowerboundpartialzS2} is clear. 
		        Note that here we don't need to assume item \eqref{itm distance2} in this case.
				This completes the proof of   \eqref{itm 2 DZSbound2}.
			\end{proof}

  The lower bounds for $\DZ S$ in Lemma \ref{eq ratioetamaxminbig} will be adequate whenever $K \geq K_0$ or  item \eqref{itm distance2} in Definition \ref{def Distancedf} holds.  The estimates we need can be summarized as follows.

		\begin{lemma}   \label{lem smallK0}  
		Assume that $(Q, Q') \in  \Trunc_{ (3)}^{\leq D}$ or $(Q, Q') \in    ( \Trunc_{D, out} \setminus  \Cent ) \setminus \Sigma_D$.
		Given an integer $0 \leq K \leq K_0$ and some  $(\tw, \tz) \in R_{K}$.  We have
		{\em 
		\ary  \label{eq lowerboundpartialzSa}
		| \DZ S(\tw, \tz)  |  \gtrsim  
		\begin{cases}
			D^{1/2}    \max( \eta_{max},  \xi_{max} )^{c_0/2}, & (Q, Q') \in    ( \Trunc_{D, out} \setminus  \Cent ) \setminus \Sigma_D, \\
 		  \langle \eta_{max}^{-\beta_1}( \eta - \eta' \partial_z \hat{F}_z(\fw, \fz) ) \rangle \gtrsim D, & (Q, Q') \in  \Trunc_{ (3)}^{\leq D}.
		\end{cases}
		\eary 
	}
		Moreover, for  any integer $2 \leq l \leq \greg $, we have
		\ary    \label{eq Dgeq3ofS2a}  
		\langle \DZ S(\tw, \tz)  \rangle^{- 1} 	\|	 \DZ^l   S(\tw, \tz) \|  \lesssim     \langle \DZ S(\tw, \tz)  \rangle^{ l/4}.
		\eary 
	\end{lemma}
	\begin{proof} 
		In both cases, $(Q, Q') \notin \Sigma_D$, in other words,  item \eqref{itm distance3} fails.
		By Lemma \ref{lem RKtoitm1},  item \eqref{itm distance1} fails. Thus item \eqref{itm distance2} holds.	    Then \eqref{eq lowerboundpartialzSa}  follows from Lemma \ref{eq ratioetamaxminbig}.   
		
		By \eqref{eq partialzFw} in Lemma \ref{lem derivativeboundsforhatF} and $ D^{400}_{K_0}  \leq   \max( \xi_{max} / \eta_{max}^{\beta_1},  \eta_{max}^{c_0} )$ (see \eqref{eq whatisK0}), we have 
		\aryst
		\|	 \DZ^l   S(\tw, \tz) \|  \lesssim_{\tinitial, l}  \eta_{max}^{1 - l ( \beta_1 - c(\esmall) ) }\max(\xi_{max}/ \eta_{max}^{\beta_1}, \eta_{max}^{c_0})^{(l + 2 ) / 400}  (\xi_{max}/  \eta_{max} + 1).
		\earyst
		This immediately implies  \eqref{eq Dgeq3ofS2a} when $\xi_{max} \leq \eta_{max}$. When $\xi_{max} > \eta_{max}$, we have  $(Q, Q') \notin \Cent$, and  \eqref{eq Dgeq3ofS2a}  follows from the above inequality and  \eqref{eq ptzSlowerbound}.
	\end{proof}

	We will need lower bounds for $\DV S$ in the region complementary to the one in Proposition \ref{lem smallK0}.   
			Aside from $\DV S$, it will also be convenient to consider the following
			\ary \label{eq BV}
			 \BV S(\tw, \tz)  = \min( \langle \DV S(\tw, \tz)  \rangle,   r_{min} \zeta_{max}  ).
			\eary
			We will divide the argument into two cases,   corresponding to $\Cent$ and $\Cent^c$.

			\begin{lemma}   \label{lem smallK}  
					Assume that $(Q, Q') \in  \Trunc_{D, out}  \cap   \Cent $.
				Given an integer $0 \leq K \leq K_0$ and some  $(\tw, \tz) \in R_{K}$.  We have 
	{\em			\ary
		&&		\langle \DZ S(\tw, \tz)  \rangle    \BV S(\tw, \tz)    \gtrsim  D^{0.99}   \langle \eta_{max}^{-\beta_1}( \eta - \eta' \partial_z \hat{F}_z(\fw, \fz) ) \rangle^{  1/20},   \label{eq lowerboundpartialzS} \\
		&& D^{0.01} \cdot \langle \DV S(\tw, \tz) \rangle / ( r_{min} \zeta_{max} )  + \langle \DZ S \rangle^{0.1}  \gtrsim  \|  \partial_{V} \hat{F}_w(\tw, \tz)   \|. \label{eq DVSrzetaDZS}
				\eary
	}
				Moreover, for  any $(l, k) \in \N^2 \setminus \{ (1, 0), (0, 1), (0, 0) \}$ with $l + k \leq \greg$,  we have
				\ary    \label{eq Dgeq3ofS2}  
				  && \langle \DV S(\tw, \tz)  \rangle^{-1}	 \|	 \DZ^l \DV^{k}   S(\tw, \tz) \|     
				 \lesssim C(k, l)	     (D_{K}  \langle \DZ S(\tw, \tz)  \rangle^{0.1}    )^{k + l - 1}.
				\eary 
			\end{lemma}

			\begin{proof}

%
By Lemma \ref{lem RKtoitm1},  item \eqref{itm distance1}  (in Definition \ref{def Distancedf})  fails. 
		  Since $  \Trunc_{D, out}  \subset \Trunc_{ (3)}^{>  D }$,  item \eqref{itm distance2}    fails as well, and we also have $\zeta_{max} \geq D \hDelta_{max}$. 
		  As a result,  item \eqref{itm distance3}  must hold. 
		   
By $(Q, Q') \in     \Cent $, we have $\eta_{max}^{\beta_1 + c_0} >  \xi_{max}$ and $\eta_{max} < 10  \eta_{min}$. In this case we must have $D_{K_0}^{400} \sim \eta_{max}^{c_0}$.
				By \eqref{eq ximaxetamaxbig}, we have  $\eta_{max} > \omegalow$.

				From the proof of Lemma \ref{lem RKtoitm1}, we see that   both \eqref{eq rangeotwtzandimage1} and \eqref{eq rangeotwtzandimage} hold. 
We may assume that
				\ary 	\label{eq partialtzSupperbound1}
				\langle \DZ S \rangle \leq D  \eta_{max}^{c_0} = D  \eta_{max}^{(1 - \beta_1)/20}.
				\eary
				In fact if \eqref{eq partialtzSupperbound1} fails, then clearly both   \eqref{eq lowerboundpartialzS}   and  \eqref{eq DVSrzetaDZS} hold.
				
				In the following, we will show that    \eqref{eq lowerboundpartialzS}   and  \eqref{eq DVSrzetaDZS} hold assuming \eqref{eq partialtzSupperbound1}.
				Recall that 
				$$\partial_{\tw} S(\tw, \tz)  =   \xi - \xi' \cdot \partial_{w} \hat{F}_w(\tw, \tz) - \eta' \partial_{w} \hat{F}_z(\tw, \tz).$$
				By \eqref{eq pullbackofestarinmodifiedchart}, we have
				\aryst
				&& 	\eta' \partial_{\tw} \hat F_z(\tw, \tz) = \eta'' \estar_{a}(\eta, \fw, \fz \mid \tw, \tz) - \eta' \estar_{a'}(\eta', \fw', \fz' \mid  \hat{F}(\tw, \tz) ) \cdot \partial_w \hat F_w(\tw, \tz) \\
				\mbox{where } &&
				\eta'' = \eta' ( \estar_{a'}(\eta', \fw', \fz' \mid  \hat{F}(\tw, \tz) ) \cdot \partial_z \hat F_w(\tw, \tz) + \partial_z \hat F_z(\tw, \tz) ). 
				\earyst
				We set
				\aryst
				\zeta_* =    \eta  \estar_{a}(\eta, \fw, \fz \mid w, z) -  \eta  \estar_{a}(\eta, \fw, \fz \mid \tw, \tz), \
				\zeta'_* =     \eta'  \estar_{a'}(\eta', \fw', \fz' \mid w', z' ) -  \eta'  \estar_{a'}(\eta', \fw', \fz' \mid \hat{F}(\tw, \tz) ).
				\earyst
				Note that we have 
				\begin{align*}
				\eta - \eta'' 
				&= \eta -  \eta'  \partial_z \hat F_z(\tw, \tz)  - \eta'  \estar_{a'}(\eta', \fw', \fz' \mid  \hat{F}(\tw, \tz) ) \cdot \partial_z \hat F_w(\tw, \tz) \\
				&= \partial_{\tz} S(\tw, \tz) +  ( \xi' -  \eta' \estar_{a'}(\eta', \fw', \fz' \mid  \hat{F}(\tw, \tz) )  ) \cdot \partial_z \hat F_w(\tw, \tz)  \\
				&= \partial_{\tz} S(\tw, \tz) +  (\zeta' + \zeta'_* ) \cdot \partial_z \hat F_w(\tw, \tz).
				\end{align*}
				Then
				\begin{align*}
				&						\partial_{\tw} S(\tw, \tz) = ( \zeta + \zeta_* )    - ( \zeta' + \zeta'_*)  \cdot \partial_w \hat F_w(\tw, \tz)   + (\eta - \eta'' )   \estar_{a}(\eta, \fw, \fz \mid \tw, \tz) \\
				=&  ( \zeta + \zeta_* )    - ( \zeta' +  \zeta'_*)  \cdot \partial_w \hat F_w(\tw, \tz)   + (  \partial_{\tz} S(\tw, \tz) +  (\zeta' + \zeta'_* ) \cdot \partial_z \hat F_w(\tw, \tz) )   \estar_{a}(\eta, \fw, \fz \mid \tw, \tz) \\
				=&  ( \zeta + \zeta_* )    - ( \zeta' + \zeta'_*)  \cdot ( \partial_w \hat F_w(\tw, \tz) - \partial_z \hat F_w(\tw, \tz) \cdot  \estar_{a}(\eta, \fw, \fz \mid \tw, \tz) )  +   \partial_{\tz} S(\tw, \tz)   \estar_{a}(\eta, \fw, \fz \mid \tw, \tz).
				\end{align*}

 	By   \eqref{eq rangeotwtzandimage1} and Proposition \ref{lem mainpropertyofestar}\eqref{itm rangeofestar},     we have   
				\ary   \label{eq zeta*bound}  
				\| \zeta_*  \| \lesssim   \hDelta_{a}(\eta, \fw, \fz),  \quad
				\| \zeta'_*   \| \lesssim    \hDelta_{a'}(\eta', \fw', \fz'). 
				\eary
				By \eqref{eq rangeotwtzandimage} and Proposition \ref{lem mainpropertyofestar}\eqref{itm rangeofestar}, we  have $\eta \estar_{a}(\eta, \fw, \fz \mid \tw, \tz) = \cO_{\tinitial}( \hDelta_{max} )$. 
				Then by \eqref{eq partialtzSupperbound1},  $\eta_{max} > \omegalow$ and $D^{2/15} = D_0^{400} \leq \eta_{max}^{c_0}$,  we have 
				\aryst
				\partial_{\tz} S(\tw, \tz)   \estar_{a}(\eta, \fw, \fz \mid \tw, \tz) = \cO_{\tinitial}( D  \eta_{max}^{\beta_1 + c_0 -1} \cdot   \hDelta_{max}  ) =  \cO(   \hDelta_{max} ).
				\earyst				 	
				By \eqref{eq partialFbound0} and that clearly $\estar_{a}(\eta, \fw, \fz \mid \tw, \tz) = \cO(1)$,  for   sufficiently small $\esmall$, we have
				\aryst
				\partial_z \hat F_w(\tw, \tz) \cdot  \estar_{a}(\eta, \fw, \fz \mid \tw, \tz) = \cO(e^{C \tinitial} \esmall).
				\earyst  
			By \eqref{eq zeta*bound} and by letting $\esmall$ be small,  \eqref{tildezetatildezeta'}  in Lemma \ref{lem DvSDvF} is satisfied for $(\widetilde\zeta, \widetilde\zeta') = (\zeta + \zeta_*, \zeta' + \zeta'_*)$. 
			Putting together the above estimates with   \eqref{eq gooddirection+} in Lemma \ref{lem DvSDvF} and $\zeta_{max} \geq D \hDelta_{max}$,    we deduce
				\ary \label{eq lowerboundpartialVS}
		 	| \DV S(\tw, \tz) | \gtrsim     \min( e^{\lambda_g \tinitial},  D)^{- 0.01}    \langle \partial_{V} \hat{F}_w(\tw, \tz) \rangle   r_{min} \zeta_{max}  \geq D^{0.99} r_{min}  \hDelta_{max}.
				\eary
	 		Then  \eqref{eq DVSrzetaDZS}  is clear.   
	 		To conclude  \eqref{eq lowerboundpartialzS},   we see that this is clear when $D \langle \eta_{max}^{-\beta_1}( \eta - \eta' \partial_z \hat{F}_z(\fw, \fz) ) \rangle \lesssim    r_{min} \zeta_{max}$.
	 		Otherwise, we deduce by \eqref{eq adscalesquarerootgrowth} that $| \eta - \eta' \partial_z \hat{F}_z(\fw, \fz)  | \gtrsim D^{-1} \eta_{max}^{\beta_1 - 1/2 - c(\esmall)} \zeta_{max} \gtrsim \eta_{max}^{c_0} \xi_{max}$. 
	 		Here we have also used that $| \zeta_{max} - \xi_{max} |  \lesssim \eta_{max}^{1/2 + c_0}$ as a result of  \eqref{eq hDeltaomegarange} and \eqref{etaestarADO1}. Note that we clearly have $ | \partial_z \hat{F}_z(\fw, \fz)  - \partial_z \hat{F}_z(\tw, \tz)  | \lesssim \eta_{max}^{-1/4}$.
 	Thus $| \DZ S(\tw, \tz) | \geq \eta_{max}^{-\beta_1} | \eta - \eta' \partial_z \hat{F}_z(\fw, \fz)  | - C  - \eta_{max}^{-\beta_1}  \xi_{max} \gtrsim \eta_{max}^{-\beta_1}  | \eta - \eta' \partial_z \hat{F}_z(\fw, \fz)  |$.
	 		Then we deduce  \eqref{eq lowerboundpartialzS}   from   \eqref{eq lowerboundpartialVS} and $r_{min} \hDelta_{max} \gtrsim 1$.

				We now verify \eqref{eq Dgeq3ofS2}  for $(l, k) = (0, 2)$. 	By  \eqref{eq propertiesofvC}   in Lemma \ref{lem DvSDvF} and \eqref{eq partialzFw}, \eqref{eq partialwFz} in Proposition \ref{lem derivativeboundsforhatF}
				\ary \label{eq partialtwSbound0}
				\|	\partial_{V}^2 S(\tw, \tz) \|  &\lesssim& C(\tinitial)  D_K^4   \xi_{max} \eta_{max}^{c_0} + \eta_{min} \|  \partial_{V}^2 \hat{F}_z(\tw, \tz) \| \\
				&\lesssim& C(\tinitial)  D_K^4   \xi_{max} \eta_{max}^{c_0} +  C(\tinitial)  \eta_{min} \| \partial_w \hat{F}_z(\tw, \tz) \| +  \eta_{min} \|  \partial_{w}^2 \hat{F}_z(\tw, \tz) \|  \nonumber  \\
	 	&\lesssim& C(\tinitial)  D_K^4   \xi_{max} \eta_{max}^{c_0} + C(\tinitial) D_K^3  \eta_{max}^{1/2 + {c_0}} +  \eta_{min} \|  \partial_{w}^2 \hat{F}_z(\tw, \tz) \|.  \nonumber 
				\eary

				By \eqref{eq rangeotwtzandimage}  and 	\eqref{eq partialw^2Fz} in Proposition \ref{lem derivativeboundsforhatF}, and by letting $\esmall$ be small, we deduce
				\aryst
				\eta_{min} \|  \partial_{\tw}^2 \hat{F}_z(\tw, \tz) \| = \cO( \eta_{min} \tinitial^2   r_a(\eta, w, z)^{-1} \Delta_a(\eta, w, z) ) = \cO( \tinitial^2  r_{min}^{-1}  \hDelta_{max} ).
				\earyst
				Then by $r_{min} \leq \eta_{max}^{-1/2 + c_0}$, $\hDelta_{max} \gtrsim \eta_{max}^{1/2 - c_0}$, $D_K \leq D_{K_0} \leq \eta_{max}^{c_0/400}$ and $\beta_1 + 20 c_0  = 1$, we have
				\aryst
				\|	\DV^2 S(\tw, \tz) \|  \lesssim  C(\tinitial) r_{min}^2 \eta_{max}^{\beta_1 + 3 c_0} +   \tinitial^2 r_{min}  \hDelta_{max} \lesssim     \tinitial^2 r_{min}  \hDelta_{max}.
				\earyst   
				We have also used in the above that $\eta_{max} \geq \omegalow \gg_{\tinitial} 1$.
				If \eqref{eq partialtzSupperbound1} holds,  then we have $  \tinitial^2 r_{min}  \hDelta_{max} \ll D^{0.99} r_{min}  \hDelta_{max}  = \cO( \langle    \DV S(\tw, \tz)  \rangle )$ by   \eqref{eq lowerboundpartialVS}.    If  \eqref{eq partialtzSupperbound1} fails, then we have 
				$  \tinitial^2 r_{min}  \hDelta_{max} \ll (D \eta_{max}^{c_0})^{0.1} =  \cO( \langle    \DZ S(\tw, \tz)  \rangle^{0.1} )$ by \eqref{eq Deltaadscale>1}.
				This completes the proof of  \eqref{eq Dgeq3ofS2} for $(l, k) = (0, 2)$.

				Now we prove  \eqref{eq Dgeq3ofS2}   for   $(l, k) \in \N^2 \setminus \{ (0, 0), (1, 0), (0, 1), (0, 2) \}$. 						
							By $\eta_{max} >  \xi_{max}$ and  following  the proof of  Lemma \ref{lem bigK}, we see that 
				\ary \label{eq generalpartofpartialtwtzS}
			 	 \|	 \DZ^l  \DV^{k}  S(\tw, \tz) \| 
				\lesssim_{\tinitial, l, k} 
				 D_{K}^{k + l + 2}     \eta_{max}^{ 1 + (k + l )c(\esmall) -(1/2 - c(\esmall))k - l \beta_1 }.
				\eary 	
	 	The exponent of $\eta_{max}$ in \eqref{eq generalpartofpartialtwtzS} is at most $ - (l + k) /10$.  
	 	Since $\eta_{max}^{c_0} \geq D_{K}$, the right hand side of \eqref{eq generalpartofpartialtwtzS} is of size $\cO(D_{K}^{k + l - 1} )$.   
				This gives \eqref{eq Dgeq3ofS2}.  
								We have finished the proof. 
			\end{proof}

						\begin{lemma}   \label{lem smallK2}  
				Assume that $(Q, Q') \in    \Trunc_{D, out}  \cap \Sigma_D \setminus  \Cent  $.
				Given an integer $0 \leq K \leq K_0$ and some  $(\tw, \tz) \in R_{K}$.  
				Then  we have
				\ary
			&& 	\langle \DZ S(\tw, \tz)  \rangle    \BV S(\tw, \tz)    \gtrsim  D^{0.99}   \eta_{max}^{c_0 },   \label{eq lowerboundpartialzS2} \\
		 &&  D^{0.01} \cdot \langle \DV S(\tw, \tz)  \rangle / ( r_{min} \zeta_{max} )  + \langle \DZ S \rangle^{0.1} \gtrsim \|  \partial_{V} \hat{F}_w(\tw, \tz)   \|, \label{eq DVSrzetaDZS2}
				\eary
				and for  any $(l, k) \in \N^2 \setminus \{ (1, 0), (0, 1), (0, 0) \}$ with $l + k \leq \greg$,  we have  
				\ary    \label{eq Dgeq3ofS22}  
				&& \langle \DV S(\tw, \tz)  \rangle^{-1}	 \|	 \DZ^l \DV^{k}   S(\tw, \tz) \|     
				\lesssim C(k, l)	     (D_{K}  \langle \DZ S(\tw, \tz)  \rangle^{0.1} \BV S(\tw, \tz)^{0.6}  )^{k + l - 1}.
				\eary 
			\end{lemma}

				\begin{proof}
%
				By hypothesis,  item \eqref{itm distance3} holds.
			  Since $  \Trunc_{D, out}  \subset \Trunc_{ (3)}^{>  D }$, we have $\zeta_{max} \geq D \hDelta_{max}$.

				Since $(Q, Q') \notin \Cent$, we can  divide the   proof  into two cases.					
				
 		\noindent{\underline{Case I:}}	  $\eta_{max}^{\beta_1 + c_0} >  \xi_{max}$ and $\eta_{max} \geq 10 \eta_{min}$. In this case we must have $D_{K_0}^{400} \sim \eta_{max}^{c_0}$.
				By \eqref{eq ximaxetamaxbig}, we have  $\eta_{max} > \omegalow$.

				By Lemma \ref{lem etaeta'different}, we see that $\langle \DZ S(\tw, \tz)  \rangle	\gtrsim	 \eta_{max}^{1 - \beta_1} \geq D \eta_{max}^{c_0 }$.
				In particular, \eqref{eq lowerboundpartialzS2} holds.   By $\eta_{max} > \omegalow \gg e^{C \tinitial}$,  \eqref{eq DVSrzetaDZS2} is clear. 
		 Note that \eqref{eq partialtwSbound0} remains valid.    Then by	\eqref{eq partialzFw} in Proposition \ref{lem derivativeboundsforhatF}, we have $	\|	\DV^2 S(\tw, \tz) \| \lesssim \eta_{max}^{c_0/10} \lesssim \langle \DZ S(\tw, \tz)  \rangle^{0.1}$. This proves   \eqref{eq Dgeq3ofS22}  for $(l, k) = (0, 2)$   in Case I.

     	Now we prove  \eqref{eq Dgeq3ofS22}   in Case I for   $(l, k) \in \N^2 \setminus \{(0, 0), (1, 0), (0, 1), (0, 2) \}$. 						
				As in the proof of Lemma \ref{lem smallK}, we see that \eqref{eq generalpartofpartialtwtzS} remains valid.  
				The exponent of $\eta_{max}$ in \eqref{eq generalpartofpartialtwtzS} is at most $ - (l + k) /10$.  By  $\eta_{max}^{c_0} \geq D_{K}$, the right hand side of \eqref{eq generalpartofpartialtwtzS} is of size $\cO(D_{K}^{k + l - 1} )$.   
				This gives \eqref{eq Dgeq3ofS22}    in Case I.
				We have finished the proof in Case I.
				%
				%
				%
				%
				%

				\noindent{\underline{Case II:}}  $\eta_{max}^{\beta_1 + c_0 } \leq    \xi_{max}$.    	 
				By the definition of $K_0$, we have $D_{K_0}^{400} \sim   \xi_{max} / \eta_{max}^{\beta_1 } \geq \eta_{max}^{c_0}$. Clearly, we have $\xi_{max}  >  \omegalow^{1/2}$.
				In this case, it is clear that $\zeta_{max} \sim \xi_{max}$. Then we have $\zeta_{max} \gg  e^{C \tinitial} \hDelta_{max}$.
				

				By \eqref{eq partialwFzbound0}, we have 
				\aryst
				\eta_{min} \|  \partial_{V} \hat{F}_z(\tw, \tz) \|  \lesssim_{\tinitial}  \eta_{min} \cdot D_{K_0}^{50}  \eta_{min}^{-1/4}       \ll e^{- C \tinitial}  \xi_{max}.
				\earyst
				By \eqref{etaestarADO1},   we see that \eqref{tildezetatildezeta'}  in Lemma \ref{lem DvSDvF} is satisfied for $(\widetilde\zeta, \widetilde\zeta') = (\xi, \xi')$. 
				Then by  \eqref{eq gooddirection+} in Lemma \ref{lem DvSDvF}, we deduce
				\ary \label{eq lowerboundpartialtwS}
				| \DV S(\tw, \tz) | \gtrsim \min( e^{\lambda_g \tinitial},  D)^{- 0.01}       \langle \partial_{V} \hat{F}_w(\tw, \tz) \rangle      r_{min}     \zeta_{max}   \gtrsim e^{- \lambda_g \tinitial/100} r_{min}  \zeta_{max}.
				\eary
		 	Then \eqref{eq DVSrzetaDZS2} is clear,  and we deduce  \eqref{eq lowerboundpartialzS2}   from   
				$  r_{min}\zeta_{max}  \gtrsim \zeta_{max}  \eta_{max}^{-1/2-c(\esmall)} \geq     e^{C \tinitial} D \eta_{max}^{c_0}$.


				Now we prove   \eqref{eq Dgeq3ofS22}   in Case II.	Let  $(l, k) \in \N^2 \setminus \{ (1, 0), (0, 1), (0, 0) \}$.
				We see that the following estimate in the proof of  Lemma \ref{lem bigK} remains valid:  
				\aryst 
				\|	 \DZ^l  \DV^{k}  S(\tw, \tz) \| 
				&\lesssim_{\tinitial, l, k, \esmall} &   D_{K}^{k + l + 2}  (\xi_{max} \eta_{max}^{- 1} +1 ) 	\eta_{max}^{ 1 + (k + l )c(\esmall) -(1/2 - c(\esmall))k - l \beta_1 }  \\
				&\lesssim_{\tinitial, l, k, \esmall} &   D_{K}^{k + l - 1} \eta_{max}^{(k + l) c(\esmall)} (\xi_{max} \eta_{max}^{- \beta_1} )^{1/100} (\xi_{max} \eta_{max}^{- 1} +1 ).
				\earyst
				By   \eqref{eq BV} and \eqref{eq lowerboundpartialtwS},   it is clear that   $\|	\DZ^l  \DV^{k}   S(\tw, \tz) \| \lesssim D_{K}^{k + l - 1}   \BV S(\tw, \tz)^{  0.6 (k + l)}$. 
				This gives \eqref{eq Dgeq3ofS22}   in Case II. 				
				We have finished the proof in Case II.	   		 
			\end{proof}

	It is direct to deduce from  Lemma \ref{lem smallK0} to Lemma \ref{lem smallK2}  the following.

	\begin{cor} \label{rem LargenessofDVSDzD}
	For any $(Q, Q') \in  \Trunc_{D, out}  $, at least one of the following is true:
			(1)  $\log \eta_{max}  \gg   \tinitial$;    (2)  $(Q, Q') \in \Sigma_D$ and $ \log  \BV S(\tw, \tz) \gg   \tinitial$.
			\end{cor}

			\subsection{Derivatives of $u$} \label{subsec partialsofu}

			Throughout this section, we fix arbitrary integers $K, \ell, l, k \geq 0$ with $l + k \leq \greg$.
			By   Theorem \ref{thm DWPT}\eqref{thm DWPT item 2},  we have
			\ary 
		&&	\|  \DZ^l  \DV^{k}   (   \widetilde{\phi}^{\dagger}_a(  \eta, w, z \mid  \fw, \fz \mid  \tw, \tz )   ) \|\leq C(k, l )   \Phi_{a, \ell}(\eta, \fw, \fz \mid w, z) \widehat \Phi_{a, \ell}(\eta, w, z \mid \tw, \tz     ).    \label{eq partialsofrho}
			\eary

			We also have the following.
			
			\begin{lemma} \label{lem partialofphia'hatF}
				For  any $(\tw, \tz) \in R_{K} \setminus R_{K-1}$,  the following is true.
				\enmt
				\item \label{itm 189} If  $(Q, Q') \in     \Trunc_{D, out} \cap ( \Cent  \cup \Sigma_{D} )$, then we have
				{\em			\ary  \label{eq partialsofphia'hatF} 
					&&  \|     \DZ^l  \DV^{k}   (   \widetilde{\phi}_{a'}(  \eta', w', z' \mid  \fw', \fz' \mid \hat F( \tw, \tz) )   \cdot    \rho_F (\tw, \tz - d_{a, \fw, \fz, \eta}(\tw))     )  \|     \\
					&\leq&   C(k, l, \ell)   D_{K-1}^{- \ell}    (    D_K \langle \DZ S \rangle^{0.1} \BV S^{0.1}     )^{  l} (     D_K^{51}     \langle \DV S \rangle^{0.1}   + \langle \DZ S \rangle^{0.1} +  D^{0.01}  \langle  \DV S \rangle / ( r_{min} \zeta_{max}  )  )^k  \nonumber \\
					&& \Phi_{a, \ell }(\eta', \fw', \fz' \mid w', z')   \widehat  \Phi_{a, \ell}(\eta', w', z' \mid  \hat F( \tw, \tz)   ),  \nonumber
					\eary
				}
				\item  \label{itm 289} If  $(Q, Q') \in  \Trunc_{ (3)}^{\leq D}$ or $(Q, Q') \in  ( \Trunc_{D, out} \setminus  \Cent ) \setminus \Sigma_{D}$, then we have 
		 	{\em			\ary  \label{eq partialsofphia'hatF2}
					&&  \|     \DZ^l      (   \widetilde{\phi}_{a'}(  \eta', w', z' \mid  \fw', \fz' \mid \hat F( \tw, \tz) )  \cdot    \rho_F (\tw, \tz - d_{a, \fw, \fz, \eta}(\tw))    )  \|     \\
					&\leq&   C( l, \ell)  D_{K-1}^{- \ell}  (  D_K^2     \langle \DZ S \rangle^{0.1}   )^{l   }    \Phi_{a, \ell }(\eta', \fw', \fz' \mid w', z')  \widehat  \Phi_{a, \ell}(\eta', w', z' \mid  \hat F( \tw, \tz)   ).    \nonumber  
					\eary
				}
				\eenmt
			\end{lemma}
			
			\begin{proof} 
	  		Recall that  by  \eqref{eq partialzFzbound0} and \eqref{eq partialFbound0}, we have $| \partial_z  \hat F_z( \tw, \tz)  | \lesssim  \extra^{-1} D_K$ and $	\| \partial_z \hat{F}_w(\tw, \tz) \| \leq  e^{C\tinitial}\esmall \leq 1$.
					By   \eqref{eq partialsFbound0}, we have
				\ary
				\label{eq DwFzlessim}			 \| \partial_w^k \partial_z^l \hat F (\tw, \tz) \| \lesssim_{\tinitial, l, k, \esmall}  C(l, k)  D_K^{l + k + 2} \eta_{max}^{(k + l)c(\esmall)}.
				\eary
				
				Let us first assume that $0 \leq K \leq K_0$. 
				By  \eqref{eq DVSrzetaDZS} in Lemma \ref{lem smallK} and \eqref{eq DVSrzetaDZS2} in Lemma \ref{lem smallK2}, we have
			\ary
						\label{eq DVFwlessim1}		 \| \partial_V  \hat F_w( \tw, \tz)  \| \lesssim D^{0.01}  \langle \DV S(\tw, \tz) \rangle / ( r_{min} \zeta_{max} )  +  \langle \DZ S( \tw, \tz) \rangle^{0.1}.
			\eary

			We will prove the following claim: For $(Q, Q')$ in   \eqref{itm 189}, we have 
			 	\ary  
			 	\label{eq etabeta1DVSbound}
			&&  	\langle \eta'  \rangle^{\beta_1} \| \DV  \hat F_z( \tw, \tz)  \| \lesssim D_K^{51}     \langle \DV S( \tw, \tz) \rangle^{0.1},  \\
					 	\label{eq etabeta1DzVSbound}
			&&    | \partial_z \hat{F}_z(\tw, \tz) | \lesssim  
				   D_K \langle \DZ S( \tw, \tz) \rangle^{0.1} \BV S(\tw, \tz)^{0.1};
				\eary
	For $(Q, Q')$ in   \eqref{itm 289}, we have 
				 	\ary   
	\label{eq etabeta1DzVSbound2}
  | \partial_z \hat{F}_z(\tw, \tz) | \lesssim  
		D_K^2     \langle \DZ S( \tw, \tz) \rangle^{0.1}.
	\eary

To prove \eqref{eq etabeta1DVSbound} for $(Q, Q') \in     \Trunc_{D, out} \cap ( \Cent  \cup \Sigma_{D} )$, we note that by  \eqref{eq adscalesquarerootgrowth} and \eqref{eq partialwFzbound0}
				\aryst   
				 \langle \eta' \rangle^{\beta_1} \| \DV  \hat F_z( \tw, \tz)  \| \lesssim_{\tinitial}    D_K^{50}  r_{min} \langle \eta' \rangle^{\beta_1 -1/2}   \eta_{max}^{c_0} \lesssim_{\tinitial}  D_K^{50}  \eta_{max}^{-c_0}.
				\earyst
 	If $\eta_{max} \gg_{\tinitial} 1$, then  clearly the leftmost term above is bounded by $D_K^{50}$. Otherwise, then by \eqref{eq ximaxetamaxbig} and by letting $\omegalow \gg_{\tinitial} 1$, we have $\xi_{max} \geq \omegalow \gg \eta_{max}$.
 	In particular, $(Q, Q') \in     ( \Trunc_{D, out} \cap \Sigma_D ) \setminus  \Cent  $.
 	Then  item \eqref{itm distance22} in Definition \ref{def Distancedf} fail.  
 	By $K \leq K_0$ and by Lemma \ref{lem RKtoitm1}, we see that \eqref{itm distance12} also fails. Then  item \eqref{itm distance32} holds by $d_f(Q, Q') \geq D$.  By the proof of Lemma \ref{lem smallK2},   \eqref{eq lowerboundpartialtwS} holds.   Then \eqref{eq etabeta1DVSbound} holds. 
 	The proofs of \eqref{eq etabeta1DzVSbound} and \eqref{eq etabeta1DzVSbound2} follow from \eqref{eq partialzFzbound0}, Lemma \ref{lem smallK0} and Corollary \ref{rem LargenessofDVSDzD}. We leave the details to the readers.   
 
Note that by $(\tw, \tz) \notin R_{K-1}$, we may add to the right hand side of the inequalities in Theorem \ref{thm DWPT}\eqref{thm DWPT item 2} a factor $D_{K-1}^{-\ell}$. Then to conclude the proofs of \eqref{eq partialsofphia'hatF} and  \eqref{eq partialsofphia'hatF2} under $K \leq K_0$, it suffices to combine Theorem \ref{thm DWPT}\eqref{thm DWPT item 2} with \eqref{eq DwFzlessim} to   \eqref{eq etabeta1DzVSbound2}.

	If $K > K_0$, then $D_K > D_{K_0} \gtrsim \eta_{max}^{c_0/400}$. The factor $D_{K-1}^{-\ell}$ is so small that the crude estimate \eqref{eq DwFzlessim} is enough for proving  \eqref{eq partialsofphia'hatF} and  \eqref{eq partialsofphia'hatF2}. This completes the proof.
			\end{proof}
			An immediate corollary of Lemma \ref{lem partialofphia'hatF} and \eqref{eq partialsofrho} is the following.
			\begin{cor} \label{cor partialofphia'hatF}
			Recall that $u$ is given by \eqref{eq amplitude}. 	For any $(\tw, \tz) \in R_{K} \setminus R_{K-1}$, the following is true.
			\enmt
			\item \label{itm 1911} If $(Q, Q')$ satisfies item \eqref{itm 189} of Lemma \ref{lem partialofphia'hatF}, then   $	|     \DZ^l  \DV^{k}   u(\tw, \tz)  | $ is bounded from above by
		{\em
				\aryst
		&& C(\ell, l, k) D_{K-1}^{- \ell}  (    D_K \langle \DZ S \rangle^{0.1} \BV S^{0.1}     )^{  l} (   D_K^{51}     \langle \DV S \rangle^{0.1}       + \langle \DZ S \rangle^{0.1} +  D^{0.01}  \langle  \DV S \rangle / ( r_{min} \zeta_{max}  ) )^k \cdot  \\ 
		&&  \Phi_{a', \ell}(\eta', \fw', \fz' \mid w', z') \cdot     \Phi_{a, \ell}(\eta , \fw, \fz \mid w, z) \cdot  
		  \widehat\Phi_{a', \ell}(\eta', w', z' \mid \hat F(\tw, \tz) )  \cdot   \widehat \Phi_{a, \ell}(\eta , w, z \mid \tw, \tz ).
				\earyst
			}
			\item   \label{itm 2911}   If $(Q, Q')$ satisfies item \eqref{itm 289} of Lemma \ref{lem partialofphia'hatF}, then  $| \DZ^l     u(\tw, \tz)  | $ is bounded from above by
					{\em
				\aryst
			 C(\ell, l ) D_{K-1}^{- \ell} (  D_K^2     \langle \DZ S \rangle^{0.1}   )^{l   }  \Phi_{a', \ell}(\eta', \fw', \fz' \mid w', z') \cdot     \Phi_{a, \ell}(\eta , \fw, \fz \mid w, z) \cdot  \widehat\Phi_{a', \ell}(\eta', w', z' \mid \hat F(\tw, \tz) )  \cdot  \widehat \Phi_{a, \ell}(\eta , w, z \mid \tw, \tz ).
				\earyst
			}
			\eenmt
			\end{cor}

			\subsection{Proof of Item \eqref{itm 1 Prop81} in Proposition \ref{prop mainprop}: Central case} \label{subsec proofitm1prop81}
			
		In this subsection, we let  $(Q, Q') \in \Trunc_{D, out} \cap \Cent$.		
		 We apply integration-by-parts with respect to  $\cD_{\tz}$ and  $\cD_{\tw}$ to write $\cK_{+}(Q' \mid Q)$ as 
			\ary   \label{eq cK0'}
			&&  1_{| \eta | \geq | \eta' |}     e^{ - i (\xi, \eta) \cdot (w, z) + i (\xi', \eta') \cdot (w', z')}     \int
			e^{i  S }    (\cD_{\tz}^*)^{l}  (\cD_{V}^*)^k  u d\tw d\tz.
			\eary

			 Let  $K \in \N$.
		We denote by $\cK_{ K}(Q' \mid Q)$ the restriction of the above integral to the parameters $(\tw, \tz) \in R_K \setminus R_{K-1}$.

			\begin{lemma} \label{lem integralboundforK}
					For any integer $l \geq 0$, there is an universal constant $\greg_0(l) > 0$ such that if $\greg > \greg_0(l)$
					 and  $\ell \geq 1$,   we have   
		{\em 	\begin{align}   
		\label{eq  integralboundforK}		 | \cK_{K}  (Q' \mid Q) | \leq&  C( l, \ell)   1_{\{|\eta| > |\eta'|\}}      \Lambda \Phi_{a', \ell}(\eta', \fw', \fz' \mid w', z') \cdot   \Phi_{a, \ell}(\eta , \fw, \fz \mid w, z) \cdot   \\
				&   \int   \widehat\Phi_{a', \ell}(\eta', w', z' \mid \hat F(\tw, \tz) )  \cdot  \widehat \Phi_{a, \ell}(\eta , w, z \mid \tw, \tz )   d\tw d\tz \nonumber
			\end{align}
		}
		where 
		{\em
		\aryst
		\Lambda = 
			D_K^{- l}  \BV S^{- l / 4}  \langle \eta_{max}^{-\beta_1}( \eta - \eta' \partial_z \hat{F}_z(\fw, \fz) ) \rangle^{- l / 80}.
		\earyst
	}
			\end{lemma}
			\begin{proof}
				We will only detail the proof for $0 \leq K \leq K_0$ as the proof for $K > K_0$ is similar but significantly easier.

				Assume that $0 \leq K \leq K_0$. 
				By our hypothesis, $S$ is a $C^{\greg}$-function. Without loss of generality we may assume  $\greg/ 4 > l \gg 1$ and $\ell \geq 10^4 l$. We let $k = 3 l$.

			We may define
\aryst
U_{-1} = \langle \DZ S \rangle^{-1}, \quad U_{k} = \langle  \DZ S \rangle^{-1}   \DZ^{k+1} S, \quad 0 \leq k \leq  \greg - 1.
\earyst
Note that we have 
\ary \label{eq DZrule}
\DZ(U_{-1}) = - U_{-1} U_0 U_1, \quad \DZ(U_k) = U_{k+1} - U_0 U_1 U_k, \quad 0 \leq k \leq  \greg - 2.
\eary
Let $\PolyU^{ k,  C_1,   C_2}$ denote the collection of functions of form
\aryst
\sum_{{\bf d} = (d_{-1}, \cdots, d_k) \in \N^{k+2}} C_{{\bf d}}(\tw, \tz) \prod_{j = -1}^{k} U_j^{d_j}
\earyst 			
where ${\bf d}$ satisfies $\sum_{j = -1}^{k}  d_j \leq C_1$ and $\sum_{j = -1}^{k} j d_j \leq C_2$. Moreover, given $P \in \PolyU^{ k,  C_1,   C_2}$ as above, we denote $||| P ||| = \sup_{{\bf d}} \| C_{\bf d} \|$.
By \eqref{eq DZrule}, we see that 
\aryst
\DZ(\PolyU^{  k,   C_1,   C_2}) \subset \PolyU^{  k+1,   C_1 + 2,   C_2 + 1}.
\earyst 
We may write
\aryst
 \cD_{\tz}^{*} =  U_{-1}^2 - i U_{-1} U_{1} + 2  i U_{-1} U_0^2 U_{1} - i U_{-1} U_{0} \DZ. 
\earyst 
By a simple induction, we see that for $1 \leq l \leq \greg-2$,  $(\cD_{\tz}^{*})^l$ can be written as
\ary \label{eq Dtzexpression}
U_{-1}^l Q_0 + U_{-1}^l Q_1 \DZ + \cdots + U_{-1}^l Q_l \DZ^l 
\eary 
with $Q_{l'} \in   \PolyU^{l - l',  3l - 2 l', l - l'}$   and  $||| Q_{l'} ||| \lesssim C(l)$    for every $0 \leq l' \leq l$.

				We now compute $(\cD_{\tz}^{*})^l (\cD_{V}^{*})^k$.
				For any integers $l \geq -1, k \geq 0$ with $l + k \leq \greg - 2$, we denote by $B(l, k)$   the set  $\{-1, \cdots, l\} \times \{0, \cdots, k\} \cup \{(-1, -1)\}$. 
				We define
				\aryst
					W_{-1, -1} = \langle \DV S \rangle^{-1}, \quad 
				W_{l', k'} =  \langle \DV S \rangle^{-1}   \DZ^{l'+1}   \DV^{k'+1} S, \quad (l', k') \in B(l, k) \setminus \{(-1, -1)\}.
				\earyst
				By direct computations, we have
				\ary
			&&	 \DZ( W_{-1, -1}  ) = - W_{-1, 0} W_{0, 0} W_{-1, -1}, \  \DV( W_{-1, -1}  ) = - W_{-1, 0} W_{ -1, 1} W_{-1, -1}  \label{eq DZrule2}   \\
	 	&&	\DV( W_{-1, k'}  ) =  W_{-1, k' + 1}   - W_{-1, 0} W_{-1, 1}  W_{-1, k'},  \quad (-1, k') \in  B(l, k) \setminus \{(-1, -1)\}, \ k'  \leq \greg - 1,  \label{eq DZrule4} \\
			&&	\DZ( W_{l', k'}  ) = W_{l'+1, k'} - W_{-1, 0} W_{0, 0}  W_{l', k'},  \quad (l', k') \in  B(l, k) \setminus \{(-1, -1)\}, \ l' + k' \leq \greg - 3.  \label{eq DZrule3} 
				\eary
			Given integers $l \geq -1, k \geq 0$ with $l + k \leq \greg - 2$, we let $\PolyW^{l, k,  C_1,   C_2}$ denote  the set of functions of form
			\aryst
	 	\sum_{{\bf b} = (b_{s, t} )_{(s,t)} \in \N^{B(l, k)} } C_{{\bf b}}(\tw, \tz)  \prod_{(s,t) \in B(l, k)} W_{s, t}^{b_{s,t}}
			\earyst
			where   ${\bf b}$ satisfies $\sum_{(s,t) \in B(l, k)}  b_{s,t} \leq C_1$ and $\sum_{(s,t) \in B(l, k)}  (s+t+1) b_{s,t} \leq C_2$.  Moreover, given $P \in \PolyW^{l, k,  C_1,   C_2}$ as above, we denote $||| P ||| = \sup_{{\bf b}} \| C_{\bf b} \|$.
			We see that for $l \geq -1, k \geq 0$ with  $l + k \leq \greg - 2$
\aryst
\DZ( \PolyW^{l, k, C_1, C_2} ) \subset  \PolyW^{l+1, k, C_1+2, C_2+1}, \quad\DV( \PolyW^{-1, k, C_1, C_2} ) \subset     \PolyW^{-1, k + 1, C_1+2, C_2+1}.
\earyst  
					By a simple induction, we see that  $(\cD_{\tz}^{*})^l (\cD_{V}^{*})^k$ can be written as
			\ary \label{eq DtzDVexpression}
		  U_{-1}^l W_{-1, -1}^k   \sum_{l' = 0}^{l} \sum_{k' = 0}^k  R_{l', k'} \DZ^{l'} \DV^{k'}  
			\eary
						  with $R_{l', k'} \in \PolyW^{l - l', k - k', 3k + 3l - 2k' - 2l',  l+k - l' - k'}$ and $||| R_{l', k'} ||| \lesssim C(l, k)$.

 Recall that $(Q, Q') \in \Trunc_{D, out} \cap \Cent$.
			 We have
			\aryst
	 | U_{-1} |, | U_{0} |,  |W_{-1, -1} |, |W_{-1, 0} |  \leq 1,
			\earyst 		
 	and by Lemma \ref{lem smallK}, for every $(s, t) \in B(l, k) \setminus \{(-1, -1), (-1, 0) \}$, we have
			\aryst
			| W_{s, t} | \lesssim  D_{K}^{s + t + 1}  \langle \DZ S(\tw, \tz)  \rangle^{0.1(s + t + 1)}.
			\earyst 
			Then we have 
			\aryst
		|	R_{l', k'}  | \leq C(  l, k)    (D_{K}  \langle \DZ S(\tw, \tz)  \rangle^{0.1}   )^{k + l - k' - l'}.
			\earyst 
	We can apply  Corollary \ref{cor partialofphia'hatF}   to see that  $	(\cD_{\tz}^{*})^l (\cD_{V}^{*})^k u $ is bounded by  
	 	\aryst
	 C(\ell, l, k) \Lambda \Phi_{a', \ell}(\eta', \fw', \fz' \mid w', z') \cdot     \Phi_{a, \ell}(\eta , \fw, \fz \mid w, z) \cdot  \widehat\Phi_{a', \ell}(\eta', w', z' \mid \hat F(\tw, \tz) )  \cdot  \widehat \Phi_{a, \ell}(\eta , w, z \mid \tw, \tz )
	 	\earyst
	 	with 
	 	\begin{align*}
	 \Lambda  =&  D_{K-1}^{- \ell} \cdot \langle \DV S  \rangle^{- k} \langle  \DZ S \rangle^{- l}  \cdot \sum_{l' = 0}^{l} \sum_{k' = 0}^{k}      (D_{K}  \langle \DZ S \rangle^{0.1}   )^{k + l - k' - l'} \cdot  \\
	  & (    D_K \langle \DZ S \rangle^{0.1} \BV S^{0.1}     )^{l'} (   D_K^{51}     \langle \DV S \rangle^{0.1}     +  \langle \DZ S \rangle^{0.1} +   D^{0.01} \langle  \DV S \rangle / ( r_{min} \zeta_{max}  ) )^{k'}.
	 	\end{align*}
	 	Since $k = 3l$, we deduce that
	 	\ary \label{eq Lambdabound}
	 	\Lambda \lesssim C(\ell, l, k)  D_{K-1}^{- \ell} \cdot D_K^{51(l + k)} \cdot \BV S^{- k/5}  \langle  \DZ S \rangle^{- l/2}.
	 	\eary 	
 	If $K > 0$, then by  Lemma \ref{lem bigK} and  	
 \eqref{eq lowerboundpartialzS}   in Lemma \ref{lem smallK}, we have 
	 $$\Lambda \lesssim C(\ell, l) D_K^{- \ell/2} \cdot   \BV S^{- l / 4 }   \langle \eta_{max}^{-\beta_1}( \eta - \eta' \partial_z \hat{F}_z(\fw, \fz) ) \rangle^{- l / 80}.$$
 	  If $K = 0$, then by   \eqref{eq lowerboundpartialzS}   in Lemma \ref{lem smallK} and $D = D_0^{3000}$, we obtain 
	 $$\Lambda \lesssim C(l) D_0^{- l} \cdot  \BV S^{- l / 4}    \langle \eta_{max}^{-\beta_1}( \eta - \eta' \partial_z \hat{F}_z(\fw, \fz) ) \rangle^{- l / 80}.$$
		This completes the proof.	
\end{proof}

 	Now we will bound 
			\ary
			\label{eq cKbigQ'Qupperbound}  
			\int_{Q : (Q, Q') \in \Trunc_{D, out} \cap \Cent} | \cK_{K}(Q' \mid Q)| dQ.
			\eary
			We write $r = r_{a'}(\eta', \fw', \fz') $ and $R =  r_{a'}(\extra\eta', \fw', \fz')$.
					We write $C'$ as an universal constant which may vary from line to line.
		We have $R \gtrsim e^{ \hat\lambda_g \tinitial} r$ by	\eqref{eq Randrratio}.
  By \eqref{eq Rmaxminrmaxmin},  we have   
			\begin{align} 
\label{eq phibound1}				& \Phi_{a, 3}(\eta' , \fw', \fz' \mid w', z') \lesssim   D_K^{C'} R^{-1}  (\extra \eta_{max})^{\beta_1/2},  \\
	\label{eq phibound2}				& \Phi_{a, 3}(\eta , \fw, \fz \mid w, z) \widehat \Phi_{a, 3}(\eta , w, z \mid \tw, \tz  )    \\
				\lesssim&     D_K^{C'}   R^{-1} r^{-1}   \langle R^{-1} (w - \fw)  \rangle^{- 3}  \langle r^{-1} (w - \tw)  \rangle^{- 3}   \cdot \eta_{max}^{\beta_1/2} (\extra \eta_{max})^{\beta_1/2}  \langle (\extra \eta_{min})^{\beta_1} ( z - \fz ) \rangle^{- 3}   \langle \eta_{min}^{\beta_1}   (  \tz  - z   ) \rangle^{- 3}, \nonumber \\
	\label{eq phibound3}		\mbox{ and }  \ \ \	&   \widehat \Phi_{a', 3}(\eta', w', z' \mid   \hat F(\tw, \tz)   )  
				\lesssim D_K^{C'}  r^{-1}  \langle r^{-1} ( \hat F_w(\tw, \tz) - w'  )  \rangle^{- 3}  \cdot  \eta_{max}^{\beta_1/2}  \langle \eta_{min}^{\beta_1}   (  \hat F_z(\tw, \tz) - z'   ) \rangle^{- 3}.
			\end{align}


		By the above estimates and  Lemma \ref{lem integralboundforK}, \eqref{eq cKbigQ'Qupperbound} is bounded by  
			\begin{align}
			&  C(l)  D_K^{- l + C'}  \int   1_{\{|\eta| \geq |\eta'|\}}  \Big \{  \int_{W(Q', \eta) \times Z(Q',  \eta)}   \Big [ \int \Big (  \int  \BV S^{- l / 4}   d\xi  \Big ) \cdot     \langle \eta_{max}^{-\beta_1}( \eta - \eta' \partial_z \hat{F}_z(\fw, \fz) ) \rangle^{- l / 80}      \cdot   \nonumber  \\
			&    r^{-2}  R^{-2}  \cdot  (\extra \eta_{max})^{\beta_1}  \eta_{max}^{\beta_1}        \langle R^{-1} (w - \fw)  \rangle^{- 3}  \langle r^{-1} (w - \tw)  \rangle^{- 3}   \cdot    \langle (\extra \eta_{min})^{\beta_1} ( z - \fz ) \rangle^{- 3}  \cdot    \nonumber   \\
			&    \langle \eta_{min}^{\beta_1}   (  \tz  - z   ) \rangle^{- 3}  \cdot  \langle r^{-1} ( \hat F_w(\tw, \tz) - w'  )  \rangle^{- 3}  \cdot   \langle \eta_{min}^{\beta_1}   (  \hat F_z(\tw, \tz) - z'   ) \rangle^{- 3}  dw dz d\tw d\tz  \Big ]  d\fw d\fz  \Big \}    d\eta. \nonumber 
			\end{align} 
			Here $W(Q', \eta), Z(Q', \eta)$ are given by Lemma \ref{lem rangeoffwfz},centered at $w_0$, $z_0$ respectively.  
			We will evaluate the above integral according to the order of the variables shown above.
 			Note that 
			\aryst
			 \int  \BV S^{- l / 4}   d\xi  \lesssim  D_K^{C'} r^{-2}.  
			\earyst
			After integrating against $dw dz$, left hand side of \eqref{eq cKbigQ'Qupperbound} can be bounded by  
			\aryst
			&& C(  l)  D_K^{- l  + C'}  \int   1_{\{|\eta| \geq |\eta'|\}}  \Big \{  \int_{W(Q', \eta) \times Z(Q',  \eta)}   \Big [ \int   \langle \eta_{max}^{-\beta_1}( \eta - \eta' \partial_z \hat{F}_z(\fw, \fz) ) \rangle^{- l / 80}    \cdot   r^{- 2}  R^{-2}  \cdot  (\extra \eta_{max})^{\beta_1}       \cdot     \\
			&&     \langle r^{-1} ( \hat F_w(\tw, \tz) - w'  )  \rangle^{- 3}  \cdot   \langle \eta_{min}^{\beta_1}   (  \hat F_z(\tw, \tz) - z'   ) \rangle^{- 2}    d\tw d\tz  \Big ] d\fw  d\fz  \Big \}    d\eta . 
			\earyst
       We obtain after integrating against $d\tw d\tz$   the following bound
			\aryst
			&& C(  l)  D_K^{- l + C'}   \int 1_{\{|\eta| \geq |\eta'|\}}    1_{W(Q', \eta) \times Z(Q', \eta)}     \langle \eta_{max}^{-\beta_1}( \eta - \eta' \partial_z \hat{F}_z(\fw, \fz) ) \rangle^{- l / 4}    \cdot    R^{-2}   (\extra \eta_{max})^{\beta_1}   \eta_{max}^{- \beta_1}        d\fw d\fz  d\eta . 
			\earyst	 
By $(Q, Q') \in \Trunc_{D, out} \cap \Cent$, we have $\eta_{min} \sim \eta_{max} \gg e^{C \tinitial}$, and hence $| Z(Q', \eta) | = \cO(D_K^{C'} \langle \extra \eta_{max} \rangle^{- \beta_1})$. We may replace $( \eta - \eta' \partial_z \hat{F}_z(\fw, \fz) )$ in the above expression by $( \eta - \eta' \partial_z \hat{F}_z(w_0, z_0) )$ at the cost of adding a factor $\cO(D_K^C)$. 
Then after integrating against $d\fw d\fz$, we see that when $\greg$  is larger than some universal constant,  we have  
			\begin{align} \label{eq Schur11} 
	 	 \int_{Q : (Q, Q') \in \Trunc_{D, out} \cap \Cent}  | \cK_{K}(Q' \mid Q)| dQ &\lesssim   C(l) D_K^{- l}   \int_{|\eta| \geq | \eta'|}    \langle \eta_{max}^{-\beta_1}( \eta - \eta' \partial_z \hat{F}_z(w_0, z_0) ) \rangle^{- l / 80}   \eta_{max}^{- \beta_1}     d\eta  \\
	 	 & \lesssim C(l) D_K^{- l/2}. \nonumber
			\end{align}
			The same estimate holds if we switch the roles of $Q$ and $Q'$.

		\subsection{Proof of Item \eqref{itm 1 Prop81} in Proposition \ref{prop mainprop}: Non-central case} \label{subsec nonCent}

 		 	In this subsection, we let  $(Q, Q') \in \Trunc_{D, out} \setminus \Cent$. 
 		If $(Q, Q') \in \Sigma_D$ (see \eqref{def SigmaD}),	 then  item \eqref{itm distance3} in Definition \ref{def Distancedf} holds, and  we  will apply integration-by-parts with respect to  $\cD_{\tz}$ and  $\cD_{\tw}$ to write $\cK_{+}(Q' \mid Q)$ as 
		\ary   \label{eq cK0'1}
		&&  1_{\{|\eta| \geq |\eta'|\}}    e^{ - i (\xi, \eta) \cdot (w, z) + i (\xi', \eta') \cdot (w', z')}    \int 
		e^{i  S }    (\cD_{\tz}^*)^{l}  (\cD_{V}^*)^k  u d\tw d\tz;
		\eary 			 
			If $(Q, Q') \notin \Sigma_D$, then we apply integration-by-parts with respect to  $\cD_{\tz}$   to write $\cK_{+}(Q' \mid Q)$ as 
				\ary   \label{eq cK0'2}
		&&   1_{\{|\eta| \geq |\eta'|\}}      e^{ - i (\xi, \eta) \cdot (w, z) + i (\xi', \eta') \cdot (w', z')}   \int 
		e^{i  S }    (\cD_{\tz}^*)^{l}    u d\tw d\tz.
		\eary

		 Let  $K \in \N$.  If $(Q, Q') \in \Sigma_D$ (resp.  $(Q, Q') \notin \Sigma_D$), then we denote by $\cK_{ K}(Q' \mid Q)$ the restriction of the integral in \eqref{eq cK0'1} (resp. \eqref{eq cK0'2}) to the parameters $(\tw, \tz) \in R_K \setminus R_{K-1}$.

		\begin{lemma} \label{lem integralboundforK2}
			For any integer $l \geq 0$, there is an universal constant $\greg_0(l) > 0$ such that if $\greg > \greg_0(l)$
			and $\ell \geq 1$,   we have   
			{\em 	\begin{align}   
					\label{eq  integralboundforK}		 | \cK_{K}  (Q' \mid Q) | \leq&  C( l, \ell)   1_{\{|\eta| \geq  |\eta'|\}}      \int  \Lambda \Phi_{a', \ell}(\eta', \fw', \fz' \mid w', z') \cdot   \Phi_{a, \ell}(\eta , \fw, \fz \mid w, z) \cdot   \\
					&  \widehat\Phi_{a', \ell}(\eta', w', z' \mid \hat F(\tw, \tz) )  \cdot  \widehat \Phi_{a, \ell}(\eta , w, z \mid \tw, \tz )   d\tw d\tz \nonumber
				\end{align}
			}
			where 
			\aryst
			\Lambda = \begin{cases}
				D_K^{- l} \cdot  \eta_{max}^{- l c_0/ 4}    \BV S^{- l / 4} , & \mbox{$(Q, Q') \in \Sigma_D$,} \\
				D_K^{- l}  \cdot  \max( \eta_{max},  \xi_{max} )^{- c_0 l / 10}, & \mbox{$(Q, Q') \notin \Sigma_D$}.
			\end{cases}
			\earyst
		\end{lemma}

		\begin{proof}  
			We will only detail the proof when $0 \leq K \leq K_0$ as the proof for $K > K_0$ is similar but significantly easier. From now on we assume that $0 \leq K \leq K_0$.

			We first assume that $(Q, Q') \in \Sigma_D$ (that is, item \eqref{itm distance32} in Definition \ref{def Distancedf} holds).
			We will bound \eqref{eq cK0'1} using \eqref{eq DtzDVexpression}.
			Following the proof of Lemma \ref{lem integralboundforK}, using Lemma \ref{lem smallK2} in place of Lemma \ref{lem smallK}, we  see that  $	(\cD_{\tz}^{*})^l (\cD_{V}^{*})^k u $ is bounded by  
			\aryst
			C(\ell, l, k) \Lambda \Phi_{a', \ell}(\eta', \fw', \fz' \mid w', z') \cdot     \Phi_{a, \ell}(\eta , \fw, \fz \mid w, z) \cdot  \widehat\Phi_{a', \ell}(\eta', w', z' \mid \hat F(\tw, \tz) )  \cdot  \widehat \Phi_{a, \ell}(\eta , w, z \mid \tw, \tz )
			\earyst
			with 
			\aryst
			\Lambda  &=&   D_{K-1}^{- \ell} \cdot \langle \DV S  \rangle^{- k} \langle  \DZ S \rangle^{- l}  \cdot \sum_{l' = 0}^{l} \sum_{k' = 0}^{k}      (D_{K}  \langle \DZ S  \rangle^{0.1}  \BV S^{ 0.6 }   )^{k + l - k' - l'} \cdot  \\
		&& 	(    D_K \langle \DZ S \rangle^{0.1} \BV S^{0.1}     )^{l'} (   D_K^{51}     \langle \DV S \rangle^{0.1}     +  \langle \DZ S \rangle^{0.1} +   D^{0.01}  \langle  \DV S \rangle / ( r_{min} \zeta_{max}  ) )^{k'}.
			\earyst 
			It is direct to see that \eqref{eq Lambdabound} remains valid. 
			By Lemma \ref{lem smallK2}, the right hand side of \eqref{eq Lambdabound} is bounded by  
			\ary
			\Lambda \lesssim
				D_K^{- l} \cdot  \eta_{max}^{- l c_0/ 4}   \BV S^{- l / 4}.
			\eary

			Now assume that $(Q, Q') \notin \Sigma_D$ (that is, item \eqref{itm distance32} fails).
			We will bound \eqref{eq cK0'2} using \eqref{eq Dtzexpression}. 
			By Lemma \ref{lem smallK0}, we have $| U_{l'} | \lesssim \langle \DZ S \rangle^{l'/2}$ for every  $0 \leq l' \leq l$.     Consequently, we have $Q_{l'} \leq C(g, l)  \langle \DZ S \rangle^{l/2} $ for every  $0 \leq l' \leq l$.   
			We  apply  Corollary \ref{cor partialofphia'hatF}   to see that  $	(\cD_{\tz}^{*})^l u $ is bounded by  
			\aryst
			C(\ell, l, k) \Lambda \Phi_{a', \ell}(\eta', \fw', \fz' \mid w', z') \cdot     \Phi_{a, \ell}(\eta , \fw, \fz \mid w, z) \cdot  \widehat\Phi_{a', \ell}(\eta', w', z' \mid \hat F(\tw, \tz) )  \cdot  \widehat \Phi_{a, \ell}(\eta , w, z \mid \tw, \tz )
			\earyst
			with $ \Lambda =   D_{K-1}^{- \ell} \cdot ( D_K^{2}  \langle  \DZ S \rangle^{0.1}  )^{l } \cdot   \langle  \DZ S \rangle^{- l /2 }$.
 	We conclude the proof in a similar way as before using \eqref{eq lowerboundpartialzSa} in Lemma \ref{lem smallK0} in place of Lemma \ref{lem smallK2}.	 
		\end{proof}

Denote $r = r_{a'}(\eta', \fw', \fz') $ and $R =  r_{a'}(\extra\eta', \fw', \fz')$ as before.
		By Lemma \ref{lem integralboundforK2}, \eqref{eq Rmaxminrmaxmin}  and \eqref{eq phibound1} to \eqref{eq phibound3} (which remain valid), we see that  
		\begin{align}
			\label{eq cKbigQ'Qupperbound2}  
			& \int_{Q : (Q, Q') \in  \Trunc_{D, out} \setminus \Cent } | \cK_{K}(Q' \mid Q)| dQ     \\
			\leq&  C( l)  D_K^{- l + C'}  \int  1_{\{|\eta| \geq  |\eta'|\}}  \Big \{  \int_{W(Q', \eta) \times Z(Q',  \eta)}   \Big [ \int \Big (  \int   \eta_{max}^{- l c_0/ 4}  \cdot \BV S^{- l /  4 } 1_{\Sigma_D}(Q, Q')   +  \nonumber  \\
			&   \max( \eta_{max},  \xi_{max} )^{- c_0 l / 10}   1_{\Sigma_D^c}(Q, Q')  d\xi  \Big ) \cdot      r^{-2}  R^{-2}  \cdot  (\extra \eta_{max})^{\beta_1}  \eta_{max}^{\beta_1}         \langle r^{-1} (w - \tw)  \rangle^{- 3}   \cdot    \nonumber   \\
			&     \langle \eta_{min}^{\beta_1}   (  \tz  - z   ) \rangle^{- 3 }  \cdot    \langle r^{-1} ( \hat F_w(\tw, \tz) - w'  )  \rangle^{- 3}  \cdot   \langle \eta_{min}^{\beta_1}   (  \hat F_z(\tw, \tz) - z'   ) \rangle^{- 3}  dw dz d\tw d\tz  \Big ]  d\fw d\fz  \Big \}    d\eta. \nonumber 
		\end{align} 
		Here $W(Q', \eta)$, $Z(Q', \eta)$ are given by Lemma \ref{lem rangeoffwfz}.
		Note that  by Lemma \ref{lem smallK0}, we have 
		\aryst
		&&  \int   \eta_{max}^{- l c_0/ 4}   \BV S^{- l / 4} 1_{\Sigma_D}(Q, Q')   +   \max( \eta_{max},  \xi_{max} )^{- c_0 l / 10}     1_{\Sigma_D^c}(Q, Q')  d\xi      \lesssim D_K^C r^{-2}  \eta_{max}^{- l c_0/ 20} . 
		\earyst
 	 We may bound \eqref{eq cKbigQ'Qupperbound2} following the argument in the previous subsection (using in addition Corollary \ref{rem LargenessofDVSDzD}).   We deduce by \eqref{eq Rmaxminrmaxmin} that when $\greg$  is larger than some universal constant, we have 
		\ary
		\int_{Q : (Q, Q') \in \Trunc_{D, out} \setminus \Cent }  | \cK_{K}(Q' \mid Q)| dQ  
		 \lesssim C(l) D_K^{- l/2}.   \label{eq Schur12} 
		\eary
		The same estimate holds if we switch the roles of $Q$ and $Q'$.

	\begin{proof}[Finishing the proof of  Item \eqref{itm 1 Prop81} in Proposition \ref{prop mainprop}]
		Since   $| \cK_{+} | \leq \sum_{K = 0}^{\infty} | \cK_{K} |$, by summing up the above inequalities over all $K \in \N$, and by  \eqref{eq Schur11} and  \eqref{eq Schur12}, we see that the inequalities in \eqref{itm 1 Prop81} of Proposition \ref{prop mainprop} hold for $\cK_{+}$ in place of $\cK$. We obtain a parallel statement for $\cK_{-}$ by a similar argument. This concludes the proof since $\cK =  \cK_{+} + \cK_{-}$.
	\end{proof}

		\subsection{Proof of Item \eqref{itm 2 Prop81} in Proposition \ref{prop mainprop}}  \label{subsec proofitm2prop81}

		We denote by $\widetilde\Gamma_a$ the set of $P = (w, z, \hat{w}, \eta \mid \fw, \fz)$ with $(w, z,  \eta \mid \fw, \fz) \in \Gamma_a$ and $\hw \in \R^2$.
		We define  a linear operator $\cT_a :  L^2(\hat\Gamma_a) \to  L^2(\widetilde\Gamma_a)$ by 
		\aryst
		(\cT_a v)(w, z, \hat{w}, \eta \mid \fw, \fz) =  \frac{1}{2\pi}  \int  e^{ i \xi ( \hat{w} - w )} v(w, z, \xi, \eta \mid \fw, \fz ) d\xi.
		\earyst
		It is direct to see that $\cT_a$ is a  $L^2$-isometry, and the inverse of $\cT_a$ is given by
		\aryst
		(\cT_a^{-1} \tilde{v} )(w, z, \xi, \eta \mid \fw, \fz) =  \frac{1}{2\pi}  \int  e^{ -  i \xi ( \hat{w} -  w ) } \tilde{v}(w, z, \hat{w}, \eta \mid \fw, \fz ) d\hat{w}.
		\earyst
		We define $\cT_{a'}$   in a similar way.
		It suffices to prove $\|  \cT_{a'} \hL_{a \to a'}^{R}  \cT_{a}^{-1} \|_{L^2  \to L^2  } \leq C( l) D^{- l}$.
		
		By comparing the expressions in \eqref{eq cK}-\eqref{eq amplitude}, the kernel of $\cT_{a'} \hL_{a \to a'}^{R}  \cT_{a}^{-1}$ is given by
		\ary
		\label{eq cK00}	  	\widetilde\cK ( w', z', \hw', \eta', \fw', \fz' \mid w, z, \hw, \eta, \fw, \fz) =   e^{ - i  \eta z + i  \eta'  z'}   	\int  	
		e^{i  S_0 }    u_1 u_2     d\tw d\tz
		\eary
		where $	S_0 = S_0(\hat{w}, \eta, \hat{w}', \eta', \tw, \tz) =  \eta  \tz   -   \eta' \hat{F}_z(\tw, \tz)$ and
		\begin{align}  
			u_1 &= \widetilde{\phi}_{a'}(  \eta', w', z' \mid  \fw', \fz' \mid \hat{F}(  \tw, \tz ) )      \widetilde{\phi}^{\dagger}_a(  \eta, w, z \mid  \fw, \fz \mid    \tw, \tz  )   \rho_F(\tw, \tz - d_{a, \fw, \fz, \eta}(\tw) ),  \\
			u_2 &= (2\pi)^{-2} \int e^{i \xi' ( \hw' - \hat{F}_w(\tw, \tz)) - i \xi (\hw - \tw) }  R(w, z, \xi, \eta, \fw, \fz, w', z', \xi', \eta', \fw', \fz' ) d\xi d\xi'.   
		\end{align}
		We set
		\aryst
		\cD_{\xi} = \frac{1 + i (D \hDelta_a(\eta, \fw, \fz) )^2 (\hw -\tw ) \cdot \partial_{\xi} }{1 + \|  D \hDelta_a(\eta, \fw, \fz) (\hw -  \tw )  \|^2 }, \quad
		\cD_{\xi'} = \frac{1 - i (D \hDelta_{a'}(\eta', \fw', \fz') )^2 (\hw' - \hat{F}_w(\tw, \tz) ) \cdot \partial_{\xi'} }{1 + \| D  \hDelta_{a'}(\eta', \fw', \fz') (\hw' - \hat{F}_w(\tw, \tz))  \|^2 }.
		\earyst
		By applying integration-by-parts with respect to $\cD_{\xi}$ and $\cD_{\xi'}$ several times, we may rewrite $u_2$ as
		\begin{align} \label{eq u2expression}
			(2\pi)^{-2} \int e^{i \xi' ( \hw' - \hat{F}_w(\tw, \tz)) - i \xi (\hw - \tw) } (\cD_{\xi}^{*})^{k}  (\cD_{\xi'}^{*})^{k} R(w, z, \xi, \eta, \fw, \fz, w', z', \xi', \eta', \fw', \fz' ) d\xi d\xi'. 
		\end{align}
		
		Given $P = (w, z, \hat{w}, \eta \mid \fw, \fz)$  and $P' = (w', z', \hat{w}', \eta' \mid \fw', \fz')$, we will use notations in  \eqref{eq Rminmax} to \eqref{eq etaminmax}.
		By applying integration-by-parts with respect to $\cD_{\tz} =   ( 1 - i \DZ S_0 \cdot \DZ  ) / ( 1 +     | \DZ S_0 |^2 )$, we obtain
		\ary \label{eq K3}
		\widetilde\cK (P' \mid P) =    e^{ - i  \eta z + i  \eta'  z'}   	\int  e^{i  S_0 }  (\cD_{\tz}^{*})^l  (u_1 u_2 )   d\tw d\tz.
		\eary
		Now we may follow the proof of  item \eqref{itm 1 Prop81} in Proposition \ref{prop mainprop}. 
		For each $K \in \N$, 	we denote by $\widetilde\cK_{ K}$ the restriction of the integral in  \eqref{eq K3} to the parameters $(\tw, \tz) \in R_K \setminus R_{K-1}$.
		By  Lemma \ref{lem bigK}, Lemma \ref{eq ratioetamaxminbig},   Lemma \ref{lem smallK0} and their proofs, we see that for every $(\tw, \tz) \in R_{K}$
		\aryst
		\langle \DZ S_0(\tw, \tz)  \rangle \gtrsim 
		\begin{cases}
	 	D_K^{- 500}  ( \eta_{max}/ \eta_{min} )^{1 - \beta_1}, &  K > K_0,  \\
			   \langle \eta_{max}^{-\beta_1}( \eta - \eta' \partial_z \hat{F}_z(\fw, \fz) ) \rangle \gtrsim D, &  0 \leq K \leq K_0
		\end{cases} 
		\earyst
		and $\|	 \DZ^l   S_0(\tw, \tz) \|  \lesssim     \langle \DZ S_0(\tw, \tz)  \rangle^{1 + l/4}$ for every $2 \leq l \leq \greg - 1$.

		By hypothesis, we have $\supp(R) \subset \Trunc_{D, in} := \Trunc_{ (1)}^{> \omegalow } \cap ( \Trunc_{ (2)}^{> D} \setminus \Trunc_{ (2)}^{> 2D} ) \cap  \Trunc_{ (3)}^{\leq D}$. 
		Then in the support of the integral of \eqref{eq u2expression}, we have $\zeta_{max} := \max(\langle \zeta \rangle, \langle \zeta' \rangle) \lesssim D \hDelta_{max} \lesssim_{\esmall} D \eta_{max}^{1/2 + c(\esmall)}$. 
		By definition we have $D_{max}  =  \max( \xi_{max} / \eta_{max}^{\beta_1  }, \eta_{max}^{c_0} ) > D$. 
		By letting $\omegalow \gg_{\tinitial} 1$, we can deduce that  $\xi_{max} < \eta_{max}^{\beta_1}$ and $D < \eta_{max}^{c_0}$.
		The volume of the support of the integrand in \eqref{eq u2expression} is of size $\cO( ( D \hDelta_a(\eta, \fw, \fz) )^2 (D \hDelta_{a'}(\eta', \fw', \fz') )^2 )$. Using the  derivative bounds of $R$ in item \eqref{itm 2 Prop81} of Proposition \ref{prop mainprop},   \eqref{eq partialzFw} in Proposition \ref{lem derivativeboundsforhatF} and \eqref{eq partialFbound0},  we deduce by straightforward computations that $| \DZ^l u_2 | $ is bounded by
		\aryst
  C(l, k)  ( D  \eta_{max}^{- \beta_1}  \hDelta_{max}   )^{l }  \langle D \hDelta_a(\eta, \fw, \fz) (\hw -  \tw )  \rangle^{- k} \langle D \hDelta_{a'}(\eta', \fw', \fz') (\hw' - \hat{F}_w(\tw, \tz)) \rangle^{- k} ( D \hDelta_a(\eta, \fw, \fz) )^2 (D \hDelta_{a'}(\eta', \fw', \fz') )^2.
		\earyst
		Moreover, we see that \eqref{eq partialsofphia'hatF2} in Lemma \ref{lem partialofphia'hatF} and \eqref{eq partialsofrho}  hold.   Consequently, we see that $	(\cD_{\tz}^{*})^l (u_1 u_2) $  is bounded by 	
		\aryst
		&& C(\ell, l, k) \Lambda \Phi_{a', \ell}(\eta', \fw', \fz' \mid w', z') \cdot     \Phi_{a, \ell}(\eta , \fw, \fz \mid w, z) \cdot  \widehat\Phi_{a', \ell}(\eta', w', z' \mid \hat F(\tw, \tz) )  \cdot  \widehat \Phi_{a, \ell}(\eta , w, z \mid \tw, \tz ) \cdot \\
		&&  \langle D \hDelta_a(\eta, \fw, \fz) (\hw -  \tw )  \rangle^{- k} \langle D \hDelta_{a'}(\eta', \fw', \fz') (\hw' - \hat{F}_w(\tw, \tz)) \rangle^{- k} (D \hDelta_a(\eta, \fw, \fz) )^2 ( D \hDelta_{a'}(\eta', \fw', \fz') )^2
		\earyst
		with $ \Lambda =   D_{K-1}^{- \ell}  \cdot ( D_K^{2}  \langle  \DZ S_0 \rangle^{0.1}  )^{l } \cdot   \langle  \DZ S_0 \rangle^{- l}$.
		Here we have also used that $D  \eta_{max}^{- \beta_1}  \hDelta_{max} \lesssim   \langle  \DZ S_0 \rangle^{- 1}$ when $K \leq K_0$, and that $D  \eta_{max}^{- \beta_1}  \hDelta_{max} \lesssim  D_{K-1}^{C}$ when $K > K_0$.
		We may bound the above expression  as in Subsection \ref{subsec proofitm1prop81}.
		Then by direct computations, we obtain
		\aryst
		\sup_{P' \in \widetilde\Gamma_{a'}} \int_{\widetilde\Gamma_{a}}  | \widetilde \cK_{K}(P' \mid P)| dP \lesssim C(l) D_K^{- l/2}.
		\earyst
		We may obtain $\sup_{P \in \widetilde\Gamma_{a}}	 \int_{\widetilde\Gamma_{a'}}  | \widetilde \cK_{K}(P' \mid P)| dP'  \lesssim C(l) D_K^{- l/2}$ by a similar argument. 
		Since   $| \widetilde \cK | \leq \sum_{K = 0}^{\infty} |\widetilde \cK_{K} |$, we may conclude the proof by   Schur's test.

					\section{Proof of Proposition \ref{lem formerL6.8}}   \label{sec lem formerL6.8}

 	 As in  Section \ref{sec distanceonphasespace}, we denote  $f^{\tinitial}_{a \to a'} = F = (F_w, F_z)$.
Throughout this section, we will use $\fq$, $\fq'$, $q$, $q'$, etc., to denote $(\fw, \fz)$, $(\fw', \fz')$, $(w, z)$, $(w', z')$, etc., and use   $P$, $P'$   to denote $(\xi, \eta)$, $(\xi', \eta')$, etc.

We first choose some partitions $\cP$ along  the $\eta$-variable.
 We may let $\cP$  be the partition of $\R$ into  left-open right-closed intervals separated by points $ \pm k^{\frac{1}{1 - \beta_1}}$, $ k \in \N$. 
For $\eta \in \R$, we denote by $\cP(\eta)$ the unique atom in $\cP$ containing $\eta$.
We have  $| \cP(\eta) | \sim \langle \eta \rangle^{\beta_1}$ for every $\eta \in \R$.
 For each $B \in \cP$, we let $\eta_{B}$ denote the mid-point of $B$.

By Lemma \ref{lem adaptedscaleistame}, 
for any $B \in \cP$ with $B \not\subset [-\omegalow, \omegalow]$, there is a finite set $D_{a, B} \subset \widehat\Omega$ (depending only on $B$), and  nonnegative functions $ \varphi_{\frq_0}$ on $\widehat\Omega$,  $\frq_0 \in D_{a, B}$, such that
for every $\frq_0 \in D_{a, B}$, we have
\ary \label{eq varphiandwidetildevarphi}
\supp( \varphi_{\frq_0} ) \subset \NGB_a^{   \extra^{-2}}(\eta_{B}, \frq_0 ) = B( \frq_0,     \extra^{-2}   r_a(\eta_{B}, \frq_0 ))
\eary 
and for every $\fq \in \widehat\Omega$,  we have  
\aryst 
\sum_{\frq_0 \in D_{a, B}}    \varphi_{\frq_0}(\fq)   \equiv 1  \mbox{ and }  \sum_{\frq_0 \in D_{a, B}}  1_{\supp( \varphi_{\frq_0})}(\fq)  \lesssim 1.   
\earyst
This gives a partition of unity in $\fq$-variables. 
 Similarly, there is a finite set $\widetilde{D}_{a', B'} \subset \widehat\Omega$ (depending on $B'$), and a collection of nonnegative functions $\widetilde\varphi_{\frq'_0}$ on $\widehat\Omega$, $\frq'_0 \in \widetilde{D}_{a', B'}$, such that for every $\frq'_0 \in \widetilde{D}_{a', B'}$, we have 
\ary
\label{eq varphiandwidetildevarphi'}
\supp( \widetilde\varphi_{\frq'_0} ) \subset \NGB_{a'}^{  \extra^{-1}}(\eta_{B'}, \frq'_0 ) = B( \frq'_0,  \extra^{-1}   r_{a'}(\eta_{B'}, \frq'_0 ))
\eary
and for every $\fq' \in \widehat\Omega$, we have 
\aryst
\sum_{\frq'_0 \in \widetilde{D}_{a', B'}}    \widetilde\varphi_{\frq'_0}(\fq')   \equiv 1  \mbox{ and }  \sum_{\frq'_0 \in \widetilde{D}_{a', B'}}  1_{\supp( \widetilde\varphi_{\frq'_0})}(\fq')  \lesssim 1.  
\earyst

Finally, we define a cut-off function in $\xi$-variable. 
Fix a $C^{\infty}$ function $\chi  : \R^2 \to [0, 1]$ satisfying $\chi |_{B(0, 2)} \equiv 1$ and $\supp(\chi) \subset B(0, 3)$.    We define a  $C^{\infty}$ function $\chi_{a} :  \hat\Gamma_a   \to [0,1 ]$ by 
\ary \label{eq defchiam}
\chi_{a}( q, P \mid \fq ) =  \chi ( \langle (  \xi  -   \eta \estar_a(\eta, \fq \mid q )  )  / (  e^{\deltasharp \tinitial }    \hDelta_{a}(\eta, \fq) )   \rangle  ) \ \mbox{ where $P = (\xi, \eta)$}.
\eary
It is clear that $\chi_a \equiv 1$ on $\hat\Gamma_a( e^{\deltasharp \tinitial})$, and $\supp(\chi_a) \subset \hat\Gamma_a( 3 e^{\deltasharp \tinitial})$.
We have
\ary \label{eq smoothnessofchiam}
\| \partial^k_{\xi}  \big ( \chi_{a}( q, P \mid \fq )  \big )  \| \leq C(k)  ( e^{\deltasharp \tinitial }  \hDelta_{a}(\eta, \fq ) )^{ - k},  \quad  k \geq 0.
\eary

We let $\cJ_a$ be the set of tuples $\cI = (B, \frq)$ with $B \in \cP$ satisfying $B \not\subset  [- \omegalow, \omegalow] $, and  $\frq \in D_{a, B}$.
We define for evey $\cI = (B, \frq) \in \cJ_{a}$ that 
\ary \label{eq defpsiamBd}
\psi_{a, \cI}( q, P \mid \fq ) = \chi_{a}( q, P \mid \fq )  1_{B}(\eta) \varphi_{\frq}(\fq )   .
\eary

 By construction, it is clear that 
\ary \label{eq sumofallvarphidis1}
\sum_{\cI \in \cJ_a} \psi_{a,  \cI} \equiv 1 \ \mbox{ on $\hat\Gamma_a(e^{\deltasharp \tinitial})$}  \mbox{ and } \sum_{\cI \in \cJ_a}  1_{\supp(\psi_{a,  \cI})} \lesssim 1.   
\eary
By \eqref{eq smoothnessofchiam}, for every $\cI \in \cJ_a$, we have 
\ary	  \label{eq psismoothness1}
\| \partial^k_{\xi}    \psi_{a,  \cI} ( q, P \mid \fq  )   \| \leq C(k)     ( e^{\deltasharp \tinitial }  \hDelta_{a}(\eta, \fq ) )^{ - k},  \quad k \geq 0.
\eary

In an analogous way, we define $\chi_{a'}$, $D_{a', B}$ for each $B \in \cP$,  $\cJ_{a'}$ and $ \widetilde\psi_{a',  \cI'}  $ for each $\cI' \in \cJ_{a'}$  using $\widetilde{\varphi}_{\frq'}$ in place of $\varphi_{\frq}$.   Parallel estimates as \eqref{eq psismoothness1} and \eqref{eq sumofallvarphidis1} hold for $\psi_{a', \cI'}$ for each $\cI'  \in \cJ_{a'}$.

Given $\cI = (B, \frq_0 ) \in \cJ_a$ and $\cI' = (B', \frq'_0 ) \in \cJ_{a'}$, we say $\cI \to \cI'$ if 
$F( \NGB_a^{2 \extra^{- 2}}(\eta_{B}, \frq_0) ) \cap \NGB^{2 \extra^{- 2}}_{a'}(\eta_{B'}, \frq'_0) \neq \emptyset$ and $| \eta - \eta' \partial_z F_z(\frq_0) |  \leq   e^{ \deltasharp \tinitial}   | \eta_{B} |^{\beta_1}$.
We can see that for every $\cI' \in \cJ_{a'}$, there are $\cO(e^{\deltasharp \tinitial})$ many $\cI \in \cJ_{a}$ with $\cI \to  \cI'$, since the support of $\widetilde\varphi_{\frq'_0}$ is much smaller than that of $\varphi_{\frq_0}$.

Given any $Q = (  q, P \mid \fq ) \in \hat\Gamma_a$ and $Q' = (q', P' \mid \fq') \in \hat\Gamma_{a'}$, we define 
\aryst
R_0(Q, Q') = \sum_{\substack{\cI \in \cJ_a,  \cI' \in \cJ'_{a'}: \cI \to \cI' } }   \psi_{a, \cI }(Q)   \widetilde\psi_{a',  \cI'}  (Q').
\earyst
By definition, we see that for every   $(Q, Q') \in \Trunc_{ (1)}^{>  \omegalow } \cap \Trunc_{ (2)}^{\leq  e^{\deltasharp \tinitial / 4}   } \cap  \Trunc_{ (3)}^{\leq  e^{\deltasharp \tinitial} }$, we have $\hat{F}( \NGB^{\extra^{- 1}}_a(\eta, \fq) ) \cap \NGB^{\extra^{- 1}}_{a'}(\eta', \fq') \neq \emptyset$
and  $| \eta - \eta' \partial_z \hat{F}_z(\fq) |   \lesssim e^{\deltasharp \tinitial / 2} \eta_{max}^{\beta_1}$.
Thus for every $\cI \in \cJ_{a}$ with $\psi_{a, \cI}(Q) \neq 0$ and every $\cI'  \in \cJ_{a'}$ with $\psi_{a', \cI'}(Q') \neq 0$, we have $\cI \to \cI'$.  
Consequently, we see that $R_0 = 1$ on $\Trunc_{ (1)}^{>  \omegalow } \cap \Trunc_{ (2)}^{\leq   e^{\deltasharp \tinitial / 4}  } \cap  \Trunc_{ (3)}^{\leq  e^{\deltasharp \tinitial}}$.

For each $\cI = (B, \frq)  \in \cJ_{a}$, we define
\ary  \label{eq defpsia'I'}
\Psi_{a, \cI} :=  \Big (  \sum_{\cI'  \in \cJ_{a'}: \cI \to  \cI' } \widetilde\psi_{a', \cI'}^2 \Big )^{1/2}.
\eary
Fix an arbitrary $\cI' = (B', \frq'_0 ) \in \cJ_{a'}$ with $\cI \to \cI'$.
By definition, any $(q', P' \mid \fq')  \in \supp(\Psi_{a, \cI} )$ with $P' = (\xi', \eta')$ satisfies that
\ary \label{eq rangeofcI'}
\fq' \in  B( F(\frq),  \extra^{-3}   r_{a'}(\eta_{B'}, F(\frq) )) \ \mbox{ and } \ \eta' \in B(\eta_{B'}, e^{2 \deltasharp \tinitial} | \eta_{B'} |^{\beta_1} ).
\eary

 We define  
 \begin{align} \label{eq  	defwidetildepsi}
 	\widetilde\psi(q,  P \mid \fq)	=&
 	\chi ( \langle (  \xi  -   \eta \estar_a(\eta, \fq \mid q)  )  /  ( 8 e^{\deltasharp \tinitial }  \hDelta_{a}(\eta, \fq) )    \rangle  ) \cdot  \\
 	& 	1_{[0, 1)}(    \| \fq - \frq_0 \| / (     \extra^{-3}  r_{\cI} ))  ) \cdot	1_{(-1, 1)}( (\eta -  \eta_{\cI} ) / ( e^{2 \delta_{\sharp } \tinitial}  \langle   \eta_{\cI} \rangle^{\beta_1} ) ). \nonumber
 \end{align}
 We have $\widetilde\psi \equiv 1$ on $\supp( \psi_{a, \cI } )$. Moreover,   for every   $Q = (q, P \mid \fq ) \in \hat\Gamma_a \cap \supp( \psi_{a, \cI } )$ and  $Q' = ( q', P' \mid \fq') \in \hat\Gamma_a \cap \supp(1 -  	\widetilde\psi )$, we have the following: 
 (i) $Q \in \hat\Gamma_a( 3 e^{\deltasharp \tinitial} )$ and $\eta \in [\eta_B - C | \eta_B |^{\beta_1},  \eta_B + C | \eta_B |^{\beta_1} ]$,
 (ii)  either  $| \eta - \eta' | > e^{3 \deltasharp \tinitial} \eta_{max}^{\beta_1}$, or $\NGB^{1/ \extra }_a(\eta, \fw, \fz) \cap \NGB^{1/ \extra }_{a}(\eta', \fw', \fz') = \emptyset$, or $Q' \not\in   \hat\Gamma_a( 8 e^{\deltasharp \tinitial }  )$.

 We fix an arbitrary smooth function $\widetilde\rho$ supported in $\Omega$ such that
 $\widetilde\rho \equiv 1$ on $\widetilde\varphi_{a}    \circ \kappa_{a}^{-1}$. In particular, we have  $ \rho  \widetilde\rho =  \rho$. 
 Assume that $\tinitial$ is large enough depending on $\delta_{g}$ and $\esmall$ is small enough depending on $\tinitial$ and $g$, the following lemma holds.
 
 \begin{lemma} \label{lem localization1}
 	For every $\ell > 0$, we have
 	$\|  M( 1 - \widetilde\psi ) \Barg_a M( \widetilde\rho) \Barg_a^{\dagger}  \|_{L^2(\supp(  \psi )) \to L^2} \leq C(\ell) e^{-  \ell \deltasharp \tinitial }$. 
 \end{lemma}
 
 \begin{proof}
 	The proof of this lemma is similar to that of Proposition \ref{prop mainprop}, except that the situation here is much simpler since we don't need to analyse the effect of $F$   (that is why the inequality holds for every $\ell$ even though we only assume finite smoothness of $f$).   
 \end{proof}

 In order to make our argument symmetric for taking inverse $f^{-1}$, we will consider an operator $$\hK_{a \to a'} = \Barg_{a'} \CC_{F^{-1}}  M( \rho)  \Barg_a^{\dagger}$$ 
 for some $C^{\greg - 1}$ function $\rho$ satisfying \eqref{eq smoothnessofrho} in place of $\rho_{a \to a'} $.
 	 Assume that $\tinitial$ is large enough depending on $\delta_{g}$,  $\esmall$ is small enough depending on $\tinitial$ and $g$, and $\omegalow$ is large enough depending on $ \tinitial$ and $g$, the following lemma holds.

\begin{lemma} \label{lem formerL6.82} 
	We have
\ary \label{eq cI'Kaa'cI}
\| M(  \Psi_{a, \cI}  ) \hK_{a \to a'} M( \widetilde\psi  ) \|_{L^2 \to L^2}  \lesssim    \exp( - 8 \deltasharp  \tinitial ). 
\eary
In particular, we have 
$$ \| M(   \Psi_{a, \cI } ) \hL_{a \to a'} M( \psi_{a, \cI } ) \|_{L^2 \to L^2}  \lesssim    \exp( - 8 \deltasharp  \tinitial ),$$
since  $M(   \Psi_{a, \cI } ) \hL_{a \to a'} M( \psi_{a, \cI } )  = M(  \Psi_{a, \cI}  ) \hK_{a \to a'} M( \widetilde\psi )  M( \psi_{a, \cI } ) $ for $\hK_{a \to a'}$ given by $\rho = \rho_{a \to a'}$.
\end{lemma}
We can now deduce Proposition \ref{lem formerL6.8} from Lemma \ref{lem formerL6.82}.
\begin{proof}[Proof of Proposition \ref{lem formerL6.8}] 
 	Take $u \in C^{\infty}_c(\hat\Gamma_a)$. We have 
	\aryst
	\|  \hL^{R_0}_{a \to a'} u \|_{L^2}^2 &\lesssim&  \sum_{\cI' \in \cJ_{a'}}	\|  M(  \widetilde\psi_{a', \cI'}   ) \hL_{a \to a'}  \sum_{\cI \in \cJ_{a}: \cI \to  \cI'}  M( \psi_{a, \cI} )  u \|_{L^2}^2  \\
\mbox{(Cauchy's inequality)} 	 &\lesssim& e^{\deltasharp \tinitial}   \sum_{ \cI'  \in \cJ_{a'}}  	  \sum_{\cI  \in \cJ_{a}: \cI \to  \cI' } \| M(  \widetilde \psi_{a', \cI' } ) \hL_{a \to a'} M( \psi_{a, \cI } )  u \|_{L^2}^2  \\
\mbox{ (\eqref{eq defpsia'I'}) } 	 &=&   e^{\deltasharp \tinitial}	  \sum_{ \cI  \in \cJ_{a}}   \| M(   \Psi_{a, \cI } ) \hL_{a \to a'} M( \psi_{a, \cI } )  u \|_{L^2}^2  \\
 \mbox{(Lemma \ref{lem formerL6.82})} &\lesssim&       	  e^{   - 15 \deltasharp  \tinitial}   \sum _{\cI \in \cJ_{a}}   \|   1_{\supp( \psi_{a,  \cI } )}  u \|_{L^2}^2   
 \lesssim    e^{- 15 \deltasharp \tinitial}    \|     u \|_{L^2}^2.
	\earyst
	The first and the last inequality above follow from  \eqref{eq sumofallvarphidis1} and the parallel inequality for $\widetilde\psi_{a',  \cI'} $;
	the second inequality follows from  $|\{ \cI \in \cJ_{a} \mid \cI \to  \cI' \} | = \cO(e^{ \deltasharp \tinitial} )$. This concludes the proof.
\end{proof}


				 We now proceed to the proof of Lemma \ref{lem formerL6.82}. 
We let $\cI = (B, \frq_0)$.  
We fix an arbitrary $\cI' = (B', \frq'_0 )$ with $\cI \to \cI'$.
We abbreviate $r_{\cI} =  r_a( \eta_{B}, \frq_0 )$, $r_{\cI'} =  r_{a'}( \eta_{B'},   \frq'_0     )$, $\hDelta_{\cI} = \hDelta_a(\eta_{B}, \frq_0 )$,  $\hDelta_{\cI'} =  \hDelta_{a'}(\eta_{B'},   \frq'_0   )$, $\eta_{\cI} = \eta_{B}$, $\eta_{\cI'} = \eta_{B'}$.  
We also abbreviate $\psi = \psi_{a, \cI}$ and  $\psi' = \Psi_{a, \cI}$.

 We consider the following operator
					\aryst
				   \hP_2 =  M( \psi' )  \hK_{a \to a'}		 M( \widetilde \psi ) \Barg_{a}   M(\widetilde\rho).
					\earyst
					Note that $\hP_2$ depends on    $\omegalow$   through  the definitions of $\cJ_a$ and $\cJ_{a'}$. 
					Recall that $\tinitial$ is large enough depending on $\delta_{g}$,  $\esmall$ is small enough depending on $\tinitial$ and $g$, and $\omegalow$ is large enough depending on $ \tinitial$ and $g$.
					The main estimate  is the following.
					\begin{lemma} \label{lem hp2}
 	We have 
						\aryst
						\|\hP_2\|_{L^2 \to L^2} \lesssim \exp( - 8 \deltasharp  \tinitial).
						\earyst
					\end{lemma}

					We are now ready to deduce Lemma \ref{lem formerL6.82} from the above lemmas. 
					\begin{proof}[Concluding the proof of Lemma \ref{lem formerL6.82}] 
						
						We have the following identity
						\aryst
						M( \psi'  )    \hK_{a \to a'}  = M( \psi'  )    \hK_{a \to a'}		  M( 1 - \widetilde\psi )   \Barg_a M( \widetilde \rho) \Barg_a^{\dagger}  +    \hP_2  \Barg_a^{\dagger}.
						\earyst
						Here we have  used that   $\hK_{a \to a'} = \hK_{a \to a'} \Barg_a M( \widetilde\rho) \Barg_a^{\dagger}$, which follows from Theorem \ref{thm DWPT}\eqref{thm DWPT item 1}. 
						By Lemma \ref{lem localization1},
						the norm of the first term on the right hand side above (as an operator from $ L^2(\supp(\psi ))$ to $L^2$) is of size $C(\ell) e^{-  l \deltasharp \tinitial }$.
						The last term above is bounded by $\cO(e^{- 8 \deltasharp  \tinitial})$ by Lemma \ref{lem hp2}. 
					\end{proof}

				To complete the proof of Proposition \ref{lem formerL6.8}, it remains to prove  Lemma \ref{lem hp2}. This will be the object of the rest of the section.
					\begin{proof}[Proof of Lemma \ref{lem hp2}]

						Recall that $\cI = (B, \frq_0)$.
						We denote $\widetilde{B} = B( \eta_{\cI},    e^{ 2 \deltasharp \tinitial}  \langle  \eta_{\cI} \rangle^{\beta_1})$.
By $\cI \to \cI'$, we have 						
\ary \label{eq etacIetacI'bigandclose}
| \eta_{\cI} | \sim | \eta_{\cI'} | \ \mbox{ and } \  | \eta_{\cI} | \gtrsim \omegalow.
\eary
We may assume that $\| \frq'_0 - F(\frq_0) \| \lesssim  \extra^{-3} r_{\cI}$.
By letting $\esmall$ be small and by Lemma \ref{lem adaptedscaleistame},   for $(q, P \mid \fq ) \in \supp(\widetilde\psi)$  we have $\eta \in \widetilde{B}$, and
						\begin{align} \label{eq rrdeltadeltaratio} 
				    \log(	r_a(\eta, \fq ) / r_{\cI} ),   \log(	r_{a}(\eta, q ) / r_{\cI} ) = \cO(1),  
					 \   \log(  \hDelta_a(\eta, \fq ) / \hDelta_{\cI} ) =  \cO(  \log \tinitial  ).  
						\end{align}
						We have a parallel statement  for   $(r_{a'}(\eta'', q'' ), r_{a'}(\eta'', \fq''), \hDelta_{a'}(\eta'', \fq'' ))$ provided  $(q'', P'' \mid \fq'') \in \supp(\psi')$. 				 
		By    Lemma \ref{lem behaviorofDeltalaongorbit},  Lemma \ref{lem adaptedscaleistame},    \eqref{eq etacIetacI'bigandclose}, we have  
						\ary  \label{eq rrdeltadeltaratio0} 
						\log(	r_{\cI'} / r_{\cI} ) = \cO(1), \  \log( \hDelta_{\cI'} / \hDelta_{\cI} ) =  \cO( \log \tinitial  ). 
						\eary
 		As \eqref{eq etacIetacI'bigandclose}, \eqref{eq rrdeltadeltaratio}, \eqref{eq rrdeltadeltaratio0} will be used extensively in the rest of the proof, we will sometimes use them without mentioning.
 			By the construction of $\widetilde\psi$, for each $(q,  \eta \mid \fq) \in \Gamma_a$ the measure of $\xi$ such that $\widetilde\psi(q, P \mid \fq) \neq 0$ (recall $P = (\xi, \eta)$) is of size $ \cO( \tinitial^C e^{2 \deltasharp  \tinitial} \hDelta_{\cI}^2)$. A parallel statement holds for $\psi'$.
 						
 Denote
\begin{align*}
& \hat{F} =  H_{a', \fq'', \eta''} F H_{a, \fq, \eta}^{-1}, \ \hat{F}' =  H_{a', \fq'', \eta''} F H_{a, \fq', \eta'}^{-1}, \ G^* = F^{-1} H_{a', \fq'', \eta''}^{-1} \\
\mbox{and}  \ \ \ \   &	G =	\hat{F}^{-1} = H_{a, \fq, \eta} G^*, \ 
					G' = (\hat{F}')^{-1} =  H_{a, \fq', \eta'} G^*.
						\end{align*}
						
						We write $G^* = (G^*_w, G^*_z) =  (G^*_x, G^*_y, G^*_z) $. We use similar notations to express $G$ and $G'$. 
						By Proposition \ref{prop decompositionofkappaapeta},  for any $\tq = (\tw, \tz) \in   B(  \fq'',    C \extra^{-2} r_{\cI'}  )$,   we   have  
						\ary \label{eq D3ofGbound}
						\| D^3 G(\tq)  \|,	\| D^3 G' (\tq)  \|,	\| D^3 G^* (\tq)  \|   =   \cO_{\esmall}( | \eta_{\cI} |^{c(\esmall)}).
						\eary


		 Recall that the kernel of $\hK_{a \to a'}$, denoted by $\cK$, is given by \eqref{eq cK} in Section \ref{sec proofoflem formerL6.3}.
		 The kernel of $\hP_2^* \hP_2$, denoted by $\cK_2$ from now on, is given by
 				\aryst
						&& \cK_2(\hq' \mid \hq) = \int \widetilde\rho(\hq')   \widetilde\rho(\hq)  \overline{ \phi_{a}( \eta', q'   \mid \fq' \mid \hq')  } \phi_{a}(\eta, q \mid \fq \mid \hq )   \\
						&&  e^{i[ P' \cdot ( H_{a, \fq', \eta'}(\hq') - q' ) -  P  \cdot ( H_{a, \fq, \eta}(\hq) - q ) ]}   \overline{\cK(q'', P'', \fq'' \mid q', P', \fq') }  \\
						&&  \cK( q'', P'', \fq'' \mid q, P, \fq)  (\psi')^2( q'', P'', \fq'' )  \widetilde \psi( q',  P', \fq' ) \\
						&& \widetilde  \psi( q, P, \fq )    dq dP d\fq   dq' dP' d\fq' dq'' dP'' d\fq''.
						\earyst
						Write $q'' = (w'', z'')$ and make substitutions $\tq \mapsto \tq + (w'', 0)$, $\tq' \mapsto \tq' + (\tw'', 0)$, we obtain
						\ary
						\label{eq expressioncKtypeIII} 
						&&  \cK_2(\hq' \mid \hq) = \int  \widetilde\rho(\hq')  \widetilde\rho(\hq)   \overline{ \phi_{a}( \eta', q'   \mid \fq' \mid \hq')  } \phi_{a}(\eta, q  \mid \fq \mid \hq)    \\
						&&  e^{i ( P' \cdot H_{a, \fq', \eta'}(\hq')  -  P \cdot  H_{a, \fq, \eta}(\hq)  )}        [\int  I( \tq, \cdots,   \fq' , z''  \mid w'' )  dw'']  dz''  \nonumber  \\
						&&   d\tq d\tq'  dP'' d\fq'' dq dP   d\fq dq' dP' d\fq'  \nonumber
						\eary
						where $ I( \tq, \cdots,   \fq' , z'' \mid w'' ) $ is the abbreviation for  
						\begin{align}
						 \label{eq expressionI}  
						&    \exp(i P'' \cdot (\tq' - \tq))  \cdot  \exp(i P \cdot G(\tw + w'', \tz) - i P' \cdot G'(\tw' + w'', \tz') )   \\
						& \overline{\widetilde{\phi}_{a'}( \eta'',  q'' \mid \fq''  \mid \tw' + w'', \tz') }   \overline{ \widetilde{\phi}^{\dagger}_{a}( \eta',  q' \mid \fq' \mid G'(\tw'+ w'', \tz'))  }   \widetilde{\phi}_{a'}( \eta'', q'' \mid \fq''  \mid \tw + w'', \tz)   \nonumber \\
						&  \widetilde{\phi}^{\dagger}_{a}( \eta, q \mid \fq \mid G(\tw + w'', \tz))     (\psi')^2(q'', P'' \mid \fq'')  \widetilde \psi (q', P' \mid \fq')  \widetilde\psi (q, P \mid \fq)  \nonumber \\
						&  \rho(G'(\tw' + w'', \tz')) \rho(G(\tw + w'', \tz)).  \nonumber
						\end{align}

						We will show that by letting $\tinitial$  be large depending on $g$, we have
						\ary \label{eq integralcK3}
						\sup_{\hq'} \int | \cK_2(\hq' \mid \hq) | d\hq ,  \ 
						\sup_{\hq} \int | \cK_2(\hq' \mid \hq) | d\hq'  \lesssim   e^{ - 16 \deltasharp  \tinitial}.
						\eary
By Schur's test, we see that \eqref{eq integralcK3} implies that $\|\hP_2\|_{L^2 \to L^2} \lesssim  e^{ - 8 \deltasharp  \tinitial}$. This would conclude the proof of Lemma \ref{lem hp2}.
						
 	From now on we fix a large integer $\ell$, and abbreviate $\Phi_{a, \ell}$, $\widehat \Phi_{a, \ell}$,  $\Phi_{a', \ell}$, $\widehat \Phi_{a', \ell}$ (defined in  \eqref{eq formuaofPhia} and \eqref{eq formuaofPhia2}), as $\Phi_a$, $\widehat\Phi_a$, $\Phi_{a'}$, $\widehat\Phi_{a'}$ respectively.   
						By Theorem \ref{thm DWPT}\eqref{thm DWPT item 2}, we have $|I | \leq I_3$ where
						\begin{align}  \label{eq IandI3}
									I_3  =&  \Phi_{a'}(\eta'', \fq'' \mid q'')^2  \widehat \Phi_{a'}(\eta'', q'' \mid \tw+ w'' , \tz)  \widehat \Phi_{a'}(\eta'', q'' \mid \tw' + w'' , \tz')    \\
						&    \Phi_a(\eta, \fq \mid q )    \Phi_{a}(\eta', \fq' \mid q')    \widehat  \Phi_{a}(\eta, q \mid G(  \tw + w'' , \tz) )  \widehat \Phi_{a}(\eta', q' \mid G'(   \tw'+ w'' , \tz') )     \nonumber  \\
					    & 1_{\supp( \psi' )}( q'', P'' \mid \fq'' )    1_{\supp(\widetilde \psi )}( q, P \mid \fq )   1_{\supp(\widetilde \psi )}( q', P' \mid \fq' )  \nonumber  \\
						& 1_{\supp(\rho)}(G'(\tw' + w'', \tz')) 1_{\supp(\rho)}(G(\tw + w'', \tz)). \nonumber 
						\end{align}
						Then we have  
						\begin{align} \label{eq expressionbfKtypeIII}
						| \cK_2(\hq' \mid \hq) | 	\lesssim&  {\bf K}_2(\hq' \mid \hq) := \int  \rho(\hq')  \rho(\hq)    \Phi_a(\eta', \fq' \mid q') \widehat \Phi_a(\eta', q' \mid \hq')  \\
						& \Phi_a(\eta, \fq \mid q) \widehat \Phi_a(\eta, q \mid \hq)   [\int I_3( \tq, \cdots,   \fq', z''  \mid w'' )    dw'']   dz''   \nonumber  \\
						&   d\tq d\tq'  dP'' d\fq'' dq dP   d\fq dq' dP' d\fq'.  \nonumber
						\end{align}
						
 			We let  $c \in (0, 1/ 4)$ denote a small absolute constant which may vary from line to line. 
						\begin{lemma} \label{lem integralbfK}
							When $\tinitial \gg_g  1$ and $\omegalow \gg_{g, \tinitial, c} 1$, we have 
							\aryst
							\sup_{\hq'} \int  {\bf K}_2(\hq' \mid \hq)   d\hq ,  \ 
							\sup_{\hq} \int  {\bf K}_2(\hq' \mid \hq)  d\hq'  \lesssim e^{10  \deltasharp \tinitial   }  \hDelta_{\cI}^6 r_{\cI}^{6} =    \cO_{c}(   \langle \eta_{\cI} \rangle^{c} ).
							\earyst
						\end{lemma}
						\begin{proof}   
							We explain the estimate for $\sup_{\hq'} \int | {\bf K}_2(\hq' \mid \hq) |  d\hq$. The other inequality follows from a parallel argument.
							
 					Fix an arbitrary $\hq' = (\hw', \hz')$. Under our notations, we have $dP dP' dP'' = d\xi d\xi' d\xi'' d\eta d\eta' d\eta''$.
							We first integrate $d\xi d\xi' d\xi''$. 
 			We have seen that the measure of $(\xi, \xi', \xi'')$ in the support of the integrand  is bounded by $C   \tinitial^{C}      e^{  6 \deltasharp    \tinitial }       \hDelta_{\cI}^4  \hDelta_{\cI' }^2$ (as one fixes the other parameters).   
							Then we apply \eqref{eq integralphiaupperbound} in Lemma \ref{lem intPhiadwdz} to $\fq = (\fw, \fz)$, $\fq' = (\fw', \fz')$ and $\fq'' = (\fw'', \fz'')$; and  apply \eqref{eq integralphiaupperbound2} in  Lemma \ref{lem intPhiadwdz} to  $q = (w, z)$, $q' = (w', z')$ and $q'' = (w'', z'')$. Thus the term $\int | {\bf K}_2(\hw', \hz' \mid \hw, \hz) |  d\hw d\hz$ is bounded by  
						    \ary 
							&& C(\ell)    \tinitial^{C}      e^{  6 \deltasharp    \tinitial }    \hDelta_{\cI}^4  \hDelta_{\cI' }^2  \int  \langle \langle \eta \rangle^{\beta_1}(G_z(\tw, \tz) - \hz) \rangle^{-\ell} \langle r_{\cI}^{-1}(G_w(\tw, \tz) - \hw) \rangle^{-\ell}  \nonumber   \\
							&&  \langle \langle \eta' \rangle^{\beta_1}(G'_z(\tw', \tz') - \hz') \rangle^{-\ell} \langle r_{\cI}^{-1} (G'_w(\tw', \tz') - \hw') \rangle^{-\ell}    \langle \langle \eta'' \rangle^{\beta_1}(\tz' - \tz) \rangle^{-\ell}  \nonumber    \\
							&& \langle r_{\cI'}^{-1} (\tw - \tw') \rangle^{-\ell} 	1_{ \widetilde{B} } (\eta)   1_{ \widetilde{B} } (\eta')  1_{B'}(\eta'')     d\hw d\hz d\tw d\tz d\tw' d\tz' d\eta d\eta' d\eta'' \nonumber   \\
							&\leq& C(\ell)   \tinitial^{C}        e^{  9  \deltasharp     \tinitial}  \hDelta_{\cI}^4  \hDelta_{\cI' }^2  r_{\cI}^{4}  r_{\cI'}^{2}. \label{eq boundofK}
							\eary
				 We then finish the proof by \eqref{eq rrdeltadeltaratio0} and \eqref{eq Deltaadscale>1}.
						\end{proof}

						By Lemma \ref{lem integralbfK} and \eqref{eq etacIetacI'bigandclose},  in the proof of \eqref{eq integralcK3} we are free to remove from the expression of $I$ any term of size $\cO(  | \eta_{\cI} |^{-c} I_3)$, or remove from  the expression of $ \cK_2$ any term of size $\cO(  | \eta_{\cI} |^{-c} {\bf K}_2)$, as long as $| \eta_{\cI'} |$ is sufficiently large depending on $g, c$.
						We will use this observation to replace $\cK_2$ and $I$ by some approximants to simplify the computations.

						\noindent{\bf (Range of parameters)}: 					 We may assume in the rest of the proof
						\ary
						 \label{eq parameterrange1}
				    	&&   \eta, \eta' \in \widetilde{B} = B( \eta_{\cI},   e^{2 \deltasharp \tinitial}  \langle  \eta_{\cI} \rangle^{\beta_1}), \quad   \eta'' \in B' \subset B( \eta_{\cI'},   e^{2 \deltasharp \tinitial}   \langle    \eta_{\cI'} \rangle^{\beta_1}), \\  
				    	 \label{eq parameterrange0}
				        &&   \fq,   \fq' \in B(  \frq_0,   \extra^{-3} r_{\cI} ), \   \fq'' \in B(  \frq'_0,   \extra^{- 3}  r_{\cI'} ),  \\
				        &&  q \in \ngb^1_{a}(\extra \eta, \fq), \ q' \in \ngb^1_{a}( \extra \eta', \fq'),   \   q'' \in \ngb^1_{a'}( \extra \eta'', \fq'' ), 	        \label{eq parameterrange8}   \\
				         \label{eq parameterrange4}
						&&   \|  w -  G_w(\tw + w'', \tz) \| \leq   r_{a}(\eta,  \fq ),  \
					\	 \|   w' -  G'_w(\tw' + w'', \tz' )  \| \leq  r_{a}(\eta',  \fq'), \label{eq parameterrange2} \\
						&&  \| \tw \|, \| \tw' \|  \leq     r_{a'}(\eta'', \fq''),   \label{eq parameterrange5} \\
						\label{eq parameterrange3}
					&&	\| \tz - \fz'' \|, \  \| \tz' - \fz'' \| \leq   \langle\eta_{\cI} \rangle^{-\beta_1 + c},  \\ 
											\label{eq parameterrange7}
				  &&  \| G_z(\tw + w'', \tz)  - \fz \|,  \   \| G'_z(\tw + w'', \tz') -  \fz' \| \leq    \langle\eta_{\cI} \rangle^{-\beta_1 + c}.
						\eary
						Indeed, by studying the supports of $I$, $\psi'$ and $\widetilde\psi$, the expression in the integral in \eqref{eq expressioncKtypeIII} is of  negligible error of size $\cO( \eta_{\cI}^{-c} I_3 )$ if any one of
						\eqref{eq parameterrange1} to  \eqref{eq parameterrange7} fails. 
						
						Given some $\fq$, $\fq'$ and $\fq''$, by a change of variable  to  
						 $w''$, and making suitable substitutions of $\widetilde{\phi}^{\dagger}_{a}$, $\widetilde{\phi}_{a'}$, we may   reduce to the case that
						$A_a(\fq) = A_a(\fq') = A_{a'}(\fq'') = {\rm Id}$  (see Proposition \ref{prop decompositionofkappaapeta}). 
						Then   
						\ary \label{eq expansionofDwGw}
						\partial_w G_w(\tw + w'', \tz )  = \partial_w G'_w(\tw' + w'', \tz') =     \Upsilon^{-1}  + \cO(\eta_{\cI}^{-1/2 + c})
						\eary
						where we have  
						\ary \label{eq Upsilon}
						\Upsilon =  
						\begin{bmatrix}
							\lambda_u
							& 0 \\ 0 &  \lambda_s
						\end{bmatrix} + \cO(\esmall), 
 						\eary
						where we have, by letting $\tinitial$ be large, that 
						\ary \label{eq rangeoflambdaus}
						 \lambda_u = e^{\cO(1)} \mu^u ( \kappa_{a'}^{-1}(\fq''),  \tinitial  )  \geq e^{ \lambda_g \tinitial} \  \ \mbox{ and } \ \ 	 \lambda_s = e^{\cO(1)} \mu^s (\kappa_{a'}^{-1}(\fq''), \tinitial ) \leq e^{- \lambda_g \tinitial}.
						\eary
						We can use \eqref{eq expansionofDwGw} and \eqref{eq Upsilon} to simplify \eqref{eq expressionI}:
						\begin{align}
						\label{eq Gw}
						  G_w( \tw + w'', \tz) &=  \Upsilon^{-1}(\tw + w'' - \fw'' ) +    G_w(\fq'') + \cO( \eta_{\cI}^{-1 + c}), \\
						  \label{eq G'w}
						   G'_w( \tw' + w'', \tz') &=  \Upsilon^{-1}(\tw' + w'' - \fw'' ) +    G'_w(\fq'') + \cO( \eta_{\cI}^{-1 + c}).
						\end{align}
						Thus, up to a negligible error of size $\cO(\eta_{\cI}^{-c} I_3)$, we may replace the factor $ \exp(i \xi  \cdot G_w(\tw + w'', \tz) - i \xi' \cdot G'_w(\tw' + w'', \tz') )  $ in  \eqref{eq expressionI} (recall that $P = (\xi, \eta)$ and $P' = (\xi', \eta')$)	by $$\exp(i \xi \cdot [\Upsilon^{-1}(\tw + w'' - \fw'' ) +    G_w(\fq'') ]  - i \xi' \cdot [ \Upsilon^{-1}(\tw' + w'' - \fw'' ) +    G'_w(\fq'') ] ).$$
					    We apply a change of variables to the resulting expression,
						replacing $\xi, \xi'$ and $\xi''$ respectively  by 
						\begin{align} \label{eq xiandzeta}
						 \zeta   +      \eta \estar_{a}(\eta, \fq \mid q) ),   \zeta'  +     \eta'  \estar_{a}(\eta', \fq' \mid q')
						\mbox{ and }  \zeta''   +     \eta''  \estar_{a'}(\eta'', \fq'' \mid  q'').
						\end{align}
						By writing $P'' = (\xi'', \eta'') = ( \zeta''   +      \eta''  \estar_{a'}(\eta'', \fq'' \mid  q''), \eta'' )$,
						we may denote
						\aryst
						 \widetilde\Psi''(q'', \zeta'', \eta'' \mid \fq'') :=   \psi'( q'', \zeta''   +      \eta''  \estar_{a'}(\eta'', \fq'' \mid  q'') , \eta'' \mid \fq'' ).
						\earyst
							Note that $	 \widetilde\Psi'' $ is independent of $q''$. Thus in the following we will simply abbreviate $\widetilde\Psi''  (q'', \zeta'', \eta'' \mid \fq'' ) $ as $\widetilde\Psi''( \zeta'', \eta'' \mid \fq'')$ instead.
							We define $\widetilde\Psi ( \zeta, \eta \mid \fq )$ in a similar way.

Up to a negligible error of size $\cO(\eta_{\cI}^{-c} I_3)$, we may substitute in \eqref{eq expressioncKtypeIII}  the term $I$ (defined in \eqref{eq expressionI})  by $I_2$, which is the abbreviation for
		\begin{align}   \label{eq threelineexpressionofI} 
			&   \exp(iS)      \overline{\widetilde{\phi}_{a'}( \eta'', q'' \mid \fq''  \mid \tw' + w'', \tz') }   \overline{ \widetilde{\phi}^{\dagger}_{a}( \eta', q' \mid \fq' \mid G'(\tw'+ w'', \tz'))  }   \\
			&  \widetilde{\phi}_{a'}( \eta'', q'' \mid \fq''  \mid \tw + w'', \tz)    \widetilde{\phi}^{\dagger}_{a}( \eta,  q \mid \fq  \mid G(\tw + w'', \tz))      \nonumber \\
			&      (\widetilde\Psi'' )^2( \zeta'', \eta'' \mid \fq'') \widetilde\Psi (  \zeta, \eta \mid \fq)  \widetilde\Psi ( \zeta', \eta'  \mid \fq'  )  \cdot   \rho(G'(\tw' + w'', \tz')) \rho(G(\tw + w'', \tz)) \nonumber
		\end{align}
		where
		\begin{align} \label{def S}
		S =&    \zeta'' \cdot (\tw' - \tw) +   \eta'' (\tz' - \tz)   	   +   (\zeta   +     \eta \estar_{a}(\eta, \fq \mid q ) ) \cdot  (  \Upsilon^{-1}(\tw + w'' - \fw'' ) +    G_w(\fq'') )  \\
	&	  -   (\zeta'  +     \eta'  \estar_{a}(\eta', \fq' \mid q')) \cdot  (\Upsilon^{-1}(\tw' + w'' - \fw'' ) +    G'_w(\fq'') )     \nonumber  \\
		&     +  (\eta  G_z(\tw + w'', \tz) -   \eta'  G'_z(\tw' + w'', \tz') )      +    \eta''  \estar_{a'}(\eta'', \fq'' \mid  q'') \cdot (\tw' - \tw).    \nonumber 
		\end{align}

	We will also use the following lemma.
	\begin{lemma}  \label{lem oscpwG}
		Given \emph{$(\eta, \fq), (\eta', \fq')$, $(\eta'', \fq'')$},   $\tq = (\tw, \tz)$, $\tq' = (\tw', \tz')$,   we have the following:
		For every $y''$,  there exists {\em $L  \in \R$} such that
		\aryst
		\partial_{x''} ( G_z(\tw + w'', \tz) -  G'_z(\tw' + w'', \tz') ) = L  +  \cO(  \eta_{\cI}^{-1/2 - c} );
		\earyst
		For every $x''$,  there exists  {\em $L'   \in \R$} such that
		\aryst
		\partial_{y''} ( G_z(\tw + w'', \tz) -  G'_z(\tw' + w'', \tz') ) = L' + \cO(   \eta_{\cI}^{-1/2 - c} ).
		\earyst 
	\end{lemma}
	\begin{proof}    
		We detail the proof of the first equality. The second one is completely analogous.
		
		By \eqref{eq pullbackofestarinmodifiedchart} and the comment below it, we have
		\begin{align*}
			& \partial_{w''} ( G_z(\tw + w'', \tz) ) 
			= \estar_{a'}(\eta'', \fw'', \fz''  \mid \tw + w'', \tz) \big [  \partial_z G_z(\tw + w'', \tz) + \\
			&  \estar_a(\eta, \fw, \fz \mid G(\tw + w'',  \tz))  \cdot \partial_zG_w(\tw + w'',  \tz ) \big ] - \estar_{a}( \eta, \fw, \fz \mid G(\tw + w'', \tz) ) \partial_w G_w(\tw + w'', \tz).
		\end{align*}
		By \eqref{eq expansionofDwGw} and \eqref{eq D3ofGbound}, the above expression equals to
		\aryst
		\estar_{a'}(\eta'', \fw'', \fz''  \mid \tw + w'', \tz)  \partial_z G^*_z( w_{\cI'},  z_{\cI'})  - \estar_{a}( \eta, \fw, \fz \mid G(\tw + w'', \tz) ) \Upsilon^{-1}    + \cO(\eta_{\cI}^{-1/2-c}).
		\earyst
		Then by Lemma \ref{lem mainapproximation} and Proposition \ref{lem mainpropertyofestar}, we deduce
		\begin{align*}
		\partial_{x''} ( G_z(\tw + w'', \tz) ) =& \theta^u_{a'}(\eta'', \fw'', \fz''  \mid \tx + \fx'', \ty + y'', \tz) \partial_z G^*_z( w_{\cI'},  z_{\cI'}) \\
	& 	-   \lambda_u^{-1}  \theta^u_{a}( \eta, \fw, \fz \mid G( \tx + \fx'', \ty + y'', \tz) )   + \cO(\eta_{\cI}^{-1/2-c}).   
		\end{align*}
		We have a similar expression for $\partial_{x''} ( G'_z(\tw' + w'', \tz') )$.
		Since the main term on the right hand side above is independent of $x''$,
		this completes the proof.
	\end{proof}

						Now we continue with the proof of Lemma \ref{lem hp2}.
						As in \cite{Tsu}, we divide the proof   into two cases.

						\smallskip
						\noindent{\bf Case I.}
						We assume that $r_{\cI} \hDelta_{\cI} \leq    e^{ 150 \deltasharp \tinitial }$. 
			In this case, we will exploit the transversality condition $({\bf NI})_{\sigmaNI, \CNI}$ in Definition \ref{def NIrho}.
		  By symmetry, to prove \eqref{eq cI'Kaa'cI} we are free to replace $F$ by $F^{-1}$ if necessary.  
Then by \eqref{lem adaptedscalefluctuation},   we may assume that 
\ary \label{eq factorbeforethetas}
| \eta_{\cI'} | r_ {\cI'}  \flu^s( \kappa_{a'}^{-1}( \frq'_0  ),    r_ {\cI'}) \sim 1.
\eary

Note that by \eqref{eq partialwFz} in Proposition \ref{lem derivativeboundsforhatF}, Remark \ref{rema worksforinverseofF}
 and (Range of parameters), 
 we have
\aryst
&&  | ( \eta  G_z(\tw + w'', \tz)  - \eta_{\cI}  G_z(\tw + w'', \tz) ) -  ( \eta  G_z(\tw + \fw'', \tz)  - \eta_{\cI}  G_z(\tw + \fw'', \tz ) )   | \\
&\leq& | \eta   - \eta_{\cI}  | | G_z(\tw + w'', \tz)  -  G_z(  \tw + \fw'', \tz ) | \lesssim_{\tinitial, \esmall}   \eta_{\cI}^{\beta_1} \cdot  r_{\cI} \cdot    \eta_{\cI}^{-1/2 + c(\esmall)} \lesssim \eta_{\cI}^{-  c}.
\earyst
We have a similar inequality as above if we replace $(G_z, \eta, \tw, \tz)$ by $(G'_z, \eta', \tw', \tz')$.
  We also recall that  by our hypothesis $A_{a'}(\fw'', \fz'') = {\rm Id}$,   and by Proposition \ref{lem mainpropertyofestar}, \eqref{eq parameterrange8} and Lemma \ref{lem mainapproximation}, we deduce
\aryst
\theta^{u}_{a'}(\eta'', \fq'' \mid q'') = \theta^{u}_{a'}(\eta'', \fq'' \mid \fx'', y'', \fz'') + \cO(\eta_{\cI}^{- 1/2 - c}), \\
\theta^{s}_{a'}(\eta'', \fq'' \mid q'') = \theta^{s}_{a'}(\eta'', \fq'' \mid x'', \fy'', \fz'') + \cO(\eta_{\cI}^{- 1/2 - c}).
\earyst
 By Lemma \ref{lem oscpwG}, for $L = L( \eta, \fq, \cdots, \tq', y'') $ we have 
\aryst
 \partial_{x''}( \exp(i   \eta_{\cI} ( G_z(\tw + w'', \tz) -    G'_z(\tw' + w'', \tz') - L  x'' ) ) ) = \cO(\eta_{\cI}^{ - 1/2 - c}).
\earyst
By the above approximations, we may rewrite   \eqref{eq threelineexpressionofI} as $\exp(i Q_0 ) \exp( i Q_1)  Q + \cO( |\eta_{\cI}|^{-c} I_3)$  
						where  $Q_0$ is a real-valued function independent of $x''$,   and 
						\aryst
						Q_1 &=&    \eta''  \theta_{a'}^s(\eta'', \fq' \mid  x'', \fy'',  \fz'') \cdot (\ty' - \ty)  \nonumber  \\ 
						&&    +    [( \zeta +      \eta \estar_{a}(\eta, \fq \mid q) ) - (\zeta' +      \eta' \estar_{a}(\eta', \fq' \mid q')) ] \Upsilon^{-1}  w''   +  \eta_{\cI} L   x'', \\
						Q &=&   \widetilde{\phi}_{a'}( \eta'', q'' \mid \fq''  \mid \tw + w'', \tz)  \overline{\widetilde{\phi}_{a'}( \eta'', q'' \mid \fq''  \mid \tw' + w'', \tz') }   \nonumber \\
						&& \overline{ \widetilde{\phi}^{\dagger}_{a}( \eta', q' \mid \fq' \mid G'(\tw'+ w'', \tz'))  } \widetilde{\phi}^{\dagger}_{a}( \eta, q \mid \fq \mid G(\tw + w'', \tz))  \\
						&&   \widetilde\Psi(  \zeta, \eta \mid \fq)      \widetilde\Psi ( \zeta', \eta' \mid \fq')  (\widetilde\Psi'')^2(  \zeta'', \eta'' \mid \fq'')    \nonumber \\
						&&      \rho(G'(\tw' + w'', \tz')) \rho(G(\tw + w'', \tz)).   \nonumber 
		\earyst

						Now we will bound $| \int 	\exp(i Q_0 ) \exp( i Q_1)  Q  dw'' |$ from above. We will partition the domain of integral into  appropriate $x''$-segments (a bounded subset $J$ of $w''$ is a \underline{$x''$-segment} if $J$ is connected and is mapped to a singleton under the projection $(x'', y'') \mapsto y''$). We denote by $|J|$ the length of a $x''$-interval $J$.
						We will bound the integral on each such segment.
			
						
						Recall that by \eqref{eq parameterrange8}, \eqref{eq parameterrange0}, \eqref{eq rrdeltadeltaratio} and \eqref{eq rrdeltadeltaratio0}, we may assume  $| y'' - \fy'' |, | z'' - \fz'' |  \leq \extra^{-1} r_{\cI'}$ and 
  $q'' = (w'', z'') \in \ngb^1_{a'}( \extra \eta'', \fq'' ) \subset B(  \frq'_0,  2  \extra^{- 3}  r_{\cI'} )$.
						
						From now on, we fix some $y'', z''$.
						We let $E_1$ denote
						\ary \label{eq tyty'distance}
		 \{  (\tw, \tw') \mid \tw = (\tx, \ty), \tw' = (\tx', \ty') \mbox{ such that \eqref{eq parameterrange5} holds and }	|\ty - \ty'| \leq  e^{ - \lambdag \tinitial / 8} r_{\cI'} \}.
					\eary 
 Now fix some $( \tw, \tw' ) \notin E_1$. 
			Take an arbitrary $x''$-interval  $J \subset  B( \fw'',  \extra^{-1} r_{\cI'} )$ satisfying
			\ary
			\label{eq lengthboundofJ}
		  |J|	\leq  e^{ \lambdag \tinitial / 4}  r_{\cI'}.
			\eary
 	By considering a line-by-line approximation of $Q$ (the expression of $Q$ has $4$ lines), we see that:
			\enmt
			\item
			by  Theorem \ref{thm DWPT}\eqref{thm DWPT item 2}, \eqref{eq rrdeltadeltaratio},
			the first line of $Q$ has oscillation on $J$ bounded by $\cO(  e^{  \lambdag \tinitial/4} \cdot \extra^{1/2} ) = \cO(e^{- \hat\lambda_g \tinitial/ 2})$ times the first line of \eqref{eq IandI3},
			\item by  Theorem \ref{thm DWPT}\eqref{thm DWPT item 2}, \eqref{eq rrdeltadeltaratio}, \eqref{eq expansionofDwGw}, \eqref{eq Upsilon}, \eqref{eq rangeoflambdaus},
			the second line of $Q$ has oscillation on $J$ bounded by $\cO(  e^{  \lambdag \tinitial / 4}  \cdot e^{- \lambdag \tinitial}  ) = \cO(e^{- \lambdag \tinitial/ 2})$ times the second line of \eqref{eq IandI3},
			\item the third line of $Q$ is independent of $x''$,
			\item by  $	\| D \rho  \| < e^{C  \tinitial}$  and $r_{\cI'} = \cO( |\eta_{\cI'}|^{ - 1/4} )$, the  fourth line of $Q$ has oscillation on  $J$  bounded by $\cO( |\eta_{\cI'}|^{-1/8} ) = \cO(e^{- \hat\lambda_g \tinitial})$ times the fourth line of \eqref{eq IandI3}.
			\eenmt 
			Then   $|Q(x''_0, y'', z'') - Q(x''_1, y'', z'')| =  \cO( e^{- \lambdag \tinitial / 5 } I_3(x''_0, y'', z'') )$ for any $(x''_0, y'' ), (x''_1, y'' ) \in J$.

			In the expression of $Q_1$,		the second  line are linear in $x''$, and
				the first line   is 
				$$h_1^{-1} \theta_{a'}^s(\eta'', \fq'' \mid  x'', \fy'',  \fz'') \ \mbox{ with } \ 	 h_1  =  1/[ \eta''  ( \ty' - \ty )]. $$

			Take an arbitrary $\ponM  \in \kappa_{a'}^{-1}(   \ngb^1_{a'}( \extra \eta'', \fq'' )   )$.   
			Let $n_u(\ponM)$ be given by Definition \ref{def Adapted scales}. 
			By \eqref{eq factorbeforethetas},     Lemma \ref{lem adaptedscaleistame},   we have $r_{\cI'} \cdot  \mu^u(\ponM , n_u(\ponM )) \sim 1$ and $|\eta_{\cI'}| \cdot r_{\cI'} \cdot  \mu^s(\ponM , n_u(\ponM ))  / \mu^c(\ponM, n_u(\ponM )) \sim  1$.
			Take some $n_{\ponM } > 0$ such that  $r_{\cI'} \cdot \mu^u(\ponM , n_{\ponM }) \sim   e^{ - \lambdag \tinitial / 4}$. Then by \eqref{eq almostconversative} and $\esmall \ll_{\tinitial} 1$, we have  
\ary \label{eq choiceofnp}
   | \eta_{\cI'}|  r_{\cI'}  \mu^s(\ponM, n_{\ponM})  / \mu^c(\ponM, n_{\ponM})  \sim   e^{ \lambdag \tinitial / 4}.
\eary
By     \eqref{eq tyty'distance} and \eqref{eq choiceofnp}, we have 
\ary \label{eq rangeofD0}
|h_1|^{-1} \gtrsim   e^{ - \lambdag \tinitial / 8}  r_{\cI'} | \eta_{\cI'} |  \sim    e^{ \lambdag \tinitial / 8}  (\mu^s(\ponM, n_{\ponM}) )^{-1}  \mu^c(\ponM, n_{\ponM}).
\eary

We claim that we can divide any $x''$-interval of length at most $ e^{ \lambdag \tinitial / 4}  r_{\cI'}$ (aside from finite many points) into a disjoint union of intervals $\bigsqcup J_k$, with the following proprety:
		  For each $k$,  we have 
		  $$ J_k = (x_k - l_k / \mu^u(\ponM_k, m_k), x_k + l_k /  \mu^u(\ponM_k, m_k) ) \times \{  \fy'' \} $$
		  with  $x_k \in B( \fx'',    \extra^{-1} r_{\cI'}   )$, $\ponM_k = \kappa^{-1}_{a'}(x_k, \fy'', \fz'')$, 			$1/2 \leq    l_k   \leq 2$,  and  an integer $m_k \geq 1$ satisfying
			\aryst
			|h_1|^{-1}   \mu^s(\ponM_k, m_k)  / \mu^c(\ponM_k, m_k)   \in (   e^{ \lambdag \tinitial / 10},   e^{  \lambdag \tinitial / 8}  ).
			\earyst 
			Indeed, given $\ponM = \kappa_{a'}^{-1}(x'', \fy'', \fz'')$ with some $x'' \in B(\fx'',     \extra^{-1}  r_{\cI'} )$, we may let $m_{\ponM}$ be the smallest integer  such that $|h_1|^{-1}   \mu^s(\ponM, m_{\ponM})  / \mu^c(\ponM, m_{\ponM}) \leq   e^{ \lambdag \tinitial / 9}  <   e^{  \lambdag \tinitial / 8}$. 
			Then by  \eqref{eq rangeofD0} and $r_{\cI'} \cdot \mu^u(\ponM, n_{\ponM}) \sim   e^{ - \lambdag \tinitial / 4}$, we deduce $1/\mu^u(\ponM, m_{\ponM}) \leq   e^{  \lambdag \tinitial / 4}  r_{\cI'}$ by letting $\tinitial$ be large.
			The claim then follows from   a simple covering argument.
			
			Since  $g$ has property $({\bf NI})_{\sigmaNI, \CNI}$,  we see that $f$ has property  $({\bf NI})_{\sigmaNI, \CNI,  e^{ 10 \hat\lambda_g \tinitial}}$ if $\esmall$ is sufficiently small. For each $J_k$,
			we apply Proposition \ref{lem mainpropertyofestar}\eqref{itm NIinequality} for $(x, y, z) = (x_k, \fy'', \fz'')$, $h_1 =  1/[ \eta''  ( \ty' - \ty )] $, $h_0 =  l_k / \mu^u(p_k, m_k)$  and $C_0 = \extra^{-1}$,  $C_1 =   e^{  \lambdag \tinitial/ 10}, \hat{C}_g = e^{7 \hat\lambda_g \tinitial}$, $c = 1$. We obtain
			\aryst
			|J|^{-1} \Big |  \int_J \exp(i Q_1(x'', y'', z'')) dx''  \Big | \lesssim   e^{ - \lambdag \sigmaNI \tinitial / 10}.
			\earyst
		 Then we have the following (recall that $I_3$ is given by \eqref{eq IandI3})
			\aryst
			| \int \exp(i Q_0) \exp(i Q_1) Q dw''|  \lesssim  e^{ - \lambdag \sigmaNI \tinitial / 10}  \int   I_3(w'') dw''.
			\earyst

			On the other hand, by  Cauchy's inequality, we see that for each fixed $\tw' = (\tx', \ty')$,   we have
			$$
			\int_{\small \tw :	 (\tw, \tw') \in E_1  }   \langle      r_{\cI'}^{-1}  (\tw' - \tw)  \rangle^{- \ell} d\tw \lesssim  e^{ - \lambdag \tinitial / 16}      r_{\cI'}^2.  
		$$
 		Then the integral expression of ${\bf K}_2$ in \eqref{eq expressionbfKtypeIII} restricted to the region where $ (\tw, \tw')  \in E_1$   is bounded by   $ e^{ - \lambdag \tinitial / 16}   {\bf K}_2(\hw', \hz' \mid \hw, \hz)$.
		Then by  \eqref{eq expressionbfKtypeIII},  \eqref{eq boundofK} and  Lemma \ref{lem integralbfK},     the left hand of  \eqref{eq integralcK3} is bounded by
$$   ( e^{ -   \sigmaNI \lambda_g \tinitial / 10 } +  e^{ - \lambdag \tinitial / 16}    )   e^{10  \deltasharp \tinitial   }  \hDelta_{\cI}^6 r_{\cI}^{6}.$$  
 	Recall that $ \deltasharp < \min(1, \sigmaNI ) \lambda_g  /20000$ by \eqref{eq deltasharpdefine}, and $\hDelta_{\cI}^6 r_{\cI}^{6} \leq e^{900 \deltasharp \tinitial} $ by hypothesis, we deduce  \eqref{eq integralcK3}.  This concludes the proof in Case I.
			
			\smallskip

	\noindent{\bf Case II.}
Now we assume that $r_{\cI} \hDelta_{\cI}  > e^{ 150 \deltasharp \tinitial }$.

The ranges of $\xi, \xi', \xi''$ could be very large: we only know an upper bound of size $\cO_{\esmall}( \eta^{1/2 +  c(\esmall)} )$  by \eqref{eq Deltaadscale>1}. Thus we need to apply integration-by-parts with respect to these variables. We denote $ D_{\zeta''}  =      e^{   \deltasharp  \tinitial}   \hDelta_{\cI'} \partial_{\zeta''}, \  D_{\zeta} =  e^{    \deltasharp   \tinitial}      \hDelta_{\cI} \partial_{\zeta}  $, and 
\begin{align*}
	\cD_{\zeta''} \phi &=  \frac{1 - i  \langle    e^{    \deltasharp    \tinitial}    \hDelta_{\cI'}   (\tw' - \tw)    ,  D_{\zeta''} \phi \rangle }{\langle   e^{     \deltasharp   \tinitial}  \hDelta_{\cI'}  (\tw' - \tw)     \rangle^2}, \\
	\cD_{\zeta} \phi &=  \frac{1- i  \langle   e^{    \deltasharp   \tinitial}   \hDelta_{\cI}  ( \Upsilon^{-1}(\tw + w'' - \fw'' ) +    G_w(\fq'')  - \hw ),  D_{\zeta} \phi \rangle }{ \langle    e^{   \deltasharp   \tinitial}  \hDelta_{\cI} (  \Upsilon^{-1}(\tw + w'' - \fw'' ) +    G_w(\fq'')  - \hw )     \rangle^2}.
\end{align*}
We define $D_{\zeta'}$ and  $\cD_{\zeta'}$  in a similar way. 
The adjoint of $\cD_{\zeta''}$ writes 
\aryst
\cD^*_{\zeta''} \phi =  \frac{1 -  i  \langle   e^{    \deltasharp   \tinitial}   \hDelta_{\cI'}   (\tw' - \tw)    ,  D_{\zeta''} \phi \rangle }{ \langle   e^{    \deltasharp    \tinitial}  \hDelta_{\cI'}  (\tw' - \tw)     \rangle^2}.
\earyst
We have parallel expressions for $\cD^{*}_{\zeta}$ and $\cD^{*}_{\zeta'}$   as well.
By  the second equality in \eqref{eq rrdeltadeltaratio}, we see that   
\begin{align}
	  \label{eq extrafactor''} 
	& | ( \cD^*_{\zeta''})^{l} (\widetilde\Psi'')^2 ( \zeta'', \eta'' \mid \fq'' ) | \leq C(l)   \tinitial^{C l }  \langle    e^{    \deltasharp  \tinitial}    \hDelta_{\cI'}  (\tw' - \tw)  \rangle^{- l} 1_{\supp(\widetilde\Psi'')}, \\
	\label{eq extrafactor}
	&   |  ( \cD^*_{\zeta})^{l}   \widetilde\Psi(\zeta, \eta \mid \fq ) | \leq C(l)  \tinitial^{C l}   \langle    e^{    \deltasharp   \tinitial}   \hDelta_{\cI} (  \Upsilon^{-1}(\tw + w'' - \fw'' ) +    G_w(\fq'')  - \hw )     \rangle^{- l} 1_{\supp(\widetilde\Psi)}.
\end{align}
We have analogous estimates for $  | ( \cD^*_{\zeta'})^{l}  \widetilde\Psi( \zeta', \eta' \mid \fq' ) |$.

Denote $\ponM  = \kappa_a^{-1}(\fq)$, $\ponM'  = \kappa_{a'}^{-1}(\fq')$ and $\ponM''  = \kappa_{a'}^{-1}(\fq'')$ .
Recall that Proposition \ref{lem mainpropertyofestar},    \eqref{eq parameterrange8}, Lemma \ref{lem mainapproximation} give
\aryst 
\theta^u_{a'}(\eta'', \fq'' \mid q'') = D_{a'}( \fq'' )^{-1}  \vartheta^{u}_{\ponM'', \langle \eta'' \rangle}(0, y'' - \fy'', 0) + \cO(\eta_{\cI'}^{- 1/2 - c}).
\earyst
We have similar expressions for $\theta^u_{a}(\eta, \fq \mid q)$ and  $\theta^u_{a}(\eta', \fq' \mid q')$.
Let us denote 
$$L'' =   \midtor( \ponM'', \adscale( \langle \eta'' \rangle, \ponM'') ).$$

	By the above approximation, and by applying integration-by-parts with respect to $\cD_{\zeta''}$, we can rewrite  \eqref{eq threelineexpressionofI} as $	\exp(i Q_0 ) \exp( i Q_1)  Q + \cO( | \eta_{\cI} |^{-c} I_3)$
	where $Q_0$ is a real-valued function independent of $y''$, and 
	\aryst
	Q_1 &=& 
	  [( \zeta_y +        \eta D_{a}( \fq )^{-1}    \vartheta^{s}_{\ponM, \langle \eta \rangle}(x - \fx, 0, 0)  ) - (\zeta'_y +    \eta' D_{a'}( \fq' )^{-1}   \vartheta^{s}_{\ponM', \langle \eta' \rangle}(x' - \fx', 0, 0)  ]    y'' / \lambda_s  +  \nonumber \\ 
	&&   \eta_{\cI} (   G_z(\tw + w'', \tz) -   G'_z(\tw' + w'', \tz')  )   -     \eta'' D_{a'}( \fq'' )^{-1}  (\tx' -  \tx)  L''  y'', \\
	Q &=&  \exp(  i       \eta''  D_{a'}( \fq'' )^{-1} ( \vartheta^{u}_{\ponM'', \langle \eta'' \rangle}(0, y'' - \fy'', 0) \cdot  (\tx' -  \tx)   +  L'' (\tx' -  \tx) y'' ) )  \nonumber \\
	&& \widetilde{\phi}_{a'}( \eta'', q'' \mid \fq''  \mid \tw + w'', \tz)  \overline{\widetilde{\phi}_{a'}( \eta'', q'' \mid \fq''  \mid \tw' + w'', \tz') }   \nonumber \\
	&& \overline{ \widetilde{\phi}^{\dagger}_{a}( \eta', q' \mid \fq' \mid G'(\tw'+ w'', \tz'))  } \widetilde{\phi}^{\dagger}_{a}( \eta, q \mid \fq \mid G(\tw + w'', \tz))  \\
	&&    (\cD^*_{\zeta''})^{\ell}  \widetilde\Psi'' (  \zeta'', \eta'' \mid \fq'')  \widetilde\Psi (  \zeta, \eta \mid \fq )  \widetilde\Psi' (  \zeta', \eta'  \mid \fq'  ) \nonumber \\
	&&  \rho(G'(\tw' + w'', \tz')) \rho(G(\tw + w'', \tz)). \nonumber 
	\earyst

	We will show that
	the integral $|\int \exp(i Q_1) Q dw''|$ is small for a \lq\lq typical choice\rq\rq   of $(\tw, \tw', w,w',\fw,\fw',\fw'')$ (detailed below).
Whenever \eqref{eq parameterrange1}-\eqref{eq parameterrange7} hold, we deduce by Lemma \ref{lem oscpwG}   that
	\begin{align}
	\label{eq P1andR}
	& 	\partial_{y''} Q_1   = R +  \cO(| \eta_{\cI} |^{1/2 - c}) \\
	\mbox{ where } \ 	 & R :=  ( \eta    D_{a}( \fq )^{-1}      \vartheta^{s}_{\ponM, \langle \eta \rangle}(x - \fx, 0, 0) -   \eta' D_{a}( \fq' )^{-1}     \vartheta^{s}_{\ponM', \langle \eta' \rangle}(x' - \fx', 0, 0)  ) / \lambda_s  +   ( \zeta_y  - \zeta'_y )  / \lambda_s  +   \nonumber   \\ 
	& +  	 \eta_{\cI}  \partial_{y''}( G_z(\tx + x'', \ty + \fy'', \tz) -   G'_z(\tx' + x'', \ty' + \fy'', \tz')  )  -    \eta'' D_{a'}( \fq'' )^{-1} L'' (\tx' -  \tx).  \nonumber 
\end{align} 
Note that $R$ is independent of $y''$, and only the first term on the right hand side above depends on $x, x'$.

Note that \eqref{eq Gw}, \eqref{eq G'w} and \eqref{eq parameterrange2} imply in particular that
\ary
\label{eq parameterrange2x}
					&  \|  x  -   \lambda_u^{-1}(\tx + x'' - \fx'' ) +    G_x(\fq'') \|  \lesssim      r_{\cI},  \\
\label{eq parameterrange2x'}
 &	  \|  x'  -   \lambda_u^{-1}(\tx' + x'' - \fx'' ) +    G'_x(\fq'') \|   \lesssim   r_{\cI}.
\eary

	Let us define  
	\aryst
	E_2 =  \{ ( x,  x')   \mid \mbox{  \eqref{eq parameterrange2x} and \eqref{eq parameterrange2x'} hold and }  |  R | \leq  e^{ -  16  \deltasharp \tinitial}  \lambda_s^{-1}    \hDelta_{\cI}  \}.
	\earyst
	We have the following.
	\begin{lemma} \label{lem areaofE}
		For any sufficiently large $\tinitial$,
		for every   \emph{$(\eta, \fq, \eta', \fq', \eta'', \fq'', \tq, \tq', w'', \zeta, \zeta')$},  
		we have $\sup_{x}|\{ x' \mid (x, x') \in E_2 \}|, \sup_{x'}|\{ x \mid (x, x') \in E_2 \}| \lesssim e^{ - 16  \deltasharp \tinitial} r_{\cI}^2$.
	\end{lemma}
	\begin{proof}  
	The set of $x'$ satisfying \eqref{eq parameterrange2x'} is   of Lebesgue measure $\cO(  r_{\cI})$.
	By   \eqref{lem adaptedscalefluctuation}, \eqref{def Delta} and  $r_{\cI} \hDelta_{\cI}  > e^{ 150 \deltasharp \tinitial }$, we deduce $| \midtor(p,  r_{\cI}) | r_{\cI}  \gg   \max( \flu^s(p, r_{\cI}),   \flu^u(p, r_{\cI}) )$ and $\hDelta_{\cI} \sim | \midtor(p,  r_{\cI}) | r_{\cI} | \eta_{\cI} |$.
	By   Lemma \ref{lem mainapproximation1}, Lemma \ref{lem mainapproximation} and $r_{\cI} \hDelta_{\cI}  > e^{ 150 \deltasharp \tinitial }$,
	for each fixed $x'$,  the set of $x $ satisfying  \eqref{eq parameterrange2x} and 
		$| R | \leq  e^{ -  16 \deltasharp \tinitial}  \lambda_s^{-1}   \hDelta_{\cI}$ is of Lebesgue measure  $\cO(e^{ -  16 \deltasharp \tinitial  }    r_{\cI})$.  
	\end{proof}
	We denote by $\cK_{E_2}$ the integral in \eqref{eq expressioncKtypeIII} restricted to the parameters satisfying $(x, x') \in E_2$; and denote by $\cK_{E_2^c}$ the integral in \eqref{eq expressioncKtypeIII} restricted to the complement. Clearly, we have $\cK_2 = \cK_{E_2} + \cK_{E_2^c}$.

	Now   we will bound $\cK_{E_2^c}$ from above.
	Fix some $(x, x') \notin E_2$.
	By definition and \eqref{eq P1andR}, we have 
	\ary \label{eq lowerboundx''P1}
	| \partial_{y''} Q_1 | >    e^{ - 17  \deltasharp \tinitial }   \lambda_s^{-1}    \hDelta_{\cI}.
	\eary
	
	In the following we will apply integration-by-part for the integral of $|\int \exp(i Q_1) Q dw''|$ along the $y''$-variable.
	Note that the fourth line of $Q$ is independent of $y''$. 
We see that $y''$-derivative of the fifth line of   $Q$ is bounded by $\cO( e^{C \tinitial} )$; and
	by Theorem \ref{thm DWPT}\eqref{thm DWPT item 2}, $y''$-derivative of the second line of $Q$  is bounded by 
	\begin{align} \label{eq 3rdwrty''} 
  \extra^{1/2}   r_{\cI}^{-1}  \Phi_{a'}(\eta'', \fq'' \mid q'')^2   \widehat \Phi_{a'}(\eta'',  q'' \mid \tw+ w'' , \tz)  \widehat\Phi_{a'}(\eta'', q'' \mid \tw' + w'' , \tz').    
	\end{align}
	Similarly,  $y''$-derivative of the third line of $Q$  is bounded by 
	\begin{align}
	\label{eq 4thwrty''} 
  \lambda_s^{-1}  r_{\cI}^{-1}    \Phi_a(\eta, \fq  \mid q )   \Phi_{a}(\eta', \fq' \mid q')   \widehat  \Phi_{a}(\eta, q \mid G(  \tw + w'' , \tz) ) \widehat \Phi_{a}(\eta', q' \mid G'(   \tw'+ w'' , \tz') ).   
	\end{align}
	We still need to treat the first line of $Q$. For this purpose we will first make some regularizations.

By  \eqref{eq parameterrange5},  \eqref{eq rrdeltadeltaratio}, \eqref{eq rrdeltadeltaratio0}, we have  $| \eta'' (\tx' -  \tx) | \lesssim   \eta_{\cI'} r_{\cI'}$.
The first  line of  $Q$ can be simplified as
$$  \exp(  i     \eta'' D_{a'}( \fq'' )^{-1}  ( \tpl^{u}_{\ponM'' }( y'' - \fy'') -  \tor^{u}(   \ponM'',  \adscale(\langle \eta'' \rangle, \ponM'' ) )  y'' ) \cdot  (\tx' -  \tx) ). $$
By Lemma \ref{lem mainapproximation1}, the above term restricted to each $y''$-segment of length $\lambda_s r_{\cI}$ satisfying \eqref{eq parameterrange8} has  oscillation of size $\cO(\lambda_s^{1/2})$.
Then by  Remark \ref{rem:torfluadscale} and by applying the regularisation argument in \cite[Lemma 6.21]{Tsu} or \cite[p. 137]{BT}, we can rewrite the {\it first  line} of  $Q$ as  $Q_{main} + Q_{err}$ where
	\ary
	 \label{eq Dx''ofQu1} 
 	| \partial_{y''} Q_{main} | \lesssim  \lambda_s^{-1} r_{\cI}^{-1} \ \mbox{ and } \ 
 	| Q_{err} | \lesssim    \lambda_s^{1/2}.
	\eary

 We let $\tilde{Q}_{main}$ be obtained from $Q$ by replacing the first line of $Q$ by $Q_{main}$.   
	Then by \eqref{eq 3rdwrty''},  \eqref{eq 4thwrty''} and \eqref{eq Dx''ofQu1} we deduce
	\ary \label{eq partialy''tildeQ1}
	| \partial_{y''}  \tilde{Q}_{main}| \leq C(\ell)    \lambda_s^{-1}    r_{\cI}^{-1}   \cdot  \langle       e^{   \deltasharp  \tinitial}    \hDelta_{\cI'}  (\tw' - \tw)  \rangle^{- \ell} \cdot I_3
	\eary
		where $I_3$ is given by \eqref{eq IandI3}.
	We let  $	\tilde{Q}_{err} = Q - \tilde{Q}_{main}$.
By the second inequality in \eqref{eq Dx''ofQu1}, we see that
	\ary \label{eq tildeQ2upperbound}
	| \tilde{Q}_{err} | \leq C(\ell)  \lambda_s^{1/2}     \cdot  \langle        e^{   \deltasharp  \tinitial}  \hDelta_{\cI'}  (\tw' - \tw)  \rangle^{- \ell} \cdot  I_3.
	\eary

	By applying integration-by-part to $\int  \exp(i Q_0 ) \exp(i Q_1 ) \tilde{Q}_{main} dx'' dy''$,
	we see that
	\begin{align} \label{eq QtoQ0}
&	\int \exp(i Q_0 ) \exp(i Q_1 ) Q dw'' = \int   \exp(i Q_0) \exp(i Q_1 ) \tilde{Q}  dx'' dy'' \\
 \mbox{where } \ \  &	 \tilde{Q} = \frac{\partial_{y''}\tilde{Q}_{main}}{\partial_{y''} Q_1} + \tilde{Q}_{err}. \nonumber 
	\end{align}
	Then by \eqref{eq partialy''tildeQ1}, \eqref{eq  tildeQ2upperbound}, \eqref{eq lowerboundx''P1}, and by  $ \hDelta_{\cI} r_{\cI} >   e^{150 \deltasharp \tinitial }$  (the hypothesis of Case II),  we have 
	\aryst
	|\tilde{Q}| \leq C(\ell) (  \frac{    \lambda_s^{-1}    r_{\cI}^{-1}  }{   e^{ -  17  \deltasharp \tinitial}  \lambda_s^{-1}   \hDelta_{\cI} }     +    \lambda_s^{1/2}    )    \langle   e^{  \deltasharp    \tinitial}  \hDelta_{\cI'}  (\tw' - \tw)  \rangle^{- \ell}  I_3 \leq C(\ell)  e^{-  130 \deltasharp \tinitial}    \langle   e^{  \deltasharp   \tinitial}   \hDelta_{\cI'}  (\tw' - \tw)  \rangle^{- \ell} I_3.
	\earyst

	Note that $\tilde{Q}_{main}$  and $\tilde{Q}_{err}$ depend on $\zeta$ and $\zeta'$ only through $\widetilde\Psi$.
	Moreover, the dependences of $\partial_{y''} Q_1$ on $\zeta$ and $\zeta'$ are rather simple: we see that, by \eqref{eq lowerboundx''P1}, for any $k, k' \geq 0$  
	\aryst
	| \partial_{\zeta}^k    \partial_{\zeta'}^{k'} ( \partial_{y''} Q_1)^{-1} | \leq    C(k, k')  |  \lambda_s | ^{-  k -  k'} |  \partial_{y''} Q_1 | ^{-  k -  k' - 1} \leq C(k, k')  |  \partial_{y''} Q_1 | ^{ - 1}    e^{ 18 (k + k')  \deltasharp \tinitial}  (    e^{   \deltasharp    \tinitial} \hDelta_{\cI} )^{- k - k'}.
	\earyst
	Then  for any $(k,  k') \in \N^2 \setminus \{ (0, 0) \}$,  we have 
	\aryst
	| \partial_{\zeta}^k    \partial_{\zeta'}^{k'} \tilde{Q} |   \leq	C(k, k', \ell)    e^{ - 130 \deltasharp \tinitial}     e^{ 18 (k + k') \deltasharp   \tinitial}    (     e^{  \deltasharp   \tinitial} \hDelta_{\cI} )^{- k - k'}  \langle    e^{  \deltasharp     \tinitial}   \hDelta_{\cI'}  (\tw' - \tw)  \rangle^{- \ell} \cdot I_3.
	\earyst
	In summary, for any $(k,  k') \in \N^2 \setminus \{ (0, 0) \}$, we have
	\aryst
	|   \partial_{\zeta}^k    \partial_{\zeta'}^{k'}   \tilde{Q} |  \leq C(k, k', \ell)   e^{ - 130 \deltasharp \tinitial}    e^{ 18 (k + k') \deltasharp   \tinitial}  (  e^{   \deltasharp    \tinitial} \hDelta_{\cI}  )^{- k - k'} \langle   e^{   \deltasharp    \tinitial}  \hDelta_{\cI'}  (\tw' - \tw)  \rangle^{- \ell} \cdot I_3.
	\earyst
We deduce
\begin{align}
	 \label{eq fouroperatorsandQ0}  
&	| ( \cD^*_{\zeta})^{3}   (\cD^*_{\zeta'})^{3}     \tilde{Q} | \leq  C( \ell  )   e^{ - 21 \deltasharp \tinitial }    I_4 
	\end{align}
	where
	\begin{align}
		\label{eq defofI4} 
		I_4 & :=  \langle   e^{   \deltasharp  \tinitial}   \hDelta_{\cI'}  (\tw' - \tw)  \rangle^{- \ell} \cdot     	\langle   e^{   \deltasharp  \tinitial}  \hDelta_{\cI} (  \Upsilon^{-1}(\tw + w'' - \fw'' ) +    G_w(\fw'', \fz'')  - \hw )     \rangle^{-\ell} \nonumber  \\
		&  	\langle   e^{   \deltasharp  \tinitial}   \hDelta_{\cI} (  \Upsilon^{-1}(\tw' + w'' - \fw'' ) +    G'_w(\fw'', \fz'')  - \hw' )     \rangle^{-\ell}  \cdot I_3.  
	\end{align}
	We see that the extra factors come from \eqref{eq extrafactor} and \eqref{eq extrafactor''}. 
  
	After the above preparations, we are now ready to conclude the proof.
	By \eqref{eq QtoQ0}, we may replace the term $I$ in \eqref{eq expressioncKtypeIII} by $\exp(i Q_0) \exp( i Q_1) \tilde{Q}$ up to an error of size $\cO(| \eta_{\cI} |^{-c} I_3)$. Then we apply to the resulting expression integration-by-parts with respect to $\cD_{\zeta}$ and $\cD_{\zeta'}$ several times, we see that by \eqref{eq fouroperatorsandQ0}, up to an error of size $\cO( | \eta_{\cI} |^{-c}  \int I_4 dw'')$, we can replace   $\int I dw''$ in \eqref{eq expressioncKtypeIII}  by
	\aryst
	\int   \exp(i Q_0) \exp(i Q_1)  ( \cD^*_{\zeta})^{\ell}   (\cD^*_{\zeta'})^{\ell}  \tilde{Q}   dw''.
	\earyst
	Consequently, we have $	| \cK_2(\hq' \mid \hq) | 	\lesssim  {\bf K}'_2(\hq' \mid \hq)$ where ${\bf K}'_2$ is defined by replacing $I_3$ by $I_4$ in \eqref{eq expressionbfKtypeIII}.
	Following the proof of Lemma \ref{lem integralbfK}, we see that  
	\ary
	\sup_{\hq'} \int  {\bf K}'_2(\hq' \mid \hq)   d\hq ,  \ 
	\sup_{\hq} \int  {\bf K}'_2(\hq' \mid \hq)  d\hq'  \leq  C e^{ 4\deltasharp    \tinitial}.
	\eary
	Then by comparing the expressions in \eqref{eq fouroperatorsandQ0}, \eqref{eq expressioncKtypeIII} and \eqref{eq expressionbfKtypeIII},    we have 	$	| \cK_{E_2^c} | \leq C e^{ -  16 \deltasharp \tinitial }  {\bf K}'_2 $.

		By   Lemma \ref{lem areaofE} and   Cauchy's inequality in the application of  Lemma \ref{lem intPhiadwdz}, we may follow the proof of Lemma \ref{lem integralbfK} to deduce $	| \cK_{E_2} | \leq C e^{ - 16 \deltasharp \tinitial }  {\bf K}'_2 $. We deduce  \eqref{eq integralcK3} with any $\kappa_{\sharp} = 4 \deltasharp $.  The proof of Case II is concluded.
					\end{proof}

					\appendix

					\section{Normal coordinate system} \label{sec normalcoordinate} 
					
					\begin{proof}[Proof of Proposition \ref{prop normalcoordinatesystem}] 
						
				   Consider the vector bundle $\cE^u = TM / E^u$.
				   Following the notation in  \cite{EPZ}, we see that $\cE^u$ is smooth along $W^u$ (see \cite[Section 3.1]{EPZ}),
				   and the derivative map $Df$ induces a bundle automorphism $F_u : \cE^u \to \cE^u$ over $f: M \to M$. Since $f$ is partially hyperbolic, then   \cite[Prop 3.21]{EPZ} gives for each $p \in M$ smooth sections $V_p: W^u(p) \to \cE^u$ and $V_p^{\perp} : W^u(p) \to \cE^u$, such that:
				   \enmt
				   \item for every $t \in (- \|Df\|, \|Df\|)$, $(V_p(\imath^u_p(t)), V^{\perp}_p(\imath^u_p(t)))$ forms a basis of $\cE^u_{\imath^u_p(t)}$,
				   \item    for some  affine function $r_p$, the matrix $A(p, t)$ associated to $F_u$ between the ordered bases $(V_p(\imath^u_p(t)), V^{\perp}_p(\imath^u_p(t)))$ and $(V_{f(p)}(\imath^u_{f(p)}( \mu^u(p, 1)  t )), V^{\perp}_{f(p)}(\imath^u_{f(p)}(\mu^u(p, 1)  t)))$ is of form
				   \aryst
				   \begin{bmatrix}
				   \mu^c(p, 1) & r_p(t) \\ 0 & \mu^s(p, 1)
				   \end{bmatrix}.
				   \earyst
				   \eenmt
				   As $f$ is close to a volume preserving map, we may assume that $r_p$ is linear for every $p$. In particular, we have 
				   $V_{p}(0) \in (E^c \oplus E^u)/E^u$ and $V^{\perp}_p(0) \in (E^s \oplus E^u)/E^u$. Without loss of generality, we may assume that $V_p(0)$ and $V^{\perp}_p(0)$ are represented by unit vectors in $E^c(p)$ and $E^s(p)$ respectively.
				   
				   Analogously, we may consider $\cE^s := TM / E^s$ and the quotient map $F_s$ of $Df^{-1}$ on $\cE^s$.
				    Then we may construct $(U_p, U^{\perp}_p)_{p \in M}$ in place of  $(V_p, V^{\perp}_p)_{p \in M}$ in an analogous way. Then for each $p \in M$, we construct a smooth diffeomorphism $\imath_p:  (- 2 \|Df\|, 2\|Df\|)^3 \to M$ such that item \eqref{itm imathp0order} in the proposition holds, and
				   \aryst
			&&	   \partial_{2} \imath_p(t, 0, 0) \in V^{\perp}_p(\imath^u_p(t)) + E^u,  		\		   \partial_{1} \imath_p(0, t, 0) \in U^{\perp}_p(\imath^s_p(t)) + E^s, \\
			&&	   \partial_{3} \imath_p(t, 0, 0) \in  V_p(\imath^u_p(t)) + E^u,	\		   \partial_{3} \imath_p(0, t, 0) \in  U_p(\imath^s_p(t)) + E^s .  
				   \earyst
				   As $V_p$, $V^{\perp}_p$, $U_p$ and $U^{\perp}_p$ have uniformly bounded smooth norms, we may assume that $\imath_p$ has uniformly bounded smooth norm as well. Moreover, by construction, $p \mapsto \imath_p$ is $(1 - c(\esmall))$-H\"older continuous;  and $\{ f \mapsto \imath_p \}_{p \in M}$ is uniformly continuous.    In particular, item \eqref{item imathsmoothnorms}   holds.
				   Then it is direct to verify that $\{ \imath_p \}_{p \in M}$ satisfies item \eqref{itm imath1storder}  as well.
					\end{proof}

					\begin{proof}[Proof of Lemma \ref{lem torsionareclose}]  
						 
						Let $n$ be given by \eqref{eq rhoandn}.	Then $\rho \sim (\mu^u(p, n))^{-1}$.				
						Let $m$ be an integer such that 
                       \ary \label{eq C0muupm}
                       K \mu^u(p, m) \rho \sim 1.
                       \eary
                       Then $| n - m | \lesssim \log K$. By applying \eqref{eq torexpression0} to $\rho$ and $K \rho$, we obtain 
						\aryst
					&&	\tor^s(p, K \rho) - \tor^s(p, \rho)   =  	\frac{\mu^s(p, n)  \mu^u(p, n) }{ \mu^c(p, n)} \cdot	\frac{ \mu^c(f^m(p), n - m)}{\mu^s(f^m(p), n - m)  \mu^u(f^m(p), n - m) } \cdot  \\
					&& (   \tor^s(f^m(p), K \mu^u(p, m)\rho) -   \tor^s(f^m(p), \mu^u(p, m)\rho) ).  
						\earyst
						Then the first inequality in the lemma follows from \eqref{eq almostconversative},  \eqref{eq rhoandn},  \eqref{eq def flus} and \eqref{eq flutorbasicestimate}.
						
						We now prove the second inequality in the lemma.
						We may assume that $d(p, p') \leq \rho^{1/2}$ for otherwise the statement follows from \eqref{eq flutorbasicestimate}.
					   Let $m$ be the largest integer satisfying $m \leq n$ and $d(f^m(p), f^m(p')) \leq 1$. 
					   By projecting $p'$ to the local invariant manifolds of $p$, we deduce that for every $0 \leq i \leq m$, we have $d(f^i(p) ,f^i(p')) \leq \max(e^{- (m- i ) \lambda_g }, e^{c(\esmall) i} d(p, p') ) \leq  \max(e^{- (m- i ) \lambda_g }, e^{c(\esmall) i} \rho^{1/2} ) \lesssim e^{- (m- i ) \lambda_g/3 }$.  
					   By distortion estimate, we have $\mu^*(p, m)  \sim  \mu^*(p', m)$ for $* \in \{s, c, u\}$, and  moreover, $\mu^u(p, m)  d(p, p') \gtrsim 1 \sim  \mu^u(p, n) \rho$.
					   In particular, $n - m \lesssim \log (\langle d(p, p') / \rho \rangle)$. 
					   
                    	By \eqref{eq torexpression0}, we have
					    \begin{align} \label{eq differencetorsion} 
						&  | \tor^s(p, \rho) - 	\tor^s(p', \rho) | / (\flu^s(p, \rho)  / \rho )    \\
						 \lesssim &  \frac{\mu^c(p, n)}{\mu^s(p, n) \mu^u(p, n)}  \big ( \sum_{i = 0}^{m-1}	\frac{\mu^s(p, i)   \mu^u(p, i) }{ \mu^c(p, i)} \twist(f^i(p))  - \sum_{i = 0}^{m-1}	\frac{\mu^s(p', i)   \mu^u(p', i) }{ \mu^c(p', i)} \twist(f^i(p')) \big )  \nonumber \\
						 +& 
						  \frac{\mu^c(p, n)}{\mu^s(p, n) \mu^u(p, n)} \big (
						\frac{\mu^s(p, m)   \mu^u(p, m) }{ \mu^c(p, m)}   \tor^s(f^m(p), \mu^u(p, m)\rho) -  	\frac{\mu^s(p', m)   \mu^u(p', m) }{ \mu^c(p', m)}   \tor^s(f^m(p'), \mu^u(p', m)\rho)  \big ). \nonumber
						\end{align} 		
						By \eqref{eq almostconversative} and \eqref{eq flutorbasicestimate},
						the third line  of \eqref{eq differencetorsion} is bounded by
						\begin{align} \label{eq en-mCdpp'}
					C 	e^{(n - m) c(\esmall)} (\langle \log (\mu^u(p, m) \rho )  \rangle + \langle \log (\mu^u(p', m) \rho )  \rangle  )  \lesssim \langle d(p, p') / \rho \rangle^{c(\esmall)} \langle \log (\rho /  d(p, p')  )  \rangle.
						\end{align} 	 	 
						The  the second line of \eqref{eq differencetorsion} is bounded by
						\ary  \label{eq sumdifferenceofproductsoftheta}
						&&    \sum_{i = 0}^{m-1} \frac{\mu^c(f^i(p), n - i)}{\mu^s(f^i(p), n - i) \mu^u(f^i(p), n - i)}	 ( \twist(f^i(p))  - \twist(f^i(p')) )       \\
						&& +  \sum_{i = 0}^{m-1}    \frac{\mu^c(f^i(p), n - i)}{\mu^s(f^i(p), n - i) \mu^u(f^i(p), n - i)}( 1 - \frac{\mu^c(p, i)}{\mu^s(p, i) \mu^u(p, i)}  \frac{\mu^s(p', i)   \mu^u(p', i) }{ \mu^c(p', i)} ) \twist(f^i(p')). \nonumber 
						\eary 
						By \eqref{eq cisHolder}, \eqref{eq C0muupm},
						the first sum in \eqref{eq sumdifferenceofproductsoftheta} is bounded by 
						\begin{align*}
					  & C  \sum_{i=0}^{m-1}   \frac{\mu^c(f^i(p), n - i)}{\mu^s(f^i(p), n - i) \mu^u(f^i(p), n - i)}	 d( f^i(p), f^i(p') )^{1/2} \\
						\lesssim&      \sum_{i=0}^{m-1}   \frac{\mu^c(f^i(p), n - i)}{\mu^s(f^i(p), n - i) \mu^u(f^i(p), n - i)}	 e^{- (m- i ) \lambda_g/6 }    \lesssim \langle d(p, p') / \rho \rangle^{c(\esmall)}.
						\end{align*}
						
						We now estimate the second sum in  \eqref{eq sumdifferenceofproductsoftheta}. 
	Define
	\aryst
	\theta_i =  \frac{\mu^c( f^i(p), 1)}{\mu^s( f^i(p), 1 ) \mu^u( f^i(p), 1 )}  \frac{\mu^s( f^i(p'), 1 )   \mu^u( f^i(p'), 1 ) }{ \mu^c( f^i(p'), 1 )}.
	\earyst
	By \eqref{eq almostconversative},	we have $| \theta_0 \cdots \theta_j | \lesssim e^{c(\esmall) j}$.
	We have $| 1 - \theta_i | \lesssim  d(f^i(p), f^i(p'))^{1/2} \lesssim e^{- (m- i ) \lambda_g/6 }  $.
	Then   the second sum on the right hand side of \eqref{eq sumdifferenceofproductsoftheta} is bounded by
	\begin{align*}
	& C  \sum_{i = 0}^{m-1}     e^{c(\esmall) (n - i)}  | 1 - \theta_0 \cdots \theta_{i-1} | 
	\lesssim  \sum_{i = 0}^{m-1}   e^{c(\esmall) (n - i)}  ( | 1 - \theta_{i-1} | + \theta_{i-1} | 1 - \theta_{i-2} |  + \cdots +  \theta_1 \cdots \theta_{i-1} | 1 - \theta_{0} |) \\
	\lesssim&  \sum_{i = 0}^{m-1}    (m-i) e^{c(\esmall) (n - i)}   e^{- \lambda_g (m-i) /6 } \lesssim  \langle d(p, p') / \rho \rangle^{c(\esmall)}. 
	\end{align*}

Summarizing the above estimates, we see that  the right hand side of \eqref{eq sumdifferenceofproductsoftheta} is bounded by $C \langle d(p, p') / \rho \rangle^{c(\esmall)}$.
					   This completes the proof of the second inequality.
					\end{proof}
					
					\begin{proof}[Proof of Lemma \ref{lem mainapproximation}]

						The proof follows   that of \cite[Lemma 4.7]{Tsu}. We will detail the proofs of \eqref{eq tildethetasalongs2} and   \eqref{eq thetasalongu3}, and omit the proofs of \eqref{eq thetaualongu1} and \eqref{eq tildethetaualongs4} as they are similar.

						Denote $q = \imath_p(w, z)$.
					By the $(1-c(\esmall))$-H\"olderness of $E^c_*$ and $c < 1 / 4$, we  see $\| ( \imath_p^{-1} \circ \imath^s_q)(\tau) -  (w + (\tau, 0), z) \| + |z|  =   \cO(   \rho^{1+c})$, and hence  \eqref{eq tildethetasalongs2} is reduce to proving $|   \vartheta^s_p( ( \imath_p^{-1} \circ \imath^s_q)(\tau) )  - \vartheta^s_p( ( \imath_p^{-1} \circ \imath^s_q)(0)  )   | = \cO( (C_0 \rho)^{2 - 1/4}) $.
						By \eqref{def varthetap},  for $V^{s}(\tau) = ( \imath_p^{-1} \circ \imath^s_{q})'(\tau)$ and $V^{u}(\tau) = ( \imath_p^{-1} \circ \imath^u_{ \imath^s_q(\tau) })'(0)$
						\aryst
						( \vartheta^u_p( ( \imath_p^{-1} \circ \imath^s_q)(\tau) ) , \vartheta^s_p( ( \imath_p^{-1} \circ \imath^s_q)(\tau) ), 1 )  \cdot V^{*}(\tau) = 0, \quad * \in \{s, u\}.
						\earyst
						By the H\"olderness of $E^s, E^u$,
	 we have $V^u(\tau) = (1, 0, 0) + \cO_{\esmall}( (C_0 \rho)^{1-c(\esmall)})$,  $V^s(\tau) = (0,  1, 0) + \cO_{\esmall}( (C_0 \rho)^{1-c(\esmall)})$ and $\| ( \vartheta^u_p, \vartheta^s_p ) \| = \cO_{\esmall}( (C_0 \rho)^{1 - c(\esmall)})$. We deduce that 		$\vartheta^s_p( ( \imath_p^{-1} \circ \imath^s_q)(\tau) ) = - (0, 0, 1) \cdot V^s(\tau) + \cO_{\esmall}( (C_0 \rho)^{2 - c(\esmall )})$.
					Note that $V^s(\tau) = ( \imath_p^{-1} \circ \imath^s_{q})'(\tau)$, and as a result
						\aryst
					 \|	 (0, 0, 1) \cdot ( V^s(\tau) -  V^s(0) ) \| \lesssim \tau  \sup_{| \tau_0 | < C_0 \rho, q_0 \in B(p,  C_0 \rho) } \|  ( \imath_p^{-1} \circ \imath^s_{q_0})''(\tau_0) \| = \cO_{\esmall}((C_0 \rho)^{2 - c(\esmall)}).
						\earyst
						We have used that $(\imath_p^{-1} \circ \imath^s_p)''(0)  = 0$ in the above argument. This concludes the proof of \eqref{eq tildethetasalongs2}.
						
											By \eqref{eq tildethetasalongs2} and that $ \flu^s(p, \rho) \gtrsim \rho^{1 + c}$ (see \eqref{eq flutorbasicestimate}), the proof of  \eqref{eq thetasalongu3}  is reduced to the case  $w = (x, 0)$.     By Remark \ref{rem:Flu}, Definition \ref{def templatefunction} and Lemma \ref{lem torsionareclose}, we have 
						$|  \vartheta^s_p(x +  \tau, 0, 0)  - \vartheta^s_p(x, 0, 0) - \tor^s(p, \rho) \tau | \lesssim  C_0^{c(\esmall)} C_1  \flu^s(p, \rho)$.  
						Since $E^c_*$ is $(1-c(\esmall))$-H\"older and $|z| \leq   \rho^{1 + c}$, we have $	|  \vartheta^s_p(x +  \tau, 0, z)  -  \vartheta^s_p( x +  \tau, 0, 0) | = \cO(\rho^{1 + c/2}) = \cO( \flu^s(p, \rho) )$. Thus  \eqref{eq thetasalongu3} holds for $w = (x, 0)$.  This finishes the proof. 
					\end{proof}

					\begin{proof}[Proof of Lemma \ref{lem behaviorofDeltalaongorbit}]

						Denote $\rho = 	 \adscale(\omega, p)$ and $\rho' =  \adscale(\omega', f^n(p))$. By \eqref{eq adscalesquarerootgrowth} and \eqref{eq pushforwardofadscale}, we have $| \log (\rho / \rho') | \lesssim n + | \log(\omega / \omega') |$.
						We claim that 
						\ary \label{eq pushforwardofflu} 
					   	| \log	\frac{ 	\flu^s( f^n(p), \rho' ) / \rho' }{ 	\flu^s( p, \rho) / \rho} | \lesssim   c(\esmall)(n  + |\log ( \rho' / \rho ) |)+ \langle \log (n + 1) \rangle.
						\eary
						Indeed, let $m$, $m'$ be integers satisfying $\mu^u(p, m) = 1/\rho$ and  $\mu^u(f^n(p), m') = 1/ \rho'$. Then the left hand side of \eqref{eq pushforwardofflu} equals to
						\aryst
						|  \log  \frac{\mu^u(f^{n}(p), m' ) \mu^s(f^{n}(p), m' ) }{\mu^c(f^{n}(p), m' )} - \log  \frac{\mu^u( p , m  ) \mu^s( p , m  )}{\mu^c( p , m  )}   |.
						\earyst
						Claim \eqref{eq pushforwardofflu} then follows  from \eqref{eq almostconversative} and the following simple inequality about  symmetric difference:
						\aryst
						| \{ 0, \cdots, m-1 \}  \Delta \{ n, \cdots, n+m'-1 \} | \lesssim n + | \log ( \rho' / \rho )  |.
						\earyst

						Now by  \eqref{eq almostconversative}, \eqref{eq torexpression0}, \eqref{eq cpmexpression} and Lemma \ref{lem torsionareclose}, we have
						\begin{align*}
						& | \tor^s(p, \rho)  -	\frac{\mu^s(p, n)   \mu^u(p, n) }{ \mu^c(p, n)}   \tor^s(f^n(p),  \rho' ) - \frac{ \twist(p, m) }{ \mu^c(p, n)}   | \\
						=&  | \tor^s(p, \rho) - \tor^s(p,  \mu^u(p, n)^{-1} \rho')  | 
						\lesssim  (n + 1)	e^{n c(\esmall)} | \log (\rho / \rho') | ( \rho/\rho' +  \rho' / \rho  )^{c(\esmall)}  \flu^s(p, \rho) / \rho. 
						\end{align*}		
						We have a similar expression relating $\tor^u$ and $\flu^u$.
						Then we can deduce
						\aryst
						|   \midtor(p, \rho) -  \frac{\mu^s(p, n)   \mu^u(p, n) }{ \mu^c(p, n)}   \midtor(f^n(p), \rho')  | \lesssim (n + 1)	e^{n c(\esmall)}  | \log (\rho / \rho') | ( |\frac{\rho}{\rho'}| + |\frac{\rho'}{\rho}|)^{c(\esmall)} \max(  \flu^s(p, \rho), \flu^u(p, \rho) ) / \rho.
						\earyst
						We may then conclude the proof by \eqref{eq pushforwardofflu} and \eqref{eq pushforwardofadscale}.
					\end{proof}

					\begin{proof}[Proof of Lemma \ref{lem adaptedscaleistame}]

				The argument here is similar to the one in the proof of Lemma \ref{lem torsionareclose}.
				
			Denote $C_0 := \langle d(p, p') /  \adscale(\omega, p) \rangle$.	By  Lemma \ref{lem behaviorofDeltalaongorbit}, \eqref{eq adscalesquarerootgrowth} and by symmetry, it suffices to assume that $\omega = \omega'$ and show that 
				\aryst
				\log	\frac{ 	\adscale(\omega, p')}{ 	\adscale(\omega, p)} \lesssim   c(\esmall) \log C_0, \ 
				\log	\frac{ 	\Delta(\omega, p')}{ 	\Delta(\omega, p)}  \lesssim  \log\langle \log C_0 \rangle +  c(\esmall) \log C_0.
				\earyst
				
				Denote $\rho = \adscale(\omega, p)$ and $\rho' = \adscale(\omega, p')$.
				Without loss of generality, we may assume that $C_0 \leq \omega^{1/4}$ for otherwise the statement follows from \eqref{eq adscalesquarerootgrowth}. 	
				As in the proof of Lemma \ref{lem torsionareclose}, we may obtain distortion estimates $\mu^*(p, m)  \sim  \mu^*(p', m)$ for $* \in \{s, c, u\}$ and every $m$ such that $\sup_{i = 0}^{m}d(f^i(p), f^i(p')) \leq 1$.
			 Then by \eqref{eq almostconversative}, it is straightforward to see that  
				$\rho'    \lesssim C_0^{c(\esmall)}\rho$.
				Analogous argument gives 
				$\flu^s(p', \rho') \lesssim  C_0^{c(\esmall)}  \flu^s(p, \rho), \ 	\flu^u( p', \rho' ) \lesssim  C_0^{c(\esmall)}  \flu^u(p, \rho)$.
				By the above argument, we have already shown that 
				$| \log (\rho' / \rho) | \lesssim c(\esmall) \log C_0$.  
				Then by Lemma \ref{lem torsionareclose}, we have 
				\aryst
				 	| \tor^s(p, \rho) - \tor^s(p', \rho') | \leq | \tor^s(p, \rho) - \tor^s(p', \rho) | + 	| \tor^s(p', \rho) - \tor^s(p', \rho') |  \lesssim  \langle \log C_0 \rangle C_0^{c(\esmall)}  \flu^s(p, \rho) / \rho .
				\earyst
				A parallel estimate as above holds for $\tor^u$ and $\flu^u$ as well.
				Then
				\begin{align*}
	 			&| \tor^s(p', \rho') - \tor^u(p', \rho') | 
				\leq	| \tor^s(p, \rho) - \tor^u(p, \rho) | + 	| \tor^s(p, \rho) - \tor^s(p', \rho') |  + 	| \tor^u(p, \rho) - \tor^u(p', \rho') |  \\ 
				\leq& 	| \tor^s(p, \rho) - \tor^u(p, \rho) |   +   \langle \log C_0 \rangle C_0^{c(\esmall)}  \max(\flu^s(p, \rho),  \flu^u(p, \rho)) / \rho.
				\end{align*}
				The proof follows from combining the above estimates.
			\end{proof}

					\section{Dynamical wave-packet transform: Computations}  \label{app DWPTcomputations}

					\begin{proof}[Proof of Lemma  \eqref{lem intPhiadwdz}] 				 
	Note that inequality \eqref{eq integralphiaupperbound} follows immediately from		\eqref{eq integralphiaupperbound0} and
						\aryst
						 \widehat\Phi_{a, \ell}(\eta, w, z \mid \tw, \tz)  \lesssim  \widehat\Phi_{a, \ell/2}(\eta, \tw, \tz \mid w, z),
						\earyst
						which follows from Lemma \ref{lem adaptedscaleistame}.
						
 				Now we prove \eqref{eq integralphiaupperbound2}.
 				Denote $c = 1/20$. We have $\beta_1 > 1/2 + c + 1/4$. 
						First we note that restricted to $z$ with $\max(|z - \hz|, |z - \tz|) < \langle  \eta \rangle^{-1/2 - c}$, we can apply Lemma \ref{lem adaptedscaleistame} to see that for any   $K \geq 1$, under $\langle r_{a}(\eta, \tw, \tz)^{-1}	( w - \tw ) \rangle \in [2^{K-1}, 2^{K})$ 
						\aryst
					 \widehat\Phi_{a, \ell}(\eta, w, z \mid \tw, \tz) \lesssim  2^{- K \ell  /2 }\widehat\Phi_{a, \ell/2}(\eta, \tw,  \tz \mid w, z), \quad   \widehat\Phi_{a, \ell}(\eta, w, z \mid \hw, \hz) \lesssim   2^{c K \ell }\widehat\Phi_{a, \ell}(\eta, \tw,  \tz \mid \hw, \hz).
						\earyst
					Sum up the contributions for all  $K \geq 1$, we can bound the restriction of integral  to such $z$ by
					\aryst
			 	\int \widehat\Phi_a(\eta, \tw,  \tz \mid w, z) \widehat\Phi_a(\eta, \tw,  \tz \mid \hw, \hz) dw dz \leq C(\ell)    \langle \langle \eta \rangle^{\beta_1}(\hz - \tz) \rangle^{-\ell} \langle  r_{a}(\eta, \tw, \tz)^{-1}(\hw - \tw) \rangle^{-\ell}. 
					\earyst
					On the other hand, we may use the bound $ \langle \eta \rangle^{-1/2 - c} \leq r_a(\eta, w, z) \leq \langle \eta \rangle^{-1/2 + c}$ from  \eqref{eq adscalesquarerootgrowth} to show that  the restriction of integral to the complement is bounded by
						\aryst
						&& C \int_{\max(|z - \hz|, |z - \tz|) \geq \langle  \eta \rangle^{-1/2 - c}} \langle \eta \rangle^{\beta_1} \langle \langle \eta \rangle^{\beta_1} (\tz - z) \rangle^{- 2 \ell }     \langle \langle \eta \rangle^{\beta_1} (\hz - z) \rangle^{- 2 \ell } \cdot  \langle \eta \rangle^{ 1 + 2c}   \\
						&& \langle \langle \eta \rangle^{1/2 - c} (\tw - w) \rangle^{- 2\ell}  \langle \langle \eta \rangle^{1/2 - c} (\hw - w) \rangle^{- 2\ell} dw dz  \\
						&\leq& C(\ell) \langle \eta \rangle^{4c- \ell / 4} \langle \eta^{\beta_1}(\hz - \tz) \rangle^{-\ell} \langle \langle \eta \rangle^{1/2 - c} (\hw - \tw) \rangle^{-\ell} \\
						&\leq& C(\ell) \langle \eta \rangle^{4c +  2 c \ell  -  \ell/4} \langle \eta^{\beta_1}(\hz - \tz) \rangle^{-  \ell} \langle r_a(\eta, \tw, \tz)^{-1} (\hw - \tw) \rangle^{- \ell}.
						\earyst
						The proof is completed by putting together the above estimates.  
					\end{proof}

					\begin{proof}[Proof of Lemma \ref{lem decompositiononthereal}]

					Let $\cQ = \{ Q \}$ be a partition of $\R$ into intervals, separated by points $\{ \pm  n^{1/(1 - \beta_1)} \}_{n \in \Z_+}$.  There is an absolute constant $D > 1$ such that 
					we have $| Q | / \langle \eta_{Q}\rangle^{\beta_1} \in (D^{-1}, D)$ for every $Q \in \cQ$ where $\eta_{Q}$ denotes an endpoint of $Q$ with the maximal absolute value. 
						Clearly,  there is a partition of unity by non-negative  smooth functions $\sum_{Q \in \cQ} \chi_Q = 1$						such that for each $Q \in \cQ$ we have
						\ary \label{eq supportofchiQ}
						| \partial^{l} \chi_Q | \leq C(l) |Q|^{-l}, l \geq 0 \ \mbox{ and }  \
						\supp(\chi_Q) \subset \widehat{Q}
						\eary
						where the set $\widehat{Q}$ is the union of $Q$ and its adjacent intervals.  
						 
						Let $h_Q$ be a compactly supported $C^\infty$ function
						such that 
						\ary \label{eq supportofhQ}
						\supp(h_Q) \subset B(0,  (8 D)^{-1} |Q|^{-1}) \subset  B(0,   \langle \eta_{Q} \rangle^{- \beta_1}/8).
						\eary
						We may also assume that  $\| h_Q \|_{L^2} = \| \widehat h_Q \|_{L^2} =  |Q|^{1/2}$ and  there is $C = C(\delta) > 0$  with
						\ary
						\label{eq upperlowerboundforwidehathQ}
						C < \inf_{\eta' \in \widehat{Q}} | \widehat h_Q(\eta')| < C^{-1}.
						\eary
						We may require $\widetilde h_Q(x) = e^{- i \eta_Q x} h_Q(x)$ satisfies $| \partial^{l} 	\widetilde h_Q(x)  | \leq C(l, k) | Q |^{l  + 1}	\langle |Q| x \rangle^{- k}$ for any $l, k \geq 0$.

			   	      By \eqref{eq supportofchiQ} and \eqref{eq upperlowerboundforwidehathQ},   we may define $q_Q$ by $ \widehat{q_Q}(\eta') =  \chi_{Q}(\eta') / \widehat h_Q(\eta')$, $\eta' \in \R$.  By construction, we see that $ \chi_Q = \widehat{ q_Q * h_Q }$. Moreover, we may require $ \widetilde q_Q(x) = e^{- i \eta_Q x} q_Q(x)$ satisfy
					   \ary \label{eq smoothnessofwidetildeqQ}
					| \partial^{l} \widetilde q_Q(x)  | \leq C(l, k) | Q |^{l  + 1}	\langle |Q| x \rangle^{- k}, \ l, k \geq 0.
					   \eary

						We define
						\aryst
				 	\rho_{\parallel}(\eta, z) = \sum_{Q \in \cQ}  |Q|^{-1/2}  1_{Q}(\eta)    q_Q(z), \quad 				  \rho_{\parallel}^{\dagger}(\eta, z) = \sum_{Q \in \cQ} |Q|^{-1/2}  1_{Q}(\eta) h_Q(z).
						\earyst
						Then we clearly have that  $	\| \rho_{\parallel} \|_{L^\infty L^2}, \|  \rho_{\parallel}^{\dagger} \|_{L^\infty L^2}   < \infty$. 		By construction and \eqref{eq supportofhQ}, we see that \eqref{eq supporth} holds.
						Moreover, by \eqref{eq smoothnessofwidetildeqQ}, we see that 	  \eqref{lem partialtildeh} also holds.  Note that we also have
						\aryst
						( \rho_{\parallel}(\eta, \cdot) * \rho^{\dagger}_{\parallel}(\eta, \cdot))(z) = \sum_{Q \in \cQ} |Q|^{-1} 1_{Q}(\eta) \chi_{Q}(z).
						\earyst

						Define $V$ and $V^{\dagger}$ as in \eqref{eq VVdagger}.
						Take an arbitrary $u \in C^{\infty}_0(\R)$. We have
						\begin{align*}
 				& 	V^{\dagger} V u(\tz) 	=   \sum_{Q}  |Q|  \int ( V u(\eta_Q, \cdot) *  \rho_{\parallel}^{\dagger}(\eta_Q, \cdot) ) (\tz) dz \sum_{Q}  |Q|  \int ( u * \rho_{\parallel}(\eta_Q, \cdot) *  \rho_{\parallel}^{\dagger}(\eta_Q, \cdot) ) (\tz) dz    \\
						 = &    \sum_{Q}    \int  \hat u(\eta')  \chi_{Q}(\eta')    e^{i \eta' \tz} d\eta' 
						 =    \int  \hat u(\eta') e^{i \eta' \tz} d\eta' = u(\tz).
						\end{align*}
						This verifies the identity in \eqref{eq VVdagger}.

						\detail{
						We have 
						\aryst
						V^* v(\tz) = \int \overline{\rho_{\parallel}(\eta, z - \tz)} v(\eta, z) d\eta dz.
						\earyst
						
						We have 
						\aryst
						(V^{\dagger})^*u(\eta, z) = \int \overline{\rho^{\dagger}_{\parallel}(\eta, \tz - z) } u(\tz) d\tz.
						\earyst
						
						}
						
						To prove the boundedness of $V$, we note that the kernel $\cK_{V^* V}$ of $V^* V$ satisfies
						\aryst
					|	\cK_{V^*V}(\tz \mid \hz) | = | \int \overline{ \rho_{\parallel}(\eta, z - \tz) } \rho_{\parallel}(\eta, z - \hz) d\eta dz | \leq C(\ell) \int \langle \langle \eta \rangle^{\beta_1}(\tz - \hz) \rangle^{- \ell}  d\eta.
						\earyst			
				 	To prove the boundedness of $V^{\dagger}$, we note that the kernel $\cK_{(V^{\dagger})^* V^{\dagger}}$ of $(V^{\dagger})^* V^{\dagger}$ satisfies
						\aryst
						|	\cK_{(V^{\dagger})^* V^{\dagger}}(\eta, z \mid \eta', z') | = | \int \overline{ \rho^{\dagger}_{\parallel}(\eta, \tz - z) } \rho^{\dagger}_{\parallel}(\eta', \tz - z')  d\tz | \leq C(\ell)  \langle \eta_{min}^{\beta_1}(z - z') \rangle^{- \ell}    \langle  \eta_{max}^{- \beta_1}(\eta - \eta') \rangle^{- \ell}
						\earyst
						where $\eta_{min} = \min(\langle \eta \rangle, \langle \eta' \rangle)$ and $\eta_{max} = \max(\langle \eta \rangle, \langle \eta' \rangle)$. The boundedness then follows from Schur's test.
					\end{proof}

				\end{document}